%% file: ms.tex
\documentclass[11pt,a4paper]{amsart}
\usepackage{amsmath,amsfonts,amssymb,amsthm}
\usepackage{graphicx}
\usepackage{color}
\usepackage{url}
\usepackage{hyperref}
\usepackage{tikz}
\usetikzlibrary{3d}

\definecolor{lightgrey}{rgb}{0.5,0.5,0.5}
\definecolor{mediumgreen}{rgb}{0.0,0.7,0.0}
\definecolor{darkgreen}{rgb}{0.0,0.35,0.0}
\definecolor{darkblue}{rgb}{0.0,0.0,0.8}

\newcommand{\darkblue}[1]{{\color{blue}#1}}

\newcommand{\diag}{\mathrm{diag}}

\newcommand{\meas}{\mathrm{meas}}
\newcommand{\range}{\mathrm{range}}

\newcommand{\opcurl}{\mathop{\textnormal{\bf curl}}}
\newcommand{\opdiv}{\mathop{\textnormal{div}}}

\newcommand{\cardinality}{\#}
\newcommand{\complexi}{\textnormal{i}}

\newcommand{\productspace}{\prod}

\renewcommand{\Re}{\textnormal{Re}}
\renewcommand{\Im}{\textnormal{Im}}

\newcommand{\varphiv}{\boldsymbol{\varphi}}

\newcommand{\jv}{{\bf j}}
\newcommand{\nv}{{\bf n}}

\newcommand{\uv}{{\bf u}}
\newcommand{\vv}{{\bf v}}
\newcommand{\wv}{{\bf w}}

\newcommand{\Hv}{{\bf H}}
\newcommand{\Lv}{{\bf L}}

\newcommand{\re}{\textnormal{re}}
\newcommand{\im}{\textnormal{im}}

\newcommand{\normal}{\nu}

\renewcommand{\top}{\mathsf{T}}
\newcommand{\herm}{\mathsf{H}}

\theoremstyle{plain} 
\newtheorem{lemma}{Lemma}[section]
\newtheorem{theorem}[lemma]{Theorem}
\newtheorem{proposition}[lemma]{Proposition}
\newtheorem{corollary}[lemma]{Corollary}

\theoremstyle{definition}
\newtheorem{definition}[lemma]{Definition}
\newtheorem{remark}[lemma]{Remark}
\newtheorem{example}[lemma]{Example}

\newtheorem{assumption}{Assumption}

\newtheorem{assumptb}{Assumption}

\newtheorem{assumptc}{Assumption}

\numberwithin{equation}{section}

\begin{document}

\title[A unified theory of non-overlapping Robin-Schwarz methods]
{A unified theory of\\ non-overlapping Robin-Schwarz methods --- \\ continuous and discrete, including cross points}

\author[C.~Pechstein]{Clemens Pechstein$^*$}

\thanks{$^*$ Dassault Syst\`emes Austria GmbH, Wienerbergerstr.~51, 1120 Wien, Austria, \texttt{clemens.pechstein@3ds.com}}


\begin{abstract}
  Non-overlapping Schwarz methods with generalized Robin transmission conditions
	were originally introduced by B.~Despr\'es for time-harmonic wave propagation problems
	and have largely developed over the past thirty years.
	The aim of the paper is to provide both a review of the available formulations and methods
	as well as a consistent theory applicable to more general cases than studied until to date.
	An abstract variational framework is provided reformulating the original problem by the well-known form involving a scattering operator
	and an interface exchange operator, and the equivalence between the formulations is discussed thoroughly.
	The framework applies to a series of wave propagation problems throughout the de Rham complex,
	such as the scalar Helmholtz equation, Maxwell's equations,
	a dual formulation of the Helmholtz equation in H(div), as well as any conforming finite element discretization thereof,
	and it applies also to coercive problems.
	Three convergence results are shown. The first one (using compactness) and the second one (based on absorbtion)
	generalize Despr\'es' early findings and apply as well to the FETI-2LM formulation
	(a discrete method introduced by de La Bourdonnaye, Farhat, Macedo, Magoul\'es, and Roux).
	The third result, oriented on the work by Collino, Ghanemi, and Joly, establishes a convergence rate
	and covers cases with cross points, while not requiring any regularity of the solution.
	The key ingredient is a global interface exchange operator, proposed originally by X.~Claeys and further developed by Claeys and Parolin,
	here worked out in full generality.
  The third type of convergence theory is applicable at the discrete level as well, where the exchange operator is allowed to be even local.
	The resulting scheme can be viewed as a generalization of the 2-Lagrange-multiplier method introduced by S.~Loisel,
	and connections are drawn to another technique proposed by Gander and Santugini.
\end{abstract}

\maketitle

\section{Introduction}

Domain decomposition (DD) methods \cite{SmithBjorstadGropp:Book,QuarteroniValli:Book,ToselliWidlund:Book,Mathew:book,DoleanJolivetNataf:Book}
can be classified according to their \emph{formulation complexity} (see Table~\ref{tab:DDClass}).
In the simplest case, such as for the overlapping Schwarz method,
one has a \emph{preconditioner} (based on domain decomposition) for the original standard finite element system.
The Neumann-Neumann and the BDDC methods are preconditioners for the Schur complement formulation, eliminating interior degrees of freedom (dofs)
that are not associated with the interface.
In dual iterative substructuring, such as for classical FETI and FETI-DP methods,
the original problem is reformulated even more, involving function spaces that allow discontinuities across subdomain interfaces.
At such a stage, the domain decomposition plays an essential role in the formulation, even before any preconditioning.

\begin{table}
\begin{center}
\begin{tabular}{cccc}
  overlapping      & primal                   & dual                     & non-overlapping \\
  additive Schwarz & iterative substructuring & iterative substructuring & Robin-Schwarz \\
\hline
	$A u = f$                & $S u_\Gamma = g$            & $F \lambda = d$           & $(I - \mathcal{X} \underline{S}) \underline{\lambda} = \underline d$
\end{tabular}

\smallskip

\caption{\label{tab:DDClass}
  Examples of some DD methods and their underlying (re-)formulation. Increasing formulation complexity when moving to the right.
}
\end{center}
\end{table}

Schwarz methods with Robin transmission conditions, often found under the name \emph{optimized Schwarz methods},
are classically formulated as an iterative process involving spaces with the same discontinuity property.
The transmission conditions, making the solution and its associated flux continuous, are only reached at convergence.
These \emph{Robin-Schwarz} methods are among the most successful DD methods for wave propagation problems,
and it is a major goal of this article to provide a fundamental understanding of
the underlying formulation, the iterative process, and the convergence theory---in the continuous as well as in the discrete case.
There will be a certain emphasis on wave propagation problems, but the framework includes the coercive (positive definite) case.

\subsection{History and literature review}

The first non-overlapping Schwarz methods with Robin transmission conditions were independently proposed and analyzed around 1990 by
Pierre-Louis Lions \cite{Lions:DD03} for the Laplace equation
and by Bruno Despr\'es for the Helmholtz equation \cite{Despres:1990a,Despres:1991a,Despres:PhD}
(early results on the time-harmonic Maxwell equations can be found in \cite{Despres:PhD,DespresJolyRoberts:1992a}).
There, the original variational problem is reformulated using a decomposition of the domain into non-overlapping subdomains,
where the coupling across the interfaces happens via impedance traces (classical Robin traces,
instead of Dirichlet and Neumann traces, see \cite{ToselliWidlund:Book,DoleanJolivetNataf:Book} and references therein).
The proposed scheme can be seen as a fixed point iteration and was shown to converge.
Lions' and Despr\'es' original proofs both use compactness arguments and energy estimates,
in the Helmholtz case based on the novel concept of \emph{pseudo-energy}.
Another milestone was the classical paper by Collino, Ghanemi, and Joly \cite{CollinoGhanemiJoly:2000a} from 2000,
proving that
(i) the damped Schwarz iteration converges and
(ii) if there are no \emph{junctions} (i.e., any interface between two subdomains is either a closed manifold or empty)
and if special \emph{impedance operators} are used (leading to generalized Robin transmission conditions),
then the convergence is \emph{geometric}, i.e., the error in the $k$-th iteration can be bounded by $\rho^k$ compared to the initial error,
with a convergence rate $\rho < 1$.
Recently, this kind of convergence result was investigated in more depth in
\cite{CollinoJolyLecouvezStupfel:2014a,CollinoJolyLecouvez:2020a,Parolin:PhD}
using \emph{non-local} impedance operators based on integral operators with singular kernels,
see also the early paper \cite{CollinoDelbueJolyPiacentini:1997a}.
For an early work using a local but non-trivial impedance operator based on the surface Laplace-Beltrami operator see \cite{PiacentiniRosa:1998a}.

While in \cite{Lions:DD03,Despres:PhD,CollinoGhanemiJoly:2000a}, the method was analyzed on the continuous level 
involving Sobolev spaces,
at the very end of the 20th century
a huge development started around computational methods of Schwarz type using finite elements.
This development was greatly influenced by the finite element tearing and interconnecting (FETI) method, introduced by
Farhat and
Roux
for static structural mechanics \cite{FarhatRoux:1990a,FarhatRoux:1991a}.
Two early approaches for the Helmholtz equation are the FETI-H method \cite{FarhatMacedoTezaur:DD11,FarhatMacedoLesoinne:2000a,FarhatMacedoLesoinneRouxMagoulesDeLaBourdonnaie:2000a},
introduced by Farhat, Macedo, Tezaur, and Lesoinne,
and the FETI-2LM method \cite{Bourdonnaye:DD10,FarhatMacedoMagoulesRoux:USNCCM,FarhatMacedoLesoinneRouxMagoulesDeLaBourdonnaie:2000a}
by de La Bourdonnaye, Farhat, Macedo, Magoul\`es, and Roux.
Both methods use Robin transmission conditions on the discrete level as well as Lagrange multipliers in addition to the separated subdomain degrees of freedom (dof).
As the article at hand will demonstrate, the formulation behind FETI-2LM can be seen as one
out of many possible discrete counterparts of Despr\'es' original method, and it has a broader spectrum of applicability than the FETI-H method.
In the early works on FETI-H and FETI-2LM, the focus lay rather on the efficient parallel computation than on the analysis.
It is worth mentioning that the reformulated problem was typically solved using Krylov acceleration.
In the long run, due to the complex symmetric (but non-Hermitian) structure of the system matrix, GMRES became the iterative method of choice.
Using carefully chosen ``coarse'' modes (typically plane waves on the subdomains, inspired from a careful spectral analysis of the two-subdomain case),
it was demonstrated numerically that two-level schemes can lead to rather fast convergence.
Although such two-level approaches have been a very important topic until today (as of 2021),
they are not pursued in the paper at hand.

In 1994, Nataf, Rogier, and de Sturler \cite{NatafRogierDeSturler:1994a}
showed for an overlapping Schwarz method that it is possible to construct \emph{optimal}
transmission conditions involving non-local operators (e.g., Dirichlet-to-Neumann maps)
such that a Krylov method would converge after $N$ steps,
where $N$ subdomains cover the original domain in a strip-like fashion.
Starting with Caroline Japhet \cite{Japhet:1998a},
a huge development began following the paradigm of approximating the optimal non-local operators by parametrized local ones and then
\emph{optimizing} the parameters with respect to the rate of convergence, typically on the continuous level, often the PDE-level, using Fourier analysis for the case of two subdomains sharing a common face (often two half-spaces).
For a comprehensive survey on \emph{optimized Schwarz methods} see \cite{Gander:2006a,GanderZhang:2019a},
for the case of wave propagation see in particular
\cite{DoleanGanderLanteriLeePeng:2015,DoleanGanderGerardoGiorda:2009a,DoleanGanderVenerousZhang:DD23,
  GanderHalpernMagoules:2007a,GanderMagoulesNataf:2002a,AlonsoRodriguezGerardoGiorda:2006a}.

Around 2005, more work appeared on non-overlapping (optimized) Robin-Schwarz methods in the discrete case,
with special emphasis on \emph{cross points}, i.e., points shared by more than two subdomains.
Two approaches were inspired from the FETI-DP method \cite[Ch.~6]{ToselliWidlund:Book}.
The method by Bendali and Boubendir \cite{BendaliBoubendir:2005a} follows Despr\'es' method,
but it maintains continuity of all dofs shared by more than two subdomains
and can thus be regarded as a dual-primal version of FETI-2LM.
The FETI-DPH method \cite{FarhatAveryTezaurLi:2005a} is a generalization of the FETI-H method, keeping continuity of certain dofs, e.g., at subdomain vertices,
and was demonstrated numerically to be very efficient with regard to the problem size, the number of subdomains, and the wave number.

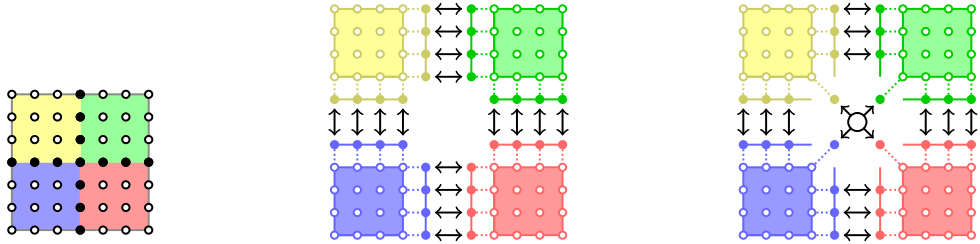
\begin{figure}
\begin{center}
  \begin{tikzpicture}
		\pgftransformscale{0.6}
		
	  \definecolor{dred}{rgb}{0.8, 0.0, 0.0}
		\definecolor{mdred}{rgb}{1.0, 0.4, 0.4}
		\definecolor{lred}{rgb}{1.0, 0.8, 0.8}
		\definecolor{mlred}{rgb}{1.0, 0.6, 0.6}
		
		\definecolor{dgreen}{rgb}{0.0, 0.4, 0.0}
		\definecolor{mdgreen}{rgb}{0.0, 0.8, 0.0}
		\definecolor{lgreen}{rgb}{0.8, 1.0, 0.8}
		\definecolor{mlgreen}{rgb}{0.6, 1.0, 0.6}
		
		\definecolor{dblue}{rgb}{0.0, 0.0, 0.8}
		\definecolor{mdblue}{rgb}{0.4, 0.4, 1.0}
		\definecolor{lblue}{rgb}{0.8, 0.8, 1.0}
		\definecolor{mlblue}{rgb}{0.6, 0.6, 1.0}
		
		\definecolor{dyellow}{rgb}{0.4, 0.4, 0.0}
		\definecolor{mdyellow}{rgb}{0.8, 0.8, 0.4}
		\definecolor{lyellow}{rgb}{1.0, 1.0, 0.8}
		\definecolor{mlyellow}{rgb}{1.0, 1.0, 0.6}
	
		\definecolor{dgrey}{rgb}{0.4, 0.4, 0.4}
		\definecolor{mdgrey}{rgb}{0.55, 0.55, 0.55}
		\definecolor{lgrey}{rgb}{0.9, 0.9, 0.9}
		\definecolor{mlgrey}{rgb}{0.75, 0.75, 0.75}
	
	  \node (A0) at (1,1) {};
	  \node (A1) at (4,1) {};
	  \node (A2) at (4,4) {};
	  \node (A3) at (1,4) {};
		\node (A01) at (2.5,1) {};
		\node (A12) at (4,2.5) {};
		\node (A23) at (2.5,4) {};
		\node (A03) at (1,2.5) {};
		\node (AM)  at (2.5,2.5) {};

		\draw[line width=0pt,color=mlblue,fill=mlblue] (A0.center)--(A01.center)--(AM.center)--(A03.center)--(A0.center);
		\draw[line width=0pt,color=mlred,fill=mlred] (A01.center)--(A1.center)--(A12.center)--(AM.center)--(A01.center);
		\draw[line width=0pt,color=mlgreen,fill=mlgreen] (AM.center)--(A12.center)--(A2.center)--(A23.center)--(AM.center);
		\draw[line width=0pt,color=mlyellow,fill=mlyellow] (A03.center)--(AM.center)--(A23.center)--(A3.center)--(A03.center);
		
		\draw[line width=0.75pt,color=mdgrey] (A0.center)--(A1.center)--(A2.center)--(A3.center)--(A0.center);

		\foreach \x in {1,1.5,...,4}
		{
		  \foreach \y in {1,1.5,...,4}
		  {
		    \draw[line width=0.74pt,color=black,fill=white] (\x,\y) circle (0.08);
      }
		}
		
		\foreach \x in {1,1.5,...,4}
		{
		  \draw[line width=0.74pt,color=black,fill=black] (\x,2.5) circle (0.08);
		  \draw[line width=0.74pt,color=black,fill=black] (2.5,\x) circle (0.08);
    }
		
  \end{tikzpicture}
  \hspace{2cm}
  \begin{tikzpicture}
		\pgftransformscale{0.6}
		
	  \definecolor{dred}{rgb}{0.8, 0.0, 0.0}
		\definecolor{mdred}{rgb}{1.0, 0.4, 0.4}
		\definecolor{lred}{rgb}{1.0, 0.8, 0.8}
		\definecolor{mlred}{rgb}{1.0, 0.6, 0.6}
		
		\definecolor{dgreen}{rgb}{0.0, 0.4, 0.0}
		\definecolor{mdgreen}{rgb}{0.0, 0.8, 0.0}
		\definecolor{lgreen}{rgb}{0.8, 1.0, 0.8}
		\definecolor{mlgreen}{rgb}{0.6, 1.0, 0.6}
		
		\definecolor{dblue}{rgb}{0.0, 0.0, 0.8}
		\definecolor{mdblue}{rgb}{0.4, 0.4, 1.0}
		\definecolor{lblue}{rgb}{0.8, 0.8, 1.0}
		\definecolor{mlblue}{rgb}{0.6, 0.6, 1.0}
		
		\definecolor{dyellow}{rgb}{0.4, 0.4, 0.0}
		\definecolor{mdyellow}{rgb}{0.8, 0.8, 0.4}
		\definecolor{lyellow}{rgb}{1.0, 1.0, 0.8}
		\definecolor{mlyellow}{rgb}{1.0, 1.0, 0.6}
		
		\definecolor{dgrey}{rgb}{0.4, 0.4, 0.4}
		\definecolor{mdgrey}{rgb}{0.55, 0.55, 0.55}
		\definecolor{lgrey}{rgb}{0.9, 0.9, 0.9}
		\definecolor{mlgrey}{rgb}{0.75, 0.75, 0.75}
	
		\draw[line width=0.75pt,color=mdblue,fill=mlblue] (0,0)--(1.5,0)--(1.5,1.5)--(0,1.5)--(0,0)--(1.5,0);
		\draw[line width=0.75pt,color=mdyellow,fill=mlyellow] (0,3.5)--(1.5,3.5)--(1.5,5)--(0,5)--(0,3.5)--(1.5,3.5);
		\draw[line width=0.75pt,color=mdred,fill=mlred] (3.5,0)--(3.5,1.5)--(5,1.5)--(5,0)--(3.5,0)--(3.5,1.5);
		\draw[line width=0.75pt,color=mdgreen,fill=mlgreen] (3.5,3.5)--(3.5,5)--(5,5)--(5,3.5)--(3.5,3.5)--(3.5,5);

		\draw[line width=0.75pt,color=mdblue] (0,2)--(1.5,2);
		\draw[line width=0.75pt,color=mdblue] (2,0)--(2,1.5);

		\draw[line width=0.75pt,color=mdyellow] (0,3)--(1.5,3);
		\draw[line width=0.75pt,color=mdyellow] (2,3.5)--(2,5);
		
		\draw[line width=0.75pt,color=mdred] (3,0)--(3,1.5);
		\draw[line width=0.75pt,color=mdred] (3.5,2)--(5,2);

		\draw[line width=0.75pt,color=mdgreen] (3,3.5)--(3,5);
		\draw[line width=0.75pt,color=mdgreen] (3.5,3)--(5,3);
		
		\foreach \x in {0,0.5,...,1.5}
		{
		  \draw[line width=0.74pt,color=mdblue,fill=mdblue] (\x,2) circle (0.08);
		  \draw[line width=0.74pt,color=mdblue,fill=mdblue] (2,\x) circle (0.08);

		  \draw[line width=0.74pt,color=mdred,fill=mdred] (3.5+\x,2) circle (0.08);
		  \draw[line width=0.74pt,color=mdred,fill=mdred] (3,\x) circle (0.08);

		  \draw[line width=0.74pt,color=mdyellow,fill=mdyellow] (\x,3) circle (0.08);
		  \draw[line width=0.74pt,color=mdyellow,fill=mdyellow] (2,3.5+\x) circle (0.08);

		  \draw[line width=0.74pt,color=mdgreen,fill=mdgreen] (3.5+\x,3) circle (0.08);
		  \draw[line width=0.74pt,color=mdgreen,fill=mdgreen] (3,3.5+\x) circle (0.08);

			\draw[line width=0.74pt,color=mdblue,fill=mdblue,densely dotted] (1.5,\x)--(2,\x);
			\draw[line width=0.74pt,color=mdblue,fill=mdblue,densely dotted] (\x,2)--(\x,1.5);

			\draw[line width=0.74pt,color=mdred,fill=mdred,densely dotted] (3,\x)--(3.5,\x);
			\draw[line width=0.74pt,color=mdred,fill=mdred,densely dotted] (3.5+\x,2)--(3.5+\x,1.5);

			\draw[line width=0.74pt,color=mdgreen,fill=mdgreen,densely dotted] (3,3.5+\x)--(3.5,3.5+\x);
			\draw[line width=0.74pt,color=mdgreen,fill=mdgreen,densely dotted] (3.5+\x,3)--(3.5+\x,3.5);

			\draw[line width=0.74pt,color=mdyellow,fill=mdyellow,densely dotted] (1.5,3.5+\x)--(2,3.5+\x);
			\draw[line width=0.74pt,color=mdyellow,fill=mdyellow,densely dotted] (\x,3)--(\x,3.5);
			
			\draw[line width=0.74pt,color=black,<->] (2.2,\x)--(2.8,\x);

			\draw[line width=0.74pt,color=black,<->] (\x,2.2)--(\x,2.8);
			
			\draw[line width=0.74pt,color=black,<->] (2.2,3.5+\x)--(2.8,3.5+\x);

			\draw[line width=0.74pt,color=black,<->] (3.5+\x,2.2)--(3.5+\x,2.8);

    }
		
		\foreach \x in {0,0.5,...,1.5}
		{
		  \foreach \y in {0,0.5,...,1.5}
		  {
		    \draw[line width=0.74pt,color=mdblue,fill=white] (\x,\y) circle (0.08);
		    \draw[line width=0.74pt,color=mdred,fill=white] (3.5+\x,\y) circle (0.08);
		    \draw[line width=0.74pt,color=mdyellow,fill=white] (\x,3.5+\y) circle (0.08);
		    \draw[line width=0.74pt,color=mdgreen,fill=white] (3.5+\x,3.5+\y) circle (0.08);
      }
		}
		
  \end{tikzpicture}
	\hspace{2cm}
  \begin{tikzpicture}
		\pgftransformscale{0.6}
		
	  \definecolor{dred}{rgb}{0.8, 0.0, 0.0}
		\definecolor{mdred}{rgb}{1.0, 0.4, 0.4}
		\definecolor{lred}{rgb}{1.0, 0.8, 0.8}
		\definecolor{mlred}{rgb}{1.0, 0.6, 0.6}
		
		\definecolor{dgreen}{rgb}{0.0, 0.4, 0.0}
		\definecolor{mdgreen}{rgb}{0.0, 0.8, 0.0}
		\definecolor{lgreen}{rgb}{0.8, 1.0, 0.8}
		\definecolor{mlgreen}{rgb}{0.6, 1.0, 0.6}
		
		\definecolor{dblue}{rgb}{0.0, 0.0, 0.8}
		\definecolor{mdblue}{rgb}{0.4, 0.4, 1.0}
		\definecolor{lblue}{rgb}{0.8, 0.8, 1.0}
		\definecolor{mlblue}{rgb}{0.6, 0.6, 1.0}
		
		\definecolor{dyellow}{rgb}{0.4, 0.4, 0.0}
		\definecolor{mdyellow}{rgb}{0.8, 0.8, 0.4}
		\definecolor{lyellow}{rgb}{1.0, 1.0, 0.8}
		\definecolor{mlyellow}{rgb}{1.0, 1.0, 0.6}
		
		\definecolor{dgrey}{rgb}{0.4, 0.4, 0.4}
		\definecolor{mdgrey}{rgb}{0.55, 0.55, 0.55}
		\definecolor{lgrey}{rgb}{0.9, 0.9, 0.9}
		\definecolor{mlgrey}{rgb}{0.75, 0.75, 0.75}
	
		\draw[line width=0.75pt,color=mdblue,fill=mlblue] (0,0)--(1.5,0)--(1.5,1.5)--(0,1.5)--(0,0)--(1.5,0);
		\draw[line width=0.75pt,color=mdyellow,fill=mlyellow] (0,3.5)--(1.5,3.5)--(1.5,5)--(0,5)--(0,3.5)--(1.5,3.5);
		\draw[line width=0.75pt,color=mdred,fill=mlred] (3.5,0)--(3.5,1.5)--(5,1.5)--(5,0)--(3.5,0)--(3.5,1.5);
		\draw[line width=0.75pt,color=mdgreen,fill=mlgreen] (3.5,3.5)--(3.5,5)--(5,5)--(5,3.5)--(3.5,3.5)--(3.5,5);

		\draw[line width=0.75pt,color=mdblue] (0,2)--(1.5,2);
		\draw[line width=0.75pt,color=mdblue] (2,0)--(2,1.5);

		\draw[line width=0.75pt,color=mdyellow] (0,3)--(1.5,3);
		\draw[line width=0.75pt,color=mdyellow] (2,3.5)--(2,5);
		
		\draw[line width=0.75pt,color=mdred] (3,0)--(3,1.5);
		\draw[line width=0.75pt,color=mdred] (3.5,2)--(5,2);

		\draw[line width=0.75pt,color=mdgreen] (3,3.5)--(3,5);
		\draw[line width=0.75pt,color=mdgreen] (3.5,3)--(5,3);

		\draw[line width=0.74pt,color=mdblue,fill=mdblue,densely dotted] (1.5,1.5)--(2,2);
		\draw[line width=0.74pt,color=mdblue,fill=mdblue] (2,2) circle (0.08);

		\draw[line width=0.74pt,color=mdred,fill=mdred,densely dotted] (3.5,1.5)--(3,2);
		\draw[line width=0.74pt,color=mdred,fill=mdred] (3,2) circle (0.08);

		\draw[line width=0.74pt,color=mdyellow,fill=mdyellow,densely dotted] (1.5,3.5)--(2,3);
		\draw[line width=0.74pt,color=mdyellow,fill=mdyellow] (2,3) circle (0.08);

		\draw[line width=0.74pt,color=mdgreen,fill=mdgreen,densely dotted] (3.5,3.5)--(3,3);
		\draw[line width=0.74pt,color=mdgreen,fill=mdgreen] (3,3) circle (0.08);
		
		\foreach \x in {0,0.5,1}
		{
		  \draw[line width=0.74pt,color=mdblue,fill=mdblue] (\x,2) circle (0.08);
		  \draw[line width=0.74pt,color=mdblue,fill=mdblue] (2,\x) circle (0.08);

		  \draw[line width=0.74pt,color=mdred,fill=mdred] (4+\x,2) circle (0.08);
		  \draw[line width=0.74pt,color=mdred,fill=mdred] (3,\x) circle (0.08);

		  \draw[line width=0.74pt,color=mdyellow,fill=mdyellow] (\x,3) circle (0.08);
		  \draw[line width=0.74pt,color=mdyellow,fill=mdyellow] (2,4+\x) circle (0.08);

		  \draw[line width=0.74pt,color=mdgreen,fill=mdgreen] (4+\x,3) circle (0.08);
		  \draw[line width=0.74pt,color=mdgreen,fill=mdgreen] (3,4+\x) circle (0.08);

			\draw[line width=0.74pt,color=mdblue,fill=mdblue,densely dotted] (1.5,\x)--(2,\x);
			\draw[line width=0.74pt,color=mdblue,fill=mdblue,densely dotted] (\x,2)--(\x,1.5);

			\draw[line width=0.74pt,color=mdred,fill=mdred,densely dotted] (3,\x)--(3.5,\x);
			\draw[line width=0.74pt,color=mdred,fill=mdred,densely dotted] (4+\x,2)--(4+\x,1.5);

			\draw[line width=0.74pt,color=mdgreen,fill=mdgreen,densely dotted] (3,4+\x)--(3.5,4+\x);
			\draw[line width=0.74pt,color=mdgreen,fill=mdgreen,densely dotted] (4+\x,3)--(4+\x,3.5);

			\draw[line width=0.74pt,color=mdyellow,fill=mdyellow,densely dotted] (1.5,4+\x)--(2,4+\x);
			\draw[line width=0.74pt,color=mdyellow,fill=mdyellow,densely dotted] (\x,3)--(\x,3.5);
			
			\draw[line width=0.74pt,color=black,<->] (2.2,\x)--(2.8,\x);

			\draw[line width=0.74pt,color=black,<->] (\x,2.2)--(\x,2.8);
			
			\draw[line width=0.74pt,color=black,<->] (2.2,4+\x)--(2.8,4+\x);

			\draw[line width=0.74pt,color=black,<->] (4+\x,2.2)--(4+\x,2.8);

    }
		
		\foreach \x in {0,0.5,...,1.5}
		{
		  \foreach \y in {0,0.5,...,1.5}
		  {
		    \draw[line width=0.74pt,color=mdblue,fill=white] (\x,\y) circle (0.08);
		    \draw[line width=0.74pt,color=mdred,fill=white] (3.5+\x,\y) circle (0.08);
		    \draw[line width=0.74pt,color=mdyellow,fill=white] (\x,3.5+\y) circle (0.08);
		    \draw[line width=0.74pt,color=mdgreen,fill=white] (3.5+\x,3.5+\y) circle (0.08);
      }
		}

		\draw[line width=0.74pt,color=black,<->] (2.14,2.14)--(2.86,2.86);
		\draw[line width=0.74pt,color=black,<->] (2.14,2.86)--(2.86,2.14);
		\draw[line width=0.74pt,color=black,fill=white] (2.5,2.5) circle (0.2);
		
  \end{tikzpicture}  \caption{\label{fig:LagrIllustrCrossp}%
	  \emph{Left:} Example of subdomain decomposition:
		$\bullet$~global interface dofs,
		$\circ$~remaining dofs.
		The dof at the center is a cross point dof and shared by four subdomains.
		\emph{Middle/right:}
	  Illustration of Lagrange multiplier layout for FETI-2LM (middle) and Loisel's method (right):
		$\circ$~local subdomain dofs,
		$\bullet$~dofs on local trace space,
		$\rightarrow$~each arrow tip indicates one Lagrange multiplier.
	}
\end{center}
\end{figure}

In the FETI-2LM method \cite{Bourdonnaye:DD10,FarhatMacedoMagoulesRoux:USNCCM,FarhatMacedoLesoinneRouxMagoulesDeLaBourdonnaie:2000a},
each dof on a \emph{facet} (a subdomain interface of codimension one, shared by two subdomains) generates two Lagrange multipliers.
For a cross point dof in 2D shared by four subdomains, this means that due to the four facets, there are 8 Lagrange multipliers,
see Fig.~\ref{fig:LagrIllustrCrossp} (left, middle); details to be shown in Sect.~\ref{sect:discreteFacetSystems}.
A different paradigm was introduced by S\'ebastien Loisel \cite{Loisel:2013a}
(therein called \emph{2-Lagrange multiplier method} 
and worked out for a finite element discretization of the Laplace equation,
later for heterogeneous diffusion \cite{LoiselNguyenScheichl:2015a}).
For this method, the number of Lagrange multipliers associated with an original dof is equal to the number of sharing subdomains (in the case above, four instead of eight, see Fig.~\ref{fig:LagrIllustrCrossp} (right).
This is achieved using a projection operator for each group of separated subdomain dofs that simply averages the values using the reciprocal multiplicity as weights.
Such averaging operators have been used early on in substructuring methods,
e.g., the balancing Neumann-Neumann methods \cite{Mandel:1993a,MandelBrezina:1996a},
and play a principal role in FETI and BDDC methods with heterogeneous coefficients
(see \cite{ToselliWidlund:Book,Pechstein:FETIBook} and references therein).
For cross point dofs the FETI-2LM leads to redundancy (the Lagrange multipliers for the solution are not unique),
whereas Loisel's method is a non-redundant formulation.
Apart from this difference, both methods are oriented on Despr\'es' method and iterate on the Lagrange multipliers.
Gander and Kwok \cite{GanderKwok:2012a} investigated the choice of the Robin parameter at cross points in order
to obtain a convergence order that is comparable to the case without cross points.

From another perspective, discrete non-overlapping Robin-Schwarz methods for cross points were investigated by Gander and Santugini
\cite{GanderSantugini:2016a}
(therein for the finite element discretization of a positive definite problem).
The authors propose two variants: \emph{discrete optimized Schwarz with auxiliary variables} and \emph{complete communication}.
As we shall see in the paper at hand, the first variant follows the FETI-2LM paradigm,
whereas the second variant is closely related to Loisel's method.
As a main (but minor) difference,
FETI-2LM and Loisel's method iterate purely on the Lagrange multipliers,
whereas the methods in \cite{GanderSantugini:2016a}
iterate on the primal subdomain dofs or on both sets of variables.

For standard nodal $H^1$-conforming finite elements, a geometric cross point always leads to a dof shared by more than two subdomains.
However, this is not true for every discretization.
Monk, Sinwel, and Sch\"oberl \cite{MonkSinwelSchoeberl:2010a} consider a dual formulation of the Helmholtz equation set up in $H(\textnormal{div})$
discretized by Raviart-Thomas elements in 2D or N\'ed\'elec face elements in 3D.
As every dof is face-based, from a discrete perspective there are no ``cross points'',
to be more precise, no cross point dofs and no redundancy.
The authors of \cite{MonkSinwelSchoeberl:2010a} explore connections between the global finite element formulation, the ultra-weak variational formulation
(UWVF) introduced by Cessenant and Despr\'es \cite{CessenantDespres:1998a},
and a novel hybridized formulation (which can be interpreted as a hybrid discontinuous Galerkin (DG) scheme but is equivalent to the original formulation).
The latter technique was further investigated and extended to a discretization scheme for the time-harmonic Maxwell equations by
M.~Huber, A.~Pechstein (n\'ee Sinwel), and J.~Sch\"oberl \cite{HuberPechsteinSchoeberl:DD20}, see also Huber's doctoral thesis \cite{Huber:PhD}.

A further contribution from the engineering community for electromagnetic wave propagation is the
FETI-2$\lambda$ method proposed by Vouvakis \cite{Vouvakis:PhD}
(see also \cite{VouvakisLee:CopperMountain8,LeeVouvakisLee:2005a,VouvakisCendesLee:2006a,Vouvakis:2015a})
and further investigated numerically by Paraschos \cite{Paraschos:PhD}.
The basic scheme follows again the FETI-2LM paradigm but the authors focus on non-matching meshes.
Related FETI-type schemes involving (generalized) Robin interface conditions can also be found in
\cite{PengLee:2010a,RawatLee:2010a}.

Non-overlapping domain decompositions naturally involve \emph{broken} spaces, in particular, broken \emph{trace spaces}
that also appear in boundary integral equations \cite{McLean:Book,Steinbach:Book2008}.
Indeed, many of the techniques above were applied to integral equations or boundary element techniques.
A boundary element counterpart to FETI-H, applied both for the Helmholtz as well as the time-harmonic Maxwell equations,
was introduced by Windisch \cite{Windisch:PhD}.
Independently, \emph{local multi-trace methods} were introduced by Hiptmair and Jerez-Hanckes \cite{HiptmairJerezHanckes:2012a}
and \emph{global multi-trace methods} by Claeys and Hiptmair \cite{ClaeysHiptmair:2013a,ClaeysHiptmair:2012a},
see also \cite{Claeys:2015a,ClaeysHiptmairJerezHanckes:2013a,ClaeysHiptmairJerezHanckesPintarelli:2015a,HiptmairJerezHanckesLeePeng:DD21}.
These formulations involve layer potentials or boundary integral operators,
often make use of Cald\'eron identities, and can be put in to the framework of operator preconditioning \cite{Hiptmair:2006a}.
Furthermore, they use exchange operators between subdomain interfaces, in some cases, similar to those from Despr\'es' method.
A connection between Schwarz methods and local multi-trace formulations was pointed out in \cite{ClaeysDoleanGander:2019a}.

Motivated from the techniques of multi-trace formulations, Xavier Claeys \cite{Claeys:2021a}
recently suggested
an interface exchange operator that is completely different to the one used so far
which simply swaps pairs of traces between subdomains.
The novel exchange operator is non-local in the sense that it involves a projection step
where a function from the global multi-trace space is projected to the single-trace space.
Albeit this operator is computationally equivalent to solving a global, coercive (positive definite) problem,
the analysis in \cite{Claeys:2021a}
shows geometric convergence, even for the case of cross points,
and does not need any regularity assumptions anymore.
In a joint work by X.~Claeys and E.~Parolin \cite{ClaeysParolin:Preprint2020},
geometric convergence was also shown for a discretization of the Helmholtz equation,
where the rate of convergence is independent of the mesh parameter, again in presence of cross points.
Parolin's doctoral thesis \cite{Parolin:PhD} includes the case of Maxwell's equations as well.

Lastly, it should be mentioned that the original purpose of Robin boundary conditions is the approximation of the exterior PDE,
and there exist improved ways to do so, which leads to a more goal-oriented construction of impedance operators.
Recently, quite some work has appeared
\cite{BoubendirAntoineGeuzaine:2012a,DespresNicolopoulosThierry:Preprint2020,DespresNicolopoulosThierry:Preprint2021,
  ElBouajajiThierryAntoineGeuzaine:2015a,ModaveGeuzaineAntoine:2020a,ModaveRoyerGeuzaineAntoine:2020a,StupfelChanaud:2018a}
that use such \emph{generalized impedance boundary conditions} (GIBC),
\emph{high order absorbing boundary conditions} (HABC),
or \emph{high order transmission conditions} (HOTC)
for domain decomposition methods, in many cases also considering cross points.

\subsection{Purpose and structure of this work}

The paper at hand provides an abstract theoretical framework for non-overlapping Schwarz methods
with Robin transmission conditions. Whereas Parolin \cite{Parolin:PhD} has already presented a high amount of abstraction
by treating the Helmholtz and Maxwell equations in a common framework, the theory in here goes one step further and
works in general Hilbert spaces with wave equations in operator or matrix form.
Not only does this improve the generality of the theory substantially,
but makes visible the essential properties. Using a convenient and compact notation, the continuous and
the discrete case can be handled to a large extent uniformly.
Instead of using the PDE level or the variational level, the whole description is \emph{operator-based},
which is close to an algorithm-oriented matrix-based notation, but more precise and independent of any chosen bases.
Furthermore, connections are drawn between many existing variants of Robin-Schwarz,
in particular
Despr\'es' method,
the FETI-2LM method \cite{Bourdonnaye:DD10,FarhatMacedoMagoulesRoux:USNCCM,FarhatMacedoLesoinneRouxMagoulesDeLaBourdonnaie:2000a},
Loisel's method \cite{Loisel:2013a},
and the two variants proposed by Gander and Santugini \cite{GanderSantugini:2016a}.
The convergence analysis is provided for three cases:
\begin{enumerate}
\item[(i)] In the general case, convergence is guaranteed, but with no information on the speed.
  In the continuous case, compactness and a regularity condition are required;
  in the discrete case, the Lagrange multipliers may be non-unique.
\item[(ii)] Assuming strong absorbtion in the problem, the compactness assumption can be dropped.
\item[(iii)] In a special case one obtains geometric convergence.
  The theory covers the classical situation where no cross points are present \cite[Sect.~4.2]{CollinoGhanemiJoly:2000a}
	as well as Claeys' choice of a global interface exchange operator in the presence of cross points,
	cf.\ \cite{Claeys:2021a,ClaeysParolin:Preprint2020}.
	It is also shown how this operator can be localized in the discrete case.
\end{enumerate}
Another achievement of the paper at hand is the precise display of the set of \emph{equations}
behind many methods that are often formulated as an iterative process,
and the clarification under which conditions these reformulations are equivalent to the original problem.
The theoretical framework is built upon a minimal set of assumptions that play the role of \emph{axioms}
and provide more generality than in previous publications, much in the spirit of the abstract overlapping Schwarz theory \cite[Ch.~2]{ToselliWidlund:Book}.
As a side product, the paper explains many methods and variants using the same compact notation,
which allows to see more clearly the differences and common building blocks.
In addition to the Robin-Schwarz variants, some related techniques are included that involve Robin transmission conditions as well.
The author hopes that this piece of work will serve as a good reference for other scientists
and be of benefit for future developments in the field.

\medskip

The remainder of this paper is organized as follows.
Section~\ref{sect:fundamentals} introduces the global problem, some abstract domain decomposition,
and a fundamental reformulation in terms of traces.
Section~\ref{sect:facetSystems} examines various choices of trace operators based on \emph{facets}.
On the one hand, this section is very technical and may initially bypassed.
On the other hand, it will be very helpful for understanding the different variants of Robin-Schwarz methods
proposed in the literature.
Section~\ref{sect:interfaceFluxes} introduces a formulation using interface fluxes
and discusses thoroughly the equivalence with the original formulation.
Section~\ref{sect:RobinTC} deals with reformulations involving generalized Robin transmission conditions.
One particular formulation is of fixed point form and involves only one set of \emph{impedance traces}
for all subdomains. At this point, some of the prominent methods are classified.
In Section~\ref{sect:convergence}, the convergence of the associated fixed point method is analyzed in the general case,
the absorbtive case, and in the special case (iii) described above.
Section~\ref{sect:interfaceExchange} shows how to construct the trace operators, spaces,
and the interface exchange operator depending on a fixed impedance operator
(with localization treated in Sect.~\ref{sect:globLocImp}),
such that the assumptions leading to the stronger convergence result are fulfilled.
Some related formulations that involve transmission conditions of Robin kind are briefly discussed in Sect.~\ref{sect:relatedFormulations},
and some technical results are contained in an appendix.

\section{Fundamental non-overlapping domain decomposition formulations}
\label{sect:fundamentals}

Before the development of the general framework (starting with Sect.~\ref{sect:globalProblem}),
let us begin with a closer look at Despr\'es' original method and fix some basic notation.

\subsection{Motivation}
\label{sect:motivation}

In his seminal thesis \cite{Despres:PhD}, Bruno Despr\'es considered the Helmholtz equation in a bounded domain with a Robin boundary condition,
\begin{align}
\label{eq:HelmholtzDespresOrigPDE}
\begin{alignedat}{2}
  -\Delta u - \kappa^2 u & = f && \qquad \text{in } \Omega \subset \mathbb{R}^d,\\
	\complexi \kappa u + \tfrac{\partial}{\partial \normal} u & = 0 && \qquad \text{on } \partial\Omega,
\end{alignedat}
\end{align}
where $u$ is the unknown \emph{phasor}\footnote{In a large part of literature, $U(x, t) = u(x) e^{-\complexi \omega t}$ is used
   (see e.g., \cite{ClaeysParolin:Preprint2020,Monk:2003a}, opposed to \cite{CollinoGhanemiJoly:2000a,GanderMagoulesNataf:2002a}),
   which would lead to a replacement of $\complexi$ by $-\complexi$ throughout this paper.}
of a time-harmonic field $U(x, t) = u(x) e^{\complexi \omega t}$ that solves the wave equation $\partial^2 U /\partial t^2 - c^2 \Delta U = F$,
with $c$~being the speed of sound, $\omega$ the angular frequency, and $\kappa = \omega/c > 0$ denoting the wave number,
$f \in L^2(\Omega)$ is a given source term such that $F(x, t) = f(x) e^{\complexi \omega t}$,
and $\normal$ is the outward unit normal on $\partial\Omega$.
The absorbing boundary condition (ABC) is an approximation of a radiation condition at infinity:
if $u$ is a plane wave propagating in direction $\pm \normal$, then $u$ can only be \emph{outgoing} with respect to $\Omega$.

For a non-overlapping decomposition $\overline\Omega = \bigcup_{i=1}^N \overline\Omega_i$ with $\Omega_i \cap \Omega_j = \emptyset$, $i \neq j$
and for a suitably chosen initial guess $(u_i^{(0)})_{i=1}^N$, the method proposed by Despr\'es is as follows.
In each iteration (index $n=0,1,\ldots$), a \emph{local} Helmholtz problem is solved on each subdomain $\Omega_i$,
where the Robin boundary data comes from the previous step and from the neighboring subdomains:
\begin{align}
\label{eq:HelmholtzDespresOrigIter}
\begin{alignedat}{2}
  -\Delta u_i^{(n+1)} - \kappa^2 u_i^{(n+1)}                       & = f && \qquad \text{in } \Omega_i\,,\\
	\complexi \kappa u_i^{(n+1)} + \tfrac{\partial}{\partial \normal_i} u_i^{(n+1)} & = \complexi \kappa u_j^{(n)} - \tfrac{\partial}{\partial \normal_j} u_j^{(n)}
	   && \qquad \text{on } \Sigma_{ij} := \partial\Omega_i \cap \partial\Omega_j\,,\\
	\complexi \kappa u_i^{(n+1)} + \tfrac{\partial}{\partial \normal_i} u_i^{(n+1)} & = 0 && \qquad \text{on } \partial\Omega_i \cap \partial\Omega,
\end{alignedat}
\end{align}
where $\normal_i$ is the unit normal on $\partial\Omega_i$, outward w.r.t.\ $\Omega_i$.
Despr\'es showed that the iterates converge to the solution of the \emph{global} problem~\eqref{eq:HelmholtzDespresOrigPDE}
subdomain-wise in $H^1$ (under assumptions that will be discussed below).
The particular choice of this method is motivated by the following characteristics:
\begin{enumerate}
\item[(i)] The global problem can be solved iteratively by solving a sequence of local problems that can be solved independently
  of each other, i.e.\ in parallel.
\item[(ii)] The communication between the subdomains is only across the interfaces $\Sigma_{ij}$ of dimension $(d-1)$.
\item[(iii)] The local problems are Robin boundary value problems and as such free of internal resonances
  (regardless of the wave number $\kappa$).
\end{enumerate}
As it is well known, the solution $u$ of \eqref{eq:HelmholtzDespresOrigPDE} satisfies the Dirichlet and Neumann transmission conditions
\begin{align}
\begin{alignedat}{2}
  u_i & = u_j \qquad && \text{on } \Sigma_{ij}\,,\\
	\tfrac{\partial}{\partial \normal_i} u_i & = - \tfrac{\partial}{\partial \normal_j} u_j
	\qquad && \text{on } \Sigma_{ij}\,,
\end{alignedat}
\end{align}
where $u_i$ denotes the restriction of $u$ to the subdomain $\Omega_i$.
Linear combination of these conditions yields the Robin transmission conditions
\begin{align}
\label{eq:RobinClassical}
  \complexi\kappa u_i \pm \tfrac{\partial}{\partial \normal_i} u_i = \complexi\kappa u_j \mp \tfrac{\partial}{\partial \normal_j} u_j
	\qquad \text{on } \Sigma_{ij}\,,
\end{align}
from which we eventually see that the solution of \eqref{eq:HelmholtzDespresOrigPDE} is a fixed point of \eqref{eq:HelmholtzDespresOrigIter}.
Here, a special role is played by the two \emph{impedance traces} $\complexi\kappa u_i \pm \tfrac{\partial}{\partial \normal_i} u_i$ of $u_i$.
To get a feel for the meaning of these traces, suppose that $\Sigma_{ij}$ is a planar face in 3D or a straight interface line in 2D such that the normal vector $\normal_i$ is constant
and $(\normal_i \cdot x)$ is constant for $x \in \Sigma_{ij}$.
If $u_i$ is a combination of an \emph{incoming} and an \emph{outgoing} wave with respect to $\Omega_i$, i.e.,
$u_i(x) = c_\text{in} e^{\complexi \kappa (\normal_i \cdot x)} + c_\text{out} e^{-\complexi \kappa (\normal_i \cdot x)}$,
then
\begin{align}
\label{eq:impedanceTraceForPlaneWave}
\begin{aligned}
  \complexi\kappa u_i + \tfrac{\partial}{\partial \normal_i} u_i & = 2 \complexi\kappa c_\text{in}  e^{\complexi \kappa (\normal_i \cdot x)},\\
	\complexi\kappa u_i - \tfrac{\partial}{\partial \normal_i} u_i & = 2 \complexi\kappa c_\text{out} e^{-\complexi \kappa (\normal_i \cdot x)},
\end{aligned}
\end{align}
i.e., the impedance trace $\complexi\kappa u_i + \tfrac{\partial}{\partial \normal_i} u_i$ is essentially $c_\text{in}$, the amplitude of the \emph{incoming} wave,
whereas the impedance trace $\complexi\kappa u_i - \tfrac{\partial}{\partial \normal_i} u_i$ is essentially $c_\text{out}$, the amplitude of the \emph{outgoing} wave.

\medskip

For the Laplace equation, the same methodology was developed independently by Pierre-Louis Lions \cite{Lions:DD03}:
Setting $\kappa = 0$ in the PDEs of \eqref{eq:HelmholtzDespresOrigPDE} and \eqref{eq:HelmholtzDespresOrigIter},
replacing the outer boundary condition by a more suitable one (e.g.\ a homogeneous Dirichlet condition),
and replacing the imaginary factor $\complexi\kappa$ in the transmission conditions of \eqref{eq:HelmholtzDespresOrigIter} by a positive factor $\gamma > 0$,
one obtains the method proposed by Lions, which works with the (real-valued) \emph{Robin traces} $\gamma u_i \pm \tfrac{\partial}{\partial \normal_i} u_i$.

\medskip

As the attentive reader might have noticed in \eqref{eq:HelmholtzDespresOrigIter},
the normal derivative $\tfrac{\partial}{\partial \normal_i} u_i^{(n)}$
of a general function $u_i^{(n)} \in H^1(\Omega_i)$ is not necessarily well-defined.
Accordingly, Despr\'es assumed that the normal derivative of the solution $u$ and of the initial functions $u_i^{(0)}$ is in $L^2$,
cf.~\cite[Lem.~4.3]{Despres:PhD}.
Then, one can guarantee that all further normal derivatives appearing in \eqref{eq:HelmholtzDespresOrigIter}
are in $L^2$ as well (see also \cite[Sect.~2.3]{CollinoGhanemiJoly:2000a}) and that the iterates $(u^{(n)}_i)$ converge.
Such kind of regularity condition can certainly be an obstacle and rises questions about possible discrete counterparts.
Collino, Ghanemi, and Joly \cite{CollinoGhanemiJoly:2000a} were able to relax the regularity condition by modifying the method
in two ways.

(i) The Robin transmission conditions are generalized to
\begin{align}
\label{eq:RobinMij}
  \complexi M_{ij} u_i \pm \tfrac{\partial}{\partial \normal_i} u_i = \complexi M_{ij} u_j \mp \tfrac{\partial}{\partial \normal_j} u_j \qquad \text{on } \Sigma_{ij}\,,
\end{align}
where the \emph{impedance operator}\footnote{In \cite{CollinoGhanemiJoly:2000a} this operator is called \emph{transmission operator}
and denoted by $T_{ij}$.}
$M_{ij}$ is an isomorphism from $H^s(\Sigma_{ij})$ to its dual, where $s \in [0, \tfrac{1}{2})$,
and fulfills a symmetry and coercivity property such that it induces a norm.
Under the condition that the normal derivatives of the solution and of the initial functions $u_i^{(0)}$
are in the dual of $H^s(\Sigma_{ij})$,
the \emph{damped} Schwarz method is well-defined and can be shown to converge.

(ii) Assume that (a) the impedance operator in \eqref{eq:RobinMij} is chosen with $s = \tfrac{1}{2}$
and (b) the subdomain partition has \emph{no junctions} \cite[Eqn.~(29)]{ClaeysParolin:Preprint2020}:
\begin{align}
\label{eq:noJunctions}
  \text{for each } (i, j),\ \Sigma_{ij} \text{ is either empty or a closed manifold of dimension } (d-1),
\end{align}
in the sense that $\Sigma_{ij}$ has no boundary, cf.\ \cite[Sect.~4.2]{CollinoGhanemiJoly:2000a}.
To obtain a well-defined iterative process, the normal derivatives of the initial functions $u_i^{(0)}$ only need to be in $H^{-1/2}(\Sigma_{ij})$,
which is a natural condition and poses no further restriction, see e.g.\ \cite{McLean:Book}.
Under these stronger conditions,
the damped Schwarz method converges geometrically with a convergence rate $< 1$.
The assumption of no junctions, however, is a severe limitation and
has only been overcome recently \cite{Claeys:2021a,ClaeysParolin:Preprint2020}.

\medskip

In the following sections, Despr\'es' and Lions' method is put into a strict functional framework in (finite- or infinite-dimensional)
Hilbert spaces using the variational level rather than the PDE level, which allows treating the case of finite element discretization too.
All involved operations will be displayed precisely, in particular the restriction to a subdomain, the action of the normal derivative,
and the exchange of data across subdomain interfaces.
The framework applies to rather general wave propagation problems and to coercive problems.

\medskip

\noindent
\emph{Notation}:
Given a (real or complex) Banach space $V$, its \emph{dual} $V^*$ is the space
of bounded linear functionals\footnote{In the literature,
  the dual of a complex Banach space is sometimes defined as the space of bounded \emph{anti-linear}
  forms and, correspondingly, sesquilinear forms are used.
	This article features linear and bilinear forms because they better correspond to the matrix-vector setting.}
with the standard dual norm.
The duality pairing is denoted by $\langle \cdot, \cdot \rangle_{V^* \times V}$,
where the subscript is omitted whenever clear from context.
All vector spaces in this paper are assumed to be (real or complex) Hilbert spaces and as such reflexive,
which permits us to simply identify the bidual $(V^*)^*$ with $V$.
However, we will not identify $V^*$ with $V$.
Given a linear operator $B \colon V \to W$, its \emph{transpose} $B^\top \colon W^* \to V^*$
is defined by $\langle B^\top \psi, v \rangle = \langle \psi, B v \rangle$ for $\psi \in W^*$, $v \in V$.
Recall that $B^\top$ is bounded if and only if $B$ is bounded.
When $V$ and $W$ are Euclidean spaces ($\mathbb{R}^n$ or $\mathbb{C}^n$),
$B$ is identified with its matrix representation, and $B^T \colon W \to V$
denotes the transpose matrix (which is, up to possible conjugation, the adjoint with respect to the Euclidean inner products).
The inner product $(\cdot,\cdot)$ of a complex Hilbert space is a sesquilinear form, including conjugation of the second argument.
In $\mathbb{C}^n$, however, we use the expression $\vv \cdot \wv = \sum_{i=1}^n v_i w_i$ (without conjugation).

\subsection{General problem setting}
\label{sect:globalProblem}

Our starting point is the linear operator equation
\begin{align}
\label{eq:global}
  \text{find } \widehat u \in \widehat U \colon \qquad
  \widehat A \widehat u = \widehat f,
\end{align}
where $\widehat U$ is a finite- or infinite-dimensional Hilbert space,
$\widehat A \colon \widehat U \to \widehat U^*$ a bounded linear operator,
and $\widehat f \in \widehat U^*$ a bounded linear functional.

Throughout the paper, the following basic properties are assumed:
\begin{enumerate}
\item[(i)] $\widehat U$ is either Hilbert space over the field of real numbers, or a \emph{complexified} Hilbert space
  of the form $\widehat U = \widehat U_\text{re} + \complexi \widehat U_\text{re}$, where $\widehat U_\text{re}$ is a real Hilbert space
	and where the inner product on $\widehat U_\re$ is extended to one on $\widehat U$.
	In the latter case, $\widehat U$ enjoys complex conjugation.
\item[(ii)] $\ker(\widehat A) = \{ 0 \}$ and $\range(\widehat A) = \widehat U^*$, i.e., Problem~\eqref{eq:global} is well-posed.
\end{enumerate}
The inner product and norm in $\widehat U$ are denoted by $(\cdot,\cdot)_{\widehat U}$ and $\|\cdot\|_{\widehat U}$, respectively.

\medskip

\begin{example}[variational formulation of the Helmholtz equation]
\label{ex:modelProblemHelmholtz}
Consider the following boundary value problem for the Helmholtz equation in strong form,
\begin{align}
\begin{alignedat}{2}
  -\Delta \widehat u - \kappa^2 \widehat u & = g_\Omega, \quad && \text{in } \Omega,\\[-1ex]
	\widehat u & = 0 && \text{on } \Gamma_D\,, \qquad
	\frac{\partial \widehat u}{\partial \normal} = g_N \quad \text{on } \Gamma_N\,, \qquad
	\frac{\partial \widehat u}{\partial \normal} + \complexi \eta \widehat u = g_R \quad \text{on } \Gamma_R\,,
\end{alignedat}
\end{align}
where $\Omega \subset \mathbb{R}^d$ is a bounded Lipschitz domain with its boundary composed of three disjoint parts $\Gamma_D$, $\Gamma_N$, $\Gamma_R$
($\Gamma_D$ and/or $\Gamma_N$ are allowed to be empty).
Let $\widehat U := H^1_D(\Omega)$ denote the subspace of functions in the complex-valued space $H^1(\Omega)$ vanishing on the Dirichlet boundary $\Gamma_D$.
Then the weak formulation reads: find $\widehat u \in \widehat U$ such that
\begin{align}
\label{eq:weakFormHelmholtz}
  \underbrace{
  \int_\Omega \nabla\widehat u \cdot \nabla\widehat v - \kappa^2 \widehat u\, \widehat v \, dx + \complexi \int_{\Gamma_R} \!\! \eta\, \widehat u\, \widehat v \, ds
	}_{\langle \widehat A \widehat u, \widehat v \rangle}
	= \underbrace{ \int_\Omega g_\Omega\, \widehat v \, dx + \int_{\Gamma_N} \!\! g_N\, \widehat v \, ds
	  + \int_{\Gamma_R} \!\! g_R\, \widehat v \, ds }_{\langle \widehat f, \widehat v \rangle}
	\quad \forall \widehat v \in \widehat U.
\end{align}
With the standard assumptions that $g_\Omega \in L^2(\Omega)$, $g_R \in L^2(\Gamma_R)$,
$\kappa \in L^\infty(\Omega)$ and $\kappa > 0$ uniformly,
$\eta \in L^\infty(\Gamma_R)$ and $\eta > 0$ uniformly,
this formulation can be easily cast into the form~\eqref{eq:global}. The energy space $\widehat U$ fulfills the basic property~(i),
and---provided that $\Gamma_R$ has positive surface measure---the system operator $\widehat A$
fulfills the basic property~(ii), cf.\ e.g.\ \cite[Sect.~4.5]{Hiptmair:2015a}.
We can also choose $\widehat U$ as a suitable conforming finite element subspace of $H^1_D(\Omega)$
such that \eqref{eq:weakFormHelmholtz} becomes a Galerkin discretization, while it is still of form~\eqref{eq:global}.
Such a setup will be referred to as a \emph{discrete case} in contrast to the previously described \emph{continuous case}.
For standard choices of finite element space with small enough mesh size, the discrete problem stays well-posed,
see e.g.\ \cite[Sect.~4.7.2]{Hiptmair:2015a}.
After having fixed a finite element basis, the operator equation can be rewritten in matrix form. Note that the resulting equation
is again of form~\eqref{eq:global} where $\widehat A$ is the stiffness matrix, $\widehat f$ the load vector,
and $\widehat U = \mathbb{C}^n$ ($n$ being the dimension of the finite element space).
\end{example}

\subsection{Abstract domain decomposition}
\label{sect:abstractDD}

The theory in this paper works on an abstract level and does not require any \emph{geometric} description of subdomains.
Rather, an \emph{algebraic} description specifies how subdomain operators (or linear functionals)
\emph{assemble} the global operator (or functional, respectively).
Nevertheless, the abstract assumptions will be accompanied by examples involving the geometric setup.
Throughout the paper, the two following definitions will be used extensively.

\begin{definition}
\label{def:abstractDD}
  An \emph{abstract domain decomposition} of $\widehat U$ is described by
\begin{enumerate}
\item[(i)] \emph{Local spaces} $U_i$, $i=1,\ldots,N$, that are assumed to be Hilbert spaces,
  with inner products $(\cdot,\cdot)_{U_i}$ and norms $\| \cdot \|_{U_i}$.
  If $\widehat U$ is finite-dimensional, infinite-dimensional, real, or complexified
  then each space $U_i$ has the corresponding property, respectively.
	We define the associated product space
	\[
	  U := \productspace_{i=1}^N U_i\,,
	\]
	also referred to as the \emph{broken space}. The $i$-th component of $u \in U$ is denoted by $u_i$,
	the inner product is $(u, v)_U := \sum_{i=1}^N (u_i, v_i)_{U_i}$, and the associated norm is denoted by $\| u \|_U$.
\item[(ii)] Bounded linear \emph{restriction} operators $R_i \colon \widehat U \to U_i$, assumed to be real-valued\footnote{Here
  real-valued means that either
  (a) $\widehat U$ and $U_i$ are real Hilbert spaces, or
	(b) $\widehat U$ and $U_i$ are complexified and $R_i$ has the form $R_i (u_\re + \complexi u_\im) = R_{i,\re} u_\re + \complexi R_{i,\re} u_\im$
  for an operator $R_{i,\re}$ acting on the real Hilbert spaces.},
  which altogether define the \emph{collective restriction operator}
	$R \colon \widehat U \to U$,
	\[
	  R \widehat u := (R_i \widehat u)_{i=1}^N\,.
	\]
\end{enumerate}
\end{definition}

\begin{definition}
\label{def:assemblingProp}
	The local bounded linear operators $A_i \colon U_i \to U_i^*$ and local functionals $f_i \in U_i^*$, $i=1,\ldots,N$
  fulfill the \emph{assembling property} iff
  \begin{align}
	\label{eq:assemblingProp}
    \widehat A = \sum_{i=1}^N R_i^\top A_i R_i, \qquad \widehat f = \sum_{i=1}^N R_i^\top f_i\,.
  \end{align}	
\end{definition}

From here on, we assume a given abstract domain decomposition $(U, R)$ as well as the existence of $A_i$, $f_i$, $i=1,\ldots,N$
fulfilling the assembling property.
	For a more compact notation,
	we define the block-diagonal operator $A = \text{diag}(A_i)_{i=1}^N \colon U \to U^*$
	and the linear functional $f = (f_i)_{i=1}^N \in U^*$ acting on the product space,
	such that \eqref{eq:assemblingProp} simply reads $\widehat A = R^\top A R$ and $\widehat f = R^\top f$,
	see also Fig.~\ref{fig:assemblingProp}.

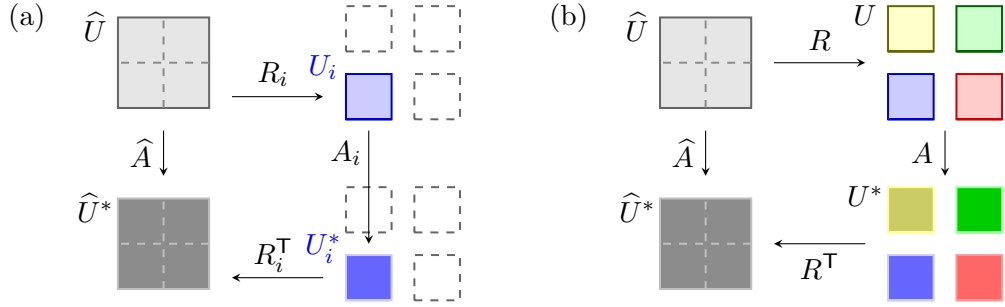
\begin{figure}
\begin{center}
  \begin{tikzpicture}
		\pgftransformscale{0.6}
		
	  \definecolor{dred}{rgb}{0.8, 0.0, 0.0}
		\definecolor{mdred}{rgb}{1.0, 0.4, 0.4}
		\definecolor{lred}{rgb}{1.0, 0.8, 0.8}
		\definecolor{mlred}{rgb}{1.0, 0.6, 0.6}
		
		\definecolor{dgreen}{rgb}{0.0, 0.4, 0.0}
		\definecolor{mdgreen}{rgb}{0.0, 0.8, 0.0}
		\definecolor{lgreen}{rgb}{0.8, 1.0, 0.8}
		\definecolor{mlgreen}{rgb}{0.6, 1.0, 0.6}
		
		\definecolor{dblue}{rgb}{0.0, 0.0, 0.8}
		\definecolor{mdblue}{rgb}{0.4, 0.4, 1.0}
		\definecolor{lblue}{rgb}{0.8, 0.8, 1.0}
		\definecolor{mlblue}{rgb}{0.6, 0.6, 1.0}
		
		\definecolor{dyellow}{rgb}{0.4, 0.4, 0.0}
		\definecolor{mdyellow}{rgb}{0.8, 0.8, 0.4}
		\definecolor{lyellow}{rgb}{1.0, 1.0, 0.8}
		\definecolor{mlyellow}{rgb}{1.0, 1.0, 0.6}
		
		\definecolor{dgrey}{rgb}{0.4, 0.4, 0.4}
		\definecolor{mdgrey}{rgb}{0.55, 0.55, 0.55}
		\definecolor{lgrey}{rgb}{0.9, 0.9, 0.9}
		\definecolor{mlgrey}{rgb}{0.75, 0.75, 0.75}
	
	  \node (A0) at (1,1) {};
	  \node (A1) at (3,1) {};
	  \node (A2) at (3,3) {};
	  \node (A3) at (1,3) {};
		\node (A01) at (2,1) {};
		\node (A12) at (3,2) {};
		\node (A23) at (2,3) {};
		\node (A03) at (1,2) {};
		
		\draw[line width=0.75pt,color=mlgrey,fill=mdgrey] (A0.center)--(A1.center)--(A2.center)--(A3.center)--(A0.center)--(A1.center);
		\draw[line width=0.75pt,color=mlgrey,dashed] (A01.center)--(A23.center);
		\draw[line width=0.75pt,color=mlgrey,dashed] (A12.center)--(A03.center);
		
		\node at (0.5,2.8) {$\widehat U^*$};
		
		\node (B0) at (1,5) {};
	  \node (B1) at (3,5) {};
	  \node (B2) at (3,7) {};
	  \node (B3) at (1,7) {};
		\node (B01) at (2,5) {};
		\node (B12) at (3,6) {};
		\node (B23) at (2,7) {};
		\node (B03) at (1,6) {};
		
		\draw[line width=0.75pt,color=dgrey,fill=lgrey] (B0.center)--(B1.center)--(B2.center)--(B3.center)--(B0.center)--(B1.center);
		\draw[line width=0.75pt,color=mdgrey,dashed] (B01.center)--(B23.center);
		\draw[line width=0.75pt,color=mdgrey,dashed] (B12.center)--(B03.center);
		
		\node at (0.5,6.8) {$\widehat U$};
		
		\draw[->, >=stealth] (2,4.5)--(2,3.5);
		\node at (1.5,4) {$\widehat A$};

		
	  \node (C0) at (6,0.75) {};
	  \node (C1) at (7,0.75) {};
	  \node (C2) at (7,1.75) {};
	  \node (C3) at (6,1.75) {};
		
		\draw[line width=0.75pt,color=lblue,fill=mdblue] (C0.center)--(C1.center)--(C2.center)--(C3.center)--(C0.center)--(C1.center);
		
		\node (D0) at (7.5,0.75) {};
	  \node (D1) at (8.5,0.75) {};
	  \node (D2) at (8.5,1.75) {};
	  \node (D3) at (7.5,1.75) {};
		
		\draw[line width=0.75pt,color=dgrey,dashed] (D0.center)--(D1.center)--(D2.center)--(D3.center)--(D0.center);
		
		\node (E0) at (6,2.25) {};
	  \node (E1) at (7,2.25) {};
	  \node (E2) at (7,3.25) {};
	  \node (E3) at (6,3.25) {};
		
		\draw[line width=0.75pt,color=dgrey,dashed] (E0.center)--(E1.center)--(E2.center)--(E3.center)--(E0.center);
		
		\node (F0) at (7.5,2.25) {};
	  \node (F1) at (8.5,2.25) {};
	  \node (F2) at (8.5,3.25) {};
	  \node (F3) at (7.5,3.25) {};
		
		\draw[line width=0.75pt,color=dgrey,dashed] (F0.center)--(F1.center)--(F2.center)--(F3.center)--(F0.center);
		
		
	  \node (G0) at (6,4.75) {};
	  \node (G1) at (7,4.75) {};
	  \node (G2) at (7,5.75) {};
	  \node (G3) at (6,5.75) {};
		
		\draw[line width=0.75pt,color=dblue,fill=lblue] (G0.center)--(G1.center)--(G2.center)--(G3.center)--(G0.center)--(G1.center);
		
		\node (H0) at (7.5,4.75) {};
	  \node (H1) at (8.5,4.75) {};
	  \node (H2) at (8.5,5.75) {};
	  \node (H3) at (7.5,5.75) {};
		
		\draw[line width=0.75pt,color=dgrey,dashed] (H0.center)--(H1.center)--(H2.center)--(H3.center)--(H0.center);
		
		\node (I0) at (6,6.25) {};
	  \node (I1) at (7,6.25) {};
	  \node (I2) at (7,7.25) {};
	  \node (I3) at (6,7.25) {};
		
		\draw[line width=0.75pt,color=dgrey,dashed] (I0.center)--(I1.center)--(I2.center)--(I3.center)--(I0.center);
		
		\node (J0) at (7.5,6.25) {};
	  \node (J1) at (8.5,6.25) {};
	  \node (J2) at (8.5,7.25) {};
	  \node (J3) at (7.5,7.25) {};
		
		\draw[line width=0.75pt,color=dgrey,dashed] (J0.center)--(J1.center)--(J2.center)--(J3.center)--(J0.center);
		
		
		\node at (5.5,5.9) {$\darkblue{U_i}$};
		\node at (5.5,1.9) {$\darkblue{U_i^*}$};
		
		\draw[->, >=stealth] (6.5,4.5)--(6.5,2);
		\node at (6.0,4) {$A_i$};
		
		\draw[->, >=stealth] (3.5,5.25)--(5.5,5.25);
		\node at (4.4,5.75) {$R_i$};
		
		\draw[<-, >=stealth] (3.5,1.25)--(5.5,1.25);
		\node at (4.4,1.75) {$R_i^\top$};

		\node at (-1,7) {(a)};
		
	\end{tikzpicture}
	\hspace{4ex}
  \begin{tikzpicture}
		\pgftransformscale{0.6}
		
	  \definecolor{dred}{rgb}{0.8, 0.0, 0.0}
		\definecolor{mdred}{rgb}{1.0, 0.4, 0.4}
		\definecolor{lred}{rgb}{1.0, 0.8, 0.8}
		\definecolor{mlred}{rgb}{1.0, 0.6, 0.6}
		
		\definecolor{dgreen}{rgb}{0.0, 0.4, 0.0}
		\definecolor{mdgreen}{rgb}{0.0, 0.8, 0.0}
		\definecolor{lgreen}{rgb}{0.8, 1.0, 0.8}
		\definecolor{mlgreen}{rgb}{0.6, 1.0, 0.6}
		
		\definecolor{dblue}{rgb}{0.0, 0.0, 0.8}
		\definecolor{mdblue}{rgb}{0.4, 0.4, 1.0}
		\definecolor{lblue}{rgb}{0.8, 0.8, 1.0}
		\definecolor{mlblue}{rgb}{0.6, 0.6, 1.0}
		
		\definecolor{dyellow}{rgb}{0.4, 0.4, 0.0}
		\definecolor{mdyellow}{rgb}{0.8, 0.8, 0.4}
		\definecolor{lyellow}{rgb}{1.0, 1.0, 0.8}
		\definecolor{mlyellow}{rgb}{1.0, 1.0, 0.6}
		
		\definecolor{dgrey}{rgb}{0.4, 0.4, 0.4}
		\definecolor{mdgrey}{rgb}{0.55, 0.55, 0.55}
		\definecolor{lgrey}{rgb}{0.9, 0.9, 0.9}
		\definecolor{mlgrey}{rgb}{0.75, 0.75, 0.75}
	
	  \node (A0) at (1,1) {};
	  \node (A1) at (3,1) {};
	  \node (A2) at (3,3) {};
	  \node (A3) at (1,3) {};
		\node (A01) at (2,1) {};
		\node (A12) at (3,2) {};
		\node (A23) at (2,3) {};
		\node (A03) at (1,2) {};
		
		\draw[line width=0.75pt,color=mlgrey,fill=mdgrey] (A0.center)--(A1.center)--(A2.center)--(A3.center)--(A0.center)--(A1.center);
		\draw[line width=0.75pt,color=mlgrey,dashed] (A01.center)--(A23.center);
		\draw[line width=0.75pt,color=mlgrey,dashed] (A12.center)--(A03.center);
		
		\node at (0.5,2.8) {$\widehat U^*$};
		
		\node (B0) at (1,5) {};
	  \node (B1) at (3,5) {};
	  \node (B2) at (3,7) {};
	  \node (B3) at (1,7) {};
		\node (B01) at (2,5) {};
		\node (B12) at (3,6) {};
		\node (B23) at (2,7) {};
		\node (B03) at (1,6) {};
		
		\draw[line width=0.75pt,color=dgrey,fill=lgrey] (B0.center)--(B1.center)--(B2.center)--(B3.center)--(B0.center)--(B1.center);
		\draw[line width=0.75pt,color=mdgrey,dashed] (B01.center)--(B23.center);
		\draw[line width=0.75pt,color=mdgrey,dashed] (B12.center)--(B03.center);
		
		\node at (0.5,6.8) {$\widehat U$};
		
		\draw[->, >=stealth] (2,4.5)--(2,3.5);
		\node at (1.5,4) {$\widehat A$};

		
	  \node (C0) at (6,0.75) {};
	  \node (C1) at (7,0.75) {};
	  \node (C2) at (7,1.75) {};
	  \node (C3) at (6,1.75) {};
		
		\draw[line width=0.75pt,color=lblue,fill=mdblue] (C0.center)--(C1.center)--(C2.center)--(C3.center)--(C0.center)--(C1.center);
		
		\node (D0) at (7.5,0.75) {};
	  \node (D1) at (8.5,0.75) {};
	  \node (D2) at (8.5,1.75) {};
	  \node (D3) at (7.5,1.75) {};
		
		\draw[line width=0.75pt,color=mlred,fill=mdred] (D0.center)--(D1.center)--(D2.center)--(D3.center)--(D0.center)--(D1.center);
		
		\node (E0) at (6,2.25) {};
	  \node (E1) at (7,2.25) {};
	  \node (E2) at (7,3.25) {};
	  \node (E3) at (6,3.25) {};
		
		\draw[line width=0.75pt,color=mlyellow,fill=mdyellow] (E0.center)--(E1.center)--(E2.center)--(E3.center)--(E0.center)--(E1.center);
		
		\node (F0) at (7.5,2.25) {};
	  \node (F1) at (8.5,2.25) {};
	  \node (F2) at (8.5,3.25) {};
	  \node (F3) at (7.5,3.25) {};
		
		\draw[line width=0.75pt,color=mlgreen,fill=mdgreen] (F0.center)--(F1.center)--(F2.center)--(F3.center)--(F0.center)--(F1.center);
		
		
	  \node (G0) at (6,4.75) {};
	  \node (G1) at (7,4.75) {};
	  \node (G2) at (7,5.75) {};
	  \node (G3) at (6,5.75) {};
		
		\draw[line width=0.75pt,color=dblue,fill=lblue] (G0.center)--(G1.center)--(G2.center)--(G3.center)--(G0.center)--(G1.center);
		
		\node (H0) at (7.5,4.75) {};
	  \node (H1) at (8.5,4.75) {};
	  \node (H2) at (8.5,5.75) {};
	  \node (H3) at (7.5,5.75) {};
		
		\draw[line width=0.75pt,color=dred,fill=lred] (H0.center)--(H1.center)--(H2.center)--(H3.center)--(H0.center)--(H1.center);
		
		\node (I0) at (6,6.25) {};
	  \node (I1) at (7,6.25) {};
	  \node (I2) at (7,7.25) {};
	  \node (I3) at (6,7.25) {};
		
		\draw[line width=0.75pt,color=dyellow,fill=lyellow] (I0.center)--(I1.center)--(I2.center)--(I3.center)--(I0.center)--(I1.center);
		
		\node (J0) at (7.5,6.25) {};
	  \node (J1) at (8.5,6.25) {};
	  \node (J2) at (8.5,7.25) {};
	  \node (J3) at (7.5,7.25) {};
		
		\draw[line width=0.75pt,color=dgreen,fill=lgreen] (J0.center)--(J1.center)--(J2.center)--(J3.center)--(J0.center)--(J1.center);
		
		
		\node at (5.5,7.05) {$U$};
		\node at (5.5,3.05) {$U^*$};
		
		\draw[->, >=stealth] (7.25,4.5)--(7.25,3.5);
		\node at (6.75,4) {$A$};
		
		\draw[->, >=stealth] (3.5,6)--(5.5,6);
		\node at (4.5,6.5) {$R$};
		
		\draw[<-, >=stealth] (3.5,2)--(5.5,2);
		\node at (4.5,1.5) {$R^\top$};

		\node at (-1,7) {(b)};
		
	\end{tikzpicture}
	\caption{\label{fig:assemblingProp}%
	  (a) Illustration of the local space $U_i$.
		(b) Illustration of the broken space $U$ and of the assembling property $\widehat A = R^\top A R$.
	}
\end{center}
\end{figure}

\begin{example}[assembling property]
\label{ex:modelHelmholtzDD}
  For the Helmholtz formulation of Example~\ref{ex:modelProblemHelmholtz},
	let $\{ \Omega_i \}_{i=1}^N$ be a non-overlapping decomposition of $\Omega$ into Lipschitz subdomains $\Omega_i$ such that
	$\bigcup_{i=1}^N \overline\Omega_i = \overline\Omega$ and $\Omega_i \cap \Omega_j = \emptyset$ for $i \neq j$.
	In the continuous case,
	the \emph{local space} $U_i$ is chosen (by default) as $U_i = H^1(\Omega_i)$ if $\meas_{d-1}(\partial\Omega_i \cap \Gamma_D) = 0$
	and $U_i = H^1_D(\Omega_i) = \{ v \in H^1(\Omega_i) \colon v_{|\Gamma_D} = 0 \}$ otherwise.
	The \emph{restriction operator} $R_i$ simply restricts a function in $H^1_D(\Omega)$ to the subdomain $\Omega_i$.
	The definition of the local operators $A_i$ and linear functionals $f_i$ follows that of $\widehat A$ and $\widehat f$,
	replacing $\Omega$ by $\Omega_i$, $\Gamma_R$ by $\Gamma_R \cap \partial\Omega_i$, etc.\
	(see also Table~\ref{tab:variationalExamples} for an example with $g_\Omega = 0$, $g_N = 0$),
	and they altogether fulfill the assembling property (Def.~\ref{def:assemblingProp}).
	If $\widehat U$ is a finite element subspace of $H^1_D(\Omega)$ based on a mesh that resolves the subdomain decomposition,
	then we define $U_i$ as the restriction of $\widehat U$ to the elements of $\Omega_i$. The assembling property holds again.
\end{example}

\begin{remark}
\label{rem:localDirichlet}
  If Dirichlet conditions are present in the original space	then they need not necessarily be inherited in the local spaces.
	Suppose for Example~\ref{ex:modelProblemHelmholtz} that we have a subdomain
	$\Omega_i$ with $\meas_{d-1}(\partial\Omega \cap \Gamma_D) > 0$.
	Then---in contrast to the default choice $U_i = H^1_D(\Omega_i)$---we	are allowed to use $U_i = H^1(\Omega_i)$,
	see Fig.~\ref{fig:detachDirichlet}.
	However, in the latter case, $\range(R_i) \subsetneq U_i$.
\end{remark}

\begin{figure}
\begin{center}
  \begin{tikzpicture}
		\pgftransformscale{0.6}
		
	  \definecolor{dred}{rgb}{0.8, 0.0, 0.0}
		\definecolor{mdred}{rgb}{1.0, 0.4, 0.4}
		\definecolor{lred}{rgb}{1.0, 0.8, 0.8}
		\definecolor{mlred}{rgb}{1.0, 0.6, 0.6}
		
		\definecolor{dgreen}{rgb}{0.0, 0.4, 0.0}
		\definecolor{mdgreen}{rgb}{0.0, 0.8, 0.0}
		\definecolor{lgreen}{rgb}{0.8, 1.0, 0.8}
		\definecolor{mlgreen}{rgb}{0.6, 1.0, 0.6}
		
		\definecolor{dblue}{rgb}{0.0, 0.0, 0.8}
		\definecolor{mdblue}{rgb}{0.4, 0.4, 1.0}
		\definecolor{lblue}{rgb}{0.8, 0.8, 1.0}
		\definecolor{mlblue}{rgb}{0.6, 0.6, 1.0}
		
		\definecolor{dyellow}{rgb}{0.4, 0.4, 0.0}
		\definecolor{mdyellow}{rgb}{0.8, 0.8, 0.4}
		\definecolor{lyellow}{rgb}{1.0, 1.0, 0.8}
		\definecolor{mlyellow}{rgb}{1.0, 1.0, 0.6}
		
		\definecolor{dgrey}{rgb}{0.4, 0.4, 0.4}
		\definecolor{mdgrey}{rgb}{0.55, 0.55, 0.55}
		\definecolor{lgrey}{rgb}{0.9, 0.9, 0.9}
		\definecolor{mlgrey}{rgb}{0.75, 0.75, 0.75}
	
	  \node (A0) at (1,1) {};
	  \node (A1) at (3,1) {};
	  \node (A2) at (3,3) {};
	  \node (A3) at (1,3) {};
		\node (A01) at (2,1) {};
		\node (A12) at (3,2) {};
		\node (A23) at (2,3) {};
		\node (A03) at (1,2) {};
		
		\draw[line width=0.75pt,color=dgrey,fill=lgrey] (A0.center)--(A1.center)--(A2.center)--(A3.center)--(A0.center)--(A1.center);
		\draw[line width=0.75pt,color=dgrey,dashed] (A01.center)--(A23.center);
		\draw[line width=0.75pt,color=dgrey,dashed] (A12.center)--(A03.center);
		
		\foreach \x in {0,0.2,...,2}
		{
		  \draw[line width=0.75pt,color=dgrey] (1 + \x, 1)--(1 + \x - 0.2, 1 - 0.2);
		}


		\node at (0.5,2.8) {$\widehat U$};
		
		
	  \node (C0) at (6,0.75) {};
	  \node (C1) at (7,0.75) {};
	  \node (C2) at (7,1.75) {};
	  \node (C3) at (6,1.75) {};
		
		\draw[line width=0.75pt,color=dblue,fill=lblue] (C0.center)--(C1.center)--(C2.center)--(C3.center)--(C0.center)--(C1.center);
		
		\node (D0) at (7.5,0.75) {};
	  \node (D1) at (8.5,0.75) {};
	  \node (D2) at (8.5,1.75) {};
	  \node (D3) at (7.5,1.75) {};
		
		\draw[line width=0.75pt,color=dred,fill=lred] (D0.center)--(D1.center)--(D2.center)--(D3.center)--(D0.center)--(D1.center);
		
		\node (E0) at (6,2.25) {};
	  \node (E1) at (7,2.25) {};
	  \node (E2) at (7,3.25) {};
	  \node (E3) at (6,3.25) {};
		
		\draw[line width=0.75pt,color=dyellow,fill=lyellow] (E0.center)--(E1.center)--(E2.center)--(E3.center)--(E0.center)--(E1.center);
		
		\node (F0) at (7.5,2.25) {};
	  \node (F1) at (8.5,2.25) {};
	  \node (F2) at (8.5,3.25) {};
	  \node (F3) at (7.5,3.25) {};
		
		\draw[line width=0.75pt,color=dgreen,fill=lgreen] (F0.center)--(F1.center)--(F2.center)--(F3.center)--(F0.center)--(F1.center);
		
		
		\node at (5.5,3.05) {$U$};
		
		\draw[->, >=stealth] (3.5,2)--(5.5,2);
		\node at (4.5,2.5) {$R$};

		\foreach \x in {0,0.2,...,2}
		{
		  \draw[line width=0.75pt,color=dgrey] (6.25 + \x, 0.25)--(6.25 + \x - 0.2, 0.25 - 0.2);
		}
		\draw[line width=0.75pt,color=dgrey] (6.25, 0.25)--(8.25, 0.25);

	\end{tikzpicture}
\end{center}
\caption{\label{fig:detachDirichlet}%
   Illustration of detaching Dirichlet boundary (diagonal hatching): none of the local spaces has an inbuilt Dirichlet condition,
	 see Remark~\ref{rem:localDirichlet}.
}
\end{figure}
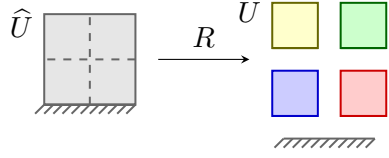

Our next assumption is on the collective restriction operator.

\medskip

\noindent\fbox{%
\begin{minipage}[c]{0.985\textwidth}
\begin{assumption}
\label{ass:A1}
{\qquad}
\begin{itemize}
\item $\ker(R) = \{ 0 \}$ \emph{(coverage property),}
\item $\range(R)$ is closed \emph{(reconstruction property).}
\end{itemize}
\end{assumption}
\end{minipage}}

\medskip

  The coverage property is equivalent to
  \[
    \big( \forall i=1,\ldots,N \colon R_i \widehat v = 0 \big) \implies \widehat v = 0,
  \]
  and is fulfilled for the typical examples as long as $\bigcup_{i=1}^N \overline\Omega_i = \overline\Omega$.
  The reconstruction property implies that $R$ has a unique bounded right inverse defined on $\range(R)$,
  so given the restrictions $R_i \widehat u$, one can always reconstruct the original function $\widehat u$.
  In other words, $R \colon \widehat U \to \range(R)$ is an isomorphism.
	In particular, $\| R \widehat v \|_U$ is an equivalent norm to $\| \widehat v \|_{\widehat U}$;
  in the typical cases, one even has $(\widehat u, \widehat v)_{\widehat U} = \sum_{i=1}^N (R_i \widehat u, R_i \widehat v)_{U_i} = (R \widehat u, R \widehat v)_U$.

\begin{example}[reconstruction property]
\label{ex:HelmholtzReconstruction}
  Let us continue with Example~\ref{ex:modelHelmholtzDD}. Any function $\widehat u \in \widehat U = H^1_D(\Omega)$
	fulfills
	\begin{align}
	\label{eq:nablaGrad}
	  \sum_{i=1}^N \int_{\Omega_i} \nabla u_i \cdot \varphiv\, dx = - \sum_{i=1}^N \int_{\Omega_i} u_i (\opdiv \varphiv)\, dx \qquad \forall \varphiv \in C^\infty_0(\Omega)^d,
	\end{align}
	where $u_i := R_i \widehat u = \widehat u_{|\Omega_i}$. Conversely, given a \emph{broken} function $u \in U = \productspace_{i=1}^N H^1_D(\Omega_i)$,
	we can form the \emph{patchwork} function $\widehat u$ in $L^2(\Omega)$ by the piecewise definition $\widehat u_{|\Omega_i} = u_i$ for $i=1,\ldots,N$.
	Apparently, $\widehat u$ is in $H^1(\Omega)$ if and only if
	\eqref{eq:nablaGrad} holds. Finally, since $u_{i|\Gamma_D} = 0$ for each $i$, it follows that $\widehat u_{|\Gamma_D} = 0$, so $\widehat u \in \widehat U$.
	To summarize, $u \in \range(R)$ if and only if \eqref{eq:nablaGrad} holds. All expressions in \eqref{eq:nablaGrad} are continuous w.r.t.\ to $u \in U$,
	and so $\range(R)$ must be closed. The analogous statement holds for the choice $U_i = H^1(\Omega_i)$, see Remark~\ref{rem:localDirichlet},
	only that the local Dirichlet boundary condition has to be added to \eqref{eq:nablaGrad}.
\end{example}

The assumptions made so far hold for a large variety of variational problems, such as the Helmholtz equation and the time-harmonic Maxwell equations,
see Table~\ref{tab:variationalExamples}, as well as for strongly coercive (``positive definite'') problems in the stated function spaces
(e.g., formed by replacing $A_i$ in Table~\ref{tab:variationalExamples} by $A_{i,0} + A_{i,1} + A_{i,2}$).
They also hold for the Galerkin discretization of these problems by standard finite elements.

\begin{table}
\begin{center}
  \begin{tabular}{c||c|c|c}
	            & Helmholtz & Maxwell & Helmholtz dual \\
	\hline
	  \rule{0pt}{3ex}
	  $A_{i,0}$
		  & $\int_{\Omega_i} \nabla u \cdot \nabla v \, dx$
		  & $\int_{\Omega_i} \mu^{-1} \opcurl \uv \cdot \opcurl \vv \, dx$
		  & $\int_{\Omega_i} \kappa^{-2} \opdiv \uv\, \opdiv \vv \, dx$ \\
		\rule{0pt}{3ex}
		$+\complexi A_{i,1}$
		  & $+\complexi \int_{\partial\Omega_i \cap \Gamma_R} \eta\,u\,v\, ds$
		  & $+\complexi\omega \big[ \int_{\Omega_i} \sigma \uv \cdot \vv \, dx + \int_{\partial\Omega_i \cap \Gamma_R} \eta\,\uv_\tau \cdot \vv_\tau \, ds \big]$
		  & $+\complexi \int_{\partial\Omega_i \cap \Gamma_R} \eta^{-1} \uv_n \vv_n \, ds$ \\
		\rule{0pt}{3ex}
		$- A_{i,2}$
		  & $-\int_{\Omega_i} \kappa^2 u\,v\, dx$
			& $-\omega^2 \int_{\Omega_i} \varepsilon \uv \cdot \vv \, dx$
			& $-\int_{\Omega_i} \uv \cdot \vv \, dx$ \\[0.5em]
		\hline
		\rule{0pt}{3ex}
		$f_i$
		  & $\int_{\partial\Omega_i \cap \Gamma_R} g_R\, v \, ds$
			& $-\complexi\omega \int_{\partial\Omega_i \cap \Gamma_R} \jv_S \cdot \vv_\tau \, ds$
			& $\int_{\partial\Omega_i \cap \Gamma_R} \eta^{-1} g_R \vv_n \, ds$ \\[0.5em]
  \end{tabular}
	\vspace{0.7ex}
	\caption{\label{tab:variationalExamples}%
	  Examples of variational formulations with $A_i = A_{i,0} + \complexi A_{i,1} - A_{i,2}$
		with operators $A_{i,k}$ corresponding to bilinear forms $\langle A_{i,k} \cdot, \cdot \rangle$ as indicated.
		Above, $\Omega_i$ is the $i$-th subdomain, $\Gamma_R$ the Robin boundary, $\kappa$ the wave number, $\omega$ the angular frequency.
		The subdomain energy space $U_i$ is a suitable subspace of $H^1(\Omega_i)$, $\Hv(\opcurl,\Omega_i)$, $\Hv(\opdiv,\Omega_i)$,
		and $\uv_\tau$ denotes the tangential trace and $\uv_n$ the (scalar) normal trace,
		such that $\uv = \uv_\tau + \uv_n \nv$ for smooth functions.
		}
\end{center}
\end{table}

\begin{remark}
	Definitions~\ref{def:abstractDD}, \ref{def:assemblingProp} and Assumption~\ref{ass:A1}
	generalize the assumptions that are usually
	made in BDDC methods, cf.\ \cite{MandelDohrmannTezaur:2005a,PechsteinDohrmann:2017a}.
  Note also that Definition~\ref{def:abstractDD}
	is similar to but different from the assumptions made in the classical abstract Schwarz theory,
	see \cite[Sect.~2.2]{ToselliWidlund:Book}. In that theory, \emph{prolongation/extension} operators $E_i \colon U_i \to \widehat U$
	are needed in the first place and their transposed operators $E_i^\top \colon \widehat U^* \to U_i^*$
	\emph{restrict} dual quantities.
	In the present theory, we need the \emph{restriction} operators $R_i \colon \widehat U \to U_i$ in the first place
	and \emph{assemble} dual quantities using the transposes $R_i^\top \colon U_i^* \to \widehat U^*$.
	According to \cite[(2.3)]{ToselliWidlund:Book}, the extension operators from the Schwarz theory 
  must fulfill the coverage property $\sum_{i=1}^N E_i(U_i) = \widehat U$, which is quite different from Assumption~\ref{ass:A1}.
\end{remark}

\subsubsection{Discrete case -- matrix notation}
\label{sect:discreteMatrixNotation}

In the discrete case, we can identify all operators with their associated matrices without changing the notation.
Then $\widehat A$ is the global stiffness matrix, $A_i$ the subdomain stiffness matrix,
$\widehat f$ the global load vector, and $f_i$ the subdomain load vector.
In most situations, each local degree of freedom (dof) of a subdomain $i$ can be associated with a global dof.
Suppose that $\widehat U = \mathbb{R}^n$ or $\mathbb{C}^n$ and that for each $i=1,\ldots,N$,
the local space of subdomain dofs is given by $U_i = \mathbb{R}^{n_i}$ or $\mathbb{C}^{n_i}$,
respectively.
Then $R_i$ must select the local dofs of subdomain $i$ out of the global dofs,
i.e., $R_i \colon \widehat U \to U_i$ is an \emph{incidence matrix} of the form
\begin{align}
\label{eq:RiIncidence}
  (R_i)_{\ell k} = \begin{cases} 1 & \text{if } \mathsf{g}_i(\ell) = k \\ 0 & \text{otherwise,} \end{cases}
\end{align}
where $\mathsf{g}_{i} \colon \{ 1,\ldots,n_i \} \to \{ 1,\ldots,n \}$ is an injective mapping, the \emph{local-to-global} mapping.
From these properties, one derives that
\begin{align}
\label{eq:RiRiTId}
  R_i R_i^T = I = \diag(1)_{\ell=1}^{n_i}, \qquad R_i^T R_i = \diag(\mu_k^{(i)})_{k=1}^n,
\end{align}
where $\mu_k^{(i)} \in \{0, 1\}$ indicates whether the global dof $k$ is shared by subdomain $i$ or not.
Based on that, we can define for each global dof $k=1,\ldots,n$ its \emph{multiplicity} $\mu_k := \sum_{i=1}^N \mu_k^{(i)}$
and the \emph{set of sharing subdomains} $\mathcal{N}_k := \{ i = 1,\ldots, N \colon \mu_k^{(i)} = 1 \}$.
The coverage property from Assumption~\ref{ass:A1} holds if and only if the minimal multiplicity is $\ge 1$.
Regarding the \emph{maximal dof multiplicity} $\mu_{\max} := \max_{k=1,\ldots,n} \mu_k$,
we distinguish three cases:
\begin{itemize}
\item $\mu_{\max} = 1$, a degenerate case (either one subdomain or no coupling between subdomains),
\item $\mu_{\max} = 2$, a special case, where some formulations below turn out to be \emph{non-redundant},
\item $\mu_{\max} > 2$, the general case, where some formulations below involve redundancy.
\end{itemize}
The last case occurs typically (but not necessarily) when the \emph{geometric} domain decomposition has \emph{cross points}.
A cross point is a \emph{geometric} point in $\overline\Omega$
that lies on at least three subdomains boundaries, cf.\ \cite{GanderSantugini:2016a}.
In \cite{GanderKwok:2012a}, a \emph{dof} $k$ with $\mu_k > 2$ is called cross point as well.
Note also that the notions of \emph{cross points} and \emph{junctions} (cf.\ \eqref{eq:noJunctions}) are slightly different, see also Fig.~\ref{fig:junctionCrossPoint}.
Decompositions without cross points are sometimes called \emph{1D} or \emph{one-way decompositions}, those without junctions are also called \emph{onion-like}.

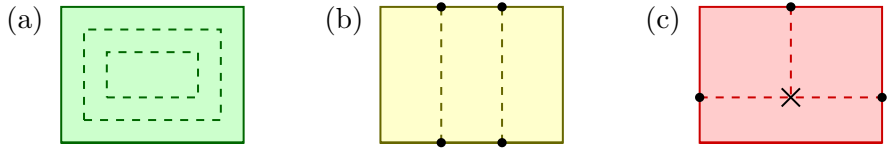
\begin{figure}
\begin{center}
	\begin{tikzpicture}
		\pgftransformscale{0.6}
		
	  \definecolor{dred}{rgb}{0.8, 0.0, 0.0}
		\definecolor{mdred}{rgb}{1.0, 0.4, 0.4}
		\definecolor{lred}{rgb}{1.0, 0.8, 0.8}
		\definecolor{mlred}{rgb}{1.0, 0.6, 0.6}
		
		\definecolor{dgreen}{rgb}{0.0, 0.4, 0.0}
		\definecolor{mdgreen}{rgb}{0.0, 0.8, 0.0}
		\definecolor{lgreen}{rgb}{0.8, 1.0, 0.8}
		\definecolor{mlgreen}{rgb}{0.6, 1.0, 0.6}
		
		\definecolor{dblue}{rgb}{0.0, 0.0, 0.8}
		\definecolor{mdblue}{rgb}{0.4, 0.4, 1.0}
		\definecolor{lblue}{rgb}{0.8, 0.8, 1.0}
		\definecolor{mlblue}{rgb}{0.6, 0.6, 1.0}
		
		\definecolor{dyellow}{rgb}{0.4, 0.4, 0.0}
		\definecolor{mdyellow}{rgb}{0.8, 0.8, 0.4}
		\definecolor{lyellow}{rgb}{1.0, 1.0, 0.8}
		\definecolor{mlyellow}{rgb}{1.0, 1.0, 0.6}

		\node (A0) at (0,0) {};
	  \node (A1) at (4,0) {};
	  \node (A2) at (4,3) {};
	  \node (A3) at (0,3) {};
		\node (Am0) at (0.5,0.5) {};
		\node (Am1) at (3.5,0.5) {};
		\node (Am2) at (3.5,2.5) {};
		\node (Am3) at (0.5,2.5) {};
		\node (Ai0) at (1,1) {};
		\node (Ai1) at (3,1) {};
		\node (Ai2) at (3,2) {};
		\node (Ai3) at (1,2) {};
		
		\draw[line width=0.75pt,color=dgreen,fill=lgreen] (A0.center)--(A1.center)--(A2.center)--(A3.center)--(A0.center)--(A1.center);
		\draw[line width=0.75pt,color=dgreen,dashed] (Am0.center)--(Am1.center)--(Am2.center)--(Am3.center)--(Am0.center);
		\draw[line width=0.75pt,color=dgreen,dashed] (Ai0.center)--(Ai1.center)--(Ai2.center)--(Ai3.center)--(Ai0.center);
		
		\node at (-0.8, 2.7) {(a)};
		
		
		\node (B0) at (7,0) {};
	  \node (B1) at (11,0) {};
	  \node (B2) at (11,3) {};
	  \node (B3) at (7,3) {};
		\node (Bl0) at (8.3333,0) {};
		\node (Bl3) at (8.3333,3) {};
		\node (Br0) at (9.6667,0) {};
		\node (Br3) at (9.6667,3) {};
		
		\draw[line width=0.75pt,color=dyellow,fill=lyellow] (B0.center)--(B1.center)--(B2.center)--(B3.center)--(B0.center)--(B1.center);
		\draw[line width=0.75pt,color=dyellow,dashed] (Bl0.center)--(Bl3.center);
		\draw[line width=0.75pt,color=dyellow,dashed] (Br0.center)--(Br3.center);
		
		\draw[line width=0.74pt,color=black,fill=black] (Bl0.center) circle (0.08);
		\draw[line width=0.74pt,color=black,fill=black] (Bl3.center) circle (0.08);
		\draw[line width=0.74pt,color=black,fill=black] (Br0.center) circle (0.08);
		\draw[line width=0.74pt,color=black,fill=black] (Br3.center) circle (0.08);
		
		\node at (6.2, 2.7) {(b)};
		
		
		\node (C0) at (14,0) {};
	  \node (C1) at (18,0) {};
	  \node (C2) at (18,3) {};
	  \node (C3) at (14,3) {};
		\node (Cl01) at (14,1) {};
		\node (Cr01) at (18,1) {};
		\node (Cm23) at (16,3) {};
		\node (Cm) at (16,1) {};
		
		\draw[line width=0.75pt,color=dred,fill=lred] (C0.center)--(C1.center)--(C2.center)--(C3.center)--(C0.center)--(C1.center);
		\draw[line width=0.75pt,color=dred,dashed] (Cl01.center)--(Cm.center)--(Cr01.center);
		\draw[line width=0.75pt,color=dred,dashed] (Cm23.center)--(Cm.center);
		
		\draw[line width=0.74pt,color=black,fill=black] (Cl01.center) circle (0.08);
		\draw[line width=0.74pt,color=black,fill=black] (Cr01.center) circle (0.08);
		\draw[line width=0.74pt,color=black,fill=black] (Cm23.center) circle (0.08);

		\draw[line width=0.74pt,color=black,fill=black] (15.8,0.8)--(16.2,1.2);
		\draw[line width=0.74pt,color=black,fill=black] (15.8,1.2)--(16.2,0.8);
	
		\node at (13.2, 2.7) {(c)};
		
	\end{tikzpicture}	  
\end{center}
  \caption{\label{fig:junctionCrossPoint}%
	  Simple illustration of junctions, cross points, and dof multiplicity.
		(a)~No junctions, no cross points, $\mu_{\max} = 2$.
		(b)~Junctions ($\bullet$), but no cross points, $\mu_{\max} = 2$.
		(c)~Cross point ($\times$); for nodal discretization $\mu_{\max} > 2$, for Raviart-Thomas discretization $\mu_{\max} = 2$ still (see Example~\ref{ex:swappingRaviartThomas}).
  }
\end{figure}

\begin{remark}
\label{rem:TFETI}
  There also cases where $R_i$ is a zero-one matrix with at most one entry of $1$ per column and per row,
	but possibly with zero rows for \emph{phantom dofs}. Then
	\begin{align}
	  R_i R_i^T = \diag(\rho_\ell^{(i)})_{\ell=1}^{n_i}\,, \qquad R_i^T R_i = \diag(\mu_k^{(i)})_{k=1}^n\,,
	\end{align}
	where $\rho_\ell^{(i)} \in \{ 0, 1 \}$ indicates whether dof $\ell$ corresponds to a global dof or is a phantom
	\cite[Remark~2.3]{PechsteinDohrmann:2017a}.
	Such phantom dofs are used in the TFETI method \cite{DostalHorakKucera:2006a} and the all-floating BETI method
	\cite{Of:PhD,OfSteinbach:2009a}
	to detach Dirichlet boundary conditions from the local systems, see also \cite{Pechstein:FETIBook}.
	If $\rho_\ell^{(i)} = 0$ for some $\ell$, then $\range(R_i) \subsetneq U_i$, see also Remark~\ref{rem:localDirichlet}
	and Fig.~\ref{fig:detachDirichlet}.
\end{remark}

\subsection{Subdomain flux formulation}

Using the assembling property~\eqref{eq:assemblingProp}, problem~\eqref{eq:global} can be rewritten as
\[
  R^\top(A R \widehat u - f) = 0.
\]
Introducing the new variables $u = R \widehat u$ and $t = A u - f$ yields the following formulation.

\medskip

\noindent\fbox{\parbox{\textwidth}{
\emph{Subdomain flux formulation:}
\begin{align}
\label{eq:subdflux}
  \text{find } (u, t) \in \range(R) \times \ker(R^\top) \colon \qquad A u - t & = f.
\end{align}}}

\medskip

Note that the equation on the right is equivalent to $A_i u_i - t_i = f_i$ for all $i=1,\ldots,N$,
so we have one individual equation for each subdomain,
while the coupling between the subdomains is expressed through the (closed) spaces $\range(R)$ and $\ker(R^\top)$.
The following lemma clarifies the relation between the original problem~\eqref{eq:global} and formulation~\eqref{eq:subdflux}.

\begin{lemma}
\label{lem:subdflux}
  Let Assumption~\ref{ass:A1} hold. Then
	\begin{enumerate}
	\item[(i)] If $\widehat u$ is a solution of \eqref{eq:global} then $(R \widehat u, A R \widehat u - f)$ is a solution of \eqref{eq:subdflux}.
	\item[(ii)] If $(u, t)$ is a solution of \eqref{eq:subdflux} then there exists $\widehat u \in \widehat U$ such that $\widehat u$ solves \eqref{eq:global}
	  and $u = R \widehat u$ and $t = A R \widehat u - f$.
	\item[(iii)] The solution of \eqref{eq:subdflux} is unique.
	\item[(iv)] There exists a bounded linear solution operator $\mathcal{S} \colon f \mapsto (u, t)$ for \eqref{eq:subdflux}.
	\end{enumerate}
\end{lemma}
\begin{proof}
  Beforehand, note that due to \ref{ass:A1}, $\range(R)$ is a closed subspace of $U$.\\
  Part~(i) is proved already.\\
	Part~(ii): Since $u \in \range(R)$,
	  \ref{ass:A1} guarantees the existence of a unique function $\widehat u \in \widehat U$ with $R \widehat u = u$.
	Application of $R^\top$ and using that $t \in \ker(R^\top)$ yields
	$R^\top(A R \widehat u - f) = 0$ which is \eqref{eq:global}.\\
	Part~(iii): Recall from Sect.~\ref{sect:globalProblem} that $\ker(\widehat A) = \{0\}$.
	So $\widehat u$ is unique and by~(ii) also $u$ and $t$.\\
	Part~(iv): Recall from Sect.~\ref{sect:globalProblem} that $\widehat A$ has a bounded inverse.
	We define
	\[
	  \mathcal{S} \colon U^* \to \range(R) \times \ker(R^\top)
		\colon f \mapsto (R \widehat A^{-1} R^\top f, A R \widehat A^{-1} R^\top f - f).
	\]
	The following is easily verified:
	\begin{enumerate}
	\item[1)] $\mathcal{S}$ is well-defined, linear, and bounded,
	\item[2)] if $(u, t) = \mathcal{S} f$ then $A u - t = f$,
	\item[3)] if $(u, t) \in \range(R) \times \ker(R^\top)$	then $\mathcal{S}(A u - t) = (u, t)$.
	\end{enumerate}
	So $\mathcal{S}$ is a bounded solution operator for \eqref{eq:subdflux}.
\end{proof}

In block-operator notation, one can write
\begin{align}
  \mathcal{S} = \begin{bmatrix} I \\ A \end{bmatrix} R \widehat A^{-1} R^\top - \begin{bmatrix} 0 \\ I \end{bmatrix}.
\end{align}

Before moving on to characterizing the spaces $\range(R)$ (Sect.~\ref{sect:traces}) and $\ker(R^\top)$ (Sect.~\ref{sect:interfaceFluxes}),
we show that the variable $t$, introduced as a distribution on the whole space $U$
(i.e., acting on all subdomains), vanishes for \emph{bubble} functions and can thus be interpreted as a distribution acting on the interface.

\begin{definition}[bubble functions]
\label{def:bubble}
  On a subdomain $i$, the bubble space $U_{i,B}$ is given by
	\[
	  U_{i,B} := \{ v_i \in U_i \colon \exists \widehat v \in \widehat U \colon R_i \widehat v = v_i,\ R_j \widehat v = 0 \ \forall j \neq i \},
	\]
	i.e., it is the space of functions on subdomain $i$ that can be \emph{extended by zero} to a function in the global space $\widehat U$.
	The product space of bubble functions is given by
	\[
	  U_B := \productspace_{i=1}^N U_{i,B} \subseteq U.
	\]
\end{definition}

\begin{proposition}
\label{prop:tFlux}
  Any $t \in \ker(R^\top)$ fulfills
	\[
	  \langle t_i, v_{i,B} \rangle = 0 \qquad \forall v_{i,B} \in U_{i,B} \quad \forall i=1,\ldots,N.
	\]
\end{proposition}
\begin{proof}
  $R^\top t = 0$ implies $\langle t, R \widehat v \rangle = 0$ for all $\widehat v \in \widehat U$.
	For fixed $i$ and $v_{i,B} \in U_{i,B}$ there exists, by Definition~\ref{def:bubble},
	a function $\widehat v_B \in \widehat U$ with $R_i \widehat v_B = v_{i,B}$ and $R_j \widehat v_B = 0$ for all $j \neq 0$.
	Hence
	\[
	  0 = \langle t, R \widehat v_B \rangle = \langle t_i, R_i \widehat v_B \rangle = \langle t_i, v_{i,B} \rangle. \qedhere
	\]
\end{proof}

Due to the property in Proposition~\ref{prop:tFlux}, it is justified to call $t$ the \emph{subdomain flux}, cf.\ \eqref{eq:subdflux}.

\begin{example}[flux on a closed subdomain boundary]
  Consider Example~\ref{ex:modelProblemHelmholtz} for an \emph{interior} subdomain $\Omega_i$ that has no intersection with the outer boundary $\partial\Omega$.
	Then $t_i \in H^1(\Omega_i)^*$ vanishes on all functions from $H^1_0(\Omega_i)$, which is why we can represent it by a distribution in $H^{-1/2}(\partial\Omega_i)$.
	Integration by parts in the principal term shows that the very same distribution generalizes the normal derivative $\widehat u/\partial \normal_i$ on $\partial\Omega_i$,
	see in particular \cite[Lemma~4.3]{McLean:Book} and \cite[Lemma~1.2.1]{QuarteroniValli:Book}.
\end{example}

\begin{example}[flux with Neumann boundary]
  Consider Example~\ref{ex:modelProblemHelmholtz} for a subdomain $\Omega_i$
	where $\Gamma_{N,i} := \partial\Omega_i \cap \Gamma_N$ is connected and has positive surface measure
	and where $\partial\Omega_i \cap \Gamma_D = \emptyset$, $\partial\Omega_i \cap \Gamma_R = \emptyset$.
	Then, since the bubble functions do have support on $\Gamma_{N_i}$,
	the flux $t_i$	vanishes on $\Gamma_{i,N}$.
	To be precise, the trace of a bubble function on $\Gamma_{N,i}$
	is in the Lions-Magenes space $H^{1/2}_{00}(\Gamma_{N,i})$, see e.g.\ \cite[Appendix~A]{ToselliWidlund:Book}
	(which is often denoted by $\widetilde H^{1/2}(\Gamma_{N,i})$, cf.\ \cite[Ch.~3]{McLean:Book}).
	This is why $t_i$ can be represented by an element in the dual of $H^{1/2}(\partial\Omega_i \setminus \Gamma_{N_i})$,
	i.e., $t_i$ is a distribution supported on $\partial\Omega_i \setminus \Gamma_{N,i}$ that can be \emph{extended by zero} to $H^{-1/2}(\partial\Omega_i)$.
	A discrete analogon is illustrated in Fig.~\ref{fig:fluxIllustr}(a).
\end{example}

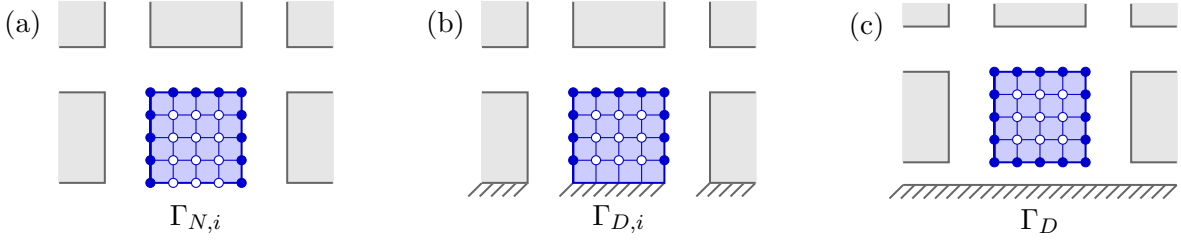
\begin{figure}
\begin{center}
  \begin{tikzpicture}
		\pgftransformscale{0.6}
				
		\definecolor{dblue}{rgb}{0.0, 0.0, 0.8}
		\definecolor{mdblue}{rgb}{0.4, 0.4, 1.0}
		\definecolor{lblue}{rgb}{0.8, 0.8, 1.0}
		\definecolor{mlblue}{rgb}{0.6, 0.6, 1.0}
		
		\definecolor{dgrey}{rgb}{0.4, 0.4, 0.4}
		\definecolor{mdgrey}{rgb}{0.55, 0.55, 0.55}
		\definecolor{lgrey}{rgb}{0.9, 0.9, 0.9}
		\definecolor{mlgrey}{rgb}{0.75, 0.75, 0.75}
	  
		
		\node at (-0.8,3.5) {(a)};
		
		\draw[line width=0.75pt,color=dblue,fill=lblue] (2,0)--(2,2)--(4,2)--(4,0)--(2,0)--(2,2);
		\foreach \x in {2.5, 3, 3.5} { \draw[line width=0.4pt,color=dblue] (\x,0)--(\x,2); }
		\foreach \x in {0.5, 1, 1.5} { \draw[line width=0.4pt,color=dblue] (2,\x)--(4,\x); }
		
		\foreach \x in {0, 0.5, ..., 2} { \draw[line width=0.4pt,color=dblue,fill=dblue] (2,\x) circle (0.1); }
		\foreach \x in {0, 0.5, ..., 1.5} { \draw[line width=0.4pt,color=dblue,fill=white] (2.5,\x) circle (0.1); }
		\foreach \x in {0, 0.5, ..., 1.5} { \draw[line width=0.4pt,color=dblue,fill=white] (3,\x) circle (0.1); }
		\foreach \x in {0, 0.5, ..., 1.5} { \draw[line width=0.4pt,color=dblue,fill=white] (3.5,\x) circle (0.1); }
		\foreach \x in {0, 0.5, ..., 2} { \draw[line width=0.4pt,color=dblue,fill=dblue] (4,\x) circle (0.1); }
		\foreach \x in {2.5, 3, 3.5} { \draw[line width=0.4pt,color=dblue,fill=dblue] (\x,2) circle (0.1); }
		
		\draw[line width=0.75pt,color=lgrey,fill=lgrey] (0,0)--(1,0)--(1,2)--(0,2)--(0,0);
		\draw[line width=0.75pt,color=dgrey] (0,0)--(1,0)--(1,2)--(0,2);
		
		\draw[line width=0.75pt,color=lgrey,fill=lgrey] (2,3)--(4,3)--(4,4)--(2,4)--(2,3);
		\draw[line width=0.75pt,color=dgrey] (2,4)--(2,3)--(4,3)--(4,4);

		\draw[line width=0.75pt,color=lgrey,fill=lgrey] (0,3)--(1,3)--(1,4)--(0,4)--(0,3);
		\draw[line width=0.75pt,color=dgrey] (0,3)--(1,3)--(1,4);
		
		\draw[line width=0.75pt,color=lgrey,fill=lgrey] (5,3)--(6,3)--(6,4)--(5,4)--(5,3);
		\draw[line width=0.75pt,color=dgrey] (5,4)--(5,3)--(6,3);
		
		\draw[line width=0.75pt,color=lgrey,fill=lgrey] (5,0)--(6,0)--(6,2)--(5,2)--(5,0);
		\draw[line width=0.75pt,color=dgrey] (6,0)--(5,0)--(5,2)--(6,2);
		
		
		\node at (3,-0.8) {$\Gamma_{N,i}$};

	\end{tikzpicture}
	\hspace{5ex}
  \begin{tikzpicture}
		\pgftransformscale{0.6}
		
		\definecolor{dblue}{rgb}{0.0, 0.0, 0.8}
		\definecolor{mdblue}{rgb}{0.4, 0.4, 1.0}
		\definecolor{lblue}{rgb}{0.8, 0.8, 1.0}
		\definecolor{mlblue}{rgb}{0.6, 0.6, 1.0}
		
		\definecolor{dgrey}{rgb}{0.4, 0.4, 0.4}
		\definecolor{mdgrey}{rgb}{0.55, 0.55, 0.55}
		\definecolor{lgrey}{rgb}{0.9, 0.9, 0.9}
		\definecolor{mlgrey}{rgb}{0.75, 0.75, 0.75}
	  
		
		\node at (-0.8,3.5) {(b)};
		
		\foreach \x in {0,0.25,...,1} { \draw[line width=0.75pt,color=dgrey] (0 + \x, 0)--(0 + \x - 0.3, 0 - 0.3); }
		\foreach \x in {0,0.25,...,2} { \draw[line width=0.75pt,color=dgrey] (2 + \x, 0)--(2 + \x - 0.3, 0 - 0.3); }
		\foreach \x in {0,0.25,...,1} { \draw[line width=0.75pt,color=dgrey] (5 + \x, 0)--(5 + \x - 0.3, 0 - 0.3); }
		
		\draw[line width=0.75pt,color=dblue,fill=lblue] (2,0)--(2,2)--(4,2)--(4,0)--(2,0)--(2,2);
				
		\foreach \x in {2.5, 3, 3.5} { \draw[line width=0.4pt,color=dblue] (\x,0)--(\x,2); }
		\foreach \x in {0.5, 1, 1.5} { \draw[line width=0.4pt,color=dblue] (2,\x)--(4,\x); }
		
		\foreach \x in {0.5, 1, ..., 2} { \draw[line width=0.4pt,color=dblue,fill=dblue] (2,\x) circle (0.1); }
		\foreach \x in {0.5, 1, ..., 1.5} { \draw[line width=0.4pt,color=dblue,fill=white] (2.5,\x) circle (0.1); }
		\foreach \x in {0.5, 1, ..., 1.5} { \draw[line width=0.4pt,color=dblue,fill=white] (3,\x) circle (0.1); }
		\foreach \x in {0.5, 1, ..., 1.5} { \draw[line width=0.4pt,color=dblue,fill=white] (3.5,\x) circle (0.1); }
		\foreach \x in {0.5, 1, ..., 2} { \draw[line width=0.4pt,color=dblue,fill=dblue] (4,\x) circle (0.1); }
		\foreach \x in {2.5, 3, 3.5} { \draw[line width=0.4pt,color=dblue,fill=dblue] (\x,2) circle (0.1); }
		
		\draw[line width=0.75pt,color=lgrey,fill=lgrey] (0,0)--(1,0)--(1,2)--(0,2)--(0,0);
		\draw[line width=0.75pt,color=dgrey] (0,0)--(1,0)--(1,2)--(0,2);
		
		\draw[line width=0.75pt,color=lgrey,fill=lgrey] (2,3)--(4,3)--(4,4)--(2,4)--(2,3);
		\draw[line width=0.75pt,color=dgrey] (2,4)--(2,3)--(4,3)--(4,4);

		\draw[line width=0.75pt,color=lgrey,fill=lgrey] (0,3)--(1,3)--(1,4)--(0,4)--(0,3);
		\draw[line width=0.75pt,color=dgrey] (0,3)--(1,3)--(1,4);
		
		\draw[line width=0.75pt,color=lgrey,fill=lgrey] (5,3)--(6,3)--(6,4)--(5,4)--(5,3);
		\draw[line width=0.75pt,color=dgrey] (5,4)--(5,3)--(6,3);
		
		\draw[line width=0.75pt,color=lgrey,fill=lgrey] (5,0)--(6,0)--(6,2)--(5,2)--(5,0);
		\draw[line width=0.75pt,color=dgrey] (6,0)--(5,0)--(5,2)--(6,2);
		

		\node at (3,-0.8) {$\Gamma_{D,i}$};

	\end{tikzpicture}
	\hspace{5ex}
  \begin{tikzpicture}
		\pgftransformscale{0.6}
		
		\definecolor{dblue}{rgb}{0.0, 0.0, 0.8}
		\definecolor{mdblue}{rgb}{0.4, 0.4, 1.0}
		\definecolor{lblue}{rgb}{0.8, 0.8, 1.0}
		\definecolor{mlblue}{rgb}{0.6, 0.6, 1.0}
		
		\definecolor{dgrey}{rgb}{0.4, 0.4, 0.4}
		\definecolor{mdgrey}{rgb}{0.55, 0.55, 0.55}
		\definecolor{lgrey}{rgb}{0.9, 0.9, 0.9}
		\definecolor{mlgrey}{rgb}{0.75, 0.75, 0.75}
	  
		
		\node at (-0.8,3) {(c)};
		
		\foreach \x in {0,0.25,...,6} { \draw[line width=0.75pt,color=dgrey] (0 + \x, -0.5)--(0 + \x - 0.3, -0.5 - 0.3); }
		\draw[line width=0.75pt,color=dgrey] (0, -0.5)--(6, -0.5);
		
		\draw[line width=0.75pt,color=dblue,fill=lblue] (2,0)--(2,2)--(4,2)--(4,0)--(2,0)--(2,2);
				
		\foreach \x in {2.5, 3, 3.5} { \draw[line width=0.4pt,color=dblue] (\x,0)--(\x,2); }
		\foreach \x in {0.5, 1, 1.5} { \draw[line width=0.4pt,color=dblue] (2,\x)--(4,\x); }
		
		\foreach \x in {0, 0.5, ..., 2} { \draw[line width=0.4pt,color=dblue,fill=dblue] (2,\x) circle (0.1); }
		\foreach \x in {0.5, 1, ..., 1.5} { \draw[line width=0.4pt,color=dblue,fill=white] (2.5,\x) circle (0.1); }
		\foreach \x in {0.5, 1, ..., 1.5} { \draw[line width=0.4pt,color=dblue,fill=white] (3,\x) circle (0.1); }
		\foreach \x in {0.5, 1, ..., 1.5} { \draw[line width=0.4pt,color=dblue,fill=white] (3.5,\x) circle (0.1); }
		\foreach \x in {0, 0.5, ..., 2} { \draw[line width=0.4pt,color=dblue,fill=dblue] (4,\x) circle (0.1); }
		\foreach \x in {2.5, 3, 3.5} { \draw[line width=0.4pt,color=dblue,fill=dblue] (\x,2) circle (0.1); }
		\foreach \x in {2.5, 3, 3.5} { \draw[line width=0.4pt,color=dblue,fill=dblue] (\x,0) circle (0.1); }
		
		\draw[line width=0.75pt,color=lgrey,fill=lgrey] (0,0)--(1,0)--(1,2)--(0,2)--(0,0);
		\draw[line width=0.75pt,color=dgrey] (0,0)--(1,0)--(1,2)--(0,2);
		
		\draw[line width=0.75pt,color=lgrey,fill=lgrey] (2,3)--(4,3)--(4,3.5)--(2,3.5)--(2,3);
		\draw[line width=0.75pt,color=dgrey] (2,3.5)--(2,3)--(4,3)--(4,3.5);

		\draw[line width=0.75pt,color=lgrey,fill=lgrey] (0,3)--(1,3)--(1,3.5)--(0,3.5)--(0,3);
		\draw[line width=0.75pt,color=dgrey] (0,3)--(1,3)--(1,3.5);
		
		\draw[line width=0.75pt,color=lgrey,fill=lgrey] (5,3)--(6,3)--(6,3.5)--(5,3.5)--(5,3);
		\draw[line width=0.75pt,color=dgrey] (5,3.5)--(5,3)--(6,3);
		
		\draw[line width=0.75pt,color=lgrey,fill=lgrey] (5,0)--(6,0)--(6,2)--(5,2)--(5,0);
		\draw[line width=0.75pt,color=dgrey] (6,0)--(5,0)--(5,2)--(6,2);
		

		\node at (3,-1.3) {$\Gamma_D$};

	\end{tikzpicture}
\end{center}
\caption{\label{fig:fluxIllustr}%
   Support of subdomain flux for a simple nodal discretization for three cases
	 ($\circ$~dofs where $t_i$ vanishes, $\bullet$~interface dofs, where $t_i$ is supported).
}
\end{figure}

\begin{example}[flux with Dirichlet boundary]
  Consider Example~\ref{ex:modelProblemHelmholtz} for a subdomain $\Omega_i$
	where $\Gamma_{D,i} := \partial\Omega_i \cap \Gamma_D$ is connected and has positive surface measure
	and where $\partial\Omega_i \cap \Gamma_N = \emptyset$, $\partial\Omega_i \cap \Gamma_R = \emptyset$.
	We treat two choices for the local space $U_i$, see also Remark~\ref{rem:localDirichlet}.\\
	(i) $U_i = H^1_D(\Omega_i)$.
	Since the flux $t_i \in H^1_D(\Omega_i)^*$ vanishes on functions from $U_{i,B} = H^1_0(\Omega_i)$,
	we conclude that $t_i$ can be represented by an element in the dual of $H^{1/2}_{00}(\partial\Omega_i \setminus \Gamma_D)$,
	i.e., it is a distribution supported on $\partial\Omega_i \setminus \Gamma_D$ that \emph{cannot necessarily be extended by zero} to $H^{-1/2}(\partial\Omega_i)$.
	For a discrete analogon see Fig.~\ref{fig:fluxIllustr}(b).\\
	(ii) $U_i = H^1(\Omega_i)$. Still, due to Definition~\ref{def:bubble}, $U_{i,B} = H^1_0(\Omega_i)$.
	Since locally, we work on the full space $U_i = H^1(\Omega_i)$, the flux $t_i$ is defined in $H^1(\Omega_i)^*$
	and it vanishes on functions from $H^1_0(\Omega_i)$.
	We conclude that $t_i$ can be represented by a distribution in $H^{-1/2}(\partial\Omega_i)$
	corresponding to the normal derivative $\widehat u/\partial \normal_i$ on the \emph{entire} boundary $\partial\Omega_i$.
	For a discrete analogon see Fig.~\ref{fig:fluxIllustr}(c).
\end{example}

\begin{example}[flux for Maxwell's equations]
  For the $E$-field formulation of Maxwell's equations (see Table~\ref{tab:variationalExamples}),
	the variable $t_i$ represents the tangential trace of $\mu^{-1} \opcurl\,\uv_i$,
	which is (up to a factor of $\pm \complexi\omega$ and a possible rotation by $90^\circ$) the electric surface current on the interface.
\end{example}

\begin{example}[discrete flux]
  In the discrete case (Sect.~\ref{sect:discreteMatrixNotation}),
	for a global dof $k$ shared by subdomains $\mathcal{N}_k$, the $k$-th row of the condition $R^\top t = 0$ reads
	\[
	  \sum_{i \in \mathcal{N}_k} t_{i,\mathsf{g}_i^{-1}(k)} = 0,
	\]
	where $\mathsf{g}_i^{-1}(k)$ denotes the local dof on subdomain $i$ corresponding to the global dof $k$.
	If $k$ is shared by two subdomains, the two fluxes must have opposite sign.
	In general, the fluxes must add up to zero, see Figure~\ref{fig:transmCrosspoint}.
\end{example}

\begin{figure}
\begin{center}
  \def\svgwidth{0.8\textwidth}
  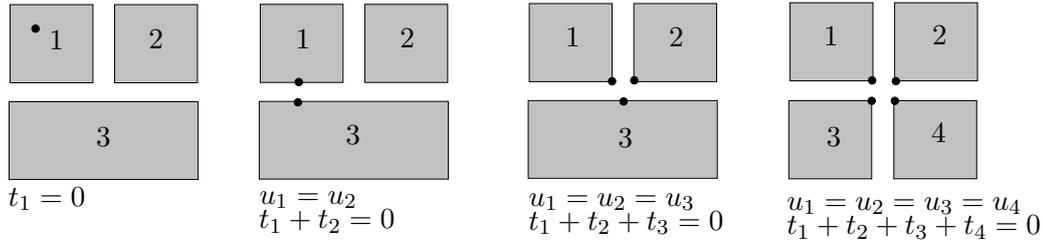
  \caption{\label{fig:transmCrosspoint}%
	  Sketch of the conditions $u \in \range(R)$, $t \in \ker(R^\top)$ in the discrete case
		for an individual global dof shared by one up to four subdomains.
		Black bullets ($\bullet$) indicate the location of the corresponding local dofs.
	}
\end{center}  
\end{figure}

\subsection{Traces}
\label{sect:traces}

Formulation~\ref{eq:subdflux} can be read as: find $(u, t) \in U \times U^*$ with
\[
  A u - t = f, \qquad u \in \range(R), \qquad R^\top t = 0.
\]
Opposed to the two equations,
the condition $u \in \range(R)$ is inconvenient for any algorithmic development,
and we will characterize it using \emph{trace operators}.
For the abstract framework of this paper, we assume the following for each subdomain $i=1,\ldots,N$:
\begin{enumerate}
\item[(i)] \emph{Local trace space} $\Lambda_i$, a Hilbert space
  with inner product $(\cdot,\cdot)_{\Lambda_i}$ and norm $\| \cdot \|_{\Lambda_i}$.
	If $\widehat U$ is finite-/infinite-dimensional, real/complexified then $\Lambda_i$ shares the same property.	
\item[(ii)] \emph{Local trace operator} $T_i \colon U_i \to \Lambda_i$, linear, bounded, and real-valued.
\end{enumerate}
We define the \emph{compound trace space} $\Lambda := \productspace_{i=1}^N \Lambda_i$, equipped with the natural inner product
$(\lambda, \mu)_\Lambda := \sum_{i=1}^N (\lambda_i, \mu_i)_{\Lambda_i}$ and corresponding norm $\| \cdot \|_{\Lambda}$,
as well as the \emph{compound trace operator}
\[
  T = \text{diag}(T_i)_{i=1}^N \colon U \to \Lambda.
\]
Before moving on to the next ingredient, the interface exchange operator, we study two important examples of trace spaces for $H^1$-formulations.

\begin{example}[natural trace operator]
\label{ex:tracesH1Surj}
  For the case $\widehat U = H^1(\Omega)$, let $\overline \Omega = \bigcup_{i=1}^N \overline\Omega_i$
	be a non-overlapping subdomain decomposition with sufficiently smooth boundaries and interfaces.
  For an interior subdomain (with positive distance from the global boundary $\partial\Omega$),
	the natural trace space is $H^{1/2}(\partial\Omega_i)$,
	and the associated trace operator $T_i \colon H^1(\Omega_i) \to H^{1/2}(\partial\Omega_i)$ is surjective,
	see e.g., \cite{ToselliWidlund:Book,Pechstein:FETIBook},
	for an illustration see Fig.~\ref{fig:traceDiagrams}(b).
  We will return to this type of choice in Sect.~\ref{sect:interfaceExchange}.
	Suppose now that we have Dirichlet conditions and that and $\partial\Omega_i \cap \Gamma_D$ is connected and has positive surface measure.
	If we wish to use $U_i = H^1_D(\Omega_i)$ then the natural trace space is
  the Lions-Magenes space $H^{1/2}_{00}(\partial\Omega_i \setminus \Gamma_D)$.
	However, we are also free to choose $U_i = H^1(\Omega_i)$ and use $H^{1/2}(\partial\Omega_i)$ as trace space.
	In the discretized case, we can use $T_i$ as the zero-one restriction matrix that selects the dofs of subdomain $i$ with multiplicity $\ge 2$
	(and possibly phantom dofs, see Remark~\ref{rem:TFETI}).
\end{example}

\begin{example}[collective trace operator]
\label{ex:tracesH1Split}
  In the classical works \cite{Lions:DD03,Despres:PhD,CollinoGhanemiJoly:2000a},
	the interface is split into \emph{facets}\footnote{In the literature, one often reads of \emph{faces},
	         in the two-dimensional case of \emph{edges}.}
	$F_{ij}$, which are open manifolds of one dimension lower than $\Omega$,
  form the interface between two subdomains, i.e., $\overline F_{ij} = \partial\Omega_i \cap \partial\Omega_j$,
  and have a non-trivial surface measure, cf.\ Fig.~\ref{fig:facetTraces}, left.
	Let $\mathcal{F}_i$ denote the facets of $\Omega_i$.
  For $\widehat U$ and $\{ \Omega_i \}$ as in Example~\ref{ex:tracesH1Surj},
  we can define a trace operator $T_{iF} \colon H^1(\Omega_i) \to L^2(F)$ for each facet $F \in \mathcal{F}_i$,
  and then define the subdomain trace operator
	\[
	  T_i \colon H^1(\Omega_i) \to \Lambda_i := \productspace\nolimits_{F \in \mathcal{F}_i} L^2(F), \quad
		T_i u_i := (T_{iF} u_i)_{F \in \mathcal{F}_i}
	\]
	of \emph{collective} type.
	In these definitions, we can replace $L^2(F)$ by $H^s(F)$ for $s \in [0, \tfrac{1}{2}]$, cf.\ \cite{CollinoGhanemiJoly:2000a}.
	Note that the compound trace space $\Lambda = \productspace_{i=1} \Lambda_i$ has \emph{two} instances of spaces on each facet (see Fig.~\ref{fig:traceDiagrams}(a)),
	which is a feature to be used a lot later on.
	Note also that so far, we have only dealt with \emph{interior} facets shared by two subdomains
	  (marked in blue in Fig.~\ref{fig:facetTraces}).
  However, there may be \emph{exterior} facets (marked in grey in Fig.~\ref{fig:facetTraces}) that only belong to one subdomain only:
	either Dirichlet facets	(see Remark~\ref{rem:localDirichlet} and Fig.~\ref{fig:fluxIllustr}(c))
	or \emph{auxiliary} facets, where we wish to evaluate traces for some other reason. 
	In the discretized case, $T_{iF} \colon U_i \to \Lambda_F$ is the zero-one matrix that selects the dofs of subdomain $i$ that are associated with the facet $F$,
	and we can set $\Lambda_i := \productspace_{F \in \mathcal{F}_i} \Lambda_F$ and define $T_i$ as above.
	This can be done for any conforming finite element discretization of $H^1(\Omega)$, $\Hv(\opcurl,\Omega)$, and $\Hv(\opdiv,\Omega)$,
	see also Examples~\ref{ex:swappingRaviartThomas}--\eqref{ex:swappingNedelec} below.
	Note that if a dof is associated with more than one facet (as it happens for cross point dofs), then the collective trace operator $T_i$
	creates multiple copies of that dof (see also Sect.~\ref{sect:bilateralDiscrFacetSys} below) and is not surjective.
\end{example}

\begin{figure}
\begin{center}
  \begin{tikzpicture}
		\definecolor{dcyan}{rgb}{0.0,0.4,0.5}
		\definecolor{dgrey}{rgb}{0.4,0.4,0.4}
	
	  \node (A0) at (0,0) {};
	  \node (A1) at (2,0) {};
	  \node (A2) at (4,0) {};
	  \node (A3) at (6,0) {};
		\node (A4) at (0,2) {};
	  \node (A5) at (2,2) {};
	  \node (A6) at (4,2) {};
	  \node (A7) at (6,2) {};
		\node (A8) at (0,4) {};
	  \node (A9) at (2,4) {};
	  \node (A10) at (4,4) {};
	  \node (A11) at (6,4) {};
	
	  \draw[line width=0.75pt] (A2.center)--(A3.center)--(A11.center)--(A8.center)--(A0.center)--(A1.center);
		\draw[line width=0.75pt, draw=dgrey] (A1.center)--(A2.center);
		\draw[dashed, line width=0.75pt, draw=dcyan] (A1.center)--(A9.center);
		\draw[dashed, line width=0.75pt, draw=dcyan] (A2.center)--(A10.center);
		\draw[dashed, line width=0.75pt, draw=dcyan] (A4.center)--(A7.center);
		
		\node at (0.7,0.7) {$\Omega_1$};
		\node at (2.7,0.7) {$\Omega_2$};
		\node at (4.7,0.7) {$\Omega_3$};
		\node at (0.7,2.7) {$\Omega_4$};
		\node at (2.7,2.7) {$\Omega_5$};
		\node at (4.7,2.7) {$\Omega_6$};
		
		\node at (1,1.7) {$\textcolor{dcyan}{F_{14}}$};
		\node at (3,1.7) {$\textcolor{dcyan}{F_{25}}$};
		\node at (5,1.7) {$\textcolor{dcyan}{F_{36}}$};
		\node at (1.6,1) {$\textcolor{dcyan}{F_{12}}$};
		\node at (1.6,3) {$\textcolor{dcyan}{F_{45}}$};
		\node at (3.6,1) {$\textcolor{dcyan}{F_{23}}$};
		\node at (3.6,3) {$\textcolor{dcyan}{F_{56}}$};
		
		\node at (3,-0.3) {$\textcolor{dgrey}{F_{20}}$};
		
	\end{tikzpicture}
  \begin{tikzpicture}
	  \definecolor{dcyan}{rgb}{0.0,0.4,0.5}
		\definecolor{dgrey}{rgb}{0.4,0.4,0.4}
		
		\node at (0,0) {};
	  \node (A0) at (0.6,1) {};
	  \node (A1) at (2.6,1) {};
	  \node (A2) at (0.6,3) {};
	  \node (A3) at (2.6,3) {};

		\node at (1.7,1.7) {$\Omega_2$};
		
		\node at (1.5,3.5) {$U_2 = H^1(\Omega_2)$};

	  \draw[line width=0.75pt] (A0.center)--(A1.center)--(A3.center)--(A2.center)--(A0.center);
		
		\node at (3.6,2.5) {$T_2$};
		\node at (3.6,2) {$\longrightarrow$};
		
		\node at (4.5,3.5) {$\Lambda_2$};
		
		\draw[line width=0.75pt, draw=dcyan] (5,0.6)--(5,2.6);
		\node at (4.5,0.3) {$\textcolor{dcyan}{L^2(F_{12})}$};
		
		\draw[line width=0.75pt, draw=dcyan] (5.4,3.0)--(7.4,3.0);
		\node at (6.4,3.4) {$\textcolor{dcyan}{L^2(F_{25})}$};
		
		\draw[line width=0.75pt, draw=dcyan] (7.8,0.6)--(7.8,2.6);
		\node at (8.3,0.3) {$\textcolor{dcyan}{L^2(F_{12})}$};
		
		\draw[line width=0.75pt, draw=dgrey] (5.4,0.0)--(7.4,0.0);
		\node at (6.4,-0.3) {$\textcolor{dgrey}{L^2(F_{20})}$};
		
	\end{tikzpicture}

  \caption{\label{fig:facetTraces}%
	  \emph{Left:} Geometric sketch of facets for a two-dimensional subdomain decomposition.
		\emph{Right:} Sketch of the subdomain trace operator for Example~\ref{ex:tracesH1Split}.
	}
\end{center}
\end{figure}
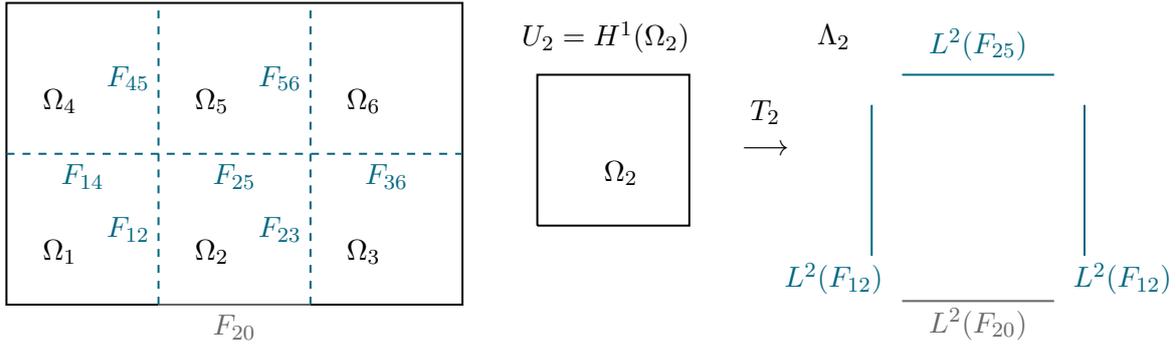

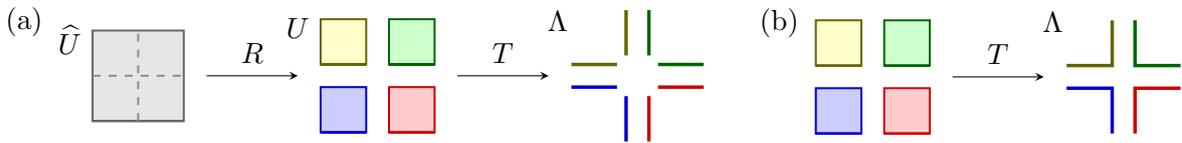
\begin{figure}
\begin{center}
  \begin{tikzpicture}
	  \pgftransformscale{0.6}
		
	  \definecolor{dred}{rgb}{0.8, 0.0, 0.0}
		\definecolor{mdred}{rgb}{1.0, 0.4, 0.4}
		\definecolor{lred}{rgb}{1.0, 0.8, 0.8}
		\definecolor{mlred}{rgb}{1.0, 0.6, 0.6}
		
		\definecolor{dgreen}{rgb}{0.0, 0.4, 0.0}
		\definecolor{mdgreen}{rgb}{0.0, 0.8, 0.0}
		\definecolor{lgreen}{rgb}{0.8, 1.0, 0.8}
		\definecolor{mlgreen}{rgb}{0.6, 1.0, 0.6}
		
		\definecolor{dblue}{rgb}{0.0, 0.0, 0.8}
		\definecolor{mdblue}{rgb}{0.4, 0.4, 1.0}
		\definecolor{lblue}{rgb}{0.8, 0.8, 1.0}
		\definecolor{mlblue}{rgb}{0.6, 0.6, 1.0}
		
		\definecolor{dyellow}{rgb}{0.4, 0.4, 0.0}
		\definecolor{mdyellow}{rgb}{0.8, 0.8, 0.4}
		\definecolor{lyellow}{rgb}{1.0, 1.0, 0.8}
		\definecolor{mlyellow}{rgb}{1.0, 1.0, 0.6}
		
		\definecolor{dgrey}{rgb}{0.4, 0.4, 0.4}
		\definecolor{mdgrey}{rgb}{0.55, 0.55, 0.55}
		\definecolor{lgrey}{rgb}{0.9, 0.9, 0.9}
		\definecolor{mlgrey}{rgb}{0.75, 0.75, 0.75}
		
		\draw[line width=0.75pt,color=dgrey,fill=lgrey] (1,0)--(3,0)--(3,2)--(1,2)--(1,0)--(3,0);
		\draw[line width=0.75pt,color=mdgrey,dashed] (1,1)--(3,1);
		\draw[line width=0.75pt,color=mdgrey,dashed] (2,0)--(2,2);
		
		\node at (0.5,1.8) {$\widehat U$};
		
		
	  \node (G0) at (6,-0.25) {};
	  \node (G1) at (7,-0.25) {};
	  \node (G2) at (7,0.75) {};
	  \node (G3) at (6,0.75) {};
		
		\draw[line width=0.75pt,color=dblue,fill=lblue] (G0.center)--(G1.center)--(G2.center)--(G3.center)--(G0.center)--(G1.center);
		
		\node (H0) at (7.5,-0.25) {};
	  \node (H1) at (8.5,-0.25) {};
	  \node (H2) at (8.5,0.75) {};
	  \node (H3) at (7.5,0.75) {};
		
		\draw[line width=0.75pt,color=dred,fill=lred] (H0.center)--(H1.center)--(H2.center)--(H3.center)--(H0.center)--(H1.center);
		
		\node (I0) at (6,1.25) {};
	  \node (I1) at (7,1.25) {};
	  \node (I2) at (7,2.25) {};
	  \node (I3) at (6,2.25) {};
		
		\draw[line width=0.75pt,color=dyellow,fill=lyellow] (I0.center)--(I1.center)--(I2.center)--(I3.center)--(I0.center)--(I1.center);
		
		\node (J0) at (7.5,1.25) {};
	  \node (J1) at (8.5,1.25) {};
	  \node (J2) at (8.5,2.25) {};
	  \node (J3) at (7.5,2.25) {};
		
		\draw[line width=0.75pt,color=dgreen,fill=lgreen] (J0.center)--(J1.center)--(J2.center)--(J3.center)--(J0.center)--(J1.center);
		
		
		\node at (5.5,2.05) {$U$};
		
		\draw[->, >=stealth] (3.5,1)--(5.5,1);
		\node at (4.5,1.5) {$R$};
		
		
		\draw[line width=1.3pt,color=dblue] (11.5,0.75)--(12.5,0.75);
		\draw[line width=1.3pt,color=dblue] (12.7,-0.45)--(12.7,0.55);
		
		\draw[line width=1.3pt,color=dyellow] (11.5,1.25)--(12.5,1.25);
		\draw[line width=1.3pt,color=dyellow] (12.7,1.45)--(12.7,2.45);
		
		\draw[line width=1.3pt,color=dred] (13.4,0.75)--(14.4,0.75);
		\draw[line width=1.3pt,color=dred] (13.2,-0.45)--(13.2,0.55);
		
		\draw[line width=1.3pt,color=dgreen] (13.4,1.25)--(14.4,1.25);
		\draw[line width=1.3pt,color=dgreen] (13.2,1.45)--(13.2,2.45);
		
		
		\node at (11.2,2.2) {$\Lambda$};
		
		\draw[->, >=stealth] (9,1)--(11,1);
		\node at (10,1.5) {$T$};


		\node at (-0.5,2.2) {(a)};
		
  \end{tikzpicture}
	\hspace{2ex}
	\begin{tikzpicture}
	  \pgftransformscale{0.6}
		
	  \definecolor{dred}{rgb}{0.8, 0.0, 0.0}
		\definecolor{mdred}{rgb}{1.0, 0.4, 0.4}
		\definecolor{lred}{rgb}{1.0, 0.8, 0.8}
		\definecolor{mlred}{rgb}{1.0, 0.6, 0.6}
		
		\definecolor{dgreen}{rgb}{0.0, 0.4, 0.0}
		\definecolor{mdgreen}{rgb}{0.0, 0.8, 0.0}
		\definecolor{lgreen}{rgb}{0.8, 1.0, 0.8}
		\definecolor{mlgreen}{rgb}{0.6, 1.0, 0.6}
		
		\definecolor{dblue}{rgb}{0.0, 0.0, 0.8}
		\definecolor{mdblue}{rgb}{0.4, 0.4, 1.0}
		\definecolor{lblue}{rgb}{0.8, 0.8, 1.0}
		\definecolor{mlblue}{rgb}{0.6, 0.6, 1.0}
		
		\definecolor{dyellow}{rgb}{0.4, 0.4, 0.0}
		\definecolor{mdyellow}{rgb}{0.8, 0.8, 0.4}
		\definecolor{lyellow}{rgb}{1.0, 1.0, 0.8}
		\definecolor{mlyellow}{rgb}{1.0, 1.0, 0.6}
		
		\definecolor{dgrey}{rgb}{0.4, 0.4, 0.4}
		\definecolor{mdgrey}{rgb}{0.55, 0.55, 0.55}
		\definecolor{lgrey}{rgb}{0.9, 0.9, 0.9}
		\definecolor{mlgrey}{rgb}{0.75, 0.75, 0.75}
				
		
	  \node (G0) at (6,-0.25) {};
	  \node (G1) at (7,-0.25) {};
	  \node (G2) at (7,0.75) {};
	  \node (G3) at (6,0.75) {};
		
		\draw[line width=0.75pt,color=dblue,fill=lblue] (G0.center)--(G1.center)--(G2.center)--(G3.center)--(G0.center)--(G1.center);
		
		\node (H0) at (7.5,-0.25) {};
	  \node (H1) at (8.5,-0.25) {};
	  \node (H2) at (8.5,0.75) {};
	  \node (H3) at (7.5,0.75) {};
		
		\draw[line width=0.75pt,color=dred,fill=lred] (H0.center)--(H1.center)--(H2.center)--(H3.center)--(H0.center)--(H1.center);
		
		\node (I0) at (6,1.25) {};
	  \node (I1) at (7,1.25) {};
	  \node (I2) at (7,2.25) {};
	  \node (I3) at (6,2.25) {};
		
		\draw[line width=0.75pt,color=dyellow,fill=lyellow] (I0.center)--(I1.center)--(I2.center)--(I3.center)--(I0.center)--(I1.center);
		
		\node (J0) at (7.5,1.25) {};
	  \node (J1) at (8.5,1.25) {};
	  \node (J2) at (8.5,2.25) {};
	  \node (J3) at (7.5,2.25) {};
		
		\draw[line width=0.75pt,color=dgreen,fill=lgreen] (J0.center)--(J1.center)--(J2.center)--(J3.center)--(J0.center)--(J1.center);
		
				
		\draw[line width=1.3pt,color=dblue] (11.5,0.75)--(12.5,0.75)--(12.5,-0.25);
		
		\draw[line width=1.3pt,color=dyellow] (11.5,1.25)--(12.5,1.25)--(12.5,2.25);
		
		\draw[line width=1.3pt,color=dred] (13.0,-0.25)--(13.0,0.75)--(14.0,0.75);
		
		\draw[line width=1.3pt,color=dgreen] (13.0,2.25)--(13.0,1.25)--(14.0,1.25);
		
		
		\node at (11.2,2.2) {$\Lambda$};
		
		\draw[->, >=stealth] (9,1)--(11,1);
		\node at (10,1.5) {$T$};


		\node at (5.2,2.2) {(b)};
		
  \end{tikzpicture}
\end{center}
\caption{\label{fig:traceDiagrams}%
  Sketch of the compound trace operator $T$ for two different choices: (a) ``torn'' traces, (b) natural traces.
}
\end{figure}

The continuity of traces will be enforced using an \emph{interface exchange operator}
acting on the local trace spaces:

\medskip

\noindent%
\fbox{%
\begin{minipage}[c]{0.985\textwidth}
\begin{assumption}
\label{ass:A2}
  The \emph{interface exchange operator} $\mathcal{X} \colon \Lambda \to \Lambda$ is linear and bounded, and
  \begin{enumerate}
  \item[(i)] $\mathcal{X}^2 = I$,
  \item[(ii)] $\range(R) = \ker((I - \mathcal{X}) T)$.
  \end{enumerate}
\end{assumption}
\end{minipage}}

\medskip

The first property states that the interface exchange operator is an \emph{involution} ($\mathcal{X}^{-1} = \mathcal{X}$),
the second property can be read as $u \in \range(R)$ $\Longleftrightarrow$ $Tu = \mathcal{X} Tu$.

Under Assumptions~\ref{ass:A1}--\ref{ass:A2}, the subdomain flux formulation~\eqref{eq:subdflux} is equivalent to
\begin{align}
\label{eq:subdfluxT}
\begin{aligned}
  \text{find } (u, t) \in U \times U^* \colon \qquad A u - t & = f,\\
	(I - \mathcal{X}) T u & = 0,\\
	R^\top t & = 0.
\end{aligned}
\end{align}
The operator $\mathcal{S}$ from Lemma~\ref{lem:subdflux} is also a solution operator for \eqref{eq:subdfluxT}.

\begin{example}[swapping operator]
\label{ex:swapping}
  For the setup from Example~\ref{ex:tracesH1Split}, let $\mathcal{X}$ be the operator that \emph{swaps} traces in the sense that
	for each interior facet $F = F_{ij}$, we have $(\mathcal{X} \lambda)_{iF} = \lambda_{jF}$
	and $(\mathcal{X} \lambda)_{jF} = \lambda_{iF}$	for $\lambda \in \Lambda$,
	where	$\lambda_{jF}$ denotes the component of $\lambda_j$ corresponding to $F$;
	see also \cite[Formula~(42)]{CollinoGhanemiJoly:2000a}.
	For exterior Dirichlet facets $F \in \mathcal{F}_i$, we can set $(\mathcal{X} \lambda)_{iF} = -\lambda_{iF}$
	such that the condition $(I-\mathcal{X})T u = 0$ enforces the homogeneous Dirichlet condition on $F$,
	see also Sect.~\ref{sect:facetSystems} below.
	Property~(i) of Assumption~\ref{ass:A2} obviously holds.
	Assume for simplicity that we have no exterior facets at all, such that $U_i = H^1_D(\Omega_i)$.
	To verify property~(ii), we have to show that for all $u \in U$,
	\begin{align}
	\label{eq:facetWiseTraceContinuity}
	  u \in \range(R)
	  \quad    \Longleftrightarrow    \quad
	  \big[ \forall F \in \mathcal{F} \colon T_{iF} u_i = T_{jF} u_j \big].
	\end{align}
	Recall from Example~\ref{ex:HelmholtzReconstruction} that the broken function $u \in U$ is in $\range(R)$
	if and only if \eqref{eq:nablaGrad} holds, i.e.,
	\[
	  \sum_{i=1}^N \int_{\Omega_i} \nabla u_i \cdot \varphiv \, dx = - \sum_{i=1}^N \int_{\Omega_i} u_i (\opdiv \varphiv) \, dx
		\qquad \forall \varphiv \in C^\infty_0(\Omega)^d.
	\]
	Integration by parts on each subdomain shows that
	\eqref{eq:nablaGrad} is equivalent to
	\[
	  \sum_{i=1}^N \int_{\partial\Omega_i} u_i (\varphiv\cdot \nu_i) \, ds = 0 \qquad \forall \varphiv \in C^\infty_0(\Omega)^d.
	\]
	Since $\varphiv$ vanishes on $\partial\Omega$ and since $u_{i|F}$ is in $L^2(F)$ for each interior facet $F \in \mathcal{F}_i$,
	the above can be rewritten as
	\[
	  \sum_{F_{ij} \in \mathcal{F}} \int_{F_{ij}} (u_i - u_j) (\varphiv\cdot \underbrace{\nu_i}_{= - \nu_j}) \, ds = 0
		\qquad \forall \varphiv \in C^\infty_0(\Omega)^d.
	\]
	Since $C^\infty_0(F)$ is dense in $L^2(F)$, the above identity holds if and only if
	$T_{iF} u_i = T_{jF} u_j$ for all $F_{ij} \in \mathcal{F}$.
	The same argument works if $L^2(F)$ is replaced by $H^s(F)$ for $s \in [0, \tfrac{1}{2}]$.
	In the discretized case, assume that we use the collective trace operators as well and that we have again no exterior facets.
	If we fix an ordering within the selected dofs on each facet $F = F_{ij}$ such that $T_{iF}$ and $T_{jF}$ map into the same space $\Lambda_F$,
	then the condition \eqref{eq:facetWiseTraceContinuity} holds and we can define $\mathcal{X}$ as above.
\end{example}

\begin{remark}
  Note that the exchange operator $\Pi$ defined in \cite[Sect.~2]{CollinoGhanemiJoly:2000a}
	is slightly different from $\mathcal{X}$ in Example~\ref{ex:swapping}: $\Pi$ swaps traces on the interior facets
	but it evaluates to zero on exterior Robin facets. Therefore, the condition $\Pi^2 = I$ only holds on the interior facets,
	cf.\ \cite[Lemma~2]{CollinoGhanemiJoly:2000a}.
\end{remark}

\begin{remark}
  As the attentive reader will have noticed, the operator $(I - \mathcal{X})T$ in Example~\ref{ex:swapping}
	evaluates the jump on each interior facet \emph{twice}:
	\[
	  ((I - \mathcal{X})T u)_{iF} = T_{iF} u_i - T_{jF} u_j\,, \qquad ((I - \mathcal{X})T u)_{jF} = T_{jF} u_j - T_{iF} u_i\,.
	\]
	Certainly, one can construct \emph{one-sided} jump operators, and these are excessively used in the classical FETI
	and FETI-DP methods as well as the FETI-H method
	\cite{FarhatMacedoTezaur:DD11,FarhatMacedoLesoinne:2000a,FarhatMacedoLesoinneRouxMagoulesDeLaBourdonnaie:2000a};
	see also Sect.~\ref{sect:HDG} and~\ref{sect:FETIH}.
	In the article at hand, the \emph{two-sided} nature of $\mathcal{X}$ will play a principal role.
\end{remark}

\begin{example}[trace and swapping for $\Hv(\opdiv)$]
\label{ex:swappingHDiv}
  For $\widehat U = \Hv(\opdiv,\Omega)$ and $U_i = \Hv(\opdiv,\Omega_i)$,
	the natural trace space for the normal trace operator (see e.g.\ \cite[Sect.~3.5]{Monk:2003a})
	on an interior facet $F \in \mathcal{F}_i$ is $H^{-1/2}_{00}(F)$,
	i.e., the \emph{dual} of $H^{1/2}_{00}(F)$ containing distributions on $F$ that are not necessarily extendible by zero to
	$H^{-1/2}(\partial\Omega_i)$.
	Let us define $T_{iF} \colon \Hv(\opdiv,\Omega_i) \to H^{-1/2}_{00}(F)$
	by $u_i \mapsto \sigma_{iF} (u_i \cdot \nu_i)_{|F}$,
	where $\sigma_{iF} = - \sigma_{jF} \in \{-1, 1\}$ is a fixed sign pattern (for each interior facet $F = F_{ij}$).
	Then Condition~\eqref{eq:facetWiseTraceContinuity} can be shown to hold as well.
	To see this, we can use integration by parts to show that a broken function $u \in U$
	is in $\range(R)$ if and only if
	\begin{align}
	\label{eq:interface1}
	  \sum_{i=1}^N \langle (u_i \cdot \nu_i), \varphi \rangle_{H^{-1/2}(\partial\Omega_i) \times H^{1/2}(\partial\Omega_i)} = 0
		\qquad \forall \varphi \in C^\infty_0(\Omega).
	\end{align}
	We restrict the test functions $\varphi$ to those that vanish in the neighborhood of $\partial F$ for all interior faces $F$,
	such that $\varphi_{|F} \in H^{1/2}_{00}(F)$.
	Then \eqref{eq:interface1} implies
	\begin{align}
	\label{eq:interface2}
		\sum_{F_{ij} \in \mathcal{F}} \sigma_{iF} \langle T_{iF} u_i - T_{jF} u_j, \varphi \rangle_{H^{-1/2}_{00}(F) \times H^{1/2}_{00}(F)}
		= 0,
	\end{align}
  which holds if and only if $T_{iF} u_i = T_{jF} u_j$ in $H^{-1/2}_{00}(F)$ for all $F_{ij} \in \mathcal{F}$
	(by a density argument).
	Conversely, one can show that \eqref{eq:interface2} implies \eqref{eq:interface1}
	by using the fact that the space of $C^\infty_0(\Omega)$-functions that vanish in the vicinity of the \emph{wirebasket}
	$\big( \bigcup_i \partial\Omega_i \big) \setminus \big( \partial\Omega \cup \bigcup_{F \in \mathcal{F}} \text{int}(F) \big)$
	is dense in $H^1_0(\Omega)$, cf.\ \cite[Lemma~3.1]{BendaliBoubendir:2005a} and \cite{ChandlerWildeHewettMoiola:2017a}.
	The same kind of argument can be used to prove \eqref{eq:facetWiseTraceContinuity} for $\Hv(\opcurl,\Omega)$,
	at least for sufficiently smooth interfaces.
\end{example}

\begin{example}[trace and swapping for Raviart-Thomas elements]
\label{ex:swappingRaviartThomas}
	Consider a lowest-order Raviart-Thomas discretization of $\Hv(\opdiv,\Omega)$ (N\'ed\'elec face elements)
	such that each dof is associated with a face $f$ of the mesh, which has a fixed orientation.
	It is then reasonable to let the subdomain restriction operators $R_i$ simply select dofs according to their ownership and not change any orientation, so $R_i$ is a zero-one matrix.
	For a facet $F$, which is the union of mesh faces, we can simply define $T_{iF}$, $T_{jF}$ as the zero-one matrices selecting the dofs on the faces $f \subset F$.
	In that case, it is easy to see that again the condition \eqref{eq:facetWiseTraceContinuity} holds and $\mathcal{X}$ can be defined as the swapping operator.
	Note that each dof has at most multiplicity two, and so the trace operator $T_i$ is surjective,
	and it coincides with the natural trace operator (up to possible reordering) that selects all dofs of subdomain $i$ with multiplicity $\ge 2$.
\end{example}

\begin{example}[trace and swapping for N\'ed\'elec elements]
\label{ex:swappingNedelec}
	Consider an $\Hv(\opcurl)$-conforming discretization by lowest-order N\'ed\'elec edge elements such that each dof is associated with an edge $e$ of the mesh,
	which has a fixed orientation.
	It is then reasonable to let the subdomain restriction operators $R_i$ simply select dofs according to their ownership and not change any orientation,
	so $R_i$ is a zero-one matrix.
	For a facet $F$, we can simply define $T_{iF}$, $T_{jF}$ as the zero-one matrices selecting the dofs on the edges $e \subset \overline F$.
	In that case, it is easy to see that again the condition \eqref{eq:facetWiseTraceContinuity} holds and $\mathcal{X}$ can be defined as the swapping operator.
	Note that if an edge $e$ with an associated dof is on the interface between two facets, $e \subset \partial F_1 \cap \partial F_2$,
	then the collective trace operator is not surjective.
\end{example}

The following lemma connects the involution $\mathcal{X}$ with its associated projections and will be helpful in many ways later on.

\begin{lemma}
\label{lem:X}
  Let $\mathcal{X} \colon \Lambda \to \Lambda$ be a linear and bounded involution ($\mathcal{X}^2 = I$). Then
	\begin{enumerate}
	\item[(i)] $\tfrac{1}{2}(I \pm \mathcal{X})$ and $\tfrac{1}{2}(I \pm \mathcal{X}^\top)$ are projections,
	\item[(ii)] $\range(I \pm \mathcal{X}) = \ker(I \mp \mathcal{X})$ and $\range(I \pm \mathcal{X}^\top) = \ker(I \mp \mathcal{X}^\top)$.
	\end{enumerate}
\end{lemma}
\begin{proof}
  $\tfrac{1}{4}(I \pm \mathcal{X})(I \pm \mathcal{X}) = \tfrac{1}{4}(I \pm 2 \mathcal{X} + \mathcal{X}^2) = \tfrac{1}{2}(I \pm \mathcal{X})$,
	same for the transposed version.
	The two projections sum up to unity, i.e., $\tfrac{1}{2}(I \mp \mathcal{X}) = I - \tfrac{1}{2}(I \pm \mathcal{X}$).
	From the projection property, it can also be shown that
	$\range(\tfrac{1}{2}(I \pm \mathcal{X})) = \ker(\tfrac{1}{2}(I \mp \mathcal{X})$, same for the transposed version.
	In particular all the ranges are closed subspaces.
\end{proof}

\begin{remark}
\label{rem:XP}
  One can also reverse the statement of Lemma~\ref{lem:X}: for any linear and bounded projection $P$,
	we can construct an involution $\mathcal{X}_P := 2 P - I$, which is the \emph{reflection operator}
	that coincides with the identity on $\range(P)$ and flips the sign on $\ker(P)$.
\end{remark}

\begin{remark}
\label{rem:XTconj}
  In the complexified case (see Sect.~\ref{sect:globalProblem}), under Assumption~\ref{ass:A2},
	\[
	  v \in \range(R)
		\Longleftrightarrow \overline v \in \range(R)
		\Longleftrightarrow (I - \mathcal{X}) T \overline v = 0
		\Longleftrightarrow (I - \overline{\mathcal{X}}) T v = 0
		\Longleftrightarrow T v \in \ker(I - \overline{\mathcal{X}}),
	\]
	where $\overline v$ is the conjugate of $v$ and $\overline{\mathcal{X}} := \mathcal{X}_\re - \complexi \mathcal{X}_\im$
	denotes the conjugate operator of $\mathcal{X} = \mathcal{X}_\re + \complexi \mathcal{X}_\im$ where $\mathcal{X}_\re$, $\mathcal{X}_\im$ are operators on the real Hilbert space
	(see also Section~\ref{sect:abstractDD}).
	Therefore,
	\[
	  \range(R) = \ker((I - \overline{\mathcal{X}})T).
	\]
	Moreover, $\overline{\mathcal{X}}^2 = I$ and $\tfrac{1}{2}(I \pm \overline{\mathcal{X}})$ and
	$\tfrac{1}{2}(I \pm \mathcal{X}^\mathsf{H})$ are projections, where $\mathcal{X}^\mathsf{H} := \overline{\mathcal{X}}^\top$ denotes the Hermitian transpose.
	This fact is used in \cite{Claeys:2021Preprint}, see also Remark~\ref{rem:HermitianImp} below.
\end{remark}

\section{Facet systems$^*$}
\label{sect:facetSystems}

This section formalizes the facets from Example~\ref{ex:tracesH1Split} in the general and in the discrete case.
Since this is a detailed and technical matter,
readers who are mainly interested in the Schwarz method itself are encouraged to (at least initially)
bypass this section and continue with Sect.~\ref{sect:interfaceFluxes} (p.~\pageref{sect:interfaceFluxes}).

\subsection{General facets systems$^*$}

Let $(U, R)$ be an abstract domain decomposition of $\widehat U$ (Def.~\ref{def:abstractDD}) with Assumption~\ref{ass:A1} fulfilled.

\begin{definition}
\label{def:facet}
  A \emph{facet} $F$ of the abstract domain decomposition is characterized by
	\begin{itemize}
	\item the \emph{adjacency set} $\mathcal{N}_F$, a non-empty set of subdomain indices \emph{linked} by the facet,
	\item the \emph{facet space} $U_F$, a Hilbert space (real or complexified in accordance with $\widehat U$, $U$), and
	\item linear bounded \emph{trace operators}
	  $T_{jF} \colon U_j \to U_F$, $j \in \mathcal{N}_F$ and $\widehat T_F \colon \widehat U \to U_F$
	  that fulfill the \emph{consistency relation}
		\begin{align}
		\label{eq:traceConsistency}
		  T_{jF} R_j = \widehat T_F \qquad \forall j \in \mathcal{N}_F\,.
		\end{align}
	\end{itemize}
	A facet $F$ is \emph{interior} if $|\mathcal{N}_F| \ge 2$ and \emph{exterior} if $|\mathcal{N}_F| = 1$.
	There are two kinds of exterior facets,	\emph{Dirichlet facets} (where Dirichlet conditions are imposed)
	and \emph{auxiliary facets} (which are included for some other reason, however, only in rare cases).
	If $|\mathcal{N}_F| = 2$, we call $F$ \emph{bilateral}.
	
	A \emph{facet system} is a collection $\mathcal{F}$ of facets,
	and it is called bilateral if all its facets are bilateral.
	We denote by $\mathcal{F}_i$ the set of facets shared by subdomain $i$.
\end{definition}

\begin{definition}
\label{def:admissibleFacetSystem}
  A facet system $\mathcal{F}$ is \emph{admissible} with the abstract domain decomposition $(U, R)$ if
	\[
	  u \in \range(R)  \ \Longleftrightarrow \
		\begin{cases}
		  T_{iF} u_i = T_{jF} u_j \quad \forall i, j \in \mathcal{N}_F         & \text{for all interior facets } F \in \mathcal{F},\\
			T_{iF} u_i = 0          \quad \text{with } \mathcal{N}_F = \{ i \}   & \text{for all exterior Dirichlet facets } F \in \mathcal{F}.
		\end{cases}
	\]
\end{definition}

For admissible facet systems, the local trace space and trace operator are defined by
\[
   \Lambda_i = \productspace_{F \in \mathcal{F}_i} U_F\,, \qquad
	 T_i u_i := (T_{iF} u_i)_{F \in \mathcal{F}_i}\,,
\]
i.e., of \emph{collective type}.
If all interior facets are bilateral,
we can define $\mathcal{X} \colon \Lambda \to \Lambda$ by
\begin{align}
\label{eq:XBilateral}
  (\mathcal{X} \lambda)_{iF} =
	\begin{cases}
	    \lambda_{jF} & \text{for bilateral facets } F \text{ with } \mathcal{N}_F = \{ i, j \},\\
	  - \lambda_{iF} & \text{for exterior Dirichlet facets } F \text{ with } \mathcal{N}_F = \{ i \},\\
		  \lambda_{iF} & \text{for exterior auxiliary facets } F \text{ with } \mathcal{N}_F = \{ i \}.
	\end{cases}
\end{align}
Then Assumption~\ref{ass:A2} holds true.

\subsection{Discrete facet systems$^*$}
\label{sect:discreteFacetSystems}

In the continuous case, the proper choice of geometric facets and the associated trace operators
in order to achieve admissibility (Def.~\ref{def:admissibleFacetSystem})
hinges on the structure of the underlying geometry and Sobolev spaces, see Examples~\ref{ex:swapping} and~\ref{ex:swappingHDiv}.
The discrete case allows to construct facets solely from the sets of dofs and the sharing subdomains
 --- albeit not uniquely.
Two particular ways of construction are discussed below,
using the assumptions and notations of Sect.~\ref{sect:discreteMatrixNotation}.

\begin{definition}[discrete facet]
\label{def:discreteFacet}
	A \emph{discrete facet} $F$ is characterized by
	\begin{enumerate}
	\item[(i)] the \emph{adjacency} set $\mathcal{N}_F$, and
	\item[(ii)] a \emph{global dof set} $\mathcal{D}_F$,
	\end{enumerate}
	such that the compatibility relation $\mathcal{N}_F \subseteq \mathcal{N}_k$ holds for all $k \in \mathcal{D}_F$
	(with $\mathcal{N}_k$ defined as in Sect.~\ref{sect:discreteMatrixNotation}).
\end{definition}

The dof set $\mathcal{D}_F$ induces a trace space $U_F = \mathbb{R}^{n_F}$ or $\mathbb{C}^{n_F}$
with $n_F$ being the number of dofs in $\mathcal{D}_F$, where we agree on a unique numbering of the dofs.
Along with that, we obtain
\begin{itemize}
\item a global trace operator $\widehat T_F \colon \widehat U \to U_F$,
  the zero-one matrix selecting the dofs $\mathcal{D}_F$ from all the global dofs, and
\item for each subdomain $j \in \mathcal{N}_F$ a local trace operator $T_{jF} \colon U_j \to U_F$,
  the zero-one matrix selecting the dofs $\mathcal{D}_F$ from the local subdomain dofs.
\end{itemize}
These operators obviously fulfill the consistency relation~\eqref{eq:traceConsistency},
and altogether the conditions in Definition~\ref{def:facet} are met.
Note that two distinct discrete facets may have the same dof set but different adjacency sets
(see e.g., Fig.~\ref{fig:facetVariants_a}, facets $F_{23}$ and $F_{14}$).

\begin{definition}
\label{def:discreteFacetClosed}
  A discrete facet $F$ with adjacency set $\mathcal{N}_F$ and dof set $\mathcal{D}_F$ is called \emph{closed}
	if $\mathcal{D}_F$ contains \emph{all} the dofs shared by the subdomains listed in $\mathcal{N}_F$, i.e., if
	\[
	  \mathcal{D}_F = \{ k = 1,\ldots,n \colon \mathcal{N}_F \subseteq \mathcal{N}_k \}.
	\]
\end{definition}

\begin{remark}
\label{rem:discreteDirichletFacets}
  If one wishes to include Dirichlet facets, one has to extend Definition~\ref{def:discreteFacet}
	and allow exterior facets $F$ with an empty global dof set but with a non-trivial trace space $U_F$
	and an associated trace operator fulfilling the property $\widehat T_F = 0$.
	The theory of Sect.~\ref{sect:bilateralDiscrFacetSys} below does not include any exterior facets,
	but can be extended without major effort.
\end{remark}

In the following, two constructions of discrete facet systems are discussed that are \emph{admissible}
in the sense of Definition~\ref{def:admissibleFacetSystem}.

\subsubsection{Bilateral discrete facet systems$^*$}
\label{sect:bilateralDiscrFacetSys}

For each pair $i \neq j$ of subdomain indices we collect the global dofs shared by $i$ and $j$,
\begin{align}
\label{eq:defGammaijClosed}
  \overline{\mathcal{D}}_{ij} := \{ k = 1,\ldots,n \colon \{i, j\} \subseteq \mathcal{N}_k \}.
\end{align}
The simplest discrete bilateral facet system is the \emph{maximal} set
\begin{align}
  \mathcal{F}^{\text{bil}}_{\max}
	  := \big\{ (\{i, j\}, \overline{\mathcal{D}}_{ij}) \colon i \neq j = 1,\ldots,n,\ \overline{\mathcal{D}}_{ij} \neq \emptyset \big\},
\end{align}
which simply contains all possible closed facets (with non-empty dof sets).
Other discrete facet systems can be obtained by removing dofs from the dof sets of individual facets,
or even removing entire facets from $\mathcal{F}^{\text{bil}}_{\max}$.
Proposition~\ref{prop:admissibleDiscreteFacetSet} below states how much one can actually remove
such that the discrete facet system is still admissible in the sense of Definition~\ref{def:admissibleFacetSystem}.

\begin{definition}[Connectivity graph]
\label{def:connectivityGraph}
  Let $\mathcal{F}$ be a discrete bilateral facet system (Def.~\ref{def:facet}, Def.~\ref{def:discreteFacet}).
	For each \emph{global interface dof} $k \in \mathcal{D}_\Gamma := \{ k = 1,\ldots,n \colon \mu_k \ge 2 \}$,
	the associated \emph{connectivity graph} $\mathcal{C}_k =(\mathcal{N}_k, \mathcal{E}_k)$
  is the undirected graph with nodes $\mathcal{N}_k$
  and edges
  \[
    \mathcal{E}_k := \big\{ \{i, j\} \colon i \neq j \in \mathcal{N}_k,\
      \exists F \in \mathcal{F},\ \mathcal{N}_F = \{i,j\} \big\},
  \]
	i.e., the subdomains are the nodes of the graph and the facets its edges.
  The graph $\mathcal{C}_k$ is \emph{connected}
	if each pair of nodes $i\neq j \in \mathcal{N}_k$ can be joined by a path of $m$ edges
  \[
    \{i_\ell, i_{\ell+1}\} \in \mathcal{E}_k \qquad \text{for } \ell=1,\ldots,m
  \]
  with $i = i_1$, $i_{m+1} = j$, and $i_\ell \in \mathcal{N}_k$ for $\ell=2,\ldots,m$,
	i.e., we can link two subdomains by passing through facets.
  The graph $\mathcal{C}_k$ is called
  \begin{itemize}
  \item \emph{maximal} if $\mathcal{E}_k = \{ \{i, j\} \colon i \neq j \in \mathcal{N}_k \}$,
  \item \emph{minimal} if $\mathcal{E}_k$ is a spanning tree for $\mathcal{N}_k$,
	  i.e., the graph is connected, has no cycles, and each node is visited by at least one edge.
		In that case $\cardinality{\mathcal{E}_k} = \cardinality{\mathcal{N}_k} - 1$.
\end{itemize}
\end{definition}

\begin{figure}
\begin{center}
  \def\svgwidth{0.9\textwidth}
  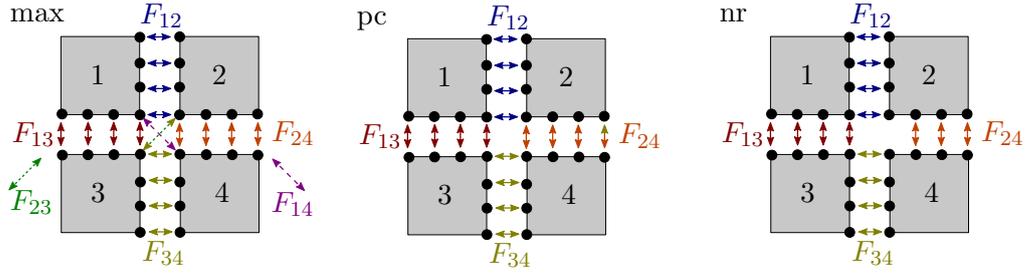
\caption{\label{fig:facetVariants_a}
  Discrete bilateral facet systems for a setting with four subdomains.
	Local dofs are visualized by $\bullet$,
	individual facets are visualized by groups of colored double arrows, each of them connecting local dofs of two subdomains.
	\emph{Left:} fully-redundant. \emph{Middle:} properly closed. \emph{Right:} non-redundant.}
\end{center}
\end{figure}

\begin{figure}
\begin{center}
  \def\svgwidth{0.83\textwidth}
  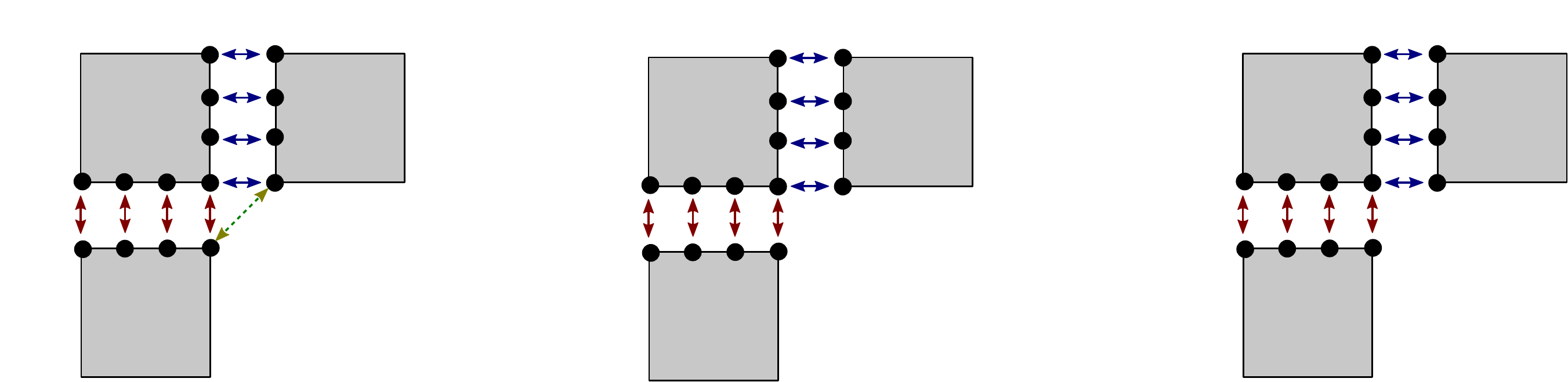
\caption{\label{fig:facetVariants_b}
  Discrete bilateral facet systems for a setting with three subdomains.
	Local dofs are visualized by $\bullet$,
	individual facets are visualized by groups of colored double arrows, each of them connecting local dofs of two subdomains.
	\emph{Left:} fully-redundant. \emph{Middle:} properly closed. \emph{Right:} non-redundant.}
\end{center}
\end{figure}

\begin{figure}
\begin{center}
  \def\svgwidth{0.83\textwidth}
  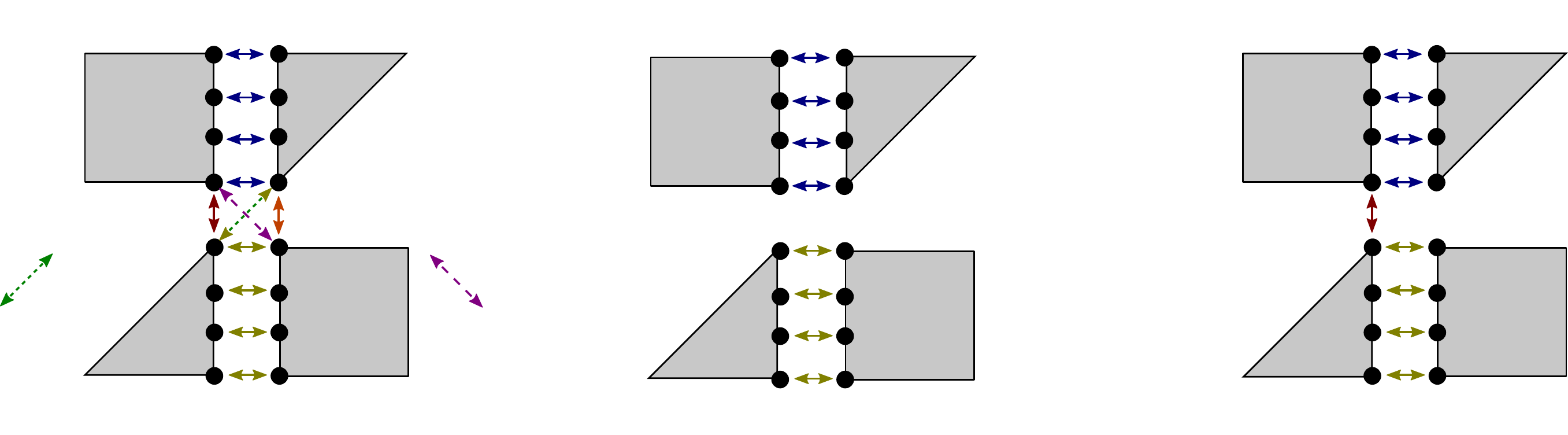
\caption{\label{fig:facetVariants_c}
  Discrete bilateral facet systems for a setting with a degenerate (non-Lipschitz) domain, divided into four subdomains.
	Local dofs are visualized by $\bullet$,
	individual facets are visualized by groups of colored double arrows, each of them connecting local dofs of two subdomains.
	\emph{Left:} fully-redundant. \emph{Middle:} properly closed. \emph{Right:} non-redundant.}
\end{center}
\end{figure}

The proof of the following proposition is left to the reader.

\begin{proposition}
\label{prop:admissibleDiscreteFacetSet}
  Let $\mathcal{F}$ be a bilateral discrete facet system (Def.~\ref{def:facet}, Def.~\ref{def:discreteFacet}).
	Then $\mathcal{F}$ is admissible (Def.~\ref{def:admissibleFacetSystem})
	if and only if the connectivity graph $\mathcal{C}_k$ of each interface dof $k \in \mathcal{D}_\Gamma$ is connected
	(Def.~\ref{def:connectivityGraph}).
	In the admissible case, $\mathcal{D}_\Gamma = \bigcup_{F \in \mathcal{F}} \mathcal{D}_F$.
\end{proposition}

\begin{definition}
  Let $\mathcal{F}$ be a discrete bilateral facet system (Def.~\ref{def:facet}, Def.~\ref{def:discreteFacet})
	and assume that for each interface dof $k \in \mathcal{D}_\Gamma$
	the connectivity graph $\mathcal{C}_k$ is connected.
	We call $\mathcal{F}$
	\begin{itemize}
	\item \emph{fully redundant} if each connectivity graph is maximal,
	\item \emph{non-redundant} if each connectivity graph is minimal.
	\end{itemize}
\end{definition}

Obviously, the \emph{maximal} discrete facet system $\mathcal{F}^{\text{bil}}_{\max}$ is fully redundant.

A non-redundant discrete facet system $\mathcal{F}^{\text{bil}}_{\text{nr}}$ can be computed in the following way:
\begin{enumerate}
\item Starting from $\mathcal{F}^{\text{bil}}_{\max}$, we run over each dof $k \in \mathcal{D}_\Gamma$
    and compute a minimal spanning tree for the connectivity graph $\mathcal{C}_k$.
\item	For each edge $\{ i, j \} \in \mathcal{E}_k$ that is not contained in the spanning tree,
	  we remove the dof $k$ from the dof set $\mathcal{D}_F$ of every facet $F$ with $\mathcal{N}_F = \{i, j\}$.
		Facets with empty dof sets are removed.
\end{enumerate}
Doing so, all the updated connectivity graphs are minimal.

Proposition~\ref{prop:admissibleDiscreteFacetSet}
states that the fully redundant and the non-redundant versions are good choices,
because admissibility (Def.~\ref{def:admissibleFacetSystem}) is guaranteed by construction.
Both variants have been used in the classical FETI method and in the FETI-DP method, cf.\ \cite[Ch.~6]{ToselliWidlund:Book}.

\medskip

A third variant, here called \emph{properly closed}, is constructed as follows:
\[
  \mathcal{F}^{\text{bil}}_{\text{pc}}
	  := \{ F \in \mathcal{F}^{\text{bil}}_{\max} \colon \exists k \in \mathcal{D}_F \colon \cardinality{\mathcal{N}_k} = 2 \},
\]
i.e., we drop those facets from $\mathcal{F}_{\max}$ where all dofs have multiplicity $>2$.
Note that indeed, all the facets in $\mathcal{F}^{\text{bil}}_{\text{pc}}$
are closed in the sense of Definition~\ref{def:discreteFacetClosed}.
In general, $\mathcal{F}^{\text{bil}}_{\text{pc}}$ is neither fully redundant nor non-redundant,
and for most examples the discretes facets of $\mathcal{F}^{\text{bil}}_{\text{pc}}$ are in accordance with the \emph{geometric}
facets of dimension $(d-1)$, see Fig.~\ref{fig:facetVariants_a}--\ref{fig:facetVariants_b}.
Note, however, that in case of \emph{degenerate} domains,
the properly closed variant may lead to non-connected connectivity graphs, see Fig.~\ref{fig:facetVariants_c}.
Nevertheless, for non-degenerate geometries,
the properly closed version usually leads to connected graphs and so Proposition~\ref{prop:admissibleDiscreteFacetSet}
guarantees admissibility as well.
The properly closed variant is (implicitly) used in the FETI-2LM formulation
\cite{Bourdonnaye:DD10,FarhatMacedoMagoulesRoux:USNCCM,FarhatMacedoLesoinneRouxMagoulesDeLaBourdonnaie:2000a}.

\medskip

The following statement marks a special case where all the above variants coincide.

\begin{proposition}
  Let $(U, R)$ be an abstract subspace decomposition of $\widehat U$ (Def.~\ref{def:abstractDD})
	and let the assumptions of Sect.~\ref{sect:discreteMatrixNotation} hold.
	In addition, assume that $\mu_{\max} \le 2$, i.e., no dof is shared by more than two subdomains.
	Then there is only one unique discrete facet system that is admissible (unless one allows exterior facets).
	In particular, any facet system following the construction from above of $\mathcal{F}^{\textnormal{bil}}_{\textnormal{nr}}$
	and $\mathcal{F}^{\textnormal{bil}}_{\textnormal{pc}}$ is equal to $\mathcal{F}^{\textnormal{bil}}_{\max}$.
\end{proposition}

The following proposition shows that once $\mu_{\max} > 2$ (i.e., once a cross point dof appears),
the trace operator necessarily fails to be surjective.

\begin{proposition}
\label{prop:bilateralCrosspoints}
  Let $(U, R)$ be an abstract subspace decomposition of $\widehat U$ (Def.~\ref{def:abstractDD}),
	let the assumptions of Sect.~\ref{sect:discreteMatrixNotation} hold,
	and let $\mathcal{F}$ be an admissible bilateral facet system.
	Then
	\[
	  \mu_{\max} \le 2 \quad \Longleftrightarrow \quad \range(T) = \Lambda.
	\]
\end{proposition}
\begin{proof}
  Assume that $\mu_{\max} > 2$ such that $\mu_k > 2$ for some global dof $k$.
	Due to Proposition~\ref{prop:admissibleDiscreteFacetSet}, the connectivity graph $\mathcal{C}_k$ must be connected.
	Since $k$ is shared by at least three subdomains, there have to be three subdomain indices,
	say 1, 2, and 3, such that the edges $(1, 2)$ and $(2, 3)$ are in the connectivity graph.
	Therefore, two facets $F_{12}$, $F_{23}$ of subdomain~$2$ must exist that contain dof $k$.
	However, in $\range(T_2)$, the copies of the local dof corresponding to $k$ on $F_{12}$, $F_{23}$
	are equal, whereas in $\Lambda_2$ the two corresponding entries may in general differ from each other,
	see Fig.~\ref{fig:bilateralCrosspointProofSketch}.
	So we have shown $\mu_{\max} > 2 \implies \range(T) \neq \Lambda$.
	To see the other implication $\mu_{\max} \le 2 \implies \range(T) = \Lambda$,
	observe that for $\mu_{\max} \le 2$, the connectivity graph $\mathcal{C}_k$ of any interface dof $k$ consists of only two subdomains,
	so there exists only \emph{one} facet between them. In other words, each local interface dof of any subdomain
	is contained in a \emph{unique} facet; this implies $\range(T) = \Lambda$.
\end{proof}

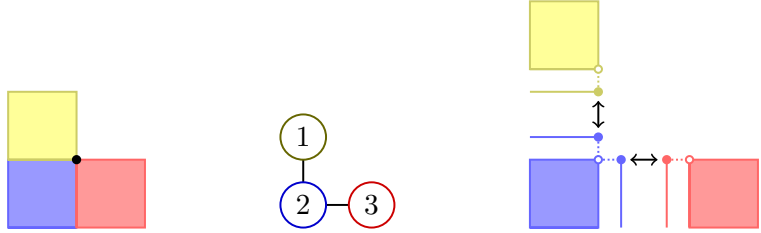
\begin{figure}
\begin{center}
  \begin{tikzpicture}
		\pgftransformscale{0.6}
		
	  \definecolor{dred}{rgb}{0.8, 0.0, 0.0}
		\definecolor{mdred}{rgb}{1.0, 0.4, 0.4}
		\definecolor{lred}{rgb}{1.0, 0.8, 0.8}
		\definecolor{mlred}{rgb}{1.0, 0.6, 0.6}
		
		\definecolor{dblue}{rgb}{0.0, 0.0, 0.8}
		\definecolor{mdblue}{rgb}{0.4, 0.4, 1.0}
		\definecolor{lblue}{rgb}{0.8, 0.8, 1.0}
		\definecolor{mlblue}{rgb}{0.6, 0.6, 1.0}
		
		\definecolor{dyellow}{rgb}{0.4, 0.4, 0.0}
		\definecolor{mdyellow}{rgb}{0.8, 0.8, 0.4}
		\definecolor{lyellow}{rgb}{1.0, 1.0, 0.8}
		\definecolor{mlyellow}{rgb}{1.0, 1.0, 0.6}
		
		\draw[line width=0.75pt,color=mdblue,fill=mlblue] (0,0)--(1.5,0)--(1.5,1.5)--(0,1.5)--(0,0)--(1.5,0);
		\draw[line width=0.75pt,color=mdyellow,fill=mlyellow] (0,1.5)--(1.5,1.5)--(1.5,3)--(0,3)--(0,1.5)--(1.5,1.5);
		\draw[line width=0.75pt,color=mdred,fill=mlred] (1.5,0)--(1.5,1.5)--(3,1.5)--(3,0)--(1.5,0)--(1.5,1.5);
		
		\draw[line width=0.74pt,color=black,fill=black] (1.5,1.5) circle (0.08);
		
  \end{tikzpicture}
	\hspace{1.5cm}
  \begin{tikzpicture}
		\pgftransformscale{0.6}
		
	  \definecolor{dred}{rgb}{0.8, 0.0, 0.0}
		\definecolor{mdred}{rgb}{1.0, 0.4, 0.4}
		\definecolor{lred}{rgb}{1.0, 0.8, 0.8}
		\definecolor{mlred}{rgb}{1.0, 0.6, 0.6}
		
		\definecolor{dblue}{rgb}{0.0, 0.0, 0.8}
		\definecolor{mdblue}{rgb}{0.4, 0.4, 1.0}
		\definecolor{lblue}{rgb}{0.8, 0.8, 1.0}
		\definecolor{mlblue}{rgb}{0.6, 0.6, 1.0}
		
		\definecolor{dyellow}{rgb}{0.4, 0.4, 0.0}
		\definecolor{mdyellow}{rgb}{0.8, 0.8, 0.4}
		\definecolor{lyellow}{rgb}{1.0, 1.0, 0.8}
		\definecolor{mlyellow}{rgb}{1.0, 1.0, 0.6}
		
		\draw[line width=0.74,color=black] (0,1.5)--(0,0)--(1.5,0);
		
		\draw[line width=0.74pt,color=dyellow,fill=white] (0,1.5) circle (0.5);
		\node at (0,1.5) {1};
		
		\draw[line width=0.74pt,color=dblue,fill=white] (0,0) circle (0.5);
		\node at (0,0) {2};
		
		\draw[line width=0.74pt,color=dred,fill=white] (1.5,0) circle (0.5);
		\node at (1.5,0) {3};
		
	\end{tikzpicture}
	\hspace{1.5cm}
  \begin{tikzpicture}
		\pgftransformscale{0.6}
		
	  \definecolor{dred}{rgb}{0.8, 0.0, 0.0}
		\definecolor{mdred}{rgb}{1.0, 0.4, 0.4}
		\definecolor{lred}{rgb}{1.0, 0.8, 0.8}
		\definecolor{mlred}{rgb}{1.0, 0.6, 0.6}
		
		\definecolor{dblue}{rgb}{0.0, 0.0, 0.8}
		\definecolor{mdblue}{rgb}{0.4, 0.4, 1.0}
		\definecolor{lblue}{rgb}{0.8, 0.8, 1.0}
		\definecolor{mlblue}{rgb}{0.6, 0.6, 1.0}
		
		\definecolor{dyellow}{rgb}{0.4, 0.4, 0.0}
		\definecolor{mdyellow}{rgb}{0.8, 0.8, 0.4}
		\definecolor{lyellow}{rgb}{1.0, 1.0, 0.8}
		\definecolor{mlyellow}{rgb}{1.0, 1.0, 0.6}
		
		\draw[line width=0.75pt,color=mdblue,fill=mlblue] (0,0)--(1.5,0)--(1.5,1.5)--(0,1.5)--(0,0)--(1.5,0);
		\draw[line width=0.75pt,color=mdyellow,fill=mlyellow] (0,3.5)--(1.5,3.5)--(1.5,5)--(0,5)--(0,3.5)--(1.5,3.5);
		\draw[line width=0.75pt,color=mdred,fill=mlred] (3.5,0)--(3.5,1.5)--(5,1.5)--(5,0)--(3.5,0)--(3.5,1.5);

		\draw[line width=0.75pt,color=mdblue] (0,2)--(1.5,2);
		\draw[line width=0.75pt,color=mdblue] (2,0)--(2,1.5);

		\draw[line width=0.75pt,color=mdyellow] (0,3)--(1.5,3);
		
		\draw[line width=0.75pt,color=mdred] (3,0)--(3,1.5);
		
		\foreach \x in {1.5}
		{
		  \draw[line width=0.74pt,color=mdblue,fill=mdblue] (\x,2) circle (0.08);
		  \draw[line width=0.74pt,color=mdblue,fill=mdblue] (2,\x) circle (0.08);

		  \draw[line width=0.74pt,color=mdred,fill=mdred] (3,\x) circle (0.08);

		  \draw[line width=0.74pt,color=mdyellow,fill=mdyellow] (\x,3) circle (0.08);

			\draw[line width=0.74pt,color=mdblue,fill=mdblue,densely dotted] (1.5,\x)--(2,\x);
			\draw[line width=0.74pt,color=mdblue,fill=mdblue,densely dotted] (\x,2)--(\x,1.5);

			\draw[line width=0.74pt,color=mdred,fill=mdred,densely dotted] (3,\x)--(3.5,\x);

			\draw[line width=0.74pt,color=mdyellow,fill=mdyellow,densely dotted] (\x,3)--(\x,3.5);
			
			\draw[line width=0.74pt,color=black,<->] (2.2,\x)--(2.8,\x);

			\draw[line width=0.74pt,color=black,<->] (\x,2.2)--(\x,2.8);

    }
		
    \draw[line width=0.74pt,color=mdblue,fill=white] (1.5,1.5) circle (0.08);
    \draw[line width=0.74pt,color=mdred,fill=white] (3.5+0,1.5) circle (0.08);
    \draw[line width=0.74pt,color=mdyellow,fill=white] (1.5,3.5+0) circle (0.08);
		
  \end{tikzpicture}
  \caption{\label{fig:bilateralCrosspointProofSketch}%
    \emph{Left:} global dof $k$ ($\bullet$) shared by three subdomains 1, 2, and 3.
		\emph{Middle:} associated connectivity graph $\mathcal{C}_k$.
		\emph{Right:} sketch relevant dofs~($\circ$) of local subdomain spaces and dofs ($\bullet$) of trace space $\Lambda$.
  }
\end{center}
\end{figure}

\subsubsection{Non-bilateral Discrete Facet Systems -- Globs$^*$}
\label{sect:globs}

While for bilateral discrete facets, continuity is imposed between the dofs of \emph{two} subdomains at a time,
we can also use conditions between the dofs of \emph{several} subdomains \emph{simultaneously}.

\begin{definition}[globs]
	The set $\mathcal{D}_\Gamma = \{ k = 1,\ldots, n \colon \mu_k \ge 2 \}$
	of interface dofs is partitioned into equivalence classes with respect to the equivalence relation
	$k \sim j \Longleftrightarrow \mathcal{N}_k = \mathcal{N}_j$,
	such that the dofs within a class are shared by the same set of subdomains.
	A discrete facet system is formed (in the sense of Definition~\ref{def:discreteFacet})
	by looping over each equivalence class:
	\begin{itemize}
	\item the equivalence class becomes the dof set $\mathcal{D}_G$,
	\item the set of (commonly) sharing subdomains becomes the adjacency set $\mathcal{N}_G$.
	\end{itemize}
	In that special case, we speak of a \emph{glob} $G$ (instead of a facet).
	For each glob $G$ there is the associated trace space $U_G$ induced by the global dofs of $G$
	and trace operators $T_{jG}$, $j \in \mathcal{N}_G$ and $\widehat T_G$ fulfilling the consistency relation~\eqref{eq:traceConsistency}.
	Finally, one may add Dirichlet globs which have an adjacency set with just one subdomain
	and a trace operator fulfilling $\widehat T_G = 0$, cf.~Remark~\ref{rem:discreteDirichletFacets}.
	The set of all the globs is denoted by $\mathcal{G}$ and the globs of subdomain $i$ by $\mathcal{G}_i$.
\end{definition}

For the standard $H^1$-conforming discretization of piece-wise linear finite elements,
the globs correspond to geometric entities that may be called subdomain faces, edges, and vertices,
cf.\ \cite{ToselliWidlund:Book,Pechstein:FETIBook}.

It turns out naturally that the glob set is admissible in the sense of Definition~\ref{def:admissibleFacetSystem},
the proof of which is left to the reader.
Figure~\ref{fig:Globs_abc} shows some examples. The glob $G = V$ in the left-most example is shared by four subdomains
$\mathcal{N}_V = \{1, 2, 3, 4\}$ and we have the trace operators $T_{1V},\ldots,T_{4V}$
which all select the vertex dof out of the respective subdomain dofs.
The conditions enforced at the vertex, as expressed in Definition~\ref{def:admissibleFacetSystem}, read
\[
  \forall i, j \in \{1,2,3,4\} \colon T_{iV} u_i = T_{jV} u_j\,,
\]
so \emph{all} the dofs associated with $V$ are imposed to be equal.
This is in contrast to the bilateral case, where only \emph{two} dofs are constrained at a time.
We note that globs are frequently used in BDDC methods \cite{Dohrmann:2003a,PechsteinDohrmann:2017a}
as well as in the analysis of FETI and FETI-DP methods \cite{ToselliWidlund:Book,Pechstein:FETIBook}.

\begin{figure}
\begin{center}
  \hspace{0.04\textwidth}
  {
    \def\svgwidth{0.35\textwidth}
    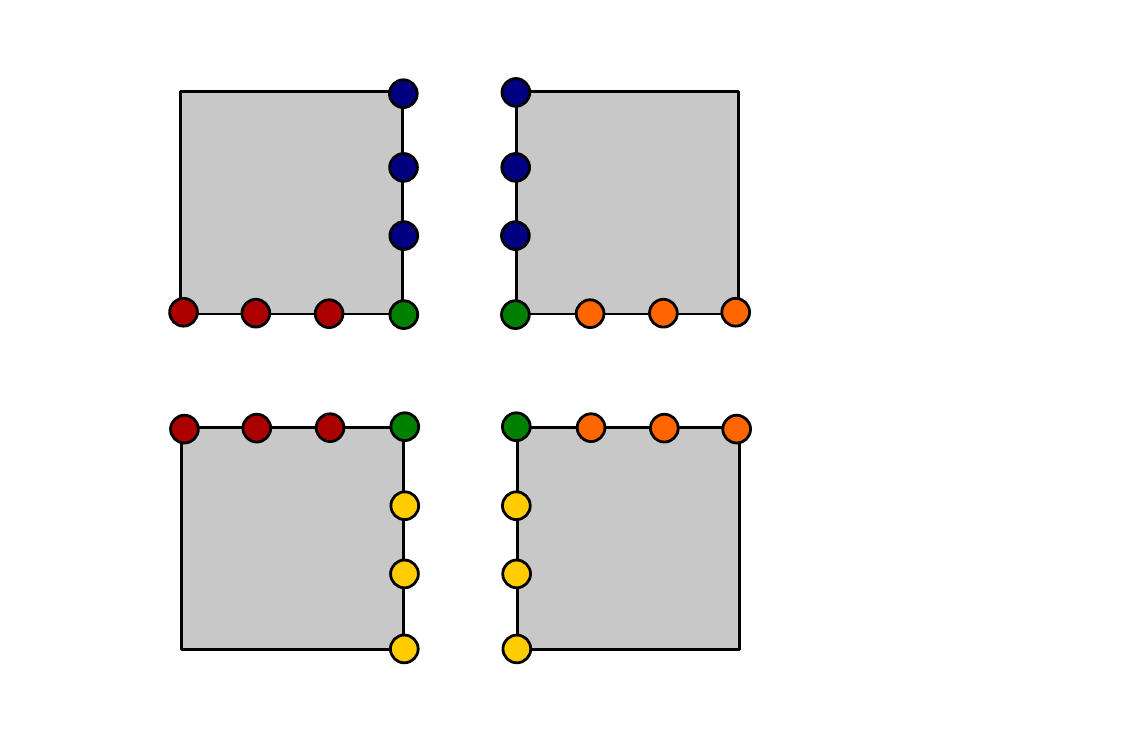
	}
  \hspace{-0.06\textwidth}
	{
	  \def\svgwidth{0.225\textwidth}
	  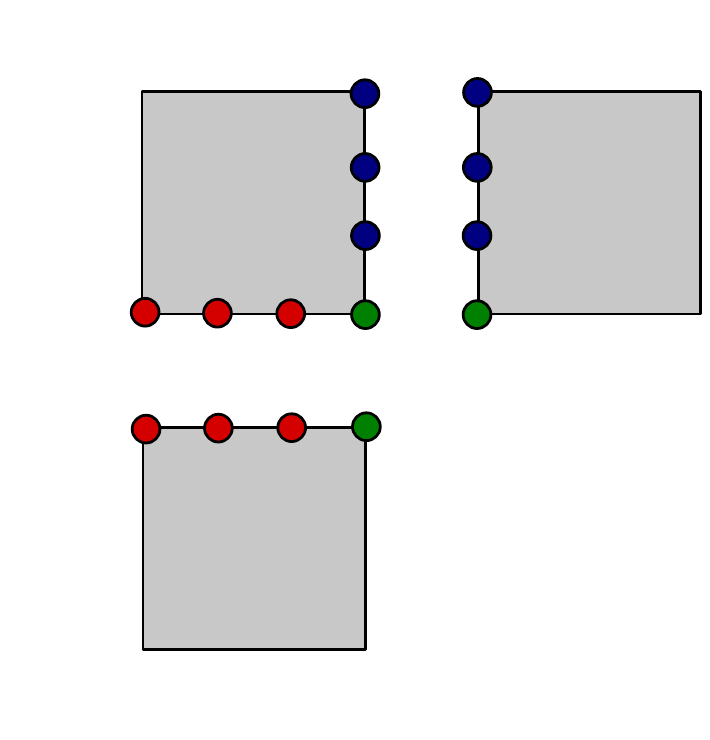
	}
  \hspace{0.08\textwidth}
	{
	  \def\svgwidth{0.18\textwidth}
	  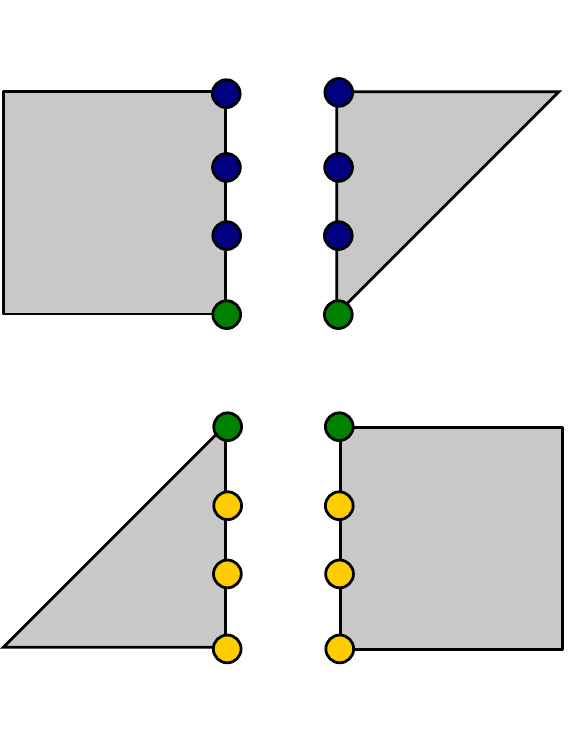
	}
  \hspace{0.04\textwidth}
	
\caption{\label{fig:Globs_abc}
  Discrete globs for three different situations;
	local interface dofs visualized by $\bullet$,
	individual globs indicated by different color.}
\end{center}

\end{figure}

The construction of the trace space and collective trace operator follows that for general discrete facets:
the subdomain trace space is given by $\Lambda_i := \productspace_{G \in \mathcal{G}_i} U_G$
and the subdomain trace operator by $T_i u_i := (T_{iG} u_i)_{G \in \mathcal{G}_i}$.
Finally, $\Lambda := \productspace_{i=1}^N \Lambda_i$ and $T u := (T_i u_i)_{i=1}^N$.
Since every local interface dof is contained in a unique glob,
there is a one-to-one correspondance between traces and local interface dofs,
which is summarized in the following proposition.

\begin{proposition}
\label{prop:globsRangeGLambda}
  Let the glob system $\mathcal{G}$ and the trace operator $T$ be constructed as above.
	Then $\range(T) = \Lambda$. In particular, there exists a right-inverse
	(an extension operator) $T^\dag \colon \Lambda \to U$ such that $T T^\dag = I$.
\end{proposition}
	
Observe that the property $\range(T) = \Lambda$
\emph{never} holds for \emph{bilateral} discrete facet systems with $\mu_{\max} > 2$
(Proposition~\ref{prop:bilateralCrosspoints}), whereas it is \emph{always} fulfilled for glob systems.

\medskip

Recall the definition~\eqref{eq:XBilateral} of the interface exchange operator $\mathcal{X}$ in the bilateral case.
In the following, we construct such an operator for the case of globs, using the
\emph{averaging projection} operator $E_m \colon \Lambda \to \Lambda$ ($m$ stands for multiplicity), given by
\begin{align}
\label{eq:EDdef}
  (E_m \lambda)_{iG} := \begin{cases}
	  \frac{1}{|\mathcal{N}_G|} \sum_{j \in \mathcal{N}_G} \lambda_{jG}
	  \quad \forall i \in \mathcal{N}_G 
		  & \text{for all } G \in \mathcal{G} \text{ that are no Dirichlet globs,}\\
		0 & \text{for all Dirichlet globs } G\text{, where } \mathcal{N}_G = \{ i \}.
	\end{cases}
\end{align}
This operator, averaging traces and redistributing them, plays a principal role in the 
\emph{2-Lagrange multiplier method} proposed by Loisel \cite{Loisel:2013a}.\footnote{The projection operator defined in~\eqref{eq:EDdef}
  is denoted by $K$ in \cite{Loisel:2013a}.}
We use it here to define the exchange operator
\begin{align}
\label{eq:XdefGeneral}
  \mathcal{X}_m := 2 E_m - I,
\end{align}
which is actually the \emph{reflection} corresponding to the projection $E_m$ (cf.\ Remark~\ref{rem:XP}).
Note that for a bilateral facet $F$ with $\mathcal{N}_F = \{i, j\}$,
\[
  (\mathcal{X}_m \lambda)_{iF} = 2 (\tfrac{1}{2} \lambda_{iF} + \tfrac{1}{2} \lambda_{jF}) - \lambda_{iF} = \lambda_{jF}\,,
\]
so \eqref{eq:XdefGeneral} is a genuine generalization of the \emph{bilateral} exchange operator from \eqref{eq:XBilateral}.

\begin{proposition}
  Let the glob system $\mathcal{G}$ and the trace operator $T$ be constructed as above.
  With $E_m$, $\mathcal{X} = \mathcal{X}_m$ defined as in \eqref{eq:EDdef}--\eqref{eq:XdefGeneral},	
	Assumption~\ref{ass:A2} holds true.
\end{proposition}
\begin{proof}
	Property~(i):
	It is easy to show that $E_m$ is a projection, i.e., $E_m^2 = E_m$.
	From this we see that $\mathcal{X}_m^2 = (2 E_m - I)^2 = 4 E_m^2 - 4 E_m + I = I$.\\
  Property~(ii):
  Apparently, $I - \mathcal{X}_m = 2(I - E_m)$. So $(I - \mathcal{X}_m) T u = 0$ if and only if
	\[
	  \begin{cases}
		  \forall i \in \mathcal{N}_G \colon
		  T_{iG} u_i - \frac{1}{|\mathcal{N}_G|} \sum_{j \in \mathcal{N}_G} T_{jG} u_j = 0
			& \text{for all } G \in \mathcal{G} \text{ that are not Dirichlet globs,}\\
			T_{iG} u_i - 0 = 0
			& \text{for all Dirichlet globs } G\text{, where } \mathcal{N}_G = \{ i \},
		\end{cases}
	\]
	which means that for non-Dirichlet globs $G$, the values $\{ T_{jG} u_j \}_{j \in \mathcal{N}_G}$ must be equal,
	and for Dirichlet globs, $T_{iG} u_i = 0$.
	Since the glob system is admissible, this concludes the proof.
\end{proof}

\begin{remark}
\label{rem:ED}
  The averaging operator $E_m$ defined in \eqref{eq:EDdef} is only one out of a whole family.
	For each glob $G$, let $\{ D_{jG} \}_{j \in \mathcal{N}_G}$ be linear operators $D_{iG} \colon U_G \to U_G$
  that form a partition of unity, i.e.,
  \begin{align}
    \sum \nolimits_{j \in \mathcal{N}_G} D_{jG} = I.
  \end{align}
  Using these, we define
  \begin{align}
	  (E_D \lambda)_{iG} := \sum \nolimits_{j \in \mathcal{N}_G} D_{jG} \lambda_{jG}
		\qquad \forall G \in \mathcal{G}\ \forall i \in \mathcal{N}_G\,.
	\end{align}
	Apparently, $E_D$ is a projection, i.e., $E_D^2 = E_D$, and we can define an interface exchange operator by
	$\mathcal{X}_D := 2 E_D - I$.
  The averaging operators $E_D$ play a principal role in FETI and balancing methods
  \cite{ToselliWidlund:Book,Dohrmann:2003a,PechsteinDohrmann:2017a}.
  The simplest choice of weights is the \emph{multiplicity scaling}
  $D_{jG} = \frac{1}{|\mathcal{N}_G|} I$,
	in which case, $E_D = E_m$. We will revisit this type of weighted projection in Sect.~\ref{sect:globLocImp}.
\end{remark}


\medskip

To summarize Section~\ref{sect:facetSystems}:
The concept of admissible facet systems leads to a natural definition of the interface exchange operator $\mathcal{X}$
such that Assumption~\ref{ass:A2} holds.
In the discrete case, one can systematically construct \emph{bilateral facet systems} or \emph{glob systems} (each of them admissible).
In the continuous case, bilateral admissible facet systems are available at least for the de Rham complex
(see Examples~\ref{ex:swapping}, \ref{ex:swappingHDiv}).

\section{Interface flux formulation}
\label{sect:interfaceFluxes}

With the help of the trace operator $T$ and the exchange operator $\mathcal{X}$, the subdomain flux formulation \eqref{eq:subdflux}
is equivalent to \eqref{eq:subdfluxT}, i.e., $A u - t = f$, $(I - \mathcal{X})T u = 0$, and $R^\top t = 0$.
Although all three conditions are proper equations, still the variable $t$ is a \emph{volumetric} distribution.
The following lemma provides a characterization of $\ker(R^\top)$ in terms of \emph{dual traces}.

\begin{lemma}
\label{lem:kerRT}
  Let Assumption~\ref{ass:A2} hold. Then
	\[
    \ker(R^\top) = \overline{ \{ T^\top \tau \colon (I + \mathcal{X}^\top) \tau = 0, \ \tau \in \Lambda^* \}}.
	\]
	The space $\{ T^\top \tau \colon (I + \mathcal{X}^\top) \tau = 0, \ \tau \in \Lambda^* \}$
	is closed if and only if $\range((I-\mathcal{X})T)$ is closed.
\end{lemma}
\begin{proof}
  The proof makes use of Banach's closed range theorem (see, e.g., \cite[p.~23ff]{McLean:Book} or \cite[Sect.~VII.5]{Yosida:Book}).
	Let $X$, $Y$ be Banach spaces. Given a subset $W \subseteq X$,
	the associated \emph{annihilator}\footnote{called \emph{polar set} in \cite[p.~58]{GiraultRaviart:Book1986}}
	is defined as $W^0 := \{ \psi \in X^* \colon \langle \psi, w \rangle = 0 \ \forall w \in W \}$.
  For any bounded linear operator $B \colon X \to Y$,
	\begin{align}
	\label{eq:polarKerRange}
	  \ker(B^\top) = \range(B)^0\,, \qquad \ker(B)^0 = \overline{\range(B^\top)},
	\end{align}
	cf.\ \cite[Lem.~2.10, Lem.~2.11]{McLean:Book}.
	In our context, since $\range(R) = \ker((I - \mathcal{X}) T)$, it follows that
	\[
	  \ker(R^\top) = \range(R)^0 = \ker((I - \mathcal{X}) T)^0 = \overline{\range(T^\top (I - \mathcal{X}^\top))}.
	\]
	Apparently,
	\[
	  \range(T^\top (I - \mathcal{X}^\top)) = \{ T^\top \tau \colon \tau \in \range(I - \mathcal{X}^\top) \},
	\]
	and this space is closed if and only if $\range((I - \mathcal{X})T)$ is closed \cite[Thm.~2.13]{McLean:Book}.
	The proof is concluded by noting that $\range(I - \mathcal{X}^\top) = \ker(I + \mathcal{X}^\top)$, see Lemma~\ref{lem:X}.
\end{proof}

The result of the previous lemma gives rise to the following reformulation.

\medskip

\noindent%
\fbox{\parbox{\textwidth}{
\emph{Interface flux formulation}: $\text{find } (u, \tau) \in U \times \Lambda^* \colon$
\begin{align}
\label{eq:traceflux}
\begin{aligned}
  A u - T^\top \tau & = f,\\
	(I - \mathcal{X}) T u & = 0,\\
	(I + \mathcal{X}^\top) \tau & = 0.
\end{aligned}
\end{align}}}

\medskip

\begin{remark}
  Using Remark~\ref{rem:XTconj}, we obtain
  $\ker(R^\top) = \overline{ \{ T^\top \tau \colon (I + \mathcal{X}^\herm) \tau = 0,\ \tau \in \Lambda^* \} }$.
  Therefore, we are allowed to replace $\mathcal{X}$ by $\overline{\mathcal{X}}$ in the second line and/or
  $\mathcal{X}^\top$ by $\mathcal{X}^\herm$ in the third line of \eqref{eq:traceflux}.
  For the simple exchange operators $\mathcal{X}$ from~\eqref{eq:XBilateral} and from~\eqref{eq:XdefGeneral},
  $\mathcal{X}^\herm = \mathcal{X}^\top$ anyway.
\end{remark}

Before discussing the connection between \eqref{eq:traceflux} and \eqref{eq:global},
let us investigate the uniqueness of solutions to \eqref{eq:traceflux}.

\begin{definition}
\label{def:spaceZ}
  The \emph{interface flux redundancy space} is given by
	\[
	  \mathcal{Z} := \ker(T^\top) \cap \ker(I + \mathcal{X}^\top).
	\]
\end{definition}

\begin{proposition}
\label{prop:uniqueTau}
  Let \ref{ass:A1}--\ref{ass:A2} hold and let $f = 0$. Then $(u, \tau)$ is a solution of \eqref{eq:traceflux} if and only if
	\[
	  u = 0 \quad \text{and} \quad \tau \in \mathcal{Z}.
	\]
\end{proposition}
\begin{proof}
  From the second line of \eqref{eq:traceflux} we find by \ref{ass:A2}
	that there exists $\widehat u \in \widehat U$ with $u = R \widehat u$.
	The third line and Lemma~\ref{lem:kerRT} imply that $T^\top \tau \in \ker(R^\top)$.
	Applying $R^\top$ to the first line of \eqref{eq:traceflux} proves $R^\top A R \widehat u = 0$,
	which shows that $\widehat u = 0$ by our assumptions on $\widehat A$.
	The remaining equations yield $T^\top \tau = 0$ and $(I + \mathcal{X}^\top) \tau = 0$.
\end{proof}

Note that $\overline{\range(T)} = \Lambda$ implies $\ker(T^\top) = \{ 0 \}$,
which means that $\mathcal{Z}$ is trivial.
Therefore, the only interesting case where $\mathcal{Z}$ can be non-trivial is that of finite dimensions.
As it turns out, each cycle of the connectivity graph corresponds to a non-trivial element of the redundancy space.

\begin{theorem}
\label{thm:Zcharact}
  In the finite-dimensional case,
  let Assumption~\ref{ass:A1}--\ref{ass:A2} hold and let $\mathcal{F}$ be a bilateral discrete facet system that is admissible.
	For each global interface dof $k \in \mathcal{D}_\Gamma$,
	let $\ell_k = \cardinality{\mathcal{E}_k} + 1 - \cardinality{\mathcal{N}_k}$
	denote the number of independent cycles of the connectivity graph $\mathcal{C}_k$ (Def.~\ref{def:connectivityGraph}).
	Then the dimension of the redundancy space $\mathcal{Z}$
	is given by $\dim(\mathcal{Z}) = \sum_{k \in \mathcal{D}_\Gamma} \ell_k$.
\end{theorem}
\begin{proof}
  See Appendix~\ref{apx:ProofZCharact}, where even a basis for $\mathcal{Z}$ is constructed.
\end{proof}

After having characterized the redundancy space $\mathcal{Z}$,
our next goal is finding conditions under which formulation~\eqref{eq:traceflux}
is equivalent to the original formulation~\eqref{eq:global}.
In the finite-dimensional case, the space in Lemma~\ref{lem:kerRT} is always closed
and so we can parametrize any $t \in \ker(R^\top)$ as $t = T^\top \tau$ with $(I + \mathcal{X}^\top) \tau = 0$.
In the infinite-dimensional case, there are two possibilities: $\range((I-\mathcal{X})T)$ can be closed or not.
While the following lemma provides a sufficient condition for this space to be closed,
Lemma~\ref{lem:H1} below helps in the non-closed case.

\begin{lemma}
\label{lem:Tsurjective}
  Let \ref{ass:A1}--\ref{ass:A2} hold and assume in addition that $\range(T) = \Lambda$.
	Then the space $\range((I - \mathcal{X})T)$ is closed.
	In particular, together with Lemma~\ref{lem:kerRT}, this implies
	\[
	  \ker(R^\top) = \{ T^\top \tau \colon (I + \mathcal{X}^\top) \tau = 0, \ \tau \in \Lambda^* \}.
	\]
\end{lemma}
\begin{proof}
	Due to Lemma~\ref{lem:X}, $\range(I - \mathcal{X})$ is closed.
	By assumption $\range(T) = \Lambda$, so altogether we can conclude that $\range((I - \mathcal{X})T)$ is closed.
\end{proof}

\begin{lemma}
\label{lem:H1}
  Let \ref{ass:A1}--\ref{ass:A2} hold and assume that $\overline{\range(T)} = \Lambda$. Then
	\[
	  \range(T^\top) \cap \ker(R^\top) = \big\{ T^\top \tau \colon (I + \mathcal{X}^\top) \tau = 0, \ \tau \in \Lambda^* \big\}.
	\]
\end{lemma}
\begin{proof}
  ``$\supseteq$'':
	The space on the right is, by Lemma~\ref{lem:kerRT}, contained in $\ker(R^\top)$ and it is obviously also contained in $\range(T^\top)$.\\
	``$\subseteq$'':
  Suppose $\tau \in \Lambda^*$ with $T^\top \tau \in \ker(R^\top) = \range(R)^0$ by \eqref{eq:polarKerRange}. Then
	\[
	  \langle T^\top \tau, v \rangle = 0 \qquad \forall v \in \range(R),
	\]
	which implies
	\[
	  \langle \tau, T v \rangle = 0 \qquad \forall v \in \range(R).
	\]
	Since by \ref{ass:A2}, $v \in \range(R)$ if and only if $(I - \mathcal{X})T v = 0$, we can conclude that
	\[
		\langle \tau, \lambda \rangle = 0 \qquad \forall \lambda \in \range(T) \text{ with } (I - \mathcal{X})\lambda = 0.
	\]
  By assumption $\range(T)$ is dense in $\Lambda$, so it follows that
	\[
	  \langle \tau, \lambda \rangle = 0 \qquad \forall \lambda \in \ker(I - \mathcal{X}).
	\]
  This implies that $\tau \in \ker(I - \mathcal{X})^0 = \range(I - \mathcal{X}^\top) = \ker(I + \mathcal{X}^\top)$ by Lemma~\ref{lem:X}.
\end{proof}

With these tools available, we can state the main theorem of this section.

\begin{theorem}
\label{thm:traceflux}
  Under Assumptions~\ref{ass:A1}--\ref{ass:A2}, the following statements hold.
	\begin{enumerate}
	\item[(i)] If $(u, \tau)$ solves \eqref{eq:traceflux} then $u = R \widehat u$ where $\widehat u$ is the unique solution of \eqref{eq:global}.
	\item[(ii)] If $\widehat u$ solves \eqref{eq:global} and, in addition, either
	  \begin{enumerate}
		\item[(a)] all spaces are finite-dimensional, or
		\item[(b)] $\range(T) = \Lambda$, or
		\item[(c)] $\overline{\range(T)} = \Lambda$ and $A R \widehat u - f \in \range(T^\top)$,
		\end{enumerate}
    then there exists $\tau \in \Lambda^*$such that $(R \widehat u, \tau)$ solves \eqref{eq:traceflux}.
		In cases~(b) and (c), $\tau$ is guaranteed to be unique,
		whereas in the finite-dimensional case (a), $\tau$ is only unique up to an element from the space $\mathcal{Z}$,
		see Proposition~\ref{prop:uniqueTau}.
	\item[(iii)]
			In cases~(a) and (b), there exists a bounded linear solution operator
			$\mathcal{S}^{(\tau)} \colon f \mapsto (u, \tau)$ for \eqref{eq:traceflux}.
	\end{enumerate}
\end{theorem}

\begin{remark}
  The assumption $\overline{\range(T)} = \Lambda$ in Case~(c) is merely of technical type.
	If this assumption is not fulfilled for an infinite-dimensional setting, it means that the trace space $\Lambda$
	is chosen unnecessarily large.
\end{remark}

\begin{proof}[Proof of Theorem~\ref{thm:traceflux}]
  (i) If $(u, \tau)$ solves \eqref{eq:traceflux} then $T^\top \tau \in \ker(R^\top)$ due to Lemma~\ref{lem:kerRT}. Hence,
	$(u, T^\top \tau)$ solves \eqref{eq:subdfluxT} and so $u = R \widehat u$, where $\widehat u$ is the unique solution of \eqref{eq:global}.\\
	(ii) Suppose $\widehat u$ solves \eqref{eq:global}.
	Then there exists $(u, t)$ solving \eqref{eq:subdfluxT}, in particular with $t \in \ker(R^\top)$.\\
	In cases~(a) and~(b), the space $\{ T^\top \tau \colon (I + \mathcal{X}^\top) \tau = 0, \ \tau \in \Lambda^* \}$ is closed (Lemma~\ref{lem:Tsurjective})
	and coincides with $\ker(R^\top)$, see Lemma~\ref{lem:kerRT}.
	Hence there exists $\tau \in \Lambda^*$ with $(I + \mathcal{X}^\top)\tau = 0$ such that $t = T^\top \tau$.
	Apparently, $(u, \tau)$ solves \eqref{eq:traceflux}.\\
	In case~(c), it follows (by assumption) that $t = A R \widehat u - f \in \range(T^\top) \cap \ker(R^\top)$, and so by Lemma~\ref{lem:H1}
	there exists $\tau \in \Lambda^*$ with $t = T^\top \tau$ and $(I + \mathcal{X}^\top) \tau = 0$.
	Again, $(u, \tau)$ solves \eqref{eq:traceflux}.\\
	Due to Proposition~\ref{prop:uniqueTau}, $\tau$ in \eqref{eq:traceflux} is only unique up to an element from $\mathcal{Z}$.
	In case~(b), however, $\ker(T^\top) = \range(T)^0 = \Lambda^0 = \{ 0 \}$.
	In case~(c), $\ker(T^\top) = \range(T)^0 = \overline{\range(T)}^0 = \{ 0 \}$.\\
	(iii)	In cases~(a) and~(b), $\ker(R^\top) = \{ T^\top \tau \colon (I + \mathcal{X}^\top) \tau = 0, \ \tau \in \Lambda^* \}$
	is closed and so there exists a bounded linear operator $Q \colon \ker(R^\top) \to \ker(I + \mathcal{X}^\top)$
	with the property that $T^\top Q t = t$ for $t \in \ker(R^\top)$.
	Recall that $\mathcal{S} \colon U^* \to \range(R) \times \ker(R^\top)$
	from Lemma~\ref{lem:subdflux} is a bounded solution operator for \eqref{eq:subdflux}.
	We define
	\begin{align*}
	  \mathcal{S}^{(\tau)} \colon U^* \to \ker((I-\mathcal{X})T) \times \ker(I + \mathcal{X}^\top)
		\colon f \mapsto (\mathcal{S}_u f, Q \mathcal{S}_t f),
	\end{align*}
	where $\mathcal{S} f = (\mathcal{S}_u f, \mathcal{S}_t f)$.
	We verify three properties of $\mathcal{S}^{(\tau)}$.
	\begin{enumerate}
	\item[1)] The operator $\mathcal{S}^{(\tau)}$ is well-defined, linear, and bounded.
	\item[2)] If $(u, \tau) = \mathcal{S}^{(\tau)} f$ then $A u - T^\top \tau = f$.
	  This follows from the properties of $\mathcal{S}$ and the fact that $T^\top Q \mathcal{S}_t f =  \mathcal{S}_t f$
		since $\mathcal{S}_t f \in \ker(R^\top)$.
	\item[3)] Any $(u, \tau) \in \ker((I - \mathcal{X})T) \times \ker(I + \mathcal{X}^\top)$ fulfills $\mathcal{S}^{(\tau)}(A u - T^\top \tau) = (u, \tau + z)$
	  for some element $z \in \mathcal{Z}$.
		To see this, we define $(v, t) := \mathcal{S}(A u - T^\top \tau) \in \range(R) \times \ker(R^\top)$.
		By construction, $A v - t = A u - T^\top \tau$,
		and Assumption~\ref{ass:A2} and Lemma~\ref{lem:kerRT} imply that $(u, T^\top \tau) \in \range(R) \times \ker(R^\top)$.
		Therefore, $(u - v, t - T^\top)$ solve the homogeneous problem,
		and so Lemma~\ref{lem:subdflux}(iii) implies $v = u$ and $t = T^\top \tau$.
		The second component of $\mathcal{S}^{(\tau)}(A u - T^\top \tau)$ is therefore given by $\sigma = Q t = Q T^\top \tau$.
		Applying $T^\top$ and using that $t \in \ker(R^\top)$ shows that
		\[
		  T^\top \sigma = T^\top \tau.
		\]
		Therefore, $\sigma - \tau \in \ker(T^\top)$. By construction, $\sigma = Q t$ is also in $\ker(I + \mathcal{X}^\top)$,
		which is a property that it shares with $\tau$. Hence, $\sigma - \tau \in \mathcal{Z}$.
	\end{enumerate}
	Altogether, $\mathcal{S}^{(\tau)}$ is a bounded solution operator for \eqref{eq:traceflux}.
\end{proof}

A short summary of Theorem~\ref{thm:traceflux}:
In case~(b), i.e., if the trace operator is surjective,
  Formulation~\eqref{eq:traceflux} is well-posed and equivalent to \eqref{eq:global}.
In the finite-dimensional case~(a), the same holds, up to possible non-uniqueness of $\tau$.
For the infinite-dimensional case with non-surjective trace operator, case~(c),
equivalence of~\eqref{eq:traceflux} and~\eqref{eq:global}
can be guaranteed under a mild density assumption and the
\emph{regularity assumption} $A R \widehat u - f \in \range(T^\top)$.
The solution operator, however, is possibly unbounded.
In the following, some examples are given that apply to the primal formulation of the Helmholtz or Laplace equation in $H^1$.

\begin{example}[$L^2$ traces]
\label{ex:regularityL2}
  Consider the setup of Example~\ref{ex:tracesH1Split} with $\Lambda_i = \bigcup_{F \in \mathcal{F}_i} L^2(F)$.
	Then already the single-facet trace operator $T_{iF} \colon H^1(\Omega_i) \to L^2(F)$ fails to have a closed range,
	and $\range(T)$ is not closed as well. Therefore, we are in case~(c). The condition $\overline{\range(T)} = \Lambda$ is fulfilled
	and the regularity assumption is equivalent to $\partial \widehat u/\partial \normal_i \in L^2(\Gamma_i)$ for the interface
	$\Gamma_i := \partial\Omega_i \cap \bigcup_{j \neq i}\partial\Omega_j$
	of each subdomain	(cf.\ \cite[p.~314]{Despres:1990a} and \cite[p.~10]{CollinoGhanemiJoly:2000a}).
\end{example}

\begin{example}[a problem with cross points]
\label{ex:crosspointProblem}
  Consider the setup from Example~\ref{ex:tracesH1Split} with $\Lambda_i = \bigcup_{F \in \mathcal{F}_i} H^{1/2}(F)$
	and with a cross point. More precisely, assume that for a subdomain $\Omega_i$
	there are two faces that share a common edge in three dimensions or a common vertex in two dimensions.
	In contrast to Example~\ref{ex:regularityL2}, the single-facet trace operator $T_{iF}$ \emph{does} have closed range,
	but the \emph{collective} trace operator $T_i \colon H^1(\Omega_i) \to \bigcup_{F \in \mathcal{F}_i} H^{1/2}(F)$
	does not; in particular it is not surjective. This kind of obstruction is analyzed in detail in Grisvard's monograph \cite{Grisvard:Book1985}.
	To get the idea, let $F$ and $G$ be two edges of a rectangular subdomain $\Omega_i \subset \mathbb{R}^2$ that share a common vertex.
	Due to the peculiar property that $C^\infty_0(F)$ is dense in $H^{1/2}(F)$ \cite[Thm.~3.40]{McLean:Book},
	we can find a sequence of $C^\infty$ functions in $\Omega_i$ that vanish entirely on $G$ and whose trace on $F$ converges to the constant function $1$
	in the $\|\cdot\|_{H^{1/2}(F)}$-norm.
	So the collective trace of this sequence has a limit in the product space $H^{1/2}(F) \times H^{1/2}(G)$, but this limit is not the collective trace of any $H^1$ function.
	To summarize, $\range(T) \subsetneq \Lambda$ but $\overline{\range(T)} = \Lambda$.
	The regularity condition from case~(c)
	is equivalent to saying (for an interior subdomain $\Omega_i)$
	that the restriction of $\partial \widehat u/\partial \normal_i \in H^{-1/2}(\partial\Omega_i)$
	to each facet $F \in \mathcal{F}_i$ is in $H^{-1/2}(F)$,
	which is the dual of $H^{1/2}(F)$ and contains distributions that are \emph{extendible by zero} to $H^{-1/2}(\partial\Omega_i)$.
\end{example}

\begin{example}[a non-collective, surjective trace operator]
  Consider the case of $\widehat U = H^1(\Omega)$ and $U_i = H^1(\Omega_i)$
	for a general subdomain partition with cross points,
	where we do \emph{not} split the interface into faces, i.e.,
	we use $T_i \colon H^1(\Omega_i) \to H^{1/2}(\Gamma_i)$ with
	$\Gamma_i := \partial\Omega_i \cap \bigcup_{j \neq i}\partial\Omega_j$,
	which is perfectly surjective,
	so the strong assumption $\range(T) = \Lambda$ holds.
	However, in general, one cannot use the simple swapping operator $\mathcal{X}$,
	see also Remark~\ref{rem:strongAssRangeTLambda} and Sect.~\ref{sect:interfaceExchange}.
	Note also that if the interface touches the Dirichlet boundary,
	some traces spaces may have to be adapted in order to maintain surjectivity.
\end{example}

\begin{figure}
\begin{center}
  \def\svgwidth{0.55\textwidth}
  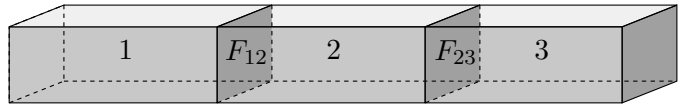
	\caption{\label{fig:1DDecomp}%
	  Illustration of a typical ``1D decomposition'': facets are separated.}
\end{center}
\end{figure}

\begin{figure}
\begin{center}
  \def\svgwidth{0.37\textwidth}
  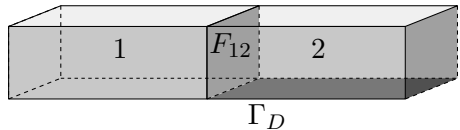
	\caption{\label{fig:1DDecompDirichlet}%
	  Example of non-matching trace spaces: the local space on the left is $H^1(\Omega_1)$ with its trace space $H^{1/2}(F_{12})$.
		Due to the Dirichlet boundary $\Gamma_D$, the space on the right, however is constrained and the trace space is a genuine subspace of $H^{1/2}(F_{12})$
		with its functions properly decaying to zero on the bottom edge of $F_{12}$.}
\end{center}
\end{figure}

\begin{example}[no junctions]
\label{ex:HelmholtzNoJunctions}
  Consider the case of $\widehat U = H^1(\Omega)$ and $U_i = H^1(\Omega_i)$
	for a general subdomain partition with \emph{no junctions} in the sense of \eqref{eq:noJunctions},
	such that the interface naturally splits into facets that are closed manifolds of co-dimension one,
	each of them with two subdomains on each side.
	We can use $T_i \colon H^1(\Omega_i) \to H^{1/2}(\Gamma_i)$ with
	$\Gamma_i := \partial\Omega_i \cap \bigcup_{j \neq i}\partial\Omega_j$,
	which is perfectly surjective,
	so the strong assumption $\range(T) = \Lambda$ holds.
	Apparently, the simple type of swapping operator $\mathcal{X}$ can be used without any complications.
\end{example}

\begin{example}[no cross points]
\label{ex:HelmholtzNoCrosspointButJunctions}
  Consider the case of $\widehat U = H^1(\Omega)$, and $U_i = H^1(\Omega_i)$
	and suppose that we have \emph{no cross points},
	in the sense that the interface $\partial\Omega_i \cap \partial\Omega_j$
	between two subdomains is either empty or a (possibly open) manifold of co-dimension one,
	and each such interface has positive distance from each other, for an example see Figure~\ref{fig:1DDecomp}.
	Note that this assumption allows two subdomains to meet at the outer boundary, cf.\ Fig.~\ref{fig:junctionCrossPoint},
	so~\eqref{eq:noJunctions} does not necessarily hold.
	In such a case, if we use $T_i \colon H^1(\Omega_i) \to \Lambda := H^{1/2}(\Gamma_i)$,
	where $\Gamma_i := \partial\Omega_i \cap \bigcup_{j \neq i}\partial\Omega_j$,
	then $T_i$ is surjective.
	At the same time, since the local interface $\Gamma_i$
	is the union of disconnected facets, each of which is shared by a unique neighboring subdomain,
	we can use the simple type of swapping operator $\mathcal{X}$ without any complications.
	Note, however, that if $U_i$ includes a Dirichlet condition and $\Gamma_i$ intersects the Dirichlet boundary in a manifold of co-dimension two,
	then the trace space $\Lambda_i$ must be adapted to include the same Dirichlet boundary condition (leading to a $H^{1/2}_{00}$-like space).
	Otherwise, there is no chance for surjectivity.
	If the Dirichlet boundary only lies on one side of a face, this can lead to non-matching trace spaces for the same facet,
	such that the simple swapping operator cannot be used anymore, cf.\ Fig.~\ref{fig:1DDecompDirichlet}.
\end{example}

The following two examples concern the discrete case.

\begin{example}
  Let $\mathcal{F}$ be a discrete bilateral facet system with $\mu_{\max} > 2$.
	Then the conditions of Cases~(b) and~(c) are \emph{not} fulfilled as
	$\range(T) \subsetneq \Lambda$, cf.\ Proposition~\ref{prop:bilateralCrosspoints}.
	Therefore, Formulation~\ref{eq:traceflux} has a solution operator, but the component $\tau$ of the solution is not unique.
\end{example}

\begin{example}
  Assume that we have either a bilateral discrete facet system with $\mu_{\max} \le 2$
	or a glob system. Then $\range(T) = \Lambda$, cf.\ Proposition~\ref{prop:bilateralCrosspoints}
	and Proposition~\ref{prop:globsRangeGLambda}.
	Therefore, Formulation~\ref{eq:traceflux} has a unique solution operator.
\end{example}

\begin{remark}
\label{rem:strongAssRangeTLambda}
  In the infinite-dimensional case, the most appealing version is case~(b) with the strong assumption $\range(T) = \Lambda$.
	While the latter can be fulfilled quite easily by a proper choice of $\Lambda$ and $T$,
  the attentive reader may ask: \emph{do there even exist operators $\mathcal{X}$ fulfilling Assumption~\ref{ass:A2}
	under these circumstances?}
	This question is the subject of Sect.~\ref{sect:interfaceExchange} below
	and a short answer is: in fact always, but in general $\mathcal{X}$ will be non-local.
\end{remark}

The following proposition will be helpful later on in Sect.~\ref{sect:convStrong}.

\begin{proposition}
\label{prop:solOpTau}
  Let $E \colon \Lambda \to U$ let be an arbitrary linear and bounded extension operator such that $T E = I$.
	(In general, such an extension is not unique, but it can only exist if $\range(T) = \Lambda$.)
	Then the solution operator from the proof of Theorem~\ref{thm:traceflux} has the form
  \begin{align*}
    \mathcal{S}^{(\tau)} = \begin{bmatrix} I & 0 \\ 0 & E^\top \end{bmatrix} \mathcal{S}
  	= \begin{bmatrix} I \\ E^\top A \end{bmatrix} R \widehat A^{-1} R^\top - \begin{bmatrix} 0 \\ E^\top \end{bmatrix}.
  \end{align*}
\end{proposition}
\begin{proof}
  If suffices to show that $E^\top \colon U^* \to \Lambda^*$ can be used instead of the operator $Q$ in the proof of Theorem~\ref{thm:traceflux}.
	Therefore we have to show that (i) $T^\top E^\top t = t$ for all $t \in \ker(R^\top)$ and (ii) $E^\top t \in \ker(I + \mathcal{X}^T)$ for all $t \in \ker(R^\top)$.
	Property~(i) follows simply from the fact that $\ker(R^\top) \subseteq \range(T^\top)$.
	To see Property~(ii), recall from Lemma~\ref{lem:Tsurjective} that for any $t \in \ker(R^\top)$ there exists $\tau \in \ker(I + \mathcal{X}^\top)$ such that $t = T^\top \tau$.
	Hence, $E^\top t = E^\top T^\top \tau = \tau \in \ker(I + \mathcal{X}^\top)$.
\end{proof}

\section{Formulations with Robin Transmission Conditions}
\label{sect:RobinTC}

In this section, based on the interface flux formulation, another formulation is derived using generalized Robin transmission conditions.
This leads to the classical method by Despr\'es (in the continuous case) and to the FETI-2LM formulation and variants thereof
(in the discrete case).

\subsection{Robin Transmission Conditions}

For the $H^1$-setting, recall the classical Robin transmission conditions \eqref{eq:RobinClassical}
and the generalized transmission conditions \eqref{eq:RobinMij}
with the impedance operator on each facet.
If, like in Example~\ref{ex:tracesH1Split}, the trace space has the form
$\Lambda_i = \productspace_{F \in \mathcal{F}_i} U_F$, then we can use an impedance operator $M_F \colon U_F \to U_F^*$ on each facet,
with a bounded inverse.
The generalized incoming and outgoing impedance traces are given by
$\alpha M_F T_{iF} u_i \pm \tau_{iF}$, where $\alpha = \complexi$ or $1$ (depending whether we deal with the Helmholtz or the Laplace equation)
and where $\tau_{iF}$ is the component of $\tau_i \in \Lambda_i^*$ corresponding to $F$ and stands for the normal derivative $\partial u_i / \partial \normal_i$.
For the choice $U_F = L^2(F)$ and $\langle M_F \lambda, \mu \rangle = \kappa \int_F \lambda\, \mu \, ds$,
we reproduce the classical impedance traces.
Forming the operators $M_i = \diag(M_F)_{F \in \mathcal{F}_i}$
and $M = \diag(M_i)_{i=1}^N \colon \Lambda \to \Lambda^*$,
we can evaluate all these traces simultaneously, $\alpha M T u \pm \tau$.
Note that if all operators $M_F$ have a bounded inverse, so has $M$.
Although at a certain point later on, we will return to impedance operators $M$ of such particular block-diagonal structure,
the following theory covers more general situations.

\medskip

\noindent%
\fbox{%
\begin{minipage}{0.985\textwidth}
\begin{assumption}
\label{ass:A3}
  The \emph{impedance operator} $M \colon \Lambda \to \Lambda^*$ is linear and bounded,
  and the operator $(M + \mathcal{X}^\top M \mathcal{X})$ has a bounded inverse.
\end{assumption}
\end{minipage}}

\begin{example}
  Consider Example~\ref{ex:tracesH1Split}	with the swapping operator $\mathcal{X}$ from Example~\ref{ex:swapping}
	and assume a block-diagonal structure $M = \diag(M_i)_{i=1}^N$ and $M_i = \diag(M_{iF})_{F \in \mathcal{F}_i}$.
	Then Assumption~\ref{ass:A3} states that the \emph{sum} $(M_{iF} + M_{jF})$ on each facet $F \in \mathcal{F}_i \cap \mathcal{F}_j$ is invertible.
	However, $M_{iF}$ may differ from $M_{jF}$.
\end{example}

\begin{lemma}
\label{lem:Robin}
  Let \ref{ass:A2}--\ref{ass:A3} hold and let $\alpha \in \mathbb{C} \setminus \{ 0 \}$.
  Then for any $\gamma \in \Lambda$ and $\tau \in \Lambda^*$, the following statements are equivalent:
	\begin{enumerate}
	\item[(i)] $(I - \mathcal{X}) \gamma = 0$ and $(I + \mathcal{X}^\top) \tau = 0$,
	\item[(ii)] $\alpha M (I - \mathcal{X}) \gamma + (I + \mathcal{X}^\top) \tau = 0$.
	\end{enumerate}
\end{lemma}
\begin{proof}
  Obviously (i) $\implies$ (ii). To show the reverse implication, we apply $(I - \mathcal{X}^\top)$ to (ii):
	\[
	  \alpha (I - \mathcal{X}^\top) M (I - \mathcal{X}) \gamma + \underline{(I - \mathcal{X}^\top) (I + \mathcal{X}^\top)} \tau = 0.
	\]
	Due to Lemma~\ref{lem:X} the underlined expression vanishes. A side computation reveals that
	\[
	  (I - \mathcal{X}^\top) M (I - \mathcal{X})
		= M + \mathcal{X}^\top M \mathcal{X} - M \mathcal{X}
		  - \underbrace{\mathcal{X}^\top M}_{= \mathcal{X}^\top M \mathcal{X} \mathcal{X}}
		= (M + \mathcal{X}^\top M \mathcal{X}) (I - \mathcal{X}).
	\]
	Therefore, $\alpha (M + \mathcal{X}^\top M \mathcal{X}) (I - \mathcal{X}) \gamma = 0$. Since $\alpha \neq 0$ and $(M + \mathcal{X}^\top M \mathcal{X})$ is invertible,
	this shows that $(I - \mathcal{X}) \gamma = 0$.
	Insertion into (ii) proves that $(I + \mathcal{X}^\top) \tau = 0$.
\end{proof}

\begin{remark}
  In the proof of Lemma~\ref{lem:Robin}, only the property $\mathcal{X}^2 = I$ is used from \ref{ass:A2}.
\end{remark}

\begin{remark}
  In Sect.~\ref{sect:convergence}, we will use the \emph{stronger} assumptions that $M$ itself has a bounded inverse
	(Assumption~\ref{ass:A6})
	and that $\mathcal{X}^\top M \mathcal{X} = M$ (Assumption~\ref{ass:A4}).
\end{remark}

\begin{remark}
\label{rem:RobinH}
  As a viable alternative to Assumption~\ref{ass:A3}, one can assume that $M + \mathcal{X}^\herm M \mathcal{X}$ has a bounded inverse and obtain the equivalence between
	(i) $(I - \mathcal{X})\gamma = 0$ and $(I + \mathcal{X}^\herm) \tau = 0$ and
	(ii) $\alpha M(I - \mathcal{X}) \gamma + (I + \mathcal{X}^\herm) \tau = 0$.
	See also Remark~\ref{rem:HermitianImp} below.
\end{remark}

For the following, we fix the \emph{Robin parameter} $\alpha$
(later on, we will set $\alpha = 1$ for the coercive case and $\alpha = \complexi$ for the wave propagation case).
Using Lemma~\ref{lem:Robin}, Formulation~\eqref{eq:traceflux} is equivalent to
\begin{align}
\label{eq:tracefluxRobin}
  \text{find } (u, \tau) \in U \times \Lambda^* \colon \quad
	\begin{bmatrix} A & - T^\top \\ \alpha M(I - \mathcal{X}) T & (I + \mathcal{X}^\top) \end{bmatrix}
	\begin{bmatrix} u \\ \tau \end{bmatrix}
	= \begin{bmatrix} f \\ 0 \end{bmatrix},
\end{align}
where the second line enforces the generalized Robin transmission conditions.

\subsection{A formulation based on generalized Robin traces}

Using the impedance operator $M$ and the scalar Robin parameter $\alpha$ from the previous section,
we define the generalized \emph{impedance trace} (or generalized Robin trace)
$\lambda := \alpha M T u + \tau$.
In view of \eqref{eq:impedanceTraceForPlaneWave}, we can call $\lambda$ the \emph{incoming impedance trace}.
Applying the bijective transformation of variables $(u, \tau) \leftrightarrow (u, \lambda)$
to \eqref{eq:tracefluxRobin} we arrive at the

\medskip

\noindent%
\fbox{\parbox{\textwidth}{
 \emph{Interface impedance trace formulation:}
\begin{align}
\label{eq:waveRobinGeneral}
  \text{find } (u, \lambda) \in U \times \Lambda^* \colon \quad
	\begin{bmatrix} (A + \alpha T^\top M T) & - T^\top \\ -\alpha \mathcal{X}^\top (M + \mathcal{X}^\top M \mathcal{X}) T & (I + \mathcal{X}^\top) \end{bmatrix}
	\begin{bmatrix} u \\ \lambda \end{bmatrix}
	= \begin{bmatrix} f \\ 0 \end{bmatrix}.
\end{align}
}}


\begin{corollary}
\label{cor:waveRobin}
  Let \ref{ass:A1}--\ref{ass:A3} hold and let $f \in U^*$ be given. Then:
	\begin{enumerate}
	\item[(i)] If $(u, \lambda)$ solves \eqref{eq:waveRobinGeneral} then $u = R \widehat u$ where $\widehat u$ is the unique solution of \eqref{eq:global}.
  \item[(ii)] If, in addition, either
    \begin{enumerate}
		\item[(a)] all spaces are finite-dimensional, or
    \item[(b)] $\range(T) = \Lambda$, or
    \item[(c)] $\overline{\range(T)} = \Lambda$ and $A R \widehat u - f \in \range(T^\top)$,
		\end{enumerate}
		then there exists $\lambda \in \Lambda^*$ such that $(R \widehat u, \lambda)$ solves \eqref{eq:waveRobinGeneral}.
		In cases (b) and (c), $\lambda$ is guaranteed to be unique,
		whereas in the finite-dimensional case (a),
		$\lambda$ is only unique up to an element from $\mathcal{Z}$ (see Def.~\ref{def:spaceZ}).
	\item[(iii)]
	  In cases (a) and (b), there exists a bounded linear solution operator $\mathcal{S}^{(\lambda)} \colon f \mapsto (u,\lambda)$
		for \eqref{eq:waveRobin}. In case~(b), $\mathcal{S}^{(\lambda)}$ is unique.
	\end{enumerate}
\end{corollary}
\begin{proof}
  The proof follows from Theorem~\ref{thm:traceflux}, Lemma~\ref{lem:Robin},
	and the fact that the transformation~$(u, \tau) \leftrightarrow (u, \lambda)$ is isomorphic.
	In cases~(a) and (b), the solution operator is given by $\mathcal{S}^{(\lambda)} f = (u, \tau + \alpha M T u)$ where $(u, \tau) = \mathcal{S}^{(\tau)} f$.
\end{proof}

With an extension operator $E$ as in Proposition~\ref{prop:solOpTau}, the solution operator can be written as
\begin{align}
\label{eq:SlambdaDef}
  \mathcal{S}^{(\lambda)} = \begin{bmatrix} I & 0 \\ \alpha M T & I \end{bmatrix} \mathcal{S}^{(\tau)}
	= 
	\begin{bmatrix} I \\ \alpha M T + E^\top A \end{bmatrix} R \widehat A^{-1} R^\top
	- \begin{bmatrix} 0 \\ E^\top \end{bmatrix}.
\end{align}

The following assumption will mainly be used in Sect.~\ref{sect:convergence} and allows for a simplified formulation.

\medskip

\noindent%
\fbox{%
\begin{minipage}{0.985\textwidth}
\begin{assumption}
\label{ass:A4}
  $\mathcal{X}^\top M \mathcal{X} = M$.
\end{assumption}
\end{minipage}}

\medskip

This assumption can be interpreted in two ways. (i) If the exchange operator $\mathcal{X}$ is fixed,
\ref{ass:A4} restricts the choice of the impedance operator $M$ to ones that are ``the same from either side'', see Example~\ref{ex:A4} below.
(ii) If $M$ is fixed, \ref{ass:A4} restricts the choice of the exchange operator~$\mathcal{X}$;
this point of view will be adopted in Sect.~\ref{sect:interfaceExchange}.
Note also that under~\ref{ass:A4}, Assumption~\ref{ass:A3} actually requires that $M$ has a bounded inverse.

\begin{example}[facet-local impedance operator]
\label{ex:A4}
  For the setup from Example~\ref{ex:tracesH1Split} with the swapping operator $\mathcal{X}$ from Example~\ref{ex:swapping},
	assume that $M$ is block-diagonal with respect to the facets, i.e., $M = \diag(M_i)_{i=1}^N$
	and $M_i = \diag(M_{iF})_{F \in \mathcal{F}_i}$.
	Then Assumption~\ref{ass:A4} states that for a facet $F$ shared by subdomain $i$ and $j$,
	the condition $M_{iF} = M_{jF}$	must hold,
	i.e., we use \emph{the same impedance operator} on both sides of the facet.
\end{example}

\begin{remark}
  Under the additional assumption of \ref{ass:A4}, formulation~\eqref{eq:waveRobinGeneral} simplifies to
  \begin{align}
  \label{eq:waveRobin}
    \text{find } (u, \lambda) \in U \times \Lambda^* \colon \quad
  	\begin{bmatrix} (A + \alpha T^\top M T) & - T^\top \\ -2\alpha \mathcal{X}^\top M T & (I + \mathcal{X}^\top) \end{bmatrix}
  	\begin{bmatrix} u \\ \lambda \end{bmatrix}
  	= \begin{bmatrix} f \\ 0 \end{bmatrix}.
  \end{align}
  Multiplication of the second line by $\frac{1}{2\alpha} M^{-1} \mathcal{X}^\top$ leads to the formally symmetric system
	\begin{align}
	  \begin{bmatrix} A & - T^\top \\ - T & \frac{1}{2\alpha} M^{-1} (I + \mathcal{X}^\top) \end{bmatrix}
	\begin{bmatrix} u \\ \lambda \end{bmatrix}
	= \begin{bmatrix} f \\ 0 \end{bmatrix}.
	\end{align}
	This is because under Assumption~\ref{ass:A4}, $M^{-1} \mathcal{X}^\top = \mathcal{X} M^{-1}$.
\end{remark}

\begin{remark}
  Under the additional assumption of \ref{ass:A4}, the substitution $\lambda = M \underline{\lambda}$ in~\eqref{eq:waveRobinGeneral}
	leads to the system
  \begin{align}
  \label{eq:waveRobinAlt}
    \text{find } (u, \underline{\lambda}) \in U \times \Lambda \colon \quad
  	\begin{bmatrix} (A + \alpha T^\top M T) & - T^\top M \\ -2\alpha \mathcal{X} T & (I + \mathcal{X}) \end{bmatrix}
  	\begin{bmatrix} u \\ \underline{\lambda} \end{bmatrix}
  	= \begin{bmatrix} f \\ 0 \end{bmatrix},
  \end{align}
	which is used, e.g., in \cite{ClaeysParolin:Preprint2020}.
\end{remark}

The elimination of the primal variable $u$ from~\eqref{eq:waveRobinGeneral} requires the following assumption.

\medskip

\noindent%
\fbox{%
\begin{minipage}{0.985\textwidth}
\begin{assumption}
\label{ass:A5}
 The operator $(A + \alpha T^\top M T)$ has a bounded inverse.
\end{assumption}
\end{minipage}}

\medskip

If $M$ is block-diagonal w.r.t.\ the subdomains, i.e., $M = \diag(M_i)_{i=1}^N$ with $M_i \colon \Lambda_i \to \Lambda_i^*$,
then Assumption~\ref{ass:A5} means that $(A_i + \alpha T_i^\top M_i T_i)$ has a bounded inverse for each individual subdomain $i=1,\ldots,N$;
see also Assumption~\ref{ass:A6}, Sect.~\ref{sect:addAssConv}.

\begin{remark}
	For the Helmholtz equation posed in $H^1$ (Example~\ref{ex:modelProblemHelmholtz})
	with a block-diagonal impedance operator,
	the property of Assumption~\ref{ass:A5} can be shown using standard techniques.
	The case of Maxwell's equations is much more intricate, and invertibility is in some situations even an open problem \cite{Parolin:PhD}.
	Some basic techniques, however, are compiled in Appendix~\ref{apx:invertibility}.
\end{remark}

Assumption~\ref{ass:A5} allows us to form the Schur complement system of \eqref{eq:waveRobinGeneral}:
\begin{align}
\label{eq:RobinSchwarz}
  \text{find } \lambda \in \Lambda \colon \qquad
  (I - \mathcal{X}^\top S) \lambda = d,
\end{align}
where
\begin{align}
\label{eq:SdDef}
\begin{aligned}
  S & := -I + \alpha (M + \mathcal{X}^\top M \mathcal{X}) T (A + \alpha T^\top M T)^{-1} T^\top,\\
	d & := \alpha \mathcal{X}^\top (M + \mathcal{X}^\top M \mathcal{X}) T (A + \alpha T^\top M T)^{-1} f.
\end{aligned}
\end{align}


\begin{proposition}
\label{prop:scattering}
  Let Assumptions~\ref{ass:A1}--\ref{ass:A5} hold.
	Then the definitions in~\eqref{eq:SdDef} simplify to
	\[
	  S = -I + 2 \alpha M T (A + \alpha T^\top M T)^{-1} T^\top, \qquad
		d = 2 \alpha \mathcal{X}^\top M T (A + \alpha T^\top M T)^{-1} f.
	\]
	For any $(v, \sigma)$ fulfilling the homogeneous equation $A v - T^\top \sigma = 0$,
	the \emph{scattering operator} $S$ maps the incoming impedance trace $\lambda = \alpha M T v + \sigma$
	to the outgoing impedance trace $S \lambda = \alpha M T v - \sigma$.\footnote{For the Helmholtz equation and for the choice
	  $\alpha = \complexi$ and $M = \kappa I$, we have $\lambda_i = \complexi\kappa u_i + \partial u_i/\partial \nu_i$
		and $S_i \lambda_i = \complexi\kappa u_i - \partial u_i/\partial \nu_i$, cf.\ \eqref{eq:impedanceTraceForPlaneWave}.}
\end{proposition}

The connection between \eqref{eq:waveRobinGeneral} and \eqref{eq:RobinSchwarz} is established in a standard fashion:

\begin{proposition}
\label{prop:waveRobinSchur}
  Let \ref{ass:A1}--\ref{ass:A3}, and~\ref{ass:A5} hold.
  \begin{enumerate}
  \item[(i)] Let $f \in U^*$ be given. If $(u, \lambda)$ solves \eqref{eq:waveRobinGeneral} then $\lambda$ solves \eqref{eq:RobinSchwarz}
    with $d$ as in \eqref{eq:SdDef}. 
  \item[(ii)] Let $f \in U^*$ be given, let $d$ be as in \eqref{eq:SdDef}, and suppose that $\lambda$ solves \eqref{eq:RobinSchwarz}.
	  Then, with $u = (A + \alpha T^\top M T)^{-1} (f + T^\top \lambda)$,
		one obtains that $(u, \lambda)$ solves \eqref{eq:waveRobinGeneral}.
  \end{enumerate}
\end{proposition}

\begin{remark}
  Applying~$\mathcal{X}^\top$ to the Schur system~\eqref{eq:RobinSchwarz} yields (under Assumption~\ref{ass:A4}) the formulation
	\begin{align}
	\label{eq:FETI2LM:a}
	  \big[ I + \mathcal{X}^\top - 2 \alpha M T (A + \alpha T^\top M T)^{-1} T^\top \big] \lambda
		  = 2 \alpha M T (A + \alpha T^\top M T)^{-1} f,
	\end{align}
	which is essentially the one used in the method introduced by de La Bourdonnaye, Farhat, Macedo, Magoul\`es, and Roux
	\cite{Bourdonnaye:DD10,FarhatMacedoMagoulesRoux:USNCCM}.
	Note that therein,
	\eqref{eq:FETI2LM:a} is solved iteratively and with a Krylov subspace method and a preconditioner
  based on a projection to subdomain plane wave functions.
  In a related journal paper \cite{FarhatMacedoLesoinneRouxMagoulesDeLaBourdonnaie:2000a},
  the method was called \emph{regularized FETI method with two Lagrange multiplier fields (FETI-2LM)}.
  The exact ordering of the unknowns and equations is not described in detail,
  but for two subdomains, \cite[(51)--(52)]{Bourdonnaye:DD10} coincides with \eqref{eq:FETI2LM:a}.
	On the contrary, formulation~\eqref{eq:RobinSchwarz} is used in the FETI-2LM method described in \cite{Roux:DD14,Roux:DD18}.
	See also \cite{VouvakisCendesLee:2006a} for a similar formulation for Maxwell's equations.
\end{remark}

\begin{remark}
\label{rem:lambdaPrimal}
  Under the additional Assumption~\ref{ass:A4}, $M$ is invertible due to \ref{ass:A3},
	and so we can use the bijective transformation
	$\underline\lambda := M^{-1} \lambda = \alpha T u + M^{-1} \tau$.
	The transformed equation reads
	\begin{align}
	\label{eq:FETI2LM:c}
	  (I - \mathcal{X} \underline S) \underline\lambda = \underline d,
	\end{align}
	where $\underline S = - I + 2 \alpha T (A + \alpha T^\top M T)^{-1} T^\top M$
	and $\underline d = 2 \alpha \mathcal{X} T (A + \alpha T^\top M T)^{-1} f$,
	which is the Schur complement formulation of \eqref{eq:waveRobinAlt}.
  Formulation~\eqref{eq:FETI2LM:c} is, e.g., used in \cite{ClaeysParolin:Preprint2020}
	(with $\alpha = -\complexi$ and with the minus sign in front of $\mathcal{X}$ moved into $\underline S$).
\end{remark}

\subsection{The Robin-Schwarz iteration}

Given a damping parameter $\beta \in (0, 1]$,
the non-overlapping Schwarz iteration with Robin transmission conditions
is nothing else than a damped Richardson method for the Schur complement system~\eqref{eq:RobinSchwarz}:
\begin{align}
\label{eq:RobinSchwarzRichardon}
\begin{aligned}
  & \text{Given: } \lambda^{(0)} \in \Lambda^*,\\
	& \lambda^{(n+1)} := \lambda^{(n)} + \beta \big[ d - (I - \mathcal{X}^\top S) \lambda^{(n)} \big] \qquad \forall n \ge 0.
\end{aligned}
\end{align}

\begin{remark}
  Equation~\eqref{eq:RobinSchwarz} can also be written in fixed-point form, $\lambda = \mathcal{X}^\top S \lambda + d$.
  Correspondingly, \eqref{eq:RobinSchwarzRichardon} can be written as
	$\lambda^{(n+1)} = (1 - \beta) \lambda^{(n)} + \beta \big[ \mathcal{X}^\top S \lambda^{(n)} + d \big]$.
\end{remark}

\begin{remark}
	With $\underline\lambda^{(0)} \in \Lambda$ given, the iteration corresponding to \eqref{eq:FETI2LM:c} reads
	(see also \cite{ClaeysParolin:Preprint2020})
	\begin{align}
	\label{eq:RobinSchwarzRichardonMTrafo}
	  \underline\lambda^{(n+1)} = \underline\lambda^{(n)}
		  + \beta \big[ \underline d - (I - \mathcal{X} \underline S) \underline\lambda^{(n)} \big] \qquad \forall n \ge 0.
	\end{align}
\end{remark}

\medskip

Along with the \emph{dual} iterates $\lambda^{(n)}$ of~\eqref{eq:RobinSchwarzRichardon},
we define the corresponding \emph{primal} sequence
\begin{align}
\label{eq:unDef}
  u^{(n)} := (A + \alpha T^\top M T)^{-1} (f + T^\top \lambda^{(n)}) \qquad \forall n \ge 0,
\end{align}
such that
\begin{align}
\label{eq:lambdauIt}
  \lambda^{(n+1)} = \lambda^{(n)} + \beta \big[ \alpha \mathcal{X}^\top (M + \mathcal{X}^\top M \mathcal{X}) T u^{(n)} - (I + \mathcal{X}^\top) \lambda^{(n)} \big] \qquad \forall n \ge 0.
\end{align}
In that form, the scheme can be interpreted as a damped Uzawa iteration for \eqref{eq:waveRobinGeneral}.

\medskip

The iterates $\lambda^{(n)}$ of \eqref{eq:RobinSchwarzRichardon} are Lagrange multipliers on the interface.
In the discrete case, this is more preferrable compared to an iteration involving functions on the whole subdomains.
However, the classical Schwarz method \cite{Despres:PhD,Lions:DD03} was proposed in terms of iterates on the subdomains,
and the following result provides a link to such a form.

\begin{proposition}
  Let~\ref{ass:A1}--\ref{ass:A3}, and~\ref{ass:A5} hold and assume that $\overline{\range(T)} = \Lambda$.
	Then the sequence $(u^{(n)})$ defined in \eqref{eq:unDef} fulfills the recurrence relation
  \begin{align}
	\label{eq:uIt}
    u^{(n+1)} = (1 - \beta) u^{(n)} + \beta (A + \alpha T^\top M T)^{-1} \Big( f + T^\top \big[
	                \alpha M \mathcal{X} T u^{(n)} - \mathcal{X}^\top (T^\top)^\dag (A u^{(n)} - f)
	     \big] \Big),
  \end{align}
	for all $n \ge 0$,
	where $(T^\top)^\dag \colon \range(T^\top) \to \Lambda^*$ is the unique left-inverse of $T^\top$, such that $(T^\top)^\dag T^\top = I$.
	(Note that $(T^\top)^\dag$ is linear but not necessarily bounded!)
	Also,
	\[
	  A u^{(n)} - f \in \range(T^\top) \qquad \forall n \ge 0,
	\]
	which is why \eqref{eq:uIt} is well-defined.
	Recall that $u^{(0)}$ is defined by $\lambda^{(0)}$,
  but actually, we can choose $u^{(0)}$ freely as long as $A u^{(0)} - f \in \range(T^\top)$ is fulfilled.
\end{proposition}
\begin{proof}
  Since $\ker(T^\top) = \range(T)^0 = \overline{\range(T)}^0 = \Lambda^0 = \{ 0 \}$, cf.\ \eqref{eq:polarKerRange},
  there exists a unique linear right inverse $(T^\top)^\dag \colon \range(T^\top) \to \Lambda^*$ such that $(T^\top)^\dag T^\top = I$.
	From~\eqref{eq:unDef}, one can see easily that
	\[
	  A u^{(n)} - f = T^\top (\lambda^{(n)} - \alpha M T u^{(n)}),
	\]
	so $A u^{(n)} - f \in \range(T^\top)$.
	Therefore, we can multiply the previous identity by $(T^\top)^\dag$ to obtain
	\begin{align}
	\label{eq:lambdanHelp}
	   \lambda^{(n)} = (T^\top)^\dag (A u^{(n)} - f) + \alpha M T u^{(n)}.
	\end{align}
	Substitution of \eqref{eq:lambdanHelp} into the right-most occurrence of $\lambda^{(n)}$ in \eqref{eq:lambdauIt}, applying $T^\top$, adding $f$ on both sides, 
	and finally applying $(A + \alpha T^\top M T)^{-1}$ yields \eqref{eq:uIt}.
\end{proof}

\begin{example}[classical impedance operator]
  Consider the Helmholtz equation from Example~\ref{ex:modelProblemHelmholtz}
	with globally constant wave number $\kappa$,
	with the choice of $L^2$-traces (Example~\ref{ex:tracesH1Split}),
	the swapping operator $\mathcal{X}$ from Example~\ref{ex:swapping},
	and with $\alpha = \complexi$ and $M_i = \kappa I$.
	Then the expression $(T^\top)^\dag (A u^{(n)} - f)$ evaluates for each subdomain $\Omega_i$
	the normal derivative $\tfrac{\partial}{\partial \normal_i} u^{(n)}_i$
	on the interface $\partial\Omega_i \cap \Gamma$ as a quantity in $L^2$.
  For damping parameter $\beta = 1$, the equations for $u^{(n+1)}$ in strong form read
	\begin{align}
	\label{eq:DespresHelmholtz}
	\begin{alignedat}{2}
	  -\Delta u_i^{(n+1)} - \kappa^2 u_i^{(n+1)} & = f_i \qquad && \text{in } \Omega_i\,,\\
		\complexi \kappa u_i^{(n+1)} + \frac{\partial}{\partial \normal_i} u_i^{(n+1)} & = \complexi \kappa u_j^{(n)} - \frac{\partial}{\partial \normal_j} u_j^{(n)}
		  \qquad && \text{on } \partial\Omega_i \cap \partial\Omega_j\,,
	\end{alignedat}
	\end{align}
	plus the given exterior boundary condition on $\partial\Omega_i \cap \Gamma_D$, $\partial\Omega_i \cap \Gamma_N$, and $\partial\Omega_i \cap \Gamma_R$.
	In this form, the method was proposed by B.~Despr\'es in \cite[p.~29]{Despres:PhD}.
\end{example}

\begin{remark}
  The \emph{fixed point formulation} behind the primal iteration \eqref{eq:uIt} reads
	\begin{align}
	  u = (A + \alpha T^\top M T)^{-1} \Big(f + T^\top \big[ \alpha M \mathcal{X} T u - \mathcal{X}^\top (T^\top)^\dag ( A u - f ) \big] \Big),
	\end{align}
	where $(T^\top)^\dag \colon U^* \to \Lambda^*$ is a generalized inverse of $T^\top$ (not necessarily bounded!) such that $(T^\top)^\dag T^\top = I$
	(which is only possible if $\ker(T^\top) = \{ 0 \}$ or, equivalently, $\overline{\range(T)} = \Lambda$).
  Reordering the terms yields
	\begin{align}
	\label{eq:FixedPointSchwarz}
		\Big( I - (A + \alpha T^\top M T)^{-1} T^\top \big[ \alpha M \mathcal{X} T - \mathcal{X}^\top (T^\top)^\dag A \big] \Big) u
		= (A + \alpha T^\top M T)^{-1} \big[ f + \mathcal{X}^\top (T^\top)^\dag f \big],
	\end{align}
	which is of similar structure as \eqref{eq:RobinSchwarz}.
	Note that if $E \colon \Lambda \to U$ is a bounded extension operator with $T E = I$ (which requires $\range(T) = \Lambda$),
	then $(T^\top)^\dag$ can be replaced by $E^\top$.
\end{remark}

\subsection{A classification of discrete methods}

  If we are in the discrete case and $\range(T) \subsetneq \Lambda$, then $\ker(T^\top)$ is non-trivial.
	This means that once a dual variable $\lambda$ is assembled to $T^\top \lambda$, it cannot be recovered in general,
	see Figure~\ref{fig:rangeTNotComplete} and see \cite[Sect.~3]{GanderSantugini:2016a}.
	Right inverses $(T^\top)^\dag$ of $T^\top$ do exist, but they only fulfill
	$T^\top (T^\top)^\dag T^\top = T^\top$; however, $(T^\top)^\dag T^\top \neq I$.
	From the perspective of the dual formulation~\eqref{eq:RobinSchwarz} this is not a problem at all.
	The FETI-2LM method \cite{Bourdonnaye:DD10,FarhatMacedoMagoulesRoux:USNCCM,FarhatMacedoLesoinneRouxMagoulesDeLaBourdonnaie:2000a}
	was introduced exactly along these lines,
	and also the classical FETI method \cite{FarhatRoux:1991a} lives with comparable redundancies,
	cf.\ \cite[Sect.~6]{ToselliWidlund:Book}, \cite[Sect.~2.2]{Pechstein:FETIBook}.
	From the perspective of trying to leverage formulation~\eqref{eq:DespresHelmholtz} from the continuous to the discrete case
	by using the bilateral properly closed facets (see Sect.~\ref{sect:bilateralDiscrFacetSys}),
	the lack of surjectivity of $T$ turns out to be a real obstacle.
	The formulation~\eqref{eq:unDef}--\eqref{eq:lambdauIt},
	involving both the dual and primal iterate $\lambda^{(n)}$ and $u^{(n)}$,
	has been proposed by Gander and Santugini \cite[Sect.~3]{GanderSantugini:2016a}
	and is therein called \emph{auxiliary variable method} or \emph{optimized Schwarz with auxiliary variables}.

\begin{figure}
\begin{center}
  \begin{tikzpicture}
		\pgftransformscale{0.6}
		
		\definecolor{dblue}{rgb}{0.0, 0.0, 0.8}
		\definecolor{mdblue}{rgb}{0.4, 0.4, 1.0}
		\definecolor{lblue}{rgb}{0.8, 0.8, 1.0}
		\definecolor{mlblue}{rgb}{0.6, 0.6, 1.0}
		
		\draw[line width=0.75pt,color=mdblue,fill=mlblue] (0,0)--(1.5,0)--(1.5,1.5)--(0,1.5)--(0,0)--(1.5,0);

		\draw[line width=0.75pt,color=mdblue] (0,2)--(1.5,2);
		\draw[line width=0.75pt,color=mdblue] (2,0)--(2,1.5);
		
		\foreach \x in {1.5}
		{
		  \draw[line width=0.74pt,color=mdblue,fill=mdblue] (\x,2) circle (0.08);
		  \draw[line width=0.74pt,color=mdblue,fill=mdblue] (2,\x) circle (0.08);

			\draw[line width=0.74pt,color=mdblue,fill=mdblue,densely dotted] (1.5,\x)--(2,\x);
			\draw[line width=0.74pt,color=mdblue,fill=mdblue,densely dotted] (\x,2)--(\x,1.5);
    }
		
    \draw[line width=0.74pt,color=mdblue,fill=white] (1.5,1.5) circle (0.08);
		
  \end{tikzpicture}
	\caption{\label{fig:rangeTNotComplete}%
	  Local subdomain dof ($\circ$) corresponds to dofs ($\bullet$) on two bilateral facets. For an element $\tau_i \in \Lambda_i$,
		the operation $t_i = T_i^\top \tau_i$ adds up the values of the two trace dofs $\bullet$ to the local subdomain dof $\circ$.
		From $t_i$ the original values of the two trace dofs cannot be recovered.
		Therefore $\ker(T_i^\top) \neq \{ 0 \}$ and $\range(T_i) \subsetneq \Lambda_i$.
	}
\end{center}
\end{figure}
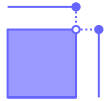

\begin{table}
\begin{center}
  \begin{tabular}{l|l|l|l}
	  method                 & proposed in & facet system               & formulation \\
	\hline
	  FETI-2LM               & \cite{Bourdonnaye:DD10,FarhatMacedoMagoulesRoux:USNCCM,FarhatMacedoLesoinneRouxMagoulesDeLaBourdonnaie:2000a}
			                                                             & bilateral, properly closed & dual   \eqref{eq:RobinSchwarz} \\
		auxiliary variables    & \cite[Sect.~3]{GanderSantugini:2016a} & bilateral, properly closed & mixed
			                                                                                    \eqref{eq:unDef}--\eqref{eq:lambdauIt} \\
		Loisel's method        & \cite{Loisel:2013a}                   & globs                      & dual   \eqref{eq:RobinSchwarz} \\
		complete communication & \cite[Sect.~4]{GanderSantugini:2016a} & globs                      & primal \eqref{eq:uIt}
	\end{tabular}
	\caption{\label{tab:methodClassification}%
    Classification of discrete methods.
	}
\end{center}
\end{table}

	Under the condition $\range(T) = \Lambda$ one can obviously use the primal form~\eqref{eq:uIt} (also in the discrete case).
	Recall that for bilateral facet systems,
	the condition $\range(T) = \Lambda$ fails to hold once a global dof is shared by three or more subdomains
	(Proposition~\ref{prop:bilateralCrosspoints}).
	Recall, however, that for \emph{glob systems}, the condition $\range(T) = \Lambda$ holds always.
	Indeed, the construction in the paper by S.~Loisel \cite{Loisel:2013a}
	(therein called \emph{2-Lagrange multiplier method}) 
	is reproduced if one uses a glob system and the \emph{dual} formulation~\eqref{eq:RobinSchwarz}.
	Independently, Gander and Santugini \cite[Sect.~4]{GanderSantugini:2016a}
  proposed a scheme called \emph{complete communication method},
	which is reproduced if one uses a glob system and the primal formulation~\eqref{eq:uIt}.
  These observations are summarized in Table~\ref{tab:methodClassification}.

\section{Convergence analysis}
\label{sect:convergence}

This section contains three types of convergence theorems that all extend available results.
Section~\ref{sect:convGeneral} generalizes the result with minimal assumptions and simple convergence
by Collino, Ghanemi, and Joly \cite{CollinoGhanemiJoly:2000a} building on Despr\'es' original proof \cite{Despres:PhD}.
Section~\ref{sect:convStrong} generlizes the result in \cite{CollinoGhanemiJoly:2000a} with stronger assumptions on the trace operator
achieving linear convergence (in \cite{CollinoGhanemiJoly:2000a} called exponential convergence).
For completeness, Section~\ref{sect:convAbs} covers the case where the wave propagation problems have a lot of absorbtions
or the coercive problems have a sufficient zero order term, leading to linear convergence as well, as e.g., demonstrated in \cite{GanderSantugini:2016a}.
Whereas the assumptions made so far have been tailored to guarantee that the reformulations in \eqref{sect:RobinTC}
are equivalent to the original equation,
one needs additional assumptions to make convergence accessible (Sect.~\ref{sect:addAssConv}).

\subsection{Additional assumptions for the convergence analysis}
\label{sect:addAssConv}

Opposed to many analyses of coercive problems where the system operator can serve as an energy norm,
here errors are measured using the impedance operator,
which requires the following stronger assumption and goes back to the concept of
\emph{pseudo-energy}, see \cite{Despres:PhD}.

\medskip

\noindent%
\fbox{%
\begin{minipage}{0.985\textwidth}
\begin{assumption}
\label{ass:A6}
  The operator $M$ from \ref{ass:A3} has the block-diagonal form $M = \diag(M_i)_{i=1}^N$,
  where each operator $M_i \colon \Lambda_i \to \Lambda_i^*$ is real-valued,
	symmetric\footnote{$M_i \colon \Lambda_i \to \Lambda_i^*$ is symmetric iff $M_i^\top = M_i$.},
	and positively bounded from below,
  i.e., there exists a constant $c_i > 0$ such that
  $\langle M_i \lambda_i, \overline \lambda_i \rangle \ge c_i \| \lambda_i \|_{\Lambda_i}^2$ for all $\lambda_i \in \Lambda_i$.
\end{assumption}
\end{minipage}}

\medskip

With Assumption~\ref{ass:A6} fulfilled, we can define the following inner products and norms:
\begin{alignat*}{3}
  (\lambda, \mu)_M & := \langle M \lambda, \overline \mu \rangle, & \qquad
	\| \lambda \|_{M} & := \langle M \lambda, \overline \lambda \rangle^{1/2}
	  \qquad && \lambda,\, \mu \in \Lambda,\\
	(\lambda, \mu)_{M^{-1}} & := \langle M^{-1} \lambda, \overline \mu \rangle, & \qquad
	     \| \mu \|_{M^{-1}} & := \langle M^{-1} \mu, \overline \mu \rangle^{1/2}
	  \qquad && \lambda,\, \mu \in \Lambda^*,
\end{alignat*}
where $\overline \mu$ denotes the complex conjugate of $\mu$.

\medskip

\noindent%
\fbox{%
\begin{minipage}{0.985\textwidth}
\begin{assumption}
\label{ass:AXR}
  The exchange operator $\mathcal{X}$ is real-valued.
\end{assumption}
\end{minipage}}

\medskip

Assumptions~\ref{ass:A4}, \ref{ass:A6}, and~\ref{ass:AXR} together imply that the exchange operator is an isometry:
\begin{align}
\label{eq:XIso}
\begin{alignedat}{2}
  \| \mathcal{X} \lambda \|_{M}       & = \| \lambda \|_{M}      \qquad && \forall \lambda \in \Lambda,\\
  \| \mathcal{X}^\top \mu \|_{M^{-1}} & = \| \mu \|_{M^{-1}}     \qquad && \forall \mu \in \Lambda^*.
\end{alignedat}
\end{align}

Our next assumption states that we are dealing either with a coercive (positive definite) problem, or with a (time-harmonic) wave propagation problem.

\medskip

\noindent%
\fbox{\begin{minipage}{0.985\textwidth}
\begin{assumption}
\label{ass:A7}
  One of the following cases holds:
	\begin{enumerate}
	\item[(i)]  \emph{Coercive case:} $\alpha = 1$ and each operator $A_i$ is real-valued, symmetric,
	  and non-negative\footnote{$A_i \colon U_i \to U_i^*$ is non-negative iff $\langle A_i v, \overline v \rangle \ge 0$
		for all $v \in U_i$.}.
	\item[(ii)] \emph{Wave propagation case:}  $\alpha = \complexi$ and each operator $A_i$ can be written as
	  \[
		  A_i = A_{i,0} + \complexi A_{i,1} - A_{i,2}
		\]
    with real-valued, symmetric, and non-negative operators $A_{i,k}$.
	\end{enumerate}
\end{assumption}
\end{minipage}}

\medskip

The following lemma states that the \emph{pseudo-energy} (cf.\ \cite[Lemme~4.3]{Despres:PhD})
of the incoming impedance trace $\lambda$ is the same as that of the outgoing impedance trace $S \lambda$ plus the interior losses.

\begin{lemma}
\label{lem:SNonExpansive}
  Let Assumptions~\ref{ass:A1}--\ref{ass:A7} hold.
	Then
	\[
	  \| S \lambda \|_{M^{-1}}^2 + 4 p = \| \lambda \|_{M^{-1}}^2 \qquad \forall \lambda \in \Lambda^*,
	\]
	where $p \ge 0$ is defined as
	\[
	  p = \left. \begin{cases}
		  \langle A v, \overline v \rangle & \text{in case~(i) of \ref{ass:A7}} \\
			\Im \langle A v, \overline v \rangle & \text{in case~(ii) of \ref{ass:A7}}
		\end{cases}
		\right\}
		\qquad \text{with } v = (A + \alpha T^\top M T)^{-1} T^\top \lambda.
	\]
	In particular, the scattering operator $S$ is non-expansive with respect to the norm $\|\cdot\|_{M^{-1}}$:
	\[
	  \| S \lambda \|_{M^{-1}} \le \| \lambda \|_{M^{-1}} \qquad \forall \lambda \in \Lambda^*.
	\]
\end{lemma}
\begin{proof}
  In addition to $v$ given above, we define $\sigma := \lambda - \alpha M T v$.
	Proposition~\ref{prop:scattering} implies
	\[
	  A v - T^\top \sigma = 0, \qquad
		\lambda = \alpha M T v + \sigma, \qquad
		S \lambda = \alpha M T v - \sigma.
	\]
	We form
	\begin{align*}
	  \| S \lambda \|_{M^{-1}}^2 - \| \lambda \|_{M^{-1}}^2
		& = \| \alpha M T v - \sigma \|_{M^{-1}}^2 - \| \alpha M T v + \sigma \|_{M^{-1}}^2\\
		& = -4 \Re \underbrace{(\sigma, \alpha M T v )_{M^{-1}}}_{\overline\alpha \langle T^\top \sigma, \overline v \rangle}
		  = -4 \underbrace{ \Re\big( \overline\alpha \langle A v, \overline v \rangle \big) }_{=: p}.
	\end{align*}
	According to \ref{ass:A7}, we have two cases:
	In case~(i), $\alpha = 1$ and $A$ is real-valued, symmetric, and non-negative, so
	  $p = \langle A v, \overline v \rangle \ge 0$.
	In case~(ii), $\alpha = \complexi$, so
	  $p = \Im \langle A v, \overline v \rangle = \sum_{i=1}^N \langle A_{i,1} v_i, \overline{v_i} \rangle \ge 0$,
	  since each operator $A_{i,1}$ is real-valued and non-negative.
\end{proof}

\begin{remark}
  For the formulation from Remark~\ref{rem:lambdaPrimal},
	$\| \underline{S}\, \underline\lambda \|_M^2 + 4 p = \| \underline \lambda \|_M^2$
	for all $\underline\lambda \in \Lambda$,
	with $p$ as in Lemma~\ref{lem:SNonExpansive} but with $v = (A + \alpha T^\top M T)^{-1} T^\top M \underline\lambda$.
\end{remark}

\subsection{Convergence in the general case}
\label{sect:convGeneral}

This section generalizes the convergence result by Collino, Ghanemi, and Joly \cite[Lemma~5]{CollinoGhanemiJoly:2000a}
which is based on the early findings by Despr\'es.

\begin{lemma}
\label{lem:kerIXS}
  Let \ref{ass:A1}--\ref{ass:A3}, and~\ref{ass:A5} hold. Then (with $\mathcal{Z}$ as in Def.~\ref{def:spaceZ})
  \[
	  \ker(I - \mathcal{X}^\top S) = \mathcal{Z}.
	\]
	If, in addition, Assumption~\ref{ass:A6} holds then $\ker(I - \mathcal{X} \underline S) = M^{-1}(\mathcal{Z})$ (see Remark~\ref{rem:lambdaPrimal}).
\end{lemma}
\begin{proof}
  Assume that $\lambda \in \Lambda^*$ with $(I - \mathcal{X}^\top S) \lambda = 0$.
	We define $f = 0$ and $u = (A + \alpha T^\top M T)^{-1} T^\top \lambda$
	and find that $(u, \lambda)$ solves \eqref{eq:waveRobin}, cf.\ Proposition~\ref{prop:waveRobinSchur}.
	With $\tau = \lambda - \alpha M T u$, this implies that	$(u, \tau)$ solves \eqref{eq:traceflux}.
	Since the solution operator $\mathcal{S}^{(\tau)}$ for \eqref{eq:traceflux} depends only on $f = 0$,
	it follows from Theorem~\ref{thm:traceflux} that $(u, \tau) = (0, z)$ for some element $z \in \mathcal{Z}$.
	This implies, in turn, that $\lambda \in \mathcal{Z}$.
	The second relation follows from $M \underline S = S M$, cf.\ Remark~\ref{rem:lambdaPrimal}.
\end{proof}

\begin{theorem}
\label{thm:convGeneral}
  Let Assumptions~\ref{ass:A1}--\ref{ass:A7} hold. In addition, either
	\begin{enumerate}
	\item[1.] all spaces are finite-dimensional, or
	\item[2.] all the following assumptions hold:
	  \begin{enumerate}
		\item[(a)] $\overline{\range(T)} = \Lambda$,
		\item[(b)] $A R \widehat u - f \in \range(T^\top)$,
		\item[(c)] $T$ is compact.
		\end{enumerate}
	\end{enumerate}
	Then the Robin-Schwarz iteration \eqref{eq:RobinSchwarzRichardon}
	with damping parameter $\beta \in (0, 1)$ converges in the sense that
	\begin{align*}
	  T^\top \lambda^{(n)} \stackrel{n \to \infty}{\longrightarrow} w^{(\infty)} \text{ in } U^*, \qquad
		u^{(n)} \stackrel{n \to \infty}{\longrightarrow} R \widehat u \text{ in } U,
	\end{align*}
	where $w^{(\infty)} := (A + \alpha T^\top M T) R \widehat u - f$ and $u^{(n)} := (A + \alpha T^\top M T)^{-1} (f + T^\top \lambda^{(n)})$.
	In general, $(\lambda^{(n)})$ only contains weakly convergent subsequences.
	Likewise, the iterates $(\underline \lambda^{(n)})$ from~\eqref{eq:RobinSchwarzRichardonMTrafo} fulfill
	$T^\top M \underline\lambda^{(n)} \stackrel{n \to \infty}{\longrightarrow} w^{(\infty)}$.
\end{theorem}
\begin{proof}
  The structure of the proof follows that of \cite[Lemma~5]{CollinoGhanemiJoly:2000a}.
  Due to the stated assumptions, Theorem~\ref{thm:traceflux}
	guarantees the existence of $\tau \in \Lambda^*$ such that $(R \widehat u, \tau)$
	solves \eqref{eq:traceflux}, or equivalently \eqref{eq:tracefluxRobin}.
	In the finite-dimensional case, $\tau$ is not necessarily unique, but we fix one possible solution.
	With $\lambda := \alpha M T R \widehat u + \tau$, we find that $(R \widehat u, \lambda)$ solves \eqref{eq:waveRobin}
	and consequently $\lambda$ solves \eqref{eq:RobinSchwarz}.
	Defining $\mu^{(n)} := \lambda^{(n)} - \lambda$, we find from \eqref{eq:RobinSchwarzRichardon} that
	\[
	  \mu^{(n+1)} = (1 - \beta) \mu^{(n)} + \beta \mathcal{X}^\top S \mu^{(n)}.
	\]
	Hence,
	\[
	  \| \mu^{(n+1)} \|_{M^{-1}}^2 = (1 - \beta)^2 \| \mu^{(n)} \|_{M^{-1}}^2
		  + 2 \beta (1 - \beta) \Re \big( \mathcal{X}^\top S \mu^{(n)}, \mu^{(n)} \big)_{M^{-1}}
			+ \beta^2 \| \mathcal{X}^\top S \mu^{(n)} \|_{M^{-1}}^2\,.
	\]
	As in the proof of \cite[Lemma~5]{CollinoGhanemiJoly:2000a}, we use the identity
	\begin{align}
	\label{eq:ReXSIdentity}
	  2 \Re \big( \mathcal{X}^\top S \mu^{(n)}, \mu^{(n)} \big)_{M^{-1}}
		  = \| \mu^{(n)} \|_{M^{-1}}^2 + \| \mathcal{X}^\top S \mu^{(n)} \|_{M^{-1}}^2
		    - \| (I - \mathcal{X}^\top S)\mu^{(n)} \|_{M^{-1}}^2
	\end{align}
	in the earlier formula to obtain
	\[
	  \| \mu^{(n+1)} \|_{M^{-1}}^2
		  = (1 - \beta) \| \mu^{(n)} \|_{M^{-1}}^2
		  - \beta (1 - \beta) \| (I - \mathcal{X}^\top S) \mu^{(n)} \|_{M^{-1}}^2
			+ \beta \| \mathcal{X}^\top S \mu^{(n)} \|_{M^{-1}}^2.
	\]
	Due to the isometry property~\eqref{eq:XIso} and Lemma~\ref{lem:SNonExpansive},
	\[
	  \| \mathcal{X}^\top S \mu^{(n)} \|_{M^{-1}}^2 = 
		\| S \mu^{(n)} \|_{M^{-1}}^2
		= \| \mu^{(n)} \|_{M^{-1}}^2 - 4 p^{(n)}\,,
	\]
	with $p^{(n)} \ge 0$ given by
	\[
	  p^{(n)} = \begin{cases}
			\Re \big\langle A e^{(n)}, \overline{e^{(n)}} \big\rangle & \text{in case (i) of \ref{ass:A7},}\\
			\Im \big\langle A e^{(n)}, \overline{e^{(n)}} \big\rangle & \text{in case (ii) of \ref{ass:A7},}
		\end{cases}
	\]
	where $e^{(n)} = (A + \alpha T^\top M T)^{-1} T^\top \mu^{(n)}$.
	By induction, we can show that
	\[
	  \| \mu^{(n+1)} \|_{M^{-1}}^2
		  + \beta (1 - \beta) \sum_{k=1}^n \| (I - \mathcal{X}^\top S) \mu^{(k)} \|_{M^{-1}}^2
		  + 4 \beta \sum_{k=1}^n p^{(k)}
		= \| \mu^{(0)} \|_{M^{-1}}^2\,.
	\]
	Since all terms on the left-hand side are non-negative and since $\beta \in (0, 1)$, this proves that
	\begin{enumerate}
	\item[(i)] the sequence $(\mu^{(n)})$ is bounded with respect to $\| \cdot \|_{M^{-1}}$,
	\item[(ii)] the series $\sum_{k=1}^\infty \| (I - \mathcal{X}^\top S) \mu^{(k)} \|_{M^{-1}}^2$ converges,
	 and so $(I - \mathcal{X}^\top S) \mu^{(k)} \stackrel{k \to \infty}{\longrightarrow} 0$ in $\Lambda^*$.
	\end{enumerate}
	Because of (i) there exists a weakly convergent subsequence $(\mu^{(n_\ell)})$ with a weak limit $\mu^{(\infty)} \in \Lambda^*$, i.e.,
	$\mu^{(n_\ell)} \rightharpoonup \mu^{(\infty)}$.
	Next, we need a case distinction:
	\begin{itemize}
	\item In case~1, all spaces are finite-dimensional and so weak convergence implies strong convergence.
	  Because of~(ii) this yields $(I - \mathcal{X}^\top S) \mu^{(\infty)} = 0$.
	\item In case~2, due to (2c), $T^\top$ is compact, and so
	  \begin{alignat*}{2}
		  T^\top \mu^{(n_\ell)} & \stackrel{\ell \to \infty}{\longrightarrow} T^\top \mu^{(\infty)} \qquad && \text{(strongly) in } U^*,\\
			T^\top (I - \mathcal{X}^\top S) \mu^{(n_\ell)}
			                      & \stackrel{\ell \to \infty}{\longrightarrow} T^\top (I - \mathcal{X}^\top S) \mu^{(\infty)} \qquad && \text{(strongly) in } U^*.
		\end{alignat*}
		Since $(I - \mathcal{X}^\top S) \mu^{(n_\ell)}$ converges to zero, it follows that $T^\top (I - \mathcal{X}^\top S) \mu^{(\infty)} = 0$.
		Due to assumption (2a) it follows that $\ker(T^\top) = \range(T)^0 = \overline{\range(T)}^0 = \{0\}$, and so
		\[
		  (I - \mathcal{X}^\top S) \mu^{(\infty)} = 0.
		\]
	\end{itemize}
	In both cases, we conclude from Lemma~\ref{lem:kerIXS} that $\mu^{(\infty)} \in \ker(T^\top) \cap \ker(I + \mathcal{X}^\top)$.
	In the finite-dimensional case, $\mu^{(\infty)}$ may be non-zero and depend on the subsequence.
	Nevertheless, it is true in general that $T^\top \mu^{(\infty)} = 0$, which shows that
	\[
	  T^\top \mu^{(n_\ell)} \stackrel{\ell \to \infty}{\longrightarrow} 0 \qquad \text{(strongly) in } U^*.
	\]
	Suppose now that the \emph{original} sequence $T^\top \mu^{(n)}$ does \emph{not} converge to zero.
	Then there must be a subsequence $(n_j)$ and some $\varepsilon > 0$ such that $\| T^\top \mu^{(n_j)} \|_{M^{-1}} \ge \varepsilon$.
	However, we can repeat the arguments from above and extract a sub-subsequence that \emph{does} converge to zero, which is a contradiction.
	Therefore,
	\[
	  T^\top \mu^{(n)} \stackrel{n \to \infty}{\longrightarrow} 0 \qquad \text{(strongly) in } U^*.
	\]	
	From the definition of $\mu^{(n)}$ we obtain that
	\[
	  T^\top \lambda^{(n)} \stackrel{n \to \infty}{\longrightarrow}
		  \alpha T^\top M T R \widehat u + \underbrace{T^\top \tau}_{A R \widehat u - f}
		  = (A + \alpha T^\top M T) R \widehat u - f.
	\]
	Using \ref{ass:A5}, the convergence property for $u^{(n)}$ follows suit.
\end{proof}

\begin{example}
  For the Helmholtz equation (Example~\ref{ex:modelProblemHelmholtz}),
	suppose that we use a bilateral facet system with facet trace space $U_F := H^s(F)$, where $0 \le s < 1/2$,
	and the usual exchange operator that swaps traces (see Example~\ref{ex:swapping}).
	Then Assumptions~\ref{ass:A1}--\ref{ass:A2} and \ref{ass:AXR} are fulfilled.
	For each facet $F$, let $M_F \colon H^s(F) \to H^s(F)^*$ be a real-valued, symmetric impedance operator which is bounded positively from below,
	and define $M_i$ as a block-diagonal operator with entries $(M_F)_{F \in \mathcal{F}_i}$.
	Then Assumptions~\ref{ass:A3}, \ref{ass:A4}, and \ref{ass:A6} hold.
	The invertibility of the local subdomain problems (Assumption~\ref{ass:A5}) is also guaranteed, see Appendix~\ref{apx:invGenRobin}.
	Finally, with $\alpha = \complexi$, Assumption~\ref{ass:A7} holds (see Table~\ref{tab:DDClass}).
	In the continuous case, $\range(T)$ is dense in $\Lambda$ because the \emph{natural} trace space $H^{1/2}(F)$ is dense in the \emph{chosen} trace space $H^s(F)$,
	and $T$ is compact because the embedding $H^{1/2}(F) \subset H^s(F)$ is compact for $s < 1/2$, so Assupmtions~2.a) and 2.c) hold.
	Assumption~2.b) states that the normal derivative of the global solution must be piecewise in $H^s(F)^*$.
	For $s=0$, this is the $L^2$ regularity used in \cite{Despres:PhD} and \cite[Sect.~2.3]{CollinoGhanemiJoly:2000a}.
\end{example}

\begin{remark}
  For the special case of the Laplace and the Helmholtz equation,
  Lions \cite{Lions:DD03} and Despr\'es \cite{Despres:PhD} proved that the \emph{undamped} Schwarz scheme (with $\beta = 1$)
	converges as well.
	A generalization of that line of proof, however, is beyond the scope of this paper
	as it would require more assumptions and appear even more technical.
\end{remark}

It would be advantageous if the scattering operator $S$ were a contraction, because this would at once imply linear convergence
\`a la Banach's fixed point theorem.

\begin{proposition}
\label{prop:SContractive}
  Let Assumptions~\ref{ass:A1}--\ref{ass:A7} hold and suppose that
	\[
	  \| S \mu \|_{M^{-1}} \le \rho \| \mu \|_{M^{-1}} \qquad \forall \mu \in \Lambda^*,
	\]
	for some contraction factor $\rho < 1$.
	Then for any damping parameter $\beta \in (0, 1]$, the Robin-Schwarz iteration~\eqref{eq:RobinSchwarzRichardon}
	converges linearly in the sense that for $n \ge 0$,
  \[
	  \| \lambda^{(n+1)} - \lambda \|_{M^{-1}} \le \rho \| \lambda^{(n)} - \lambda \|_{M^{-1}}
	\]
	and therefore $\| \lambda^{(n)} - \lambda \|_{M^{-1}} \le \rho^n \| \lambda^{(0)} - \lambda \|_{M^{-1}}$.
\end{proposition}

However, two causes can prevent $S$ from being (strongly) contractive.
\begin{enumerate}
\item[(i)] In case of redundancies, i.e., if $\mathcal{Z} = \ker(I - \mathcal{X}^\top S)$ is non-trivial,
  there exists an element $\mu \neq 0$ such that $\mathcal{X}^\top S \mu = \mu$, and so
	$\| S \mu \|_{M^{-1}} = \| \mathcal{X}^\top S \mu \|_{M^{-1}} = \| \mu \|_{M^{-1}}$.
\item[(ii)] For wave propagation problems, the typical subdomain $i$ has mostly \emph{propagative modes},
  these are functions $\mu_i \in \Lambda_i^* \setminus \{ 0 \}$ such that
  \[
	  S_i \mu_i = \xi \mu_i \qquad \text{with } \xi \in \mathbb{C},\ |\xi| = 1,
  \]
	which is why $\| S_i \mu_i \|_{M_i^{-1}} = \| \mu_i \|_{M_i^{-1}}$.
	Indeed, Lemma~\ref{lem:SNonExpansive} shows that if $A_{i,1} = 0$ then \emph{all} functions in $\Lambda_i$
	are propagative and $S_i$ is an \emph{isometry}.
\end{enumerate}

\begin{remark}
  A recent and very promising work \cite{GongGanderGrahamLafontaineSpence:Preprint2021a}
	on an overlapping Robin-Schwarz method for the Helmholtz equation proves \emph{power-contractivity}
	(but does not fit into the framework of this paper).
\end{remark}

\subsection{Convergence for strong absorbtion}
\label{sect:convAbs}

The following theorem shows convergence also for damping parameters of one and can do so without compactness,
however, under strong assumptions on the subdomain operators (cf.\ \cite[Thm.~3.2]{GanderSantugini:2016a}).

\begin{theorem}
\label{thm:convSPD}
  Let Assumptions~\ref{ass:A1}--\ref{ass:A7} hold. In addition, 
	\begin{enumerate}
	\item[1.] all spaces are finite-dimensional, or
	\item[2.] the two following assumptions hold:
	  \begin{enumerate}
		\item[(a)] $\overline{\range(T)} = \Lambda$,
		\item[(b)] $A R \widehat u - f \in \range(T^\top)$.
		\end{enumerate}
	\end{enumerate}
	Furthermore, assume that there exist positive constants $a_i > 0$ such that for each $i=1,\ldots,N$,
	\begin{alignat*}{2}
	  \langle A_i v, \overline v \rangle & \ge a_i \| v \|_{U_i}^2 \qquad & \text{in case~(i) of \ref{ass:A7},}\\
		\langle A_{i,1} v, \overline v \rangle & \ge a_i \| v \|_{U_i}^2 \qquad & \text{in case~(ii) of \ref{ass:A7}.}
	\end{alignat*}
	Then the Robin-Schwarz iteration \eqref{eq:RobinSchwarzRichardon}
	with damping parameter $\beta \in (0, 1]$ converges in the sense that
	\begin{align*}
	  T^\top \lambda^{(n)} \stackrel{n \to \infty}{\longrightarrow} w^{(\infty)} \text{ in } U^*, \qquad
		u^{(n)} \stackrel{n \to \infty}{\longrightarrow} R \widehat u \text{ in } U.
	\end{align*}
	where $w^{(\infty)} := (A + \alpha T^\top M T) R \widehat u - f$ and $u^{(n)} := (A + \alpha T^\top M T)^{-1} (f + T^\top \lambda^{(n)})$.
\end{theorem}
\begin{proof}
  With the stated assumptions, we can repeat the first few arguments from the proof of Theorem~\ref{thm:convGeneral}
	and obtain
	\[
	  \| \mu^{(n+1)} \|_{M^{-1}}^2
		  + \beta (1 - \beta) \sum_{k=1}^n \| (I - \mathcal{X}^\top S) \mu^{(k)} \|_{M^{-1}}^2
		  + 4 \beta \sum_{k=1}^n p^{(k)}
		= \| \mu^{(0)} \|_{M^{-1}}^2\,,
	\]
  with
	\[
	  p^{(k)} = \begin{cases}
			\Re \big\langle A e^{(k)}, \overline{e^{(k)}} \big\rangle = \sum_{i=1}^N \big\langle A_i e_i^{(k)}, \overline{e_i^{(k)}} \big\rangle & \text{in case (i) of \ref{ass:A7},}\\[1ex]
			\Im \big\langle A e^{(k)}, \overline{e^{(k)}} \big\rangle = \sum_{i=1}^N \big\langle A_{i,1} e_i^{(k)}, \overline{e_i^{(k)}} \big\rangle & \text{in case (ii) of \ref{ass:A7},}
		\end{cases}
	\]
	where $e^{(k)} = (A + \alpha T^\top M T)^{-1} T^\top \mu^{(k)}$.
	Since $p^{(k)} \ge 0$ and $\beta > 0$ the series $\sum_{k=1}^\infty p^{(k)}$ converges, and so $p^{(k)} \to 0$ as $k \to \infty$.
	By assumption, this implies that
	\[
	  e^{(k)} \to 0 \quad \text{(strongly) in } U.
	\]
	From the definition of $e^{(k)}$ and $\mu^{(k)} = \lambda^{(n)} - \lambda$,
	one can easily conclude that the sequence $(T^\top \lambda^{(k)})$ converges.
	Recalling that $\lambda = \alpha M T R\widehat u + \tau$ and $A R \widehat u - T^\top \tau = f$, we find that
	\begin{align*}
	  e^{(n)} & = (A + \alpha T^\top M T)^{-1} T^\top \lambda^{(n)} - (A + \alpha T^\top M T)^{-1} (\alpha T^\top M T R \widehat u + \underbrace{T^\top \tau}_{A R \widehat u - f})\\[-2ex]
		& = (A + \alpha T^\top M T)^{-1} (T^\top \lambda^{(n)} + f) - R \widehat u = u^{(n)} - R \widehat u.
	\end{align*}
	Therefore, $u^{(n)} \to R \widehat u$ in $U$.
\end{proof}

\begin{remark}
  In the typical coercive case (such as for Laplace's equation),
	the assumption in Theorem~\ref{thm:convSPD} essentially states that each subdomain operator $A_i$ has a trivial kernel.
	In the typical wave propagation case (such as for the Helmholtz equation),
	the assumption in Theorem~\ref{thm:convSPD} leads to the damping of any wave.
	Note, however, that the assumption is quite strong since $A_{i,1}$
	has to include not only a zero-order term (which would be more typical) but also a portion of the principal term.
\end{remark}

\subsection{Convergence with surjective traces}
\label{sect:convStrong}

In this section, linear convergence is shown under the additional assumption that the trace operator $T$ is surjective
($\range(T) = \Lambda$), but neither regularity nor compactness is needed anymore.
The proofs work along the lines of \cite[Sect.~4.2]{CollinoGhanemiJoly:2000a}.

\begin{lemma}
\label{lem:IXSIso}
  Let Assumptions~\ref{ass:A1}--\ref{ass:A3}, \ref{ass:A5} hold and in addition that $\range(T) = \Lambda$.
	Then the operator $(I - \mathcal{X}^\top S)$ is an isomorphism,
	and the same applies to the operator $(I - \mathcal{X} \underline{S})$ from Remark~\ref{rem:lambdaPrimal}.
\end{lemma}
\begin{proof}
  Recall the solution operator $\mathcal{S}^{(\lambda)} \colon U^* \to U \times \Lambda^*$ from Corollary~\ref{cor:waveRobin}.
  Since by assumption, $\range(T) = \Lambda$, there exists a bounded extension operator $E \colon \Lambda \to U$ such that $T E = I$,
	see also Proposition~\ref{prop:solOpTau}.
	We define $\mathcal{T} \colon \Lambda^* \to \Lambda^*$
	by
	\[
	  \mathcal{T} := \mathcal{S}^{(\lambda)}_\lambda \tfrac{1}{\alpha}(A + \alpha T^\top M T) E (M + \mathcal{X}^\top M \mathcal{X})^{-1} \mathcal{X}^\top,
	\]
	with the notation $\mathcal{S}^{(\lambda)} f = (\mathcal{S}^{(\lambda)}_u f, \mathcal{S}^{(\lambda)}_\lambda f)$.
  \begin{enumerate}
	\item[1)] $\mathcal{T}$ is well-defined, linear, and bounded,
	\item[2)] $(I - \mathcal{X}^\top S) \mathcal{T} e = e$ for all $e \in \Lambda^*$:
	  For arbitrary but fixed $e \in \Lambda^*$, we set
		\[
		  f := \tfrac{1}{\alpha}(A + \alpha T^\top M T) E (M + \mathcal{X}^\top M \mathcal{X})^{-1} \mathcal{X}^\top e,
		\]
		such that $\lambda = \mathcal{S}^{(\lambda)}_\lambda f = \mathcal{T} e$.
		From Corollary~\ref{cor:waveRobin} we see that $(I - \mathcal{X}^\top S) \lambda = d$
		where $d = \alpha \mathcal{X}^\top (M + \mathcal{X}^\top M \mathcal{X}) T (A + \alpha T^\top M T)^{-1} f$.
		As one can check, $d = e$.
	\item[3)] $\mathcal{T} (I - \mathcal{X}^\top S) \lambda = \lambda$ for all $\lambda \in \Lambda^*$:
	  For arbitrary but fixed $\lambda \in \Lambda^*$, we set
		\[
		  d := (I - \mathcal{X}^\top S) \lambda, \qquad 
			f := \tfrac{1}{\alpha}(A + \alpha T^\top M T) E (M + \mathcal{X}^\top M \mathcal{X})^{-1} \mathcal{X}^\top d.
		\]
		Obviously, $\mu := \mathcal{T} (I - \mathcal{X}^\top S) \lambda = \mathcal{S}^{(\lambda)}_\lambda f$.
		Proposition~\ref{prop:waveRobinSchur}(ii) implies that $(u, \lambda)$ with $u = (A + \alpha T^\top M T)^{-1} T^\top \lambda$ solves \eqref{eq:waveRobin}.
		Corollary~\ref{cor:waveRobin} shows that $\lambda = \mathcal{S}^{(\lambda)}_\lambda f$, so $\mu = \lambda$.
	\end{enumerate}
	Summarizing, $\mathcal{T}$ is the bounded inverse of $(I - \mathcal{X}^\top S)$.
\end{proof}

\begin{corollary}
\label{cor:IXSInvRepr}
  Let the prerequisites of Lemma~\ref{lem:IXSIso} be fulfilled.
	Then
	\begin{align*}
	  & (I - \mathcal{X}^\top S)^{-1}
			= \tfrac{1}{\alpha} E^\top \Big[ \widetilde A (R \widehat A^{-1} R^\top) \widetilde A - \widetilde A \Big] E
			  \big( M + \mathcal{X}^\top M \mathcal{X} \big)^{-1} \mathcal{X}^\top\\
		& = \tfrac{1}{\alpha} \Big[ (E^\top A + \alpha M T) R \widehat A^{-1} R^\top (A E + \alpha T^\top M)
			    - E^\top A E - \alpha M) \Big]
			  \big( M + \mathcal{X}^\top M \mathcal{X} \big)^{-1} \mathcal{X}^\top,
	\end{align*}
	where $\widetilde A = (A + \alpha T^\top M T)$ and $E \colon \Lambda \to U$ is an arbitrary extension operator such that $T E = I$.
\end{corollary}
\begin{proof}
  The statement follows from the proof of Lemma~\ref{lem:IXSIso} and the definition~\eqref{eq:SlambdaDef} of $S^{(\lambda)}$,
	see also Proposition~\ref{prop:solOpTau}.
\end{proof}

\begin{remark}
  The representation of $(I - \mathcal{X}^\top S)^{-1}$ in Corollary~\ref{cor:IXSInvRepr} is used in \cite{ClaeysParolin:Preprint2020}
	to derive a lower bound for the associated inf-sup constant in the discrete case, where this bound is independent of the mesh parameter.
	See also Remark~\ref{rem:uniformInfSup} below.
\end{remark}

\begin{corollary}
\label{cor:IXSinfSup}
  Let Assumptions~\ref{ass:A1}--\ref{ass:A3}, \ref{ass:A5}--\ref{ass:A6} hold and assume that $\range(T) = \Lambda$.
	Then there exists a constant $\gamma > 0$ such that
	\begin{align}
	\label{eq:IXSinfSup}
	  \| (I - \mathcal{X}^\top S)\lambda \|_{M^{-1}} \ge \gamma \| \lambda \|_{M^{-1}} \qquad \forall \lambda \in \Lambda^*.
	\end{align}
	Under the additional Assumptions~\ref{ass:A4}, \ref{ass:AXR} and~\ref{ass:A7}, the following coercivity estimate holds (with the same constant $\gamma$ as above):
	\begin{align}
	\label{eq:IXScoercivity}
	  \Re \big\langle M^{-1}(I - \mathcal{X}^\top S) \lambda, \overline \lambda \big\rangle \ge \frac{\gamma^2}{2} \| \lambda \|_{M^{-1}}^2
		\qquad \forall \lambda \in \Lambda^*.
	\end{align}
	Likewise, $\| (I - \mathcal{X} \underline{S}) \mu \|_M \ge \gamma \| \mu \|_M$
	and $\Re \big\langle M (I - \mathcal{X} \underline S) \mu, \overline\mu \big\rangle
	  \ge \frac{\gamma^2}{2} \| \mu \|_M^2$
	for all $\mu \in \Lambda$, cf.\ Remark~\ref{rem:lambdaPrimal}.
\end{corollary}
\begin{proof}
  The first estimate follows directly from Lemma~\ref{lem:IXSIso}, owing to the fact that $(I - \mathcal{X}^\top S)$
	has a bounded inverse and that $M^{-1}$ induces a norm due to \ref{ass:A6}.\\
	For the second part, let \ref{ass:A4}, \ref{ass:AXR} and~\ref{ass:A7} hold in addition.
  Using identity~\eqref{eq:ReXSIdentity} from the proof of Theorem~\ref{thm:convGeneral},
	\[
	  \Re \langle M^{-1} \mathcal{X}^\top S \lambda, \overline \lambda \rangle
		  = \tfrac{1}{2} \big[ \| \lambda \|_{M^{-1}}^2 + \| \mathcal{X}^\top S \lambda \|_{M^{-1}}^2
		    - \| (I - \mathcal{X}^\top S) \lambda \|_{M^{-1}}^2 \big] \qquad \forall \lambda \in \Lambda^*\,,
	\]
	as well as Property~\eqref{eq:XIso}, Lemma~\ref{lem:SNonExpansive}, and Lemma~\ref{lem:IXSIso}, we obtain
	\begin{align*}
	  & \Re \big\langle M^{-1} (I - \mathcal{X}^\top S) \lambda, \overline\lambda \big\rangle
		  = \| \lambda \|_{M^{-1}}^2 - \Re \big\langle M^{-1} \mathcal{X}^\top S \lambda, \overline \lambda \big\rangle \\
		& \qquad = \| \lambda \|_{M^{-1}}^2 - \frac{1}{2} \bigg[
		    \| \lambda \|_{M^{-1}}^2
				{} + {} \!\!\!\!\! \underbrace{ \| \mathcal{X}^\top S \lambda \|_{M^{-1}}^2 }_{= \| S \lambda \|_{M^{-1}}^2 \le \| \lambda \|_{M^{-1}}^2 } \!\!\!\!\!
		    {} - {} \underbrace{ \| (I - \mathcal{X}^\top S) \lambda \|_{M^{-1}}^2 }_{\ge \gamma^2 \| \lambda \|_{M^{-1}}^2 } \bigg]
		\ge \frac{\gamma^2}{2} \| \lambda \|_{M^{-1}}^2\,. \qedhere
	\end{align*}
\end{proof}

\begin{remark}
  The constant $\gamma$ in \eqref{eq:IXSinfSup} may be called an \emph{inf-sup constant} as \eqref{eq:IXSinfSup} is equivalent to
	\[
	  \inf_{\lambda \in \Lambda^*} \sup_{\mu \in \Lambda} \frac{ |\langle (I - \mathcal{X}^\top S) \lambda, \mu \rangle|
		          }{     \| \lambda \|_{M^{-1}} \| \mu \|_{M}    }
		\ge \gamma.
	\]
  A stronger property is
	\begin{align}
	\label{eq:op:fieldOfValues}
	  |\langle M^{-1} (I - \mathcal{X}^\top S) \lambda, \overline\lambda \rangle| \ge \gamma_{\text{FV}} \| \lambda \|_{M^{-1}}^2 \qquad \forall \lambda \in \Lambda^*.
	\end{align}
  This can be equally expressed by saying that the \emph{numerical range} (in the finite dimensional case also called \emph{field of values}) of $(I - \mathcal{X}^\top S)$
	with respect to the inner product $(\cdot,\cdot)_{M^{-1}}$ has a distance of at least $\gamma_{\text{FV}}$ from the origin:
	\[
	  \min_{z \in W} |z| \ge \gamma_{\text{FV}}\,, \quad \text{where }
		W = \big\{ \big( (I - \mathcal{X}^\top S) \lambda, \lambda \big)_{M^{-1}} \colon \| \lambda \|_{M^{-1}} = 1 \big\} \subset \mathbb{C}.
	\]
	Obviously, \eqref{eq:op:fieldOfValues} implies \eqref{eq:IXSinfSup} with $\gamma = \gamma_{\text{FV}}$:
	\[
	  \gamma_{\text{FV}} \| \lambda \|_{M^{-1}}^2 \le \big( (I - \mathcal{X}^\top S) \lambda, \lambda \big)_{M^{-1}}
		\le \| (I - \mathcal{X}^\top S) \lambda \|_{M^{-1}} \| \lambda \|_{M^{-1}}\,.
	\]
	An even stronger property is
	\begin{align}
	\label{eq:op:coercivity}
	  \Re \langle M^{-1} (I - \mathcal{X}^\top S) \lambda, \overline\lambda \rangle \ge \gamma_{\text{co}} \| \lambda \|_{M^{-1}}^2 \qquad \forall \lambda \in \Lambda^*,
	\end{align}
	as \eqref{eq:op:coercivity} implies \eqref{eq:op:fieldOfValues} with $\gamma_\text{FV} = \gamma_{\text{co}}$.
	Note that in the finite-dimensional case, \eqref{eq:op:coercivity} states that the Hermitian part of $(I - \mathcal{X}^\top S)$ is positive definite,
	where Hermitian and positive definite are to be understood with respect to the inner product $(\cdot,\cdot)_{M^{-1}}$ in $\Lambda^*$.

	We have seen that, in general, the coercivity property~\eqref{eq:op:coercivity} implies the inf-sup property~\eqref{eq:IXSinfSup}
	with $\gamma = \gamma_{\text{co}}$.
	The second part of Corollary~\ref{cor:IXSinfSup}, however, states that for the special operator $(I - \mathcal{X}^\top S)$,
	the inf-sup property~\eqref{eq:IXSinfSup} implies the coercivity property~\eqref{eq:op:coercivity}
	with $\gamma_\text{co} = \tfrac{1}{2}\gamma^2$.
\end{remark}

\begin{theorem}
\label{thm:convSchwarzSurjTrace}
  Let Assumptions~\ref{ass:A1}--\ref{ass:A7} hold and assume, in addition, that $\range(T) = \Lambda$.
	Then the Robin-Schwarz iteration~\eqref{eq:RobinSchwarzRichardon}
	with damping parameter $\beta \in (0, 1)$ converges linearly in the sense that for $n \ge 0$,
	\begin{itemize}
	\item $\| \lambda^{(n+1)} - \lambda \|_{M^{-1}} \le \rho \| \lambda^{(n)} - \lambda \|_{M^{-1}}$ and therefore
	      $\| \lambda^{(n)} - \lambda \|_{M^{-1}} \le \rho^n \| \lambda^{(0)} - \lambda \|_{M^{-1}}$,
	\item $\| u^{(n)} - R \widehat u \|_U \le C\,\rho^n \| \lambda^{(0)} - \lambda \|_{U}$,
	\end{itemize}
	where $C$ is constant and
	$\rho = \sqrt{1 - (1-\beta)\beta \gamma^2} < 1$, with the inf-sup constant $\gamma$ from Corollary~\ref{cor:IXSinfSup}.
	Likewise, $\| \underline\lambda^{(n+1)} - \underline\lambda \|_M \le \rho \| \underline\lambda^{(n)} - \underline\lambda \|_M$
	for the transformed iteration~\ref{eq:RobinSchwarzRichardonMTrafo}.
\end{theorem}
\begin{proof}
  As in the proof of Theorem~\ref{thm:convGeneral}, we define $\mu^{(n)} := \lambda^{(n)} - \lambda$
	and obtain
	\[
	  \| \mu^{(n+1)} \|_{M^{-1}}^2
		  = (1 - \beta) \| \mu^{(n)} \|_{M^{-1}}^2
		  - \beta (1 - \beta) \| (I - \mathcal{X}^\top S) \mu^{(n)} \|_{M^{-1}}^2
			+ \beta \| \mathcal{X}^\top S \mu^{(n)} \|_{M^{-1}}^2.
	\]
	Thanks to \ref{ass:A4}, Lemma~\ref{lem:SNonExpansive}, and Corollary~\ref{cor:IXSinfSup}, this implies
	\begin{align*}
	  \| \mu^{(n+1)} \|_{M^{-1}}^2
		& \le (1 - \beta) \| \mu^{(n)} \|_{M^{-1}}^2
		  - \beta (1 - \beta) \gamma^2 \| \mu^{(n)} \|_{M^{-1}}^2
			+ \beta \| \mu^{(n)} \|_{M^{-1}}^2\\
		& = \big(1 - \beta(1 - \beta) \gamma^2 \big) \| \mu^{(n)} \|_{M^{-1}}^2.
	\end{align*}
	As in the proof of Theorem~\ref{thm:convSPD}, one easily shows that
	$e^{(n)} = u^{(n)} - R \widehat u = (A + \alpha T^\top M T)^{-1} T^\top \mu^{(n)}$.
	Due to \ref{ass:A5}, \ref{ass:A6}, there exists a constant $C$ such that $\| e^{(n)} \|_U \le C \| \mu^{(n)} \|_{M^{-1}}$.
\end{proof}

\begin{example}
  For the Helmholtz equation (Example~\ref{ex:modelProblemHelmholtz}),
	suppose that we use the natural trace space $\Lambda_i = H^{1/2}(\Gamma_i)$
	on the subdomain interface $\Gamma_i := \partial\Omega_i \cap \bigcup_{j \neq i}\partial\Omega_j$,
	then the additional assumption $\range(T) = \Lambda$ holds.
	Let $M_i$ be a real-valued, symmetric impedance operator which is bounded positively from below.
	With $\alpha = \complexi$, Assumption~\ref{ass:A7} holds (see Table~\ref{tab:DDClass}),
	and Assumption~\ref{ass:A5} is guaranteed by Appendix~\ref{apx:invGenRobin}.
	We distinguish two cases.	
	\begin{enumerate}
	\item[(i)] In the case of no junctions, the usual swapping operator can be used (see also
	  Example~\ref{ex:crosspointProblem}, Example~\ref{ex:HelmholtzNoJunctions}, Example~\ref{ex:HelmholtzNoCrosspointButJunctions},
		and Example~\ref{ex:A4}),
	  provided that $M_i$ is block-diagonal and can be written as $(M_F)_{F \in \mathcal{F}_i}$,
	  i.e., with the \emph{same} impedance $M_F$ on either side of the facet $F$.
	  Then Assumptions~\ref{ass:A1}--\ref{ass:A3}, \ref{ass:A6}, and \ref{ass:AXR} are fulfilled,
	  and so Theorem~\ref{thm:convSchwarzSurjTrace} reproduces the result of \cite[Sect.~4.2]{CollinoGhanemiJoly:2000a}.
	\item[(ii)] In the general case, one can resort to the exchange operator constructed in Section~\ref{sect:interfaceExchange}.	
	  Then Assumptions~\ref{ass:A1}--\ref{ass:A3}, \ref{ass:A6}, and \ref{ass:AXR} are fulfilled as well,
	  and Theorem~\ref{thm:convSchwarzSurjTrace} reproduces the result of \cite{Claeys:2021a,ClaeysParolin:Preprint2020,Parolin:PhD} for the Helmholtz equation.
	\end{enumerate}
\end{example}

\begin{remark}
  In the finite-dimensional case, the convergence of the weighted GMRES iteration for $(I - \mathcal{X}^\top S)$
	using the inner product $(\cdot,\cdot)_{M^{-1}}$ can be estimated
	along the classical result by Elman \cite{Elman:PhD},
	see also \cite{LiesenTichy:2006a} and references therein.
	The lower bound~\eqref{eq:IXScoercivity}
	and the upper bound
	\[
	  \sup_{\lambda \in \Lambda^*} \frac{\| (I - \mathcal{X}^\top S) \lambda \|_{M^{-1}}}{\| \lambda \|_{M^{-1}}} \le 2,
	\]
	result in the convergence estimate of
	\[
	  \| (I - \mathcal{X}^\top S) (\lambda^{n} - \lambda) \|_{M^{-1}} \le \rho_\text{GMRES}^n \| (I - \mathcal{X}^\top S)(\lambda^{0} - \lambda) \|_{M^{-1}}\,,
	\]
	for the \emph{residuals} of weighted GMRES,
	where $\rho_\text{GMRES} = \sqrt{1 - \gamma^2/4} < 1$.
	This estimate is similar to that of Theorem~\ref{thm:convSchwarzSurjTrace}, observing that for the choice $\beta=1/2$,
	the estimated convergence rate for the \emph{iterates} is $\rho = \sqrt{1 - \gamma^2/4}$, cf.\ \cite[Remark~9]{CollinoGhanemiJoly:2000a}.
\end{remark}

\begin{lemma}
\label{lem:lowerBoundInfSup}
  Let Assumptions~\ref{ass:A1}--\ref{ass:A7} hold and assume that $\range(T) = \Lambda$.
  In addition, suppose that we have estimates of the form
	\begin{alignat*}{2}
	  \| R \widehat v \|_U & \le C_R \| \widehat v \|_{\widehat U} \qquad && \forall \widehat v \in \widehat U, \\
	  \| A v \|_{U^*} & \le C_A \| v \|_U \qquad && \forall v \in U, \\
		\| T v \|_M & \le C_T \| v \|_U \qquad && \forall v \in U, \\
	  \| E \lambda \|_U & \le C_E \| \lambda \|_M \qquad && \forall \lambda \in \Lambda, \\
		\| \widehat A \widehat v \|_{\widehat U^*} & \ge c_{\widehat A} \| \widehat v \|_{\widehat U}
		        \qquad && \forall \widehat v \in \widehat U,
	\end{alignat*}
	then the inf-sup stability~\eqref{eq:IXSinfSup} holds with
	\[
	  \gamma
		= 2 \left( \frac{ (C_A C_E + C_T)^2 C_R^2 }{c_{\widehat A}} + C_A C_E^2 + 1 \right)^{-1}
		\ge \frac{c_{\widehat A}}{ (C_A C_E + C_T)^2 C_R^2 }\,.
	\]
\end{lemma}
\begin{remark}
  In applications, often $\| R \widehat v \|_U = \| \widehat v \|_{\widehat U}$, which implies $C_R = 1$.
	If, in addition, $M$ is chosen as the minimal extension, i.e., $\| \lambda \|_M = \min_{v \in U \colon T v = \lambda} \| v \|_U$,
	and if $E \lambda = \mathop{\text{argmin}}_{v \in U \colon T v = \lambda} \| v \|_U$,
	then $C_T = 1$ and $C_E = 1$, and so the bound depends on $C_A$ and $c_{\widehat A}$ only.
\end{remark}
\begin{proof}[Proof of Lemma~\ref{lem:lowerBoundInfSup}]
  Apparently, \eqref{eq:IXSinfSup} is equivalent to
	\begin{align}
	\label{eq:boundIXTSInverse}
	  \| (I - \mathcal{X}^T S)^{-1} \mu \|_{M^{-1}} \le \gamma^{-1} \| \mu \|_{M^{-1}} \qquad \forall \mu \in \Lambda^*.
	\end{align}
	Corollary~\ref{cor:IXSInvRepr} together with \ref{ass:A4} yields
	\begin{align*}
	   (I - \mathcal{X}^T S)^{-1} \mu
		 = \tfrac{1}{2\alpha} \big[ (E^\top A + \alpha M T) R \widehat A^{-1} R^\top (A E + \alpha T^\top M)
		  - E^\top A E - \alpha M \big] M^{-1} \mathcal{X}^\top \mu.
	\end{align*}
	In order to estimate the $\| \cdot \|_{M^{-1}}$-norm of the above expression, we make use of the fact that
	$\| \cdot \|_{M^{-1}}$ is the dual norm of $\| \cdot \|_M$, which implies
	\begin{alignat}{2}
	  \| E^\top \psi \|_{M^{-1}} & \le C_E \| \psi \|_{U^*} \qquad && \forall \psi \in U^* \,,\\
	\label{eq:TtopBound}
		\| T^\top \mu \|_{U^*} & \le C_T \| \mu \|_{M^{-1}} \qquad && \forall \mu \in \Lambda^*.
	\end{alignat}
	Together with the assumed bounds for the operators $\widehat A$, $A$, $R$, $T$, and $E$ this yields
	\begin{multline*}
	  \| (I - \mathcal{X}^T S)^{-1} \mu \|_{M^{-1}} \\
		\le \frac{1}{2|\alpha|} \Big[ (C_E C_A + |\alpha| C_T) \frac{C_R^2}{c_{\widehat A}} (C_A C_E + |\alpha| C_T)
		  + (C_E C_A C_E + |\alpha|) \Big] \| M^{-1} \mathcal{X}^\top \mu \|_{M}\,.
	\end{multline*}
	Recall that due to Assumptions~\ref{ass:A6} and~\ref{ass:A4}, $\| \mathcal{X}^\top \mu \|_{M^{-1}} = \| \mu \|_{M^{-1}}$,
	and that due to \ref{ass:A7}, $|\alpha| = 1$. Altogether, this implies~\eqref{eq:boundIXTSInverse} with
	\[
	  \gamma^{-1} = \frac{1}{2} \left( \frac{ (C_A C_E + C_T)^2 C_R^2 }{c_{\widehat A}} + C_A C_E^2 + 1 \right).
	\]
	Since $T E = I$, it follows that $C_T C_E \ge 1$. In addition, since $\widehat A = R^\top A R$, it can be shown that
	$C_A C_R^2 / c_{\widehat A} \ge 1$. Therefore,
	\[
	  C_A C_E^2 + 1
		\le C_E (C_A C_E + C_T)
		\le \frac{C_E C_A C_R^2}{c_{\widehat A}} (C_A C_E + C_T)
		\le \frac{C_R^2}{c_{\widehat A}} (C_A C_E + C_T)^2\,,
	\]
	which implies the second estimate.
\end{proof}

\begin{remark}
\label{rem:uniformInfSup}
  In~\cite{ClaeysParolin:Preprint2020}, it is shown that for a \emph{family} of refined meshes
	with mesh parameter $h \to 0$,
	the associated family of Schwarz methods leads to a uniform positive bound for $\gamma_h$ for the Helmholtz equation.
	This fact is reflected in Lemma~\ref{lem:lowerBoundInfSup} when considering that $C_R = 1$,
	and that all the other estimates (for the chosen set of discrete operators) can be shown to hold uniformly w.r.t.\ $h \to 0$.
\end{remark}

For completeness, a result is given for the absorbing case with surjective trace.

\begin{theorem}
  Let Assumptions~\ref{ass:A1}--\ref{ass:A7} hold.
	In addition, assume that $\range(T) = \Lambda$
	and that there exist positive constants $a_i > 0$ such that for each $i=1,\ldots,N$,
	\begin{alignat*}{2}
	  \langle A_i v, \overline v \rangle & \ge a_i \| v \|_{U_i}^2 \qquad & \text{in case~(i) of \ref{ass:A7},}\\
		\langle A_{i,1} v, \overline v \rangle & \ge a_i \| v \|_{U_i}^2 \qquad & \text{in case~(ii) of \ref{ass:A7}.}
	\end{alignat*}
  Then
	\[
	  \| S \mu \|_{M^{-1}} \le \sqrt{1 - \zeta} \| \mu \|_{M^{-1}} \quad \forall \mu \in \Lambda^*, \qquad
		\text{with } \zeta = \big( \min_{i=1,\ldots,N} a_i \big) \frac{4}{(C_A + C_T^2)^2\, C_E^2}\,,
	\]
	with $C_A$, $C_E$, and $C_T$ as in Lemma~\ref{lem:lowerBoundInfSup}.
	Therefore, by Proposition~\ref{prop:SContractive},
	the sequence $(\lambda^{(n)})$ of Robin-Schwarz iterates with damping parameter $\beta \in (0, 1]$ converges linearly.
\end{theorem}
\begin{proof}
  Recall from Lemma~\ref{lem:SNonExpansive} that
	\[
	  \| S \mu \|_{M^{-1}}^2 \le \| \mu \|_{M^{-1}}^2 - 4 p, \qquad
		\text{where } p = \begin{cases} \langle A_i v, \overline v \rangle & \text{in case~(i) of \ref{ass:A7},} \\
		  \langle A_{i,1} v, \overline v \rangle & \text{in case~(ii) of \ref{ass:A7},} \end{cases}
	\]
	with $v = (A + \alpha T^\top M T)^{-1} T^\top \mu$. Due to the stated assumptions,
	\[
	  \| S \mu \|_{M^{-1}}^2 \le \| \mu \|_{M^{-1}}^2 - 4 \big( \min_{i=1,\ldots,N} a_i \big) \| v \|_U^2\,.
	\]
	From $\| (A + \alpha T^\top M T) v \|_{U^*} \le (C_A + C_T^2) \| v \|_{U}$
	and $\| \mu \|_{\Lambda^*} \le C_E \| E^\top \mu \|_{U^*}$
	(which follows from~\eqref{eq:TtopBound} with $\psi = E^\top \mu$ using $T E = I$),
	we obtain
	\[
	  \| v \|_U \ge \frac{1}{C_A + C_T^2} \| T^\top \mu \|_{U^*} \ge \frac{1}{(C_A + C_T^2)C_E} \| \mu \|_{M^{-1}}\,.
	\]
	Combination of the two estimates concludes the proof.
\end{proof}

\section{Generalized interface exchange operators}
\label{sect:interfaceExchange}

In this section, we follow the key idea of \cite{ClaeysParolin:Preprint2020} and construct generalized interface exchange operator
based on surjective trace operators.
Compared to \cite{ClaeysParolin:Preprint2020} the situation is more general and based on just two assumptions~\ref{ass:B1}, \ref{ass:B2}
on the trace operator, to be discussed below.

\begin{proposition}
  Let \ref{ass:A1} hold. Then the bubble space $U_B$ from Definition~\ref{def:bubble} is a closed subspace of $\range(R)$.
	If $\widehat U$ is a complexified Hilbert space, then so is $U_B$.
\end{proposition}
\begin{proof}
  From Definition~\ref{def:bubble}, it is easily seen that
	$U_{i,B} = \{ v_i \colon v \in \range(R),\ v_j = 0 \quad \forall j \neq i \}$.
	Therefore, we can write
	\[
	  U_B = \sum_{i=1}^N \{ v \in \range(R) \colon v_j = 0 \quad \forall j \neq i \},
	\]
	which shows that $U_B$ is a closed subspace of $\range(R)$.
	If $\widehat U$ is a complexified Hilbert space then $U$ is complexified and $R$ real valued. Therefore, $U_B$ is complexified, too.
\end{proof}

\noindent%
\fbox{%
\begin{minipage}{0.985\textwidth}
\begin{assumptb}
\label{ass:B1}
  $\ker(T) \subseteq U_B$.
\end{assumptb}
\end{minipage}
}

\smallskip

\noindent%
\fbox{%
\begin{minipage}{0.985\textwidth}
\begin{assumptb}
\label{ass:B2}
  $\range(T) = \Lambda$.
\end{assumptb}
\end{minipage}
}

\medskip

Assumption~\ref{ass:B1} states that a function $u_i$ with $T_i u_i = 0$ can always be extended by zero to the other subdomains.
Assumption~\ref{ass:B2} states that the trace operator is surjective, cf.~Example~\ref{ex:tracesH1Surj} and Sect.~\ref{sect:globs}.

\begin{remark}
	In the special case of $\ker(T) = U_B$, together with~\ref{ass:B2},
	it follows that $\Lambda$ is isomorphic to any complementary space $U_B^\perp$ such that $U = U_B \oplus U_B^\perp$
	and $T$ isomorphic to the operator that projects a function $u \in U$ to $U_B^\perp$.
\end{remark}

\begin{lemma}
\label{lem:RLambda}
  Let \ref{ass:A1}, \ref{ass:B1} hold. Then there exists a Hilbert space $\widehat \Lambda$ and a bounded linear operator
	$R_\Lambda \colon \widehat \Lambda \to \Lambda$ such that
	\begin{enumerate}
	\item[(i)] $R_\Lambda$ is injective,
	\item[(ii)] $\range(T R) = \range(R_\Lambda)$,
	\item[(iii)] $\widehat \Lambda$ is isomorphic to any complementary subspace $\mathcal{V}$ fulfilling $\widehat U = R^{-1}(\ker(T)) \oplus \mathcal{V}$,
	  and to any complementary subspace $\mathcal{W}$ fulfilling $\range(R) = \ker(T) \oplus \mathcal{W}$,
	\item[(iv)] if $\widehat U$ is a complexified Hilbert space, then also $\widehat\Lambda$ is complexified and $R_\Lambda$ is real-valued.
	\end{enumerate}
\end{lemma}
\begin{proof}
  Without loss of generality, we may assume that all spaces are real
	(in the complex case, we can follow the construction of the real case and then complexify the space and operator).
  Since $\ker(T)$ is a closed subspace of $U_B \subseteq \range(R)$, there exists a complementary space $\mathcal{W}$ such that $\range(R) = \ker(T) \oplus \mathcal{W}$.
	We restrict $T$ to $\mathcal{W}$ and call it $R_{\Lambda,\mathcal{W}} \colon \mathcal{W} \to \Lambda$.
	With this construction, $\widehat \Lambda = \mathcal{W}$ and $R_\Lambda = R_{\Lambda,\mathcal{W}}$ fulfill properties~(i), (ii), and (iv):
	\begin{align*}
	  \ker(R_{\Lambda,\mathcal{W}}) & = \ker(T) \cap \mathcal{W} = \{ 0 \},\\
	  \range(R_{\Lambda,\mathcal{W}}) & = T(\mathcal{W}) = T(\underbrace{\ker(T) \oplus \mathcal{W}}_{=\range(R)}) = \range(T R).
	\end{align*}
	Finally, since $R$ is an isomorphism between $\widehat U$ and $\range(R)$ and since $\ker(T)$ is closed,
	the space $R^{-1}(\ker(T)) \subseteq \widehat U$ is closed,
	and any complementary space $\mathcal{V}$ is isomorphic to $\mathcal{W}$.
\end{proof}

\begin{remark}
\label{rem:RLambda}
  Under the assumptions of Lemma~\ref{lem:RLambda}, one can even show that there exists an operator $\widehat T \colon \widehat U \to \widehat\Lambda$
	such that $R_\Lambda \widehat T = T R$, see Figure~\ref{fig:TRcommDiag}.
	If $\widehat \Lambda = \mathcal{V} \subset \widehat U$ (see Lemma~\ref{lem:RLambda}(iii)),
	the operator $\widehat T$ is simply the projection to $\mathcal{V}$ and vanishes on $R^{-1}(\ker(T))$.
\end{remark}

\begin{figure}
\begin{center}
  \begin{tikzpicture}
	  \pgftransformscale{0.6}
		
	  \definecolor{dred}{rgb}{0.8, 0.0, 0.0}
		\definecolor{mdred}{rgb}{1.0, 0.4, 0.4}
		\definecolor{lred}{rgb}{1.0, 0.8, 0.8}
		\definecolor{mlred}{rgb}{1.0, 0.6, 0.6}
		
		\definecolor{dgreen}{rgb}{0.0, 0.4, 0.0}
		\definecolor{mdgreen}{rgb}{0.0, 0.8, 0.0}
		\definecolor{lgreen}{rgb}{0.8, 1.0, 0.8}
		\definecolor{mlgreen}{rgb}{0.6, 1.0, 0.6}
		
		\definecolor{dblue}{rgb}{0.0, 0.0, 0.8}
		\definecolor{mdblue}{rgb}{0.4, 0.4, 1.0}
		\definecolor{lblue}{rgb}{0.8, 0.8, 1.0}
		\definecolor{mlblue}{rgb}{0.6, 0.6, 1.0}
		
		\definecolor{dyellow}{rgb}{0.4, 0.4, 0.0}
		\definecolor{mdyellow}{rgb}{0.8, 0.8, 0.4}
		\definecolor{lyellow}{rgb}{1.0, 1.0, 0.8}
		\definecolor{mlyellow}{rgb}{1.0, 1.0, 0.6}
		
		\definecolor{dgrey}{rgb}{0.4, 0.4, 0.4}
		\definecolor{mdgrey}{rgb}{0.55, 0.55, 0.55}
		\definecolor{lgrey}{rgb}{0.9, 0.9, 0.9}
		\definecolor{mlgrey}{rgb}{0.75, 0.75, 0.75}
				
		
		\draw[line width=0.75pt,color=dgrey,fill=lgrey] (1,0)--(3,0)--(3,2)--(1,2)--(1,0)--(3,0);
		\draw[line width=0.75pt,color=mdgrey,dashed] (1,1)--(3,1);
		\draw[line width=0.75pt,color=mdgrey,dashed] (2,0)--(2,2);
		
		\node at (0.5,1.8) {$\widehat U$};

		\draw[->, >=stealth] (3.5,1)--(5.5,1);
		\node at (4.25,1.5) {$R$};

		\node at (5.5,2) {$U$};


		\draw[line width=1.3pt,color=mdgrey] (1,1-4.5)--(3,1-4.5);
		\draw[line width=1.3pt,color=mdgrey] (2,0-4.5)--(2,2-4.5);

		\node at (0.5,1.8-4.5) {$\widehat \Lambda$};

		\draw[->, >=stealth] (3.5,1-4.5)--(5.5,1-4.5);
		\node at (4.25,1.5-4.5) {$R_\Lambda$};
		

	  \node (G0) at (6,-0.25) {};
	  \node (G1) at (7,-0.25) {};
	  \node (G2) at (7,0.75) {};
	  \node (G3) at (6,0.75) {};
		
		\draw[line width=0.75pt,color=dblue,fill=lblue] (G0.center)--(G1.center)--(G2.center)--(G3.center)--(G0.center)--(G1.center);
		
		\node (H0) at (7.5,-0.25) {};
	  \node (H1) at (8.5,-0.25) {};
	  \node (H2) at (8.5,0.75) {};
	  \node (H3) at (7.5,0.75) {};
		
		\draw[line width=0.75pt,color=dred,fill=lred] (H0.center)--(H1.center)--(H2.center)--(H3.center)--(H0.center)--(H1.center);
		
		\node (I0) at (6,1.25) {};
	  \node (I1) at (7,1.25) {};
	  \node (I2) at (7,2.25) {};
	  \node (I3) at (6,2.25) {};
		
		\draw[line width=0.75pt,color=dyellow,fill=lyellow] (I0.center)--(I1.center)--(I2.center)--(I3.center)--(I0.center)--(I1.center);
		
		\node (J0) at (7.5,1.25) {};
	  \node (J1) at (8.5,1.25) {};
	  \node (J2) at (8.5,2.25) {};
	  \node (J3) at (7.5,2.25) {};
		
		\draw[line width=0.75pt,color=dgreen,fill=lgreen] (J0.center)--(J1.center)--(J2.center)--(J3.center)--(J0.center)--(J1.center);
		
				
		\draw[line width=1.3pt,color=dblue] (6,0.75-4.5)--(7,0.75-4.5)--(7,-0.25-4.5);
		
		\draw[line width=1.3pt,color=dyellow] (6,1.25-4.5)--(7,1.25-4.5)--(7,2.25-4.5);
		
		\draw[line width=1.3pt,color=dred] (7.5,-0.25-4.5)--(7.5,0.75-4.5)--(8.5,0.75-4.5);
		
		\draw[line width=1.3pt,color=dgreen] (7.5,2.25-4.5)--(7.5,1.25-4.5)--(8.5,1.25-4.5);
		
		
		\node at (5.5,-2.25) {$\Lambda$};
		
		\draw[->, >=stealth] (7.25,-0.75)--(7.25,-1.75);
		\node at (7.75,-1.25) {$T$};

		\draw[->, >=stealth] (2,-0.75)--(2,-1.75);
		\node at (2.5,-1.25) {$\widehat T$};
		
  \end{tikzpicture}
	\hspace{2cm}
  \begin{tikzpicture}
	  \pgftransformscale{0.6}
		
	  \definecolor{dred}{rgb}{0.8, 0.0, 0.0}
		\definecolor{mdred}{rgb}{1.0, 0.4, 0.4}
		\definecolor{lred}{rgb}{1.0, 0.8, 0.8}
		\definecolor{mlred}{rgb}{1.0, 0.6, 0.6}
		
		\definecolor{dgreen}{rgb}{0.0, 0.4, 0.0}
		\definecolor{mdgreen}{rgb}{0.0, 0.8, 0.0}
		\definecolor{lgreen}{rgb}{0.8, 1.0, 0.8}
		\definecolor{mlgreen}{rgb}{0.6, 1.0, 0.6}
		
		\definecolor{dblue}{rgb}{0.0, 0.0, 0.8}
		\definecolor{mdblue}{rgb}{0.4, 0.4, 1.0}
		\definecolor{lblue}{rgb}{0.8, 0.8, 1.0}
		\definecolor{mlblue}{rgb}{0.6, 0.6, 1.0}
		
		\definecolor{dyellow}{rgb}{0.4, 0.4, 0.0}
		\definecolor{mdyellow}{rgb}{0.8, 0.8, 0.4}
		\definecolor{lyellow}{rgb}{1.0, 1.0, 0.8}
		\definecolor{mlyellow}{rgb}{1.0, 1.0, 0.6}
		
		\definecolor{dgrey}{rgb}{0.4, 0.4, 0.4}
		\definecolor{mdgrey}{rgb}{0.55, 0.55, 0.55}
		\definecolor{lgrey}{rgb}{0.9, 0.9, 0.9}
		\definecolor{mlgrey}{rgb}{0.75, 0.75, 0.75}
				

		\draw[line width=1.3pt,color=mdgrey] (1-0.2,1-4.5)--(2-0.2,1-4.5);
		\draw[line width=1.3pt,color=mdgrey] (2,0-4.5-0.2)--(2,1-4.5-0.2);
		\draw[line width=1.3pt,color=mdgrey] (2+0.2,1-4.5)--(3+0.2,1-4.5);
		\draw[line width=1.3pt,color=mdgrey] (2,1-4.5+0.2)--(2,2-4.5+0.2);
		\draw[line width=1.3pt,color=mdgrey,fill=mdgrey] (2,1-4.5) circle (0.02);

		\node at (0.5,1.8-4.5) {$\widehat \Lambda$};

		\draw[->, >=stealth] (3.5,1-4.35)--(5.5,1-4.35);
		\node at (4.5,1.5-4.35) {$R_\Lambda$};
		
		\draw[<-, >=stealth] (3.5,1-4.65)--(5.5,1-4.65);
		\node at (4.5,1.5-5.75) {$E_\Lambda$};

				
		\draw[line width=1.3pt,color=dblue] (6-0.2,0.75-4.5)--(7-0.2,0.75-4.5);
		\draw[line width=1.3pt,color=dblue] (7,0.75-4.5-0.2)--(7,-0.25-4.5-0.2);
		\draw[line width=1.3pt,color=dblue,fill=dblue] (7.0, 0.75-4.5) circle (0.02);
		
		\draw[line width=1.3pt,color=dyellow] (6-0.2,1.25-4.5)--(7-0.2,1.25-4.5);
		\draw[line width=1.3pt,color=dyellow] (7,1.25-4.5+0.2)--(7,2.25-4.5+0.2);
		\draw[line width=1.3pt,color=dyellow,fill=dyellow] (7,1.25-4.5) circle (0.02);
		
		\draw[line width=1.3pt,color=dred] (7.5,-0.25-4.5-0.2)--(7.5,0.75-4.5-0.2);
		\draw[line width=1.3pt,color=dred] (7.5+0.2,0.75-4.5)--(8.5+0.2,0.75-4.5);
		\draw[line width=1.3pt,color=dred,fill=dred] (7.5,0.75-4.5) circle (0.02);
		
		\draw[line width=1.3pt,color=dgreen] (7.5,2.25-4.5+0.2)--(7.5,1.25-4.5+0.2);
		\draw[line width=1.3pt,color=dgreen] (7.5+0.2,1.25-4.5)--(8.5+0.2,1.25-4.5);
		\draw[line width=1.3pt,color=dgreen,fill=dgreen] (7.5,1.25-4.5) circle (0.02);
		
		
		\node at (5.75,-2.25) {$\Lambda$};
		

		\draw[line width=1.3pt,color=mdgrey] (1,1)--(3,1);
		\draw[line width=1.3pt,color=mdgrey] (2,0)--(2,2);

		\node at (0.5,1.8) {$\widehat \Lambda$};

		\draw[->, >=stealth] (3.5,1-4.35+4.5)--(5.5,1-4.35+4.5);
		\node at (4.5,1.5-4.35+4.5) {$R_\Lambda$};
		
		\draw[<-, >=stealth] (3.5,1-4.65+4.5)--(5.5,1-4.65+4.5);
		\node at (4.5,1.5-5.75+4.5) {$E_\Lambda$};

				
		\draw[line width=1.3pt,color=dblue] (6,0.75)--(7,0.75)--(7,-0.25);
		
		\draw[line width=1.3pt,color=dyellow] (6,1.25)--(7,1.25)--(7,2.25);
		
		\draw[line width=1.3pt,color=dred] (7.5,-0.25)--(7.5,0.75)--(8.5,0.75);
		
		\draw[line width=1.3pt,color=dgreen] (7.5,2.25)--(7.5,1.25)--(8.5,1.25);
		
		
		\node at (5.75,-2.25+4.5) {$\Lambda$};
		
  \end{tikzpicture}
\caption{\label{fig:TRcommDiag}%
  \emph{Left:} Illustration of Lemma~\ref{lem:RLambda} and Remark~\ref{rem:RLambda} (the diagram commutes).
	\emph{Top right:} Illustration of Lemma~\ref{lem:ELambda}.
	\emph{Bottom right:} Illustration of Sect.~\ref{sect:globLocalImp}.
}
\end{center}
\end{figure}
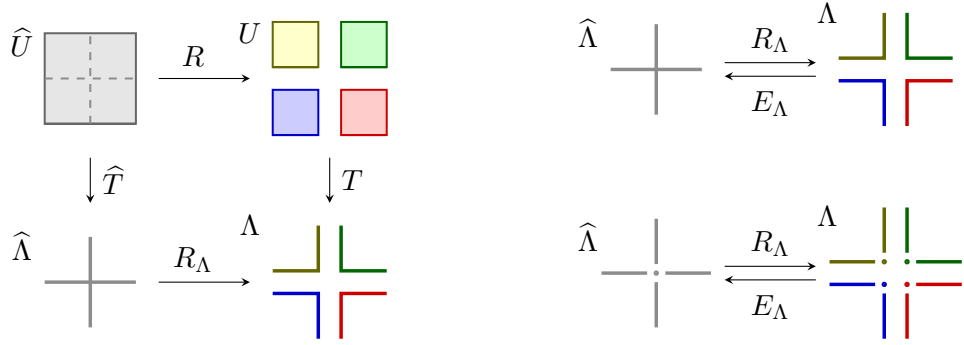

\begin{proposition}
\label{prop:rangeRbyRLambda}
  Let \ref{ass:A1}, \ref{ass:B1} hold. Then
	\[
	  u \in \range(R) \quad \Longleftrightarrow \quad T u \in \range(R_\Lambda).
	\]
\end{proposition}
\begin{proof}
  ``$\Longrightarrow$'' follows from Lemma~\ref{lem:RLambda}(ii).\\
	``$\Longleftarrow$'': Let $u \in U$ with $T u \in \range(R_\Lambda) = \range(T R)$ by Lemma~\ref{lem:RLambda}(ii).
	So there exists $\widehat v \in \widehat U$ such that $T u = T R \widehat v$, or equivalently, $u - R \widehat v \in \ker(T) \subseteq U_B$ by Assumption~\ref{ass:B1}.
	Consequently, there exists $v_B \in U_B$ such that $u = R \widehat v + v_B$.
	Since both $R \widehat v$ and $v_B$ are contained in $\range(R)$, the proof is concluded.
\end{proof}

\begin{proposition}
\label{prop:rangeTR}
  Let \ref{ass:A1}, \ref{ass:B1}, and \ref{ass:B2} hold. Then $\range(T R)$ is closed.
\end{proposition}
\begin{proof}
  Due to \ref{ass:B2}, $\range(T)$ is closed, so there exists a bounded right-inverse
	\[
	  T^\dag \colon \range(T) \to U, \qquad		T T^\dag T = T.
	\]
	Let $(\widehat u^{(k)})$ be an arbitrary but fixed sequence in $\widehat U$
	with the property that $T R \widehat u^{(k)} \stackrel{k \to \infty}{\longrightarrow} \lambda^{(\infty)} \in \Lambda$.
	Since $\range(T)$ is closed, also $\lambda^{(\infty)} \in \range(T)$. We apply $T^\dag$:
	\[
	  \underbrace{T^\dag T R \widehat u^{(k)}}_{=: w^{(k)}} \stackrel{k \to \infty}{\longrightarrow} \underbrace{T^\dag \lambda^{(\infty)}}_{=: w^{(\infty)}} \in U.
	\]
	Since $T w^{(k)} = T R \widehat u^{(k)}$ or equivalently, $w^{(k)} - R \widehat u^{(k)} \in \ker(T)$,
	we can conclude from \ref{ass:B1} that $w^{(k)} = R \widehat u^{(k)} + w_B^{(k)}$ for some bubble function $w_B^{(k)} \in U_B \subseteq \range(R)$.
	Therefore, $w^{(k)} \in \range(R)$.
	Since $(w^{(k)})$ converges and $\range(R)$ is closed (due to~\ref{ass:A1}), it follows that $w^{(\infty)} \in \range(R)$.
	Observing that $\lambda^{(\infty)} = T w^{(\infty)}$, we conclude that $\lambda^{(\infty)} \in \range(T R)$.
	Summarizing, the limit of the arbitrary sequence $T R \widehat u^{(k)}$ is again in $\range(T R)$, so $\range(T R)$ must be closed.
\end{proof}

\begin{lemma}
\label{lem:ELambda}
  Let \ref{ass:A1}, \ref{ass:B1}, and \ref{ass:B2} hold. Then
	\begin{enumerate}
	\item[(i)] $\range(R_\Lambda)$ is closed,
	\item[(ii)] there exists a bounded linear, real-valued operator $E_\Lambda \colon \Lambda \to \widehat \Lambda$ such that $E_\Lambda R_\Lambda = I$,
	\item[(iii)] for any operator $E_\Lambda$ with the properties of (ii),
	  the reflection operator $\mathcal{X} := 2 R_\Lambda E_\Lambda - I$ fulfills Assumption~\ref{ass:A2}.
	\end{enumerate}
\end{lemma}
\begin{proof}
  Part~(i): Due to Lemma~\ref{lem:RLambda}, $\range(R_\Lambda) = \range(T R)$, which is closed (Proposition~\ref{prop:rangeTR}).\\
	Part~(ii):
	Since $\range(R_\Lambda)$ is closed, there exists a pseudo-inverse $E_\Lambda \colon \Lambda \to \widehat\Lambda$
	(one out of many) such that $R_\Lambda E_\Lambda R_\Lambda = R_\Lambda$.
	Since $R_\Lambda$ is injective, also $E_\Lambda R_\Lambda = I$.\\
	Part~(iii): From~(ii), we find that $P:= R_\Lambda E_\Lambda \colon \Lambda \to \Lambda$ is a \emph{projection} (i.e., $P^2 = P$)
	and that $\range(P) = \range(R_\Lambda)$.
	The operator $\mathcal{X} := 2 P - I$ is the \emph{reflection} corresponding to $P$, and one checks easily that $\mathcal{X}^2 = I$.
	Moreover,
	\[
	  \ker\big(\tfrac{1}{2}(I - \mathcal{X})\big) = \ker(I - P) = \range(P) = \range(R_\Lambda),
	\]
	and so also $\ker((I - \mathcal{X})) = \range(R_\Lambda)$.
	Finally, Proposition~\ref{prop:rangeRbyRLambda} leads to the conclusion that
	\[
	  u \in \range(R) \quad \Longleftrightarrow \quad T u \in \ker(I - \mathcal{X}),
	\]
	which is another way of expressing that $\range(R) = \ker((I - \mathcal{X})T)$.
\end{proof}

Recall that in Sect.~\ref{sect:convGeneral} we assumed \ref{ass:A4}, i.e., $\mathcal{X}^\top M \mathcal{X} = M$,
which means a certain restriction on $M$ and $\mathcal{X}$.
The following lemma provides a construction of an operator $\mathcal{X}$ \emph{depending} on $M$.

\begin{lemma}
\label{lem:ELambdaM}
  Let \ref{ass:A1}, \ref{ass:B1}, \ref{ass:B2}, and \ref{ass:A6} hold.
	Then $E_\Lambda \colon \Lambda \to \widehat\Lambda$ given by
	\[
	  E_\Lambda = (R_\Lambda^\top M R_\Lambda)^{-1} R_\Lambda^\top M
	\]
	is well-defined and the operator $\mathcal{X} := 2 R_\Lambda E_\Lambda - I$ fulfills Assumptions~\ref{ass:A2}
	and~\ref{ass:A4}.
	Moreover, $\mathcal{X}$ is real-valued, i.e., Assumption~\ref{ass:AXR} holds.
\end{lemma}
\begin{proof}
  Firstly, $(R_\Lambda^\top M R_\Lambda) \colon \widehat\Lambda \to \widehat\Lambda^*$
	is easily seen to be real-valued, symmetric, and positively bounded from below, which is why it has a bounded inverse.
	So, $E_\Lambda$ is well-defined and real-valued. By construction, $E_\Lambda R_\Lambda = I$.
	Secondly, the operator $\mathcal{X}$ as defined above is real-valued and fulfills Assumption~\ref{ass:A2}, see Lemma~\ref{lem:ELambda}.
	Lastly, using the symmetry of $M$, one easily verifies the identity $E_\Lambda^\top R_\Lambda^\top M = M R_\Lambda E_\Lambda$,
	which implies $\mathcal{X}^\top M = M \mathcal{X}$.
	The proof is concluded by applying $\mathcal{X}$ from the right.
\end{proof}

\begin{remark}
  For $E_\Lambda$, $\mathcal{X}$ given as in Lemma~\ref{lem:ELambdaM},
	$P := R_\Lambda E_\Lambda = \tfrac{1}{2}(I + \mathcal{X})$
	is the $M$-orthogonal projector to $\range(R_\Lambda) \subseteq \Lambda$,
	cf.~\cite[Sect.~4]{ClaeysParolin:Preprint2020}.
\end{remark}

\begin{proposition}
\label{prop:XLambdaM}
  Let \ref{ass:A1}, \ref{ass:B1}, \ref{ass:B2}, and \ref{ass:A6} hold. Moreover, let $\mathcal{X} \colon \Lambda \to \Lambda$
	be an operator fulfilling \ref{ass:A2} and \ref{ass:A4}.
	Then $\mathcal{X} = 2 R_\Lambda E_\Lambda - I$ with $E_\Lambda$ from Lemma~\ref{lem:ELambdaM}.
\end{proposition}
\begin{proof}
  Due to Assumption~\ref{ass:A2} and Proposition~\ref{prop:rangeRbyRLambda},
	\[
	  u \in \range(R) \quad \Longleftrightarrow \quad (I - \mathcal{X})T u = 0 \quad \Longleftrightarrow \quad T u \in \range(R_\Lambda).
	\]
	Since $\range(T) = \Lambda$, it follows that $\range(R_\Lambda) = \ker(I - \mathcal{X})$.
	As an immediate consequence, $\mathcal{X} R_\Lambda = R_\Lambda$.
	Applying $R_\Lambda^\top$ to the identity $\mathcal{X}^\top M \mathcal{X} = M$ of \ref{ass:A4} yields
	\[
	  R_\Lambda^\top M \mathcal{X} = R_\Lambda^\top M.
	\]
	Lemma~\ref{lem:X} implies that $\range(I + \mathcal{X}) = \range(R_\Lambda)$.
	Let $\lambda \in \Lambda$ be arbitrary but fixed. Then there exists $\widehat \lambda \in \widehat\Lambda$
	such that $(I + \mathcal{X}) \lambda = R_\Lambda \widehat\lambda$.
	Combination with the above identity yields
	\[
	  R_\Lambda^\top M \underbrace{(I + \mathcal{X}) \lambda}_{R_\Lambda \widehat\lambda} = 2 R_\Lambda^\top M \lambda.
	\]
  Due to \ref{ass:A6} and the fact that $\range(R_\Lambda)$ is closed, the operator $R_\Lambda^\top M R_\Lambda$
	has a bounded inverse and so
	\[
	  \widehat\lambda = 2 (R_\Lambda^\top M \lambda R_\Lambda)^{-1} R_\Lambda^\top M \lambda = 2 E_\Lambda \lambda.
	\]
	Finally, $\mathcal{X}\lambda = (I + \mathcal{X})\lambda - \lambda = R_\lambda \widehat\lambda - \lambda
	= (2 R_\Lambda E_\Lambda - I)\lambda$.
\end{proof}

\begin{remark}
\label{rem:ELambdaM}
  The statements of Lemma~\ref{lem:ELambdaM} and Proposition~\ref{prop:XLambdaM} also hold
	if we replace \ref{ass:A6} by the weaker set of assumptions that
	(i) $M^\top = M$ and (ii) $R_\Lambda^\top M R_\Lambda$ has a bounded inverse.
	But then, in Lemma~\ref{lem:ELambdaM} $\mathcal{X}$ is no more guaranteed to be real-valued and $P = R_\Lambda E_\Lambda$
	is not necessarily an orthogonal projection either.
\end{remark}

\begin{remark}
\label{rem:HermitianImp}
  The convergence analysis from Sect.~\ref{sect:convergence} (in particular Thm.~\ref{thm:convGeneral}, Thm.~\ref{thm:convSchwarzSurjTrace})
	can without much effort also be generalized to the case of Hermitian impedance operators $M = M^\herm$.
	For this purpose, one has to work with the slightly adapted interface flux formulation
	\begin{align}
	\label{eq:tracefluxH}
	\begin{aligned}
		\text{find } (u, \tau) \in U \times \Lambda^* \colon \quad
		A u - T^\top \tau & = f,\\
		(I - \mathcal{X}) T u & = 0,\\
		(I + \mathcal{X}^\herm) \tau & = 0,
	\end{aligned}
	\end{align}
	instead of \eqref{eq:traceflux}. Assumption~\ref{ass:A3} has to be replaced by the assumption that $M + \mathcal{X}^\herm M \mathcal{X}$
	has a bounded inverse (see also Remark~\ref{rem:RobinH}),
	then the interface impedance trace formulation~\eqref{eq:waveRobinGeneral} takes the form
  \begin{align}
  \label{eq:waveRobinGeneralH}
    \text{find } (u, \lambda) \in U \times \Lambda^* \colon \quad
	  \begin{bmatrix} (A + \alpha T^\top M T) & - T^\top \\ -\alpha \mathcal{X}^\herm (M + \mathcal{X}^\herm M \mathcal{X}) T & (I + \mathcal{X}^\herm) \end{bmatrix}
	  \begin{bmatrix} u \\ \lambda \end{bmatrix}
	  = \begin{bmatrix} f \\ 0 \end{bmatrix}.
  \end{align}
	Assumption~\ref{ass:A4} has to be replaced by the assumption that $\mathcal{X}^\herm M \mathcal{X} = M$,
	Assumption~\ref{ass:A6} can be weakened to requiring $M_i$ only needs Hermitian ($M_i^\herm = M_i$) and bounded positively from below
	(still constituting a norm), and Assumption~\ref{ass:AXR} can be dropped.
	In that case, the exchange operator is an isometry by construction.
	Moreover, if $\mathcal{X}$ is constructed as in Lemma~\ref{lem:ELambdaM},
	i.e., $\mathcal{X} = 2 R_\Lambda (R_\Lambda^\top M R_\Lambda)^{-1} R_\Lambda^\top M$,
	then the identity $\mathcal{X}^\herm M \mathcal{X} = M$ holds. The exchange operator, however, is in general not real-valued anymore.
	Note that for standard wave propagation problems, the (complex-valued) operator $A$ is \emph{symmetric},
	which means that the augmented operator $(A + \complexi \alpha T^\top M T)$, being a combination of a symmetric and a Hermitian operator,
	looses such structural property.
	A treatment of even more general impedance operators (operators with positive definite Hermitian part) can be found in \cite{Claeys:2021Preprint}.
\end{remark}

\subsection{Local and quasi-local impedance operators}
\label{sect:globLocImp}

The definition of $E_\Lambda$ in Lemma~\ref{lem:ELambdaM}
rewrites as $E_\lambda \lambda = \widehat\lambda \in \widehat\Lambda$,
where
\begin{align}
\label{eq:globalInterfaceProblem}
  (R_\Lambda^\top M R_\Lambda) \widehat \lambda = R_\Lambda^\top M \lambda,
\end{align}
i.e., the application of $E_\Lambda$ requires the solution of an coercive problem on the continuous interface space $\widehat\Lambda$.
Unless the decomposition is free of cross points (such that the geometric interface splits into individual components not touching each other),
this is a \emph{global} problem.
Still, one can say that one has reduced the original non-coercive wave propagation problem to a sequence of local wave problems and a global coercive problem.
But from an algorithmic point of view, (i)~this global problem must be solved in each step of the Schwarz iteration and
(ii)~if \eqref{eq:globalInterfaceProblem} is not solved exactly, a refined convergence analysis would actually be necessary.
Fortunately, in the discrete case this drawback can be overcome using local or quasi-local impedance operators.

\subsubsection{Glob-local impedance operators}
\label{sect:globLocalImp}

Let the assumptions of Sect.~\ref{sect:discreteMatrixNotation}
hold and let $\Lambda$ and $T$ be constructed via the glob system, see Sect.~\ref{sect:globs}.
Assume furthermore that each $M_i \colon \Lambda_i \to \Lambda_i^*$ in Assumption~\ref{ass:A6} has the form
\begin{align}
  (M_i \lambda)_{iG} = M_{iG} \lambda_{iG}\,,
\end{align}
i.e., $M_i$ is \emph{block-diagonal} with respect to the glob partition.
Then, one can show that $E_\Lambda$ from Lemma~\ref{lem:ELambdaM} has the form
\begin{align}
\label{eq:ELambdaFormulaGlob}
  (E_\Lambda \lambda)_{G} = \widehat M_G^{-1} \sum_{i \in \mathcal{N}_G} M_{iG} \lambda_{iG}\,,
	\quad \text{where } \widehat M_G := \sum_{j \in \mathcal{N}_G} M_{jG}\,,
\end{align}
where we use the convention that $\widehat\Lambda = \productspace_{G \in \mathcal{G}} U_G$ and
$\widehat\lambda_G$ denotes the component of $\widehat\lambda \in \widehat\Lambda$ corresponding to glob $G$,
see also Figure~\ref{fig:TRcommDiag}.
With the definition from Remark~\ref{rem:ED}, $R_\Lambda E_\Lambda = E_D$ with weight matrices $D_{jG} = \widehat M_G^{-1} M_{jG}$.

Apparently, making the matrices $\widehat M_G^{-1}$ available is quite an affordable operation
because it requires next neighbor communication only. Similar procedures are used in FETI-DP and BDDC methods with \emph{deluxe scaling}
\cite{BeiraoDaVeigaPavarinoScacchiWidlundZampini:2014a,DohrmannWidlund:2016a,PechsteinDohrmann:2017a}.
The investigation of Robin-Schwarz methods with glob-local impedance operators is yet a topic of future research,
in particular their performance in practice and their convergence analysis with respect to the discretization parameter.
Clearly, enforcing $M_i$ to be block-diagonal with respect to the glob partition comes at the price that this operator
has no continuous counterpart anymore, so the convergence rate is likely to be no more independent of the discretization parameter.
The goal, of course, would be an impedance operator with a rate depending only very mildly on the discretization parameter.

In the following cases, the operator $E_\Lambda$ becomes fully local.
If for every glob $G$, the impedance operators have the special form
\[
  M_{iG} = k_{iG} \check M_G\,,
\]
with scalars $k_{iG}$ and an operator $\check M_G$ independent of the subdomain index, then formula~\eqref{eq:ELambdaFormulaGlob} simplifies to
\begin{align}
  (E_\Lambda \lambda)_{G} = \sum_{i \in \mathcal{N}_G} \frac{k_{iG}}{\sum_{j \in \mathcal{N}_G} k_{jG}} \lambda_{iG}\,.
\end{align}
Finally, if for every glob $G$, the impedence operators $M_{iG}$ are the same,
i.e., $M_{iG} = \check M_G$, for all $i \in \mathcal{N}_G$,
then
\begin{align}
\label{eq:ElambdaMult}
  (E_\Lambda \lambda)_G = \frac{1}{\cardinality{\mathcal{N}_G}} \sum_{i \in \mathcal{N}_G} \lambda_{iG}\,,
\end{align}
i.e., $R_\Lambda E_\Lambda$ is the same as the multiplicity projector $E_m$ from Sect.~\ref{sect:globs}.
The associated exchange operator $\mathcal{X} = 2 R_\Lambda E_\Lambda - I$ fulfills Assumption~\ref{ass:A4}
but has the simple form
\begin{align}
\label{eq:XlambdaMult}
  (\mathcal{X} \lambda)_{iG} = \Big( \frac{2}{\cardinality{\mathcal{N}_G}} - 1 \Big) \lambda_{iG}
	  + \sum_{j \in \mathcal{N}_G \setminus \{ i \}} \frac{2}{\cardinality{\mathcal{N}_G}} \lambda_{jG}\,.
\end{align}
For a glob shared by two subdomains, the operator $\mathcal{X}$ simply \emph{swaps} the two associated functions.

\subsubsection{Diagonal impedance operators}

Let the assumptions of Sect.~\ref{sect:discreteMatrixNotation}
hold and let $\Lambda$ and $T$ be constructed via the glob system, see Sect.~\ref{sect:globs}.
Assume furthermore that $M_i \colon \Lambda_i \to \Lambda_i^*$ in Assumption~\ref{ass:A6}
is a \emph{diagonal} matrix with diagonal entries $( m^{(i)}_\ell )_{\ell = 1}^{\text{dim}(\Lambda_i)}$.
Then, one can show that $E_\Lambda$ from Lemma~\ref{lem:ELambdaM} has the form
\begin{align}
  (E_\Lambda \lambda)_{k}
	  = \widehat m_k^{-1}
		\sum_{i \in \widehat{\mathcal{N}}_k} m^{(i)}_{\widehat{\mathsf{g}}_i^{-1}(k)} \lambda_{i,\widehat{\mathsf{g}}_i^{-1}(k)}\,,
	\quad \text{where } \widehat m_k := \sum_{j \in \widehat{\mathcal{N}}_k} m^{(j)}_{\widehat{\mathsf{g}}_j^{-1}(k)},
\end{align}
where the global interface dof $k = 1,\ldots, \text{dim}(\widehat\Lambda)$
is shared by subdomains $\widehat{N}_k$ and corresponds to 
the local interface dof $\ell = \widehat{\mathsf{g}}_i^{-1}(k)$ of subdomain $i$.

If for all global interface dofs $k$ all subdomain impedance values
$\{ m^{(i)}_{\widehat{\mathsf{g}}_i^{-1}(k)} \}_{i \in \widehat{\mathcal{N}}_k}$
take the same value independently of the subdomain index $i$, but only depending on $k$,
then,
like in \eqref{eq:ElambdaMult},
$R_\Lambda E_\Lambda = E_m$, where $E_m$ is the multiplicity projector from Sect.~\ref{sect:globs},
which appears in Loisel's method \cite{Loisel:2013a} (therein denoted by $K$).
The exchange operator $\mathcal{X} = 2 R_\Lambda E_\Lambda - I$ takes the simple form
\eqref{eq:XlambdaMult}.
%
%
Indeed, Loisel \cite{Loisel:2013a} assumes a diagonal impedance operator with identical values.
Moreover, an inspection of \cite[Sect.~4]{GanderSantugini:2016a} reveals that the \emph{complete communication method}
suggested by Gander and Santugini follows the same principle. 
Summarizing, Sect.~\ref{sect:interfaceExchange}
can be viewed as a generalization of the method proposed by Claeys \cite{Claeys:2021a}
and the discrete methods from \cite{Loisel:2013a} and \cite[Sect.~4]{GanderSantugini:2016a}.

\subsection{An exceptional interface exchange operator}

While certainly not of any practical value,
the following proposition marks a theoretical corner case
where the interface exchange operator is chosen in a way that leads to instantaneous convergence.

\begin{proposition}
\label{prop:algebrConvOneStep}
  Let \ref{ass:A1} hold and assume moreover that
	\begin{enumerate}
	\item $T = I$, i.e., $\Lambda = U$ and in particular $\range(T) = \Lambda$,
	\item $A^\top = A$,
	\item $A$ has a bounded inverse,
	\item we choose $M = A$, $\alpha = 1$.
	\end{enumerate}
	Then the only possible interface exchange operator $\mathcal{X}$ fulfilling \ref{ass:A2}
	and \ref{ass:A4}, i.e., $\mathcal{X}^\top A \mathcal{X} = A$,
	is given by $\mathcal{X} = 2 R \widehat A^{-1} R^\top A - I$.
	With such a choice of $\mathcal{X}$,
	the undamped Robin-Schwarz iteration~\eqref{eq:RobinSchwarzRichardon} (with $\beta = 1$) fulfills
	\[
	  u^{(1)} = R \widehat u.
	\]
\end{proposition}
\begin{remark}
  One may say that under the stated assumptions, the iteration \emph{converges after the first step}.
	More precisely, the proposition states the \emph{algebraic property} that the first iterate already reproduces to the solution,
	without the need of any estimate on the \emph{norm} of the error.
	This property does not come as a surprise, since the application of $\mathcal{X}$ involves $\widehat A^{-1}$, i.e., the solution of the original problem.
\end{remark}
\begin{proof}[Proof of Proposition~\ref{prop:algebrConvOneStep}]
  Firstly, the choice $\mathcal{X} = 2 R \widehat A^{-1} R^\top A - I$ obviously fulfills
	\ref{ass:A2} and $\mathcal{X}^\top A \mathcal{X} = A$.
	Secondly, assume that \ref{ass:A2} and $\mathcal{X}^\top A \mathcal{X} = A$ hold. Then, due to Lemma~\ref{lem:X},
	\begin{enumerate}
	\item[(i)] $\tfrac{1}{2}(I + \mathcal{X})$ is a projection,
	\item[(ii)] $\ker(I - \mathcal{X}) = \range(I + \mathcal{X}) = \range(R)$.
	\end{enumerate}
	Because of that, $\mathcal{X} R = R$.
	Applying $R^\top$ to the identity $\mathcal{X}^\top A \mathcal{X} = A$ therefore yields 
	\[
	  R^\top A (I + \mathcal{X}) = 2 R^\top A.
	\]
	For given $\lambda \in \Lambda$, due to (ii), $(I + \mathcal{X}) \lambda = R \widehat v$ for some $\widehat v \in \widehat U$, so
	\[
	  \underbrace{R^\top A R}_{\widehat A} \widehat v = 2 R^\top A \lambda.
	\]
	Therefore, $\widehat v = 2 \widehat A^{-1} R^\top A \lambda$, from which we deduce $\mathcal{X} = 2 R \widehat A^{-1} R^\top A - I$.
	
  Together with the other stated assumptions, we find that \ref{ass:A3} holds.
	Moreover, since $T = I$, $\alpha = 1$, $M = A$, and $A$ has a bounded inverse,
	also $(A + \alpha T^\top M A) = 2A$ has a bounded inverse, so \ref{ass:A5} holds, and since $M = A$, \ref{ass:A4} holds.
	Proposition~\ref{prop:scattering} and a straightforward calculation shows that $S = 0$ and $d = \mathcal{X}^\top f$, from which we deduce that
	\[
	  \lambda^{(1)} = \mathcal{X}^\top f, \qquad u^{(1)} = \tfrac{1}{2} A^{-1} (f + \lambda^{(1)}).
	\]
	Insertion of the formula for $\mathcal{X}$ and substitution yields
	\[
	  u^{(1)} = \tfrac{1}{2} A^{-1} (f + 2 A R \widehat A^{-1} R^\top f - f) = R \widehat A^{-1} \underbrace{ R^\top f}_{\widehat f}
		= R \widehat u. \qedhere
	\]
\end{proof}

\section{Related formulations}
\label{sect:relatedFormulations}

In this section, a couple of formulations have been collected from the literature that are related to
some of the formulations above, and they are displayed using the same compact notation.

\subsection{A three-field formulation}

A different way of reformulating the subdomain flux formulation~\eqref{eq:subdflux}
in terms of equations is the \emph{three-field domain decomposition method}
introduced by Brezzi and Marini in \cite{BrezziMarini:DD6}, see also \cite{QuarteroniValli:Book}.
Let Assumption~\ref{ass:A1} and Assumption~\ref{ass:B1} hold.
By Proposition~\ref{prop:rangeRbyRLambda},
\[
  u \in \range(R) \quad \Longleftrightarrow \quad T u = R_\Lambda \varphi
\]
for some $\varphi \in \widehat\Lambda$. Following the proof of Lemma~\ref{lem:RLambda},
we can write $\range(R) = \ker(T) \oplus \mathcal{W}$ with $\mathcal{W} = T^\dag(\range(R_\Lambda))$, where $T^\dag$ is some not necessarily bounded right-inverse of $T$
such that $T T^\dag = I$.
Therefore,
$\ker(R^\top) = \range(R)^\circ = \ker(T)^\circ \cap \mathcal{W}^\circ = \overline{\range(T^\top)} \cap \mathcal{W}^\circ$ and so
\begin{align*}
  \ker(R^\top) & = \overline{\{ T^\top \tau \colon \tau \in \Lambda^*,\ \langle T^\top \tau, T^\dag R_\Lambda \psi \rangle = 0 \quad \forall \psi \in \widehat\Lambda \}}
	= \overline{\{ T^\top \tau \colon \tau \in \Lambda^*,\ R_\Lambda^\top \tau = 0 \}}.
\end{align*}
Collecting the equations yields the three-field formulation
\begin{align}
\label{eq:threefield}
  \text{find } (u, \varphi, \tau) \in U \times \widehat\Lambda \times \Lambda^* \colon
	\begin{bmatrix}
	  A & 0 & - T^\top \\
		0 & 0 & R_\Lambda^\top \\
		- T & R_\Lambda & 0 
	\end{bmatrix}
	\begin{bmatrix} u \\ \varphi \\ \tau \end{bmatrix}
	= \begin{bmatrix} f \\ 0 \\ 0 \end{bmatrix}.
\end{align}
Similarly to Theorem~\ref{thm:traceflux}, we have the following result:

\begin{proposition}
  Let Assumptions~\ref{ass:A1} and~\ref{ass:B1} hold. Then
	\begin{enumerate}
	\item[(i)] If $(u, \varphi, \tau)$ solves \eqref{eq:threefield} then $u = R \widehat u$ where $\widehat u$ is the unique solution of \eqref{eq:global}.
	\item[(ii)] If $\widehat u$ solves \eqref{eq:global} and if either 
	  (a)~all spaces are finite-dimensional, or (b)~$\range(T) = \Lambda$ or (c)~$A R \widehat u - f \in \range(T^\top)$,
		then there exists $\varphi \in \widehat\Lambda$ and $\tau \in \Lambda^*$ such that $(R \widehat u, \varphi, \tau)$ solves \eqref{eq:threefield}.
  \item[(iii)] In cases~(a) and~(b), there exists a bounded solution operator for \eqref{eq:threefield}.
	\end{enumerate}
\end{proposition}

\subsection{A hybridized DG-like formulation}
\label{sect:HDG}

In this section, a technique is reviewed and generalized that was originally proposed in 
\cite{MonkSinwelSchoeberl:2010a} for a Raviart-Thomas discretization of the Helmholtz equation
(which is shown to be a transformed version of the ultra-weak variational formulation of \cite{CessenantDespres:1998a})
and further discussed in \cite{HuberPechsteinSchoeberl:DD20,Huber:PhD} for Maxwell's equations.
As a difference to the original works, the technique is applied on the subdomain rather than on the element level.

Suppose that Assumption~\ref{ass:A1} holds
and assume further that $T$ and $\mathcal{X}$ are based on an admissible facet system $\mathcal{F}$
where all interior facets are bilateral (cf.\ Sect.~\ref{sect:facetSystems}),
such that Assumption~\ref{ass:A2} holds as well.
The special structure allows the definition of a \emph{single facet space} $U_\mathcal{F} := \productspace_{F \in \mathcal{F}} U_F$
and of the \emph{one-sided, signed jump operator} $J_\mathcal{F} \colon \Lambda \to U_\mathcal{F}$, given by
\begin{align}
\label{eq:oneSideJump}
  (J_\mathcal{F} \tau)_F = \begin{cases}
	  \tau_{iF} - \tau_{jF}    & \text{for bilateral facets } F \text{ with } \mathcal{N}_F = \{i, j\}, i > j, \\
		2 \tau_{iF}              & \text{for exterior Dirichlet facets } F \text{ with } \mathcal{N}_F = \{ i \}, \\
		0                        & \text{for exterior auxiliary facets } F \text{ with } \mathcal{N}_F = \{ i \},
	\end{cases}
	\qquad \text{for } \tau \in \Lambda
\end{align}
(the factor $2$ for the Dirichlet facets may be spared).
In addition, we can define the \emph{distribution} operator
$D_\mathcal{F} \colon U_\mathcal{F} \to \Lambda$ by
\begin{align}
  (D_\mathcal{F} u_\mathcal{F})_{iF} = \begin{cases}
	  u_F & \text{for bilateral facets } F \in \mathcal{F}_i\,,\\
		0   & \text{for exterior Dirichlet facets } F \in \mathcal{F}_i\,,\\
		u_F & \text{for exterior auxiliary facets } F \in \mathcal{F}_i\,,
	\end{cases}
	\qquad \text{for } u_\mathcal{F} \in U_\mathcal{F}\,.
\end{align}

\begin{proposition}
\label{prop:JDX}
  In addition to the above, let Assumptions~\ref{ass:A3}--\ref{ass:A4} hold. Then
	\begin{itemize}
	\item $\ker(J_\mathcal{F}) = \range(D_\mathcal{F}) = \range(I + \mathcal{X}) = \ker(I - \mathcal{X})$,
	\item $\range(J_\mathcal{F}^\top) = \ker(D_\mathcal{F}^\top) = \ker(I + \mathcal{X}^\top) = \range(I - \mathcal{X}^\top)$, and
	\item the operator $D_\mathcal{F}^\top M D_\mathcal{F}$ is invertible.
	\end{itemize}
\end{proposition}

Starting with the interface impedance trace formulation~\eqref{eq:waveRobinGeneral}
and performing the change of variables
\begin{align}
\label{eq:HDGchangeVar}
  \lambda^+ := J_\mathcal{F}^\top v_\mathcal{F} + \alpha M D_\mathcal{F} u_\mathcal{F},
  \qquad
  \text{for } u_\mathcal{F} \in U_\mathcal{F},\ v_\mathcal{F} \in U_\mathcal{F}^*,
\end{align}
one obtains
\begin{align*}
  (A + \alpha T^\top M T) u - T^\top J_\mathcal{F}^\top v_\mathcal{F} - \alpha T^\top M D_\mathcal{F} v_\mathcal{F} & = f,\\
	-2 \alpha M T u + 2 M D_\mathcal{F} u_\mathcal{F} & = 0.
\end{align*}
Multiplying the last line by $\tfrac{1}{2} J_\mathcal{F} M^{-1}$ and another time by $\tfrac{1}{2} D_\mathcal{F}^\top$,
we arrive at the structurally symmetric system
\begin{align}
\label{eq:HDG}
  \begin{bmatrix}
    (A + \alpha T^\top M T) & - T^\top J_\mathcal{F}^\top & - \alpha T^\top M D_\mathcal{F} \\
    - J_\mathcal{F} T & 0 & 0 \\
    - \alpha D_\mathcal{F}^\top M T & 0 & \alpha D_\mathcal{F}^\top M D_\mathcal{F}
  \end{bmatrix}
  \begin{bmatrix}
    u \\ v_\mathcal{F} \\ u_\mathcal{F}
  \end{bmatrix}
  =
  \begin{bmatrix}
    f \\ 0 \\ 0
  \end{bmatrix}.
\end{align}
Whereas $\lambda^+ \in \Lambda^*$ represents the two (generalized) Robin traces on each facet,
the pair $(u_\mathcal{F}, v_\mathcal{F}) \in U_\mathcal{F} \times U_\mathcal{F}^*$
stands for the Dirichlet and Neumann trace.
The second line of \eqref{eq:HDG} enforces
\[
  J_\mathcal{F} T u = 0,
\]
i.e., the continuity of $u$ across facets. This implies that $T u \in \range(D_\mathcal{F})$.
The third line of \eqref{eq:HDG} can be rewritten as
\[
  \alpha D_\mathcal{F}^\top M (T u - D_\mathcal{F} u_\mathcal{F}) = 0.
\]
Since $T u \in \range(D_\mathcal{F})$ and since $D_\mathcal{F}^\top M D_\mathcal{F}$ is invertible, this implies
\[
  T u = D_\mathcal{F} u_\mathcal{F}\,,
\]
i.e., $u_\mathcal{F}$ is the continuous Dirichlet trace.
Using Proposition~\ref{prop:JDX} one can show that \eqref{eq:HDGchangeVar} is a bijective transformation of variables,
and so Formulation~\eqref{eq:HDG} is equivalent to the interface impedance trace formulation~\eqref{eq:waveRobinGeneral}.
Under Assumption~\ref{ass:A5}, the broken primal variable $u$ can be eliminated from the system.
Preconditioners for the reduced system are discussed in \cite{MonkSinwelSchoeberl:2010a,HuberPechsteinSchoeberl:DD20,Huber:PhD}.

\subsection{The FETI-H formulation}
\label{sect:FETIH}

The FETI-H method \cite{FarhatMacedoTezaur:DD11,FarhatMacedoLesoinne:2000a,FarhatMacedoLesoinneRouxMagoulesDeLaBourdonnaie:2000a}
was originally introduced to overcome the \emph{internal resonance problem} of domain decomposition methods for the Helmholtz equation.
To get the main idea, let us assume a \emph{bilateral} facet system.
As a first ingredient, we need a \emph{subdomain sign pattern} $(\sigma_i)_{i=1}^N$ and a selection of facets $\mathcal{F}^\perp \subset \mathcal{F}$
such that
\begin{enumerate}
\item[(i)] each subdomain $i$ has either a plus sign ($\sigma_i = +1$) or a minus sign ($\sigma_i = -1$),
\item[(ii)] for each facet $F \in \mathcal{F}^\perp$ with $\mathcal{N}_F = \{ i, j \}$ the sign changes, i.e., $\sigma_i = - \sigma_j$, and
\item[(iii)] for each subdomain $i$, the set $\mathcal{F}_i^\perp := \mathcal{F}_i \cap \mathcal{F}^\perp$ is non-empty, i.e.,
  each subdomain has at least one facet where the neighboring subdomain has opposite sign.
\end{enumerate}
These properties can be achieved by constructing a minimal spanning tree for the connectivity graph (with subdomains as nodes and facets as edges),
starting with $+1$ at the root, and alternating the sign when going up the tree.

The second ingredient are modified subdomain operators $\widetilde A_i \colon U_i \to U_i^*$.
For each facet $F \in \mathcal{F}^\perp$, let $M_F \colon U_F \to U_F^*$ be a fixed impedance operator and define
\begin{align}
\label{eq:FETIHAaug}
  \widetilde A_i := A_i + \complexi \sigma_i \sum_{F \in \mathcal{F}_i^\perp} T_{iF}^\top M_F T_{iF}\,,
\end{align}
as well as $\widetilde A := \diag(\widetilde A_i)_{i=1}^N \colon U \to U^*$.

\begin{proposition}
\label{prop:assemblingFETIH}
  The modified subdomain operators $(\widetilde A_i)_{i=1}^N$ satisfy the assembling property
	\[
	  \widehat A = \sum_{i=1}^N R_i^\top \widetilde A_i R_i\,.
	\]
\end{proposition}
\begin{proof}
  Expanding the definition of $\widetilde A_i$ one finds that the terms on the facets $F \in \mathcal{F}^\perp$ cancel due to the opposite signs.
\end{proof}

Apparently, if $A_i$ is loss-free (i.e., if $A_i = A_{i,0} - A_{i,2}$ for real-valued and non-negative operators $A_{i,0}$, $A_{i,2}$, cf.\ Assumption~\ref{ass:A7}),
then $\widetilde A_i$ corresponds to a more or less classical Robin problem. More details will be discussed below.
Using the jump operator $B := J_\mathcal{F} T \colon U \to U_\mathcal{F}$ with $J_\mathcal{F}$ from \eqref{eq:oneSideJump},
one derives the

\medskip

\noindent\fbox{\parbox{\textwidth}{
\emph{FETI-H formulation:}
\begin{align}
\label{eq:FETIHformulation}
  \text{find } (u, \lambda_\mathcal{F}) \in U \times U_\mathcal{F}^* \colon \qquad
	\begin{bmatrix} \widetilde A & B^\top \\ B & 0 \end{bmatrix}
	\begin{bmatrix} u \\ \lambda_\mathcal{F} \end{bmatrix} = \begin{bmatrix} f \\ 0 \end{bmatrix}
\end{align}}}

\medskip

Compared to Formulation~\eqref{eq:waveRobinGeneral} (and recalling that we assume bilateral facets) there is only one set of Lagrange parameters per facet,
which is a Robin-type trace. The second line of~\eqref{eq:FETIHformulation} still couples the Dirichlet traces
whereas~\eqref{eq:waveRobinGeneral} couples the two Robin traces.
In this light, \eqref{eq:waveRobinGeneral} is a Robin-Robin scheme and \eqref{eq:FETIHformulation} a Dirichlet-Robin scheme.

\begin{proposition}
  Under Assumptions~\ref{ass:A1}--\ref{ass:A2}, the following statements hold.
	\begin{enumerate}
	\item[(i)] If $(u, \lambda_\mathcal{F})$ solves \eqref{eq:FETIHformulation} then $u = R \widehat u$ where $\widehat u$ is the unique solution of \eqref{eq:global}.
	\item[(ii)] If $\widehat u$ solves \eqref{eq:global} and, in addition, either
	  \begin{enumerate}
		\item[(a)] all spaces are finite-dimensional, or
		\item[(b)] $\range(T) = \Lambda$, or
		\item[(c)] $\overline{\range(T)} = \Lambda$ and $A R \widehat u - f \in \range(T^\top)$,
		\end{enumerate}
    then there exists $\lambda_\mathcal{F} \in U_\mathcal{F}^*$ such that $(R \widehat u, \lambda_\mathcal{F})$ solves \eqref{eq:FETIHformulation}.
		In cases~(b) and (c), $\lambda_\mathcal{F}$ is guaranteed to be unique,
		whereas in the finite-dimensional case (a), $\lambda_\mathcal{F}$ is only unique up to an element from $\ker(B^\top)$,
		which is related to $\mathcal{Z}$, see Proposition~\ref{prop:uniqueTau}.
	\item[(iii)]
			In cases~(a) and (b), there exists a bounded linear solution operator for \eqref{eq:FETIHformulation}.
	\end{enumerate}
\end{proposition}
\begin{proof}
  (i) Assume that $(u, \lambda_\mathcal{F})$ solves \eqref{eq:FETIHformulation}.
	Then $J_\mathcal{F}^\top \lambda_\mathcal{F} \in \range(J_\mathcal{F}^\top) = \ker(I + \mathcal{X}^\top)$,
	and so by Lemma~\ref{lem:kerRT},
	\begin{align}
	\label{eq:BtopKerRTop}
	  \range(B^\top) = \range(T^\top J_\mathcal{F}^\top) \subset \ker(R^\top), \qquad \overline{\range(B^\top)} = \ker(R^\top).
	\end{align}
	Moreover, from Proposition~\ref{prop:JDX} and Assumption~\ref{ass:A2} we obtain
	$\ker(B) = \ker(J_\mathcal{F} T) = \range(R)$,
	which is why there exists $\widehat u$ such that
	$R \widehat u = u$.
	Multiplying the first line of \eqref{eq:FETIHformulation} by $R^T$ from the left, we obtain
	(using Proposition~\ref{prop:assemblingFETIH} and \eqref{eq:BtopKerRTop})
	\[
	  \underbrace{(R^\top \widehat A R)}_{= \widehat A} \widehat u + \underbrace{ R^T B^T \lambda_\mathcal{F} }_{= 0 }
		= \underbrace{ R^\top f }_{ = \widehat f}\,.
	\]
	(ii) Under the stated assumptions, there exists a solution $(u, \tau)$ of the interface flux formulation~\eqref{eq:traceflux}.
	Since $\tau \in \ker(I + \mathcal{X}^\top) = \range(J_\mathcal{F}^\top)$, there exists $\tau_\mathcal{F} \in U_\mathcal{F}^*$ with $\tau = J_\mathcal{F}^\top \tau_\mathcal{F}$
	and so
	\[
	  A u + T^\top J_F^\top \tau_\mathcal{F} = f
	\]
	with $u = R \widehat u$.
	We define $\lambda_\mathcal{F} \in U_\mathcal{F}^*$ by
	\[
	  \lambda_F := \begin{cases}
		  \tau_{iF} - \complexi \sigma_i M_F T_{iF} R_i \widehat u  & \text{for } F \in \mathcal{F}^\perp \text{ with } \mathcal{N}_F = \{i, j\}, i > j,\\
		  \tau_{iF}                                                 & \text{for } F \in \mathcal{F} \setminus \mathcal{F}^\perp \text{ with } \mathcal{N}_F = \{i, j\}, i > j.\\
			\end{cases}
	\]
	A short computation using \eqref{eq:FETIHAaug} and the definition of $J_\mathcal{F}$ reveals that, indeed,
	$\widetilde A u + T^\top J_\mathcal{F}^\top \lambda_\mathcal{F} = f$.
	The rest of the proof is straightforward.
\end{proof}

In the original FETI-H method, system~\eqref{eq:FETIHformulation} is further reduced by forming the Schur complement.
To this end, one has to assume that the modifed subdomain operators $\widetilde A_i$ have bounded inverses.
This assumption is similar to Assumption~\ref{ass:A5} in its nature, and some tools are provided in Appendix~\ref{apx:invertibility}.
However, it becomes apparent that if the original operator $A_i$ has losses, i.e., $A_i = A_{i,0} + \complexi A_{i,1} - A_{i,2}$
with non-trivial $A_{i,1}$, and if $\sigma_i = -1$, then there is a mismatch of signs in the loss terms and the theory breaks down.
As a matter of fact, the FETI-H method was originally proposed for loss-free problems
(see e.g., \cite[Sect.~2.1]{FarhatMacedoLesoinne:2000a} where the system matrices $\mathbf{K}^s - k^2 \mathbf{M}^s$ are real-valued,
corresponding to the boundary conditions in \cite[Sect.~2.2, Eqn.~(11)]{FarhatMacedoLesoinne:2000a}
of Dirichlet and Neumann type).

Under the stated assumption, the resulting Schur complement system is
\begin{align}
  \text{find } \lambda_\mathcal{F} \in U_\mathcal{F}^* \colon \qquad
	\underbrace{ B^\top \widetilde A^{-1} B }_{=:F} \lambda_\mathcal{F} = \underbrace{ B^\top \widetilde{A}^{-1} f }_{=:d}.
\end{align}
In the original FETI-H method, this equation is solved by a Krylov method,
including a projection such that the residual is orthogonal to precomputed interface modes based on plane waves, see \cite[Sect.~4.1]{FarhatMacedoLesoinne:2000a}.

\subsection{Multi-trace formulations}

This section deals with formulations of multi-trace type, related to
\cite{ClaeysDoleanGander:2019a,ClaeysHiptmair:2012a,ClaeysHiptmairJerezHanckes:2013a,
  HiptmairJerezHanckes:2012a,HiptmairJerezHanckesLeePeng:DD21,LeePeng:Book,LeeVouvakisLee:2005a,PengLee:2010a}.
The involved variables are \emph{two pairs} per interface,
namely the Dirichlet and Neumann trace on either side.

Let Assumptions~\ref{ass:A1} and~\ref{ass:A2} hold
and let the interface flux formulation \eqref{eq:traceflux} be the starting point.
Recall from Theorem~\ref{thm:traceflux} that if we either have finite dimensions, $\range(T) = \Lambda$, or regularity,
then \eqref{eq:traceflux} is equivalent to the original formulation~\eqref{eq:global}.
Suppose, in addition, that Assumption~\ref{ass:A3} holds,
such that Lemma~\ref{lem:Robin} allows us to rewrite \eqref{eq:traceflux} as
\begin{align}
\label{eq:MT:facetFluxMatrix}
  \text{find } (u, \tau) \in U \times \Lambda^* \colon \quad
  \begin{bmatrix}
    A & - T^\top \\
    \alpha M (I - \mathcal{X}) T & (I + \mathcal{X}^\top)
  \end{bmatrix}
  \begin{bmatrix} u \\ \tau \end{bmatrix}
  \ = \
  \begin{bmatrix} f \\ 0 \end{bmatrix}.
\end{align}
Separating the terms involving $\mathcal{X}$, we obtain
\begin{align}
\label{eq:MT:facetFluxSeparated}
  \bigg(\underbrace{\begin{bmatrix} A & - T^\top \\ \alpha M T & I \end{bmatrix}}_{
        \displaystyle =: \mathcal{A}}
  + \underbrace{\begin{bmatrix} 0 & 0 \\ - \alpha M \mathcal{X} T & \mathcal{X}^\top \end{bmatrix}}_{
        \displaystyle =: \mathcal{C}} \bigg)
  \begin{bmatrix} u \\ \tau \end{bmatrix}
  \ = \
  \begin{bmatrix} f \\ 0 \end{bmatrix}.
\end{align}
Next, let us assume that the augmented operator
$\widetilde A: = A + \alpha T^\top M T$ is invertible (Assumption~\ref{ass:A5}).
Then operator $\mathcal{A}$ is invertible: a block factorization shows that
\[
  \mathcal{A}^{-1} = \begin{bmatrix} I & 0 \\ - \alpha M T & I \end{bmatrix}
  \begin{bmatrix} \widetilde A^{-1} & 0 \\ 0 & I \end{bmatrix}
  \begin{bmatrix} I & T^\top \\ 0 & I \end{bmatrix}.
\]
We multiply \eqref{eq:MT:facetFluxSeparated} by $\mathcal{A}^{-1}$, like applying a preconditoner:
\begin{align}
  \big( I + \mathcal{A}^{-1} \mathcal{C} \big) \begin{bmatrix} u \\ \tau \end{bmatrix}
  \ = \ \mathcal{A}^{-1} \begin{bmatrix} f \\ 0 \end{bmatrix}.
\end{align}
Since $\mathcal{C}$ depends, besides $\tau$, only on $T u$ we can introduce
\begin{align}
   \gamma := T u \in \Lambda
\end{align}
as a new variable and multiply the first equation by $T$. This yields
\[
  \left( I + \begin{bmatrix} T & 0 \\ 0 & I \end{bmatrix} \mathcal{A}^{-1}
  \begin{bmatrix} 0 & 0 \\ - \alpha M \mathcal{X} & \mathcal{X}^\top \end{bmatrix}
  \right) \begin{bmatrix} \gamma \\ \tau \end{bmatrix}
  = \begin{bmatrix} T & 0 \\ 0 & I \end{bmatrix} \mathcal{A}^{-1} \begin{bmatrix} f \\ 0 \end{bmatrix}
\]
or, more explicitly,
\begin{align*}
  \left( I + \begin{bmatrix} T & 0 \\ - \alpha M T & I \end{bmatrix}
    \begin{bmatrix} \widetilde A^{-1} & 0 \\ 0 & I \end{bmatrix}
    \begin{bmatrix} - \alpha T^\top M & T^\top \\ - \alpha M & I \end{bmatrix}
    \begin{bmatrix} \mathcal{X} & 0 \\ 0 & \mathcal{X}^\top \end{bmatrix} \right)
  \begin{bmatrix} \gamma \\ \tau \end{bmatrix} 
  = \begin{bmatrix} T & 0 \\ - \alpha M T & I \end{bmatrix}
    \begin{bmatrix} \widetilde A^{-1} f \\ 0 \end{bmatrix}.
\end{align*}
The unknowns are now $(\gamma, \tau) \in \Lambda \times \Lambda^*$,
i.e., if $\Lambda$ is based on a facet system, we have two pairs of unknowns per facet.
The original solution $u$ can be obtained by solving
\[
  \widetilde A u = f + T^\top (\tau + \alpha M \gamma)
\]
separately on each subdomain.
Rearranging the above formulation reveals more structure:
\begin{align}
\label{eq:MT:PXdDef}
  \Bigg( I - 
	\underbrace{\begin{bmatrix} T & 0 \\ - \alpha M T & I \end{bmatrix}
    \begin{bmatrix} \widetilde A^{-1} & 0 \\ 0 & I \end{bmatrix}
    \begin{bmatrix} \alpha T^\top M & T^\top \\ \alpha M & I \end{bmatrix}
	}_{=: \mathsf{P}}
  \underbrace{ \begin{bmatrix} \mathcal{X} & 0 \\ 0 & -\mathcal{X}^\top \end{bmatrix} }_{=: \mathsf{X}}
	\Bigg)
  \begin{bmatrix} \gamma \\ \tau \end{bmatrix} 
  = \underbrace{\begin{bmatrix} I \\ - \alpha M \end{bmatrix}
         T \widetilde A^{-1} f}_{=: \mathsf{d}}.
\end{align}
As will be shown below, the operator
$\mathsf{P} \colon \Lambda \times \Lambda^* \to \Lambda \times \Lambda^*$
is an analog of the Cald\'eron projector (cf.\ e.g.\ \cite{McLean:Book,Steinbach:Book2008,ClaeysHiptmairJerezHanckes:2013a}).
Note that if $M$ is block-diagonal, then so is $\mathsf{P}$.
The operator $\mathsf{X} \colon \Lambda \times \Lambda^* \to \Lambda \times \Lambda^*$
exchanges the (candidates for the) Cauchy traces and flips the sign of the interface fluxes,
and it fulfills $\mathsf{X}^2 = I$.
We summarize:

\medskip

\noindent%
\fbox{\parbox{\textwidth}{
\textsl{Local multi-trace formulation I:}
\begin{align}
\label{eq:MT:localMTFormulation}
  \text{find } \begin{bmatrix} \gamma \\ \tau \end{bmatrix} \in \Lambda \times \Lambda^* \colon
	\qquad (I - \mathsf{P} \mathsf{X}) \begin{bmatrix} \gamma \\ \tau \end{bmatrix}
	= \mathsf{d},
\end{align}
with $\mathsf{P}$, $\mathsf{X}$, and $\mathsf{d}$ defined in \eqref{eq:MT:PXdDef}.
}}

\medskip

The above formulation is of the same form as the one in \cite[Sect.~3.1]{ClaeysDoleanGander:2019a}.

\begin{proposition}
\label{prop:MT:equiv}
  Let Assumptions~\ref{ass:A1}--\ref{ass:A3} and~\ref{ass:A5} be fulfilled, then the following statements hold:
	\begin{enumerate}
	\item[(i)] If $(\gamma, \tau) \in \Lambda \times \Lambda^*$
	  is a solution of the local multi-trace formulation~\eqref{eq:MT:localMTFormulation},
		then there exists $u \in U$ such that
		\[
		  \gamma = T u, \qquad A u - T^\top \tau = f,
		\]
		and $(u, \tau)$ is a solution of the interface flux formulation~\eqref{eq:traceflux}.
	\item[(ii)] Conversely, if $(u, \tau) \in U \times \Lambda^*$
	  is a solution of the interface flux formulation~\eqref{eq:traceflux},
		then $(T u, \tau)$ solves the local multi-trace formulation~\eqref{eq:MT:localMTFormulation}.
	\end{enumerate}
\end{proposition}
Part~(ii) has already been shown when deriving the local multi-trace formulation. 
The proof of Part~(i) is given below and requires a couple of results on the properties of $\mathsf{P}$ and $\mathsf{d}$.

\begin{lemma}
\label{lem:MT:Pproj}
  $\begin{bmatrix} \alpha M & I \end{bmatrix} \mathsf{P} = \begin{bmatrix} \alpha M & I \end{bmatrix}$
  and $\mathsf{P}^2 = \mathsf{P}$.
\end{lemma}
\begin{proof}
We can write $\mathsf{P}$ as
\begin{align}
\label{eq:MT:Pform}
  \mathsf{P} = 
  \begin{bmatrix} T & 0 \\ - \alpha M T & I \end{bmatrix}
    \begin{bmatrix} \widetilde A^{-1} & 0 \\ 0 & I \end{bmatrix}
    \begin{bmatrix} T^\top \\ I \end{bmatrix}
    \begin{bmatrix} \alpha M & I \end{bmatrix}.
\end{align}
A short calculation shows
\begin{align*}
  \begin{bmatrix} \alpha M & I \end{bmatrix} \mathsf{P}
  = \underbrace{\begin{bmatrix} \alpha M & I \end{bmatrix}
    \begin{bmatrix} T & 0 \\ - \alpha M T & I \end{bmatrix}}_{
      \begin{bmatrix} 0 & I \end{bmatrix}
    }
    \begin{bmatrix} \widetilde A^{-1} & 0 \\ 0 & I \end{bmatrix}
    \begin{bmatrix} T^\top \\ I \end{bmatrix}
    \begin{bmatrix} \alpha M & I \end{bmatrix}
  =  \begin{bmatrix} \alpha M & I \end{bmatrix}.
\end{align*}
With the above representation of $\mathsf{P}$, we see immediately that $\mathsf{P}^2 = \mathsf{P}$.
\end{proof}


\begin{definition}
  Given $g \in U^*$, we define the linear manifold
  \[
    \mathcal{CP}(g) =: 
    \left\{ \begin{bmatrix} \gamma \\ \sigma \end{bmatrix} \in \Lambda \times \Lambda^*
    \colon \exists u \in U \colon \gamma = T u, A u - T^\top \sigma = g \right\}
    \subset \Lambda \times \Lambda^*,
  \]
	which is the set of \emph{Cauchy pairs} for the linear equation involving $A$ and the right-hand side $g$.
  Apparently, $\mathcal{CP}(g_1) + \mathcal{CP}(g_2) = \mathcal{CP}(g_1 + g_2)$.
\end{definition}

The following lemma shows that the Calder\'on projector $\mathsf{P}$ maps to the space of Cauchy pairs for the homogeneous equation.

\begin{lemma}
\label{lem:MT:rangeP}
  $\range(\mathsf{P}) = \left\{ \begin{bmatrix} \gamma \\ \sigma \end{bmatrix} \in \Lambda \times \Lambda^* \colon
     \exists u \in U \colon \gamma = T u,\ A u - T^\top \sigma = 0 \right\}
   = \mathcal{CP}(0)$.
\end{lemma}
\begin{proof}
  ``$\subseteq$'': Let $(\widetilde\gamma, \widetilde\sigma) \in \Lambda \times \Lambda^*$ be arbitrary but fixed and set
  \[
    \begin{bmatrix} \gamma \\ \sigma \end{bmatrix}
    = \mathsf{P} \begin{bmatrix} \widetilde\gamma \\ \widetilde\sigma \end{bmatrix}
    = \begin{bmatrix}
        T \widetilde A^{-1} T^\top (\alpha M \widetilde\gamma
            + \widetilde\sigma) \\[0.3ex]
        (I - \alpha M T \widetilde A^{-1} T^\top)
          (\alpha M \widetilde\gamma + \widetilde\sigma)
      \end{bmatrix}.
  \]
  We define
  \[
	  u := \widetilde A^{-1} T^\top (\alpha M \widetilde\gamma + \widetilde\sigma).
  \]
  Following the definition of $\mathsf{P}$,
  we find that indeed $T u = \gamma$. Moreover,
  \begin{align*}
   A u =
   (\widetilde A - \alpha T^\top M T) u
   & = (I - \alpha T^\top M T \widetilde A^{-1}) T^\top
       (\alpha M \widetilde\gamma + \widetilde \sigma)\\
   & = T^\top (I - \alpha M T \widetilde A^{-1} T^\top)
       (\alpha M \widetilde\gamma + \widetilde \sigma)
     = T^\top \sigma.
  \end{align*}
  ``$\supseteq$'': Assume that we have $\gamma \in \Lambda$, $\sigma \in \Lambda^*$ and $u \in U$
  with $\gamma = T u$ and $A u - T^\top \sigma = 0$. Then
  \[
    \widetilde A^{-1} T^\top (\alpha M \gamma + \sigma)
    = \widetilde A^{-1} (T^\top \alpha M T u + T^\top \sigma)
    = \widetilde A^{-1} \big[ (\widetilde A - A) u + T^\top \sigma \big]
    = u.
  \]
  Therefore
  \[
    \mathsf{P} \begin{bmatrix} \gamma \\ \sigma \end{bmatrix}
    = \begin{bmatrix} T \widetilde A^{-1} T^\top (\alpha M \gamma + \sigma) \\[0.3ex]
         (I - \alpha M T \widetilde A^{-1} T^\top)
           (\alpha M \gamma + \sigma) \end{bmatrix}
    = \begin{bmatrix} T u \\
         \alpha M \gamma + \sigma - \alpha M T u \end{bmatrix}
    = \begin{bmatrix} \gamma \\ \sigma \end{bmatrix},
  \]
  which shows that $(\gamma, \sigma)$ is in the range of $\mathsf{P}$.
\end{proof}

\begin{lemma}
  $\ker(\mathsf{P}) = \left\{ \begin{bmatrix} \gamma \\ \tau \end{bmatrix} \in \Lambda \times \Lambda^* \colon \alpha M \gamma + \tau = 0 \right\}$.
\end{lemma}
\begin{proof}
  ``$\supseteq$'' follows from \eqref{eq:MT:Pform}.\\
	``$\subseteq$'': Suppose that $\mathsf{P} \mathsf{v} = 0$,
	then $\begin{bmatrix} \alpha M & I \end{bmatrix} \mathsf{P} \mathsf{v} = 0$.
	Lemma~\ref{lem:MT:Pproj} implies that $\begin{bmatrix} \alpha M & I \end{bmatrix} \mathsf{v} = 0$.
\end{proof}

\begin{lemma}
\label{lem:MT:d}
  $\mathsf{d} \in \mathcal{CP}(f)$
  and $\begin{bmatrix} \alpha M & I \end{bmatrix} \mathsf{d} = 0$.
\end{lemma}
\begin{proof}
  Recall that
  $\mathsf{d} = \begin{bmatrix} T \widetilde A^{-1} f \\
          - \alpha M T \widetilde A^{-1} f \end{bmatrix}
    =: \begin{bmatrix} \gamma \\ \sigma \end{bmatrix}$.
  By setting $u := \widetilde A^{-1} f$, we find that $\gamma = T u$ and
  \[
    A u = A \widetilde A^{-1} f = (\widetilde A - T^\top \alpha M T) \widetilde A^{-1} f
        = f + T^\top \sigma.
  \]
  The second relation follows immediately from the definition of $\mathsf{d}$.
\end{proof}

\begin{remark}
  If $M$ is block-diagonal, i.e., $M = \diag(M_i)_{i=1}^N$, then $\mathsf{P} = \diag(\mathsf{P}_i)_{i=1}^N$
	with $\mathsf{P}_i^2 = \mathsf{P}_i$ and
	\begin{align*}
	  \range(\mathsf{P}_i) & = \left\{ \begin{bmatrix} \gamma_i \\ \sigma_i \end{bmatrix} \in \Lambda_i \times \Lambda_i^* \colon
		  \exists u_i \in U_i \colon \gamma_i = T_i u_i\,,\ A_i u_i - T_i^\top \sigma_i = 0 \right\},\\
		\ker(\mathsf{P}_i) & = \left\{ \begin{bmatrix} \gamma_i \\ \sigma_i \end{bmatrix} \in \Lambda_i \times \Lambda_i^* \colon
		  \alpha M_i \gamma_i + \sigma_i = 0 \right\}.
	\end{align*}
	If $U_i = \Lambda_i$ and $\alpha = 1$, and if $M_i$ is the exterior operator w.r.t.\ to $A_i$,
	then $\widetilde A_i^{-1}$ plays the role of the fundamental solution and $\mathsf{P}_i$ is indeed the classical Cald\'eron projector.
\end{remark}

\begin{proof}[Proof of Proposition~\ref{prop:MT:equiv}, Part~(i)]
	Assume that $(\gamma, \tau)$ solves~\eqref{eq:MT:localMTFormulation}, i.e.,
  \[
    \begin{bmatrix} \gamma \\ \tau \end{bmatrix}
    - \mathsf{P} \mathsf{X} \begin{bmatrix} \gamma \\ \tau \end{bmatrix}
    = \mathsf{d}.
  \]
  Due to Lemma~\ref{lem:MT:rangeP} and Lemma~\ref{lem:MT:d},
  \[
    \begin{bmatrix} \gamma \\ \tau \end{bmatrix} \in \mathcal{CP}(f) + \mathcal{CP}(0) = \mathcal{CP}(f),
  \]
  so there exists $u \in U$ with $T u = \gamma$ and $A u - T^\top \tau = f$.
  Next we apply $\begin{bmatrix} \alpha M & I \end{bmatrix}$ to the system:
  \[
    \underbrace{\begin{bmatrix} \alpha M & I \end{bmatrix}
      \begin{bmatrix} \gamma \\ \tau \end{bmatrix}}_{= \alpha M \gamma + \tau}
    - \underbrace{\begin{bmatrix} \alpha M & I \end{bmatrix} \mathsf{P}}_{=
      \begin{bmatrix} \alpha M & I \end{bmatrix}} \mathsf{X}
       \begin{bmatrix} \gamma \\ \tau \end{bmatrix}
    = \underbrace{\begin{bmatrix} \alpha M & I \end{bmatrix} \mathsf{d}}_{=0},
  \]
  where we have used Lemma~\ref{lem:MT:Pproj} and Lemma~\ref{lem:MT:d}.
  Employing the definition of $\mathsf{X}$ we obtain the condition
  $\alpha M (I - \mathcal{X}) \gamma + (I + \mathcal{X}^\top) \tau = 0$,
	which is, due to Lemma~\ref{lem:Robin}, equivalent to
  \[
    (I - \mathcal{X}) \gamma = 0, \qquad (I + \mathcal{X}^\top) \tau = 0.
  \]
  Since $\gamma = T u$,
	we end up with the interface flux formulation~\eqref{eq:traceflux}.
\end{proof}

We have seen that $\gamma = T u$ and $A u - T^\top \tau = f$ can be expressed by
\begin{align}
  (I - \mathsf{P}) \begin{bmatrix} \gamma \\ \tau \end{bmatrix} = \mathsf{d}
\end{align}
and that $(I - \mathcal{X}) \gamma = 0$ and $(I + \mathcal{X}^\top) \tau$ can be written as
\begin{align}
  (I - \mathsf{X}) \begin{bmatrix} \gamma \\ \tau \end{bmatrix} = 0.
\end{align}
Combining the two latter conditions using a complex number $\theta \neq 0$ leads to the following formulation,
which is related to~\cite{HiptmairJerezHanckesLeePeng:DD21}.

\medskip

\noindent%
\fbox{\parbox{\textwidth}{
\textsl{Local multi-trace formulation II:}
\begin{align}
\label{eq:MT:localMTFormulation2}
  \text{find } \begin{bmatrix} \gamma \\ \tau \end{bmatrix} \in \Lambda \times \Lambda^* \colon \qquad
  \Big((I - \mathsf{P}) + \theta (I - \mathsf{X}) \Big) \begin{bmatrix} \gamma \\ \tau \end{bmatrix} = \mathsf{d}.
\end{align}
}}

\begin{proposition}
  Let Assumptions~\ref{ass:A1}--\ref{ass:A3} and~\ref{ass:A5} be fulfilled and let $\theta \neq 0$,
	then the following statements hold:
	\begin{enumerate}
	\item[(i)] If $(\gamma, \tau) \in \Lambda \times \Lambda^*$
	  is a solution of the local multi-trace formulation~\eqref{eq:MT:localMTFormulation2},
		then there exists $u \in U$ such that
		\[
		  \gamma = T u, \qquad A u - T^\top \tau = f,
		\]
		and $(u, \tau)$ is a solution of the interface flux formulation~\eqref{eq:traceflux}.
	\item[(ii)] Conversely, if $(u, \tau) \in U \times \Lambda^*$
	  is a solution of the interface flux formulation~\eqref{eq:traceflux},
		then $(T u, \tau)$ solves the local multi-trace formulation~\eqref{eq:MT:localMTFormulation2}.
	\end{enumerate}
\end{proposition}
\begin{proof}
  (ii) has already been shown.
	For (i), we apply $\begin{bmatrix} \alpha M & I \end{bmatrix}$ to \eqref{eq:MT:localMTFormulation2}
	and use Lemmas~\ref{lem:MT:Pproj} and~\ref{lem:MT:d} to obtain
	\begin{align*}
	  \theta \begin{bmatrix} \alpha M & I \end{bmatrix} (I - \mathsf{X}) \begin{bmatrix} \gamma \\ \tau \end{bmatrix} = 0.
	\end{align*}
	From Lemma~\ref{lem:Robin}, we obtain that $(I - \mathsf{X}) \begin{bmatrix} \gamma \\ \tau \end{bmatrix} = 0$.
	Insertion into \eqref{eq:MT:localMTFormulation2} yields
	$(I - \mathsf{P}) \begin{bmatrix} \gamma \\ \tau \end{bmatrix} = \mathsf{d}$,
	from which we conclude that there exists $u \in U$ with $A u - T^\top \tau = f$ and $\gamma = T u$.
	Summarizing, $(u, \tau)$ fulfills \eqref{eq:traceflux}.
\end{proof}

\begin{remark}
  The multi-trace formulation~\eqref{eq:MT:localMTFormulation} can be derived from~\eqref{eq:MT:localMTFormulation2}
  setting $\theta = -1$ and applying the bijective transformation~$(\gamma, \tau) \mapsto \mathsf{X} (\gamma, \tau)$.
\end{remark}

\begin{remark}
  In the framework of boundary integral methods, the Cald\'eron projector is based on choosing $M$ as the exterior operator to $A$
	(supposing that $\Lambda = \range(T)$).
  Then $\mathsf{P} = \tfrac{1}{2} I + \mathsf{A}$ with a block-operator $\mathsf{A}$ involving boundary integral operators,
	cf.\ \cite{HiptmairJerezHanckes:2012a,ClaeysHiptmairJerezHanckes:2013a}.
	In that case, the choice $\theta = -\tfrac{1}{2}$ leads to
	\[
	  \Big( \mathsf{A} - \frac{1}{2} \mathsf{X} \Big) \begin{bmatrix} \gamma \\ \tau \end{bmatrix} = -\mathsf{d},
	\]
	which is of the same form as in	\cite{HiptmairJerezHanckes:2012a}, \cite[Sect.~6]{ClaeysHiptmairJerezHanckes:2013a}.
\end{remark}

\subsection*{Acknowledgement}

  The author would like to express his thanks to Ortwin Farle, Timo Euler, Sabine Zaglmayr, and Hermann Schneider (Dassault Syst\`emes)
	as well as to Xavier Claeys (UPMC Paris) and Martin Gander (Universit\'e de Gen\`eve)
	for a series of fruitful discussions.
	Some helpful hints were also provided by Patrick Joly (ENSTA Paris), Ivan Graham (University of Bath),
	Sebastian Sch\"ops (TU Darmstadt), Clemens Hofreither (RICAM Linz), and Herbert Egger (that time TU Darmstadt, now JKU Linz).

\begin{appendix}

\section{Proof of Theorem~\ref{thm:Zcharact}}
\label{apx:ProofZCharact}

  Beforehand, observe that
	since none of the operators $T$, $\mathcal{X}$ couples subdomain dofs or trace dofs that correspond to two \emph{different} global dofs,
	we can treat one global interface dof at a time.
	So without loss of generality, we may assume that $\mathcal{D}_\Gamma$ consists of a single global interface dof $k$,
	such that our goal is a proof for $\dim(\mathcal{Z}) = \ell_k$.
	
	Let $\mathcal{E}_k'$ be the edges of a fixed minimal spanning tree of the connectivity graph $\mathcal{C}_k$
	(see Fig.~\ref{fig:graphCycles}, left)
	and let us collect the remaining edges $\mathcal{E}_k \setminus \mathcal{E}_k'$ in a sequence $(e_1\,,\ldots,e_m)$.
	By classical graph theory,
	$\cardinality{\mathcal{E}_k'} = \cardinality{\mathcal{N}_k} - 1$ and
	for each remaining edge $e_i$, $i=1,\ldots,m$,
	there exists an associated cycle $\mathcal{L}_i$ consisting of edges from $\mathcal{E}_k' \cup \{ e_1,\ldots,e_i \}$,
	such that the number of independent cycles in $\mathcal{C}_k$ is given by $\ell_k = m$
	(see Fig.~\ref{fig:graphCycles}, middle).

	As one observes, $\mu \in \Lambda^*$ has two values per facet (one for each subdomain),
	so accordingly, two values per edge of the connectivity graph. Since each subdomain corresponds to a node of the graph,
	we can think of these values as being associated with the endpoints of the edges.
	Under that perspective,
	the operator $T^\top$ sums up the all values per \emph{node}, while the operator $(I + \mathcal{X}^\top)$
	sums up the two values per \emph{edge}.

	For each edge $e_i$, we define an element $\widehat\mu_i \in \mathcal{Z}$
	by putting values $\pm 1$ along the associated cycle $\mathcal{L}_i$ as illustrated in Figure~\ref{fig:graphCycles} (right)
	and zero elsewhere.
	Apparently, $T^\top \widehat\mu_i = 0$ because the two non-zero values associated with each node within the cycle
	have opposite sign. At the same time, the two values associated with each edge within the cycle sum up to zero as well,
	so $(I + \mathcal{X}^\top) \widehat\mu_i = 0$. Altogether, $\widehat\mu_i \in \mathcal{Z}$.
  Moreover,	
	the element $\widehat\mu_i$ is linearly independent from $\{ \widehat \mu_j \}_{j=1}^{i-1}$
	because the latter are not supported on the edge $e_i$. Therefore $\dim(\mathcal{Z}) \ge m$.
	
	In order to see that $\dim(\mathcal{Z}) \le m$, we let $\mu \in \mathcal{Z}$ be arbitrary but fixed.
	Recall the sequence $(e_1,\ldots,e_m)$ of remaining edges and let us start with the edge $e_m$.
	Since $(I - \mathcal{X}^\top)\mu = 0$, the two values of $\mu$ on edge $e_m$ must have opposite sign.
	Therefore, we can find a coefficient $\alpha_m$ such that $\mu_m := \mu - \alpha_m \widehat\mu_m$
	has vanishing values on edge $e_m$.
	We proceed inductively.
	For $i > 1$, suppose that $\mu_i \in \mathcal{Z}$ has vanishing values on all edges $e_i,\ldots,e_m$.
	Then, since $(I - \mathcal{X}^\top)\mu_i = 0$,
	the two values of $\mu_i$ on edge $e_{i-1}$ must have opposite sign.
	So there exists a coefficient $\alpha_{i-1}$ such that $\mu_{i-1} := \mu_i - \alpha_{i-1} \widehat\mu_{i-1}$
	has vanishing values on edge $e_{i-1}$. Since $\widehat\mu_{i-1}$ has vanishing values on all the edges $e_i,\ldots,e_m$,
	the function $\mu_{i-1}$ vanishes on all the edges $e_{i-1},\ldots,e_m$ and $\mu_{i-1} \in \mathcal{Z}$.
	The inductive process stops with $\mu_1 \in \mathcal{Z}$ vanishing entirely on all remaining edges $e_1,\ldots,e_m$,
	and so the only possible non-zero values of $\mu_1$ are located at the edges of the minimal spanning tree.
	This spanning tree, however, must have nodes with just one edge attached. The condition $T^\top \mu_1 = 0$
	implies that the values of $\mu_1$ at these node and the attached edges is zero.
	Using the condition $(I + \mathcal{X}^\top) \mu_1 = 0$
	along the edges allows to show, eventually, that $\mu_1 = 0$. Therefore, $\dim(\mathcal{Z}) = m$
	and $\mathcal{Z} = \text{span}\{ \widehat\mu_1, \ldots, \widehat\mu_m \}$.

\begin{figure}
\begin{center}
  \begin{tikzpicture}
	  [x={(1cm,0cm)}, y={(0.3cm,0.4cm)}, z={(0cm,1cm)}]
		
		\pgftransformscale{0.4}
		
	  \definecolor{black}{rgb}{0.0, 0.0, 0.0}
		
		\begin{scope}[canvas is xz plane at y=0]
			\draw[color=black,fill=black] (0,0) circle (0.1);
			\draw[line width=1,color=black,fill=white] (6,0) circle (0.1);
			\draw[color=black,fill=black] (0,6) circle (0.1);
			\draw[color=black,fill=black] (6,6) circle (0.1);
		\end{scope}

		\begin{scope}[canvas is xz plane at y=6]
			\draw[color=black,fill=black] (0,0) circle (0.1);
			\draw[line width=1,color=black,fill=white] (6,0) circle (0.1);
			\draw[color=black,fill=black] (0,6) circle (0.1);
			\draw[line width=1,color=black,fill=white] (6,6) circle (0.1);
		\end{scope}
		
		\draw[line width=0.75pt,color=black] (0.5, 0.0, 6.0)--(5.5, 0.0, 6.0);
		\draw[line width=0.75pt,color=black] (0.5, 0.0, 0.0)--(5.5, 0.0, 0.0);
		\draw[line width=0.75pt,color=black] (0.0, 0.5, 0.0)--(0.0, 5.5, 0.0);
		\draw[line width=0.75pt,color=black] (0.5, 6.0, 0.0)--(5.5, 6.0, 0.0);
		\draw[line width=0.75pt,color=black] (0.0, 6.0, 0.5)--(0.0, 6.0, 5.5);
		\draw[line width=0.75pt,color=black] (6.0, 0.5, 6.0)--(6.0, 5.5, 6.0);
		\draw[line width=0.75pt,color=black] (0.0, 0.5, 6.0)--(0.0, 5.5, 6.0);
		
  \end{tikzpicture}
	\hspace{0.5cm}
  \begin{tikzpicture}
	  [x={(1cm,0cm)}, y={(0.3cm,0.4cm)}, z={(0cm,1cm)}]
		
		\pgftransformscale{0.4}
		
	  \definecolor{dred}{rgb}{0.9, 0.0, 0.0}
		\definecolor{dgreen}{rgb}{0.0, 0.7, 0.0}
		\definecolor{dblue}{rgb}{0.0, 0.0, 0.8}
		\definecolor{dorange}{rgb}{0.8, 0.3, 0.0}
		\definecolor{dmagenta}{rgb}{0.6, 0.0, 0.6}
		\definecolor{black}{rgb}{0.0, 0.0, 0.0}
		
		\begin{scope}[canvas is xz plane at y=0]
			\draw[color=black,fill=black] (0,0) circle (0.1);
			\draw[color=black,fill=black] (6,0) circle (0.1);
			\draw[color=black,fill=black] (0,6) circle (0.1);
			\draw[color=black,fill=black] (6,6) circle (0.1);
		\end{scope}

		\begin{scope}[canvas is xz plane at y=6]
			\draw[color=black,fill=black] (0,0) circle (0.1);
			\draw[color=black,fill=black] (6,0) circle (0.1);
			\draw[color=black,fill=black] (0,6) circle (0.1);
			\draw[color=black,fill=black] (6,6) circle (0.1);
		\end{scope}
		
		\draw[line width=0.75pt,color=black] (0.5, 0.0, 6.0)--(5.5, 0.0, 6.0);
		\draw[line width=0.75pt,color=black] (0.5, 0.0, 0.0)--(5.5, 0.0, 0.0);
		\draw[line width=0.75pt,color=black] (0.0, 0.5, 0.0)--(0.0, 5.5, 0.0);
		\draw[line width=0.75pt,color=black] (0.5, 6.0, 0.0)--(5.5, 6.0, 0.0);
		\draw[line width=0.75pt,color=black] (0.0, 6.0, 0.5)--(0.0, 6.0, 5.5);
		\draw[line width=0.75pt,color=black] (6.0, 0.5, 6.0)--(6.0, 5.5, 6.0);
		\draw[line width=0.75pt,color=black] (0.0, 0.5, 6.0)--(0.0, 5.5, 6.0);
		
		\draw[line width=1.2pt,color=dblue]    (6, 1, 0)--(6, 5, 0); 
		\node[text=dblue] at (6.8,3,0) {$e_1$};
		\draw[line width=1.2pt,color=dred]     (0, 0, 1)--(0, 0, 5); 
		\node[text=dred] at (-0.6,0,3) {$e_2$};
		\draw[line width=1.2pt,color=dgreen]   (1, 6, 6)--(5, 6, 6); 
		\node[text=dgreen] at (3,7,6) {$e_3$};
		\draw[line width=1.2pt,color=dmagenta] (6, 0, 1)--(6, 0, 5); 
		\node[text=dmagenta] at (5.4,0,3) {$e_4$};
		\draw[line width=1.2pt,color=dorange]  (6, 6, 1)--(6, 6, 5); 
		\node[text=dorange] at (6.6,6,3) {$e_5$};
				
		\begin{scope}[canvas is xy plane at z=0]
		  \draw[->,color=dblue] (4,3) arc (0:320:1);
		\end{scope}

		\begin{scope}[canvas is zy plane at x=0]
		  \draw[->,color=dred] (4,3) arc (0:320:1);
		\end{scope}

		\begin{scope}[canvas is xy plane at z=6]
		  \draw[->,color=dgreen] (4,3) arc (0:320:1);
		\end{scope}

		\begin{scope}[canvas is zx plane at y=0]
		  \draw[->,color=dmagenta] (4,3) arc (0:320:1);
		\end{scope}

		\begin{scope}[canvas is yz plane at x=6]
		  \draw[->,color=dorange] (4,3) arc (0:320:1);
		\end{scope}
  \end{tikzpicture}
	\hspace{0.5cm}
  \begin{tikzpicture}
	  [x={(1cm,0cm)}, y={(0.3cm,0.4cm)}, z={(0cm,1cm)}]
		
		\pgftransformscale{0.4}
		
	  \definecolor{dred}{rgb}{0.9, 0.0, 0.0}
		\definecolor{dgreen}{rgb}{0.0, 0.7, 0.0}
		\definecolor{dblue}{rgb}{0.0, 0.0, 0.8}
		\definecolor{dorange}{rgb}{0.8, 0.3, 0.0}
		\definecolor{dmagenta}{rgb}{0.6, 0.0, 0.6}
		\definecolor{black}{rgb}{0.0, 0.0, 0.0}
		
		\begin{scope}[canvas is xz plane at y=0]
			\draw[color=black,fill=black] (0,0) circle (0.1);
			\draw[color=black,fill=black] (6,0) circle (0.1);
			\draw[color=black,fill=black] (0,6) circle (0.1);
			\draw[color=black,fill=black] (6,6) circle (0.1);
		\end{scope}

		\begin{scope}[canvas is xz plane at y=6]
			\draw[color=black,fill=black] (0,0) circle (0.1);
			\draw[color=black,fill=black] (6,0) circle (0.1);
			\draw[color=black,fill=black] (0,6) circle (0.1);
			\draw[color=black,fill=black] (6,6) circle (0.1);
		\end{scope}
		
		\draw[line width=0.75pt,color=black] (0.5, 0.0, 6.0)--(5.5, 0.0, 6.0);
		\draw[line width=0.75pt,color=black] (0.5, 0.0, 0.0)--(5.5, 0.0, 0.0);
		\draw[line width=0.75pt,color=black] (0.0, 0.5, 0.0)--(0.0, 5.5, 0.0);
		\draw[line width=0.75pt,color=black] (0.5, 6.0, 0.0)--(5.5, 6.0, 0.0);
		\draw[line width=0.75pt,color=black] (0.0, 6.0, 0.5)--(0.0, 6.0, 5.5);
		\draw[line width=0.75pt,color=black] (0.0, 0.5, 6.0)--(0.0, 5.5, 6.0);		
		\draw[line width=0.75pt,color=black] (0.5, 6.0, 6.0)--(5.5, 6.0, 6.0);
		\draw[line width=0.75pt,color=black] (0.0, 0.0, 0.5)--(0.0, 0.0, 5.5);
		
		\draw[line width=1.2pt,color=dorange] (6, 0, 1)--(6, 0, 5); 
		\draw[line width=1.2pt,color=dorange] (6, 1, 0)--(6, 5, 0); 
		\draw[line width=1.2pt,color=dorange] (6, 6, 1)--(6, 6, 5); 
		\draw[line width=1.2pt,color=dorange] (6, 1, 6)--(6, 5, 6);
		
		\begin{scope}[canvas is yz plane at x=6]
		  \draw[->,color=dorange] (4,3) arc (0:320:1);
		\end{scope}
		
		\node[text=dorange] at (7,1,0) {$-1$};
		\node[text=dorange] at (7,5,0) {$+1$};

		\node[text=dorange] at (5,1.5,6) {$+1$};
		\node[text=dorange] at (5,4.5,6) {$-1$};

		\node[text=dorange] at (5,0,1) {$+1$};
		\node[text=dorange] at (5,0,5) {$-1$};

		\node[text=dorange] at (7,6,1) {$-1$};
		\node[text=dorange] at (7,6,5) {$+1$};
		
  \end{tikzpicture}
  \caption{\label{fig:graphCycles}%
	  Illustration of the proof Theorem~\ref{thm:Zcharact}.
		\emph{Left:} minimal spanning tree, $\circ$ nodes with just one edge attached.
		\emph{Middle:} black edges: minimal spanning tree, colored: sequence of remaining edges with associated cycles.
		\emph{Right:} Values of $\widehat \mu_5$ associated with the cycle of $e_5$.
	}
\end{center}
\end{figure}
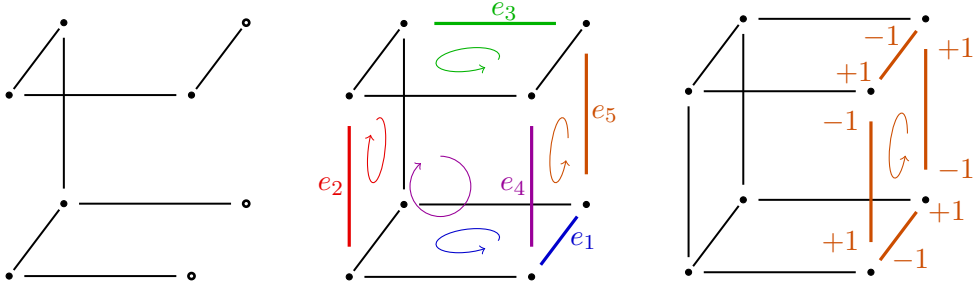

\section{Invertibility of generalized Robin problems}
\label{apx:invertibility}

In this section, we investigate the invertibility of the augmented operator $A + \alpha T^\top M T$,
cf.\ Assumption~\ref{ass:A5}.
As for classical wave propagation, the two building blocks are the \emph{Fredholm property} and the \emph{injectivity},
which altogether ensure a bounded inverse, see e.g.\ \cite{Monk:2003a,Hiptmair:2015a}.
In discrete case, invertibility is usually proved either via injectivity alone
or via an inf-sup condition derived from the continuous counterpart, cf.\ e.g.\ \cite{Hiptmair:2015a}.

We begin with the injectivity (Sect.~\ref{apx:injectivity}),
visit some general tools on the Fredholm property (Sect.~\ref{apx:FredholmTools})
and apply these for standard as well as generalized Robin problems (Sect.~\ref{apx:invStdRobin}
and Sect.~\ref{apx:invGenRobin}).
As will be noted, for some constellations in the case of Maxwell's equations, the Fredholm property remains an open problem.

\subsection{Injectivity}
\label{apx:injectivity}

We start with the assumption that the operator $M$ is block-diagonal,
which allows to treat one subdomain at a time.

\medskip

\noindent%
\fbox{%
\begin{minipage}[c]{0.985\textwidth}
\begin{assumptc}
\label{ass:C1}
  The operator $M$ from \ref{ass:A3} has the block-diagonal form $M = \diag(M_i)_{i=1}^N$,
  where each operator $M_i \colon \Lambda_i \to \Lambda_i^*$ is real-valued, symmetric, non-negative, and definite,
  i.e., $\langle M_i \lambda_i, \overline\lambda_i \rangle = 0 \implies \lambda_i = 0$
  for all $\lambda_i \in \Lambda_i$.
\end{assumptc}
\end{minipage}}

\medskip

Note that Assumption~\ref{ass:A6} implies~\ref{ass:C1} but not vice versa.

\medskip

\noindent%
\fbox{%
\begin{minipage}[c]{0.985\textwidth}
\begin{assumptc}
\label{ass:C2}
  In accordance with Assumption~\ref{ass:A7}, the following holds:
  \begin{enumerate}
  \item[(i)] In the \emph{coercive case} (with $\alpha = 1$):
  \[
    \big[ A_i v_i = 0 \text{ and } T_i v_i = 0 \big]
  	\implies v_i = 0 \qquad \forall v_i \in U_i\,.
  \]
  \item[(ii)] In the \emph{wave propagation case} (with $\alpha = \complexi$):
  \[
    \big[ (A_{i,0} - A_{i,2})v_i = 0 \text{ and } A_{i,1} v_i = 0 \text{ and } T_i v_i = 0 \big]
		\implies v_i = 0 \qquad \forall v_i \in U_i\,.
  \]
  \end{enumerate}
\end{assumptc}
\end{minipage}}

\medskip

Assumption~\ref{ass:C2} can be seen as an abstract version of Holmgreen's theorem: $A_i v = 0$ (in case~(i))
implies that the Neumann trace on the interface is zero, $T_i v = 0$ means that the Dirichlet trace on the interface is zero.
Inside the subdomain, $v$ fulfills the homogeneous PDE, so $v$ must vanish entirely;
see also \cite[Thm.~1]{FarhatMacedoLesoinne:2000a} or \cite[Thm.~A.1]{FarhatMacedoLesoinneRouxMagoulesDeLaBourdonnaie:2000a}.
In the typical coercive cases, $A_i$ has a finite-dimensional kernel (constant functions, rigid body modes),
which is fixed by the Dirichlet condition.
For Maxwell's equations, Part~(ii) of~\ref{ass:C2} is widely known as the \emph{continuation principle},
cf.\ e.g.\ \cite[Sect.~4.6]{Monk:2003a}.

In the discrete case, the following proposition, which is essentially \cite[Lemma~5.1]{Claeys:2021Preprint},
allows to derive~\ref{ass:C2} from the assumptions on the original operator $\widehat A$.

\begin{proposition}
  Let Assumption~\ref{ass:A7} hold, assume that $\ker(\widehat A) = \{ 0 \}$ (as stated in~Sect.~\ref{sect:globalProblem}),
	and let, in addition, Assumption~\ref{ass:B1} be fulfilled.
	Then Assumption~\ref{ass:C2} holds al well.
\end{proposition}
\begin{proof}
	Assume without loss of generality, that we are in Case~(ii) of Assumption~\ref{ass:A7} (the proof of Case~(i) is analogous)
	and that $(A_{i,0} - A_{i,2}) v_i = 0$ and $A_{i,1} v_i = 0$ and $T_i v_i = 0$.
  Assumption~\ref{ass:B1} guarantees that $\ker(T_i) \subseteq U_{i,B}$, and so for $T_i v_i = 0$ there exists a function $\widehat v \in \widehat U$
	such that $v_i = R_i \widehat v$ and $R_j \widehat v = 0$ for all $j \neq i$.
	From this and our initial assumptions, we can conclude that $\widehat A \widehat v = 0$, which implies $\widehat v = 0$ and therefore $v_i = 0$.
\end{proof}

\begin{proposition}[injectivity]
  Let Assumptions~\ref{ass:A7}, \ref{ass:C1}, and \ref{ass:C2} hold. Then the operator $\widetilde A_i$ is injective.
\end{proposition}
\begin{proof}
  For arbitrary but fixed $v \in U_i$ with $\widetilde A_i v = 0$, we show that $v = 0$.\\
	\emph{Coercive case} ($\alpha = 1$):
	$\widetilde A_i = A_i + T_i^\top M_i T_i$ with $A_i$ real-valued, symmetric, and non-negative.
  Due to our assumptions,
	\[
	  \langle A_i v, \overline v \rangle + \langle M_i T_i v, T_i \overline v \rangle = 0,
	\]
	and since both terms are non-negative, both must vanish. From the assumptions on $A_i$ and $M_i$,
	this implies $A_i v = 0$ and $T_i = 0$. Assumption~\ref{ass:C2}(i) guarantees that $v = 0$.\\
	\emph{Wave propagation case} ($\alpha = \complexi$):
	\[
	  \widetilde A_i = A_{i,0} + \complexi \widetilde A_{i,1} - A_{i,2}\,, \qquad \text{where }
		\widetilde A_{i,1} = A_{i,1} + T_i^\top M_i T_i\,.
	\]
	Due to our assumptions,
	\[
	  \langle (A_{i,0} - A_{i,2}) v, \overline v \rangle
		  + \complexi \Big( \underbrace{ \langle A_{i,1} v, \overline v \rangle }_{ \ge 0 }
		                  + \underbrace{ \langle M_i T_i v, T_i \overline v \rangle }_{ \ge 0 } \Big) = 0.
	\]
	Both the real and imaginary part must vanish. Since all the summands in the imaginary part are non-negative,
	\[
	  \langle A_{i,1} v, \overline v \rangle = 0 \qquad \text{and} \qquad \langle M_i T_i v, T_i v \rangle = 0.
	\]
	From the assumptions on $A_{i,1}$ and on $M_i$, it follows that $A_{i,1} v = 0$ and $T_i v = 0$.
	Recalling that $\widetilde A_i v = 0$, this implies also that $(A_{i,0} - A_{i,2})v = 0$.
	Assumption~\ref{ass:C2}(ii) guarantees that $v = 0$.
\end{proof}

\subsection{Technical Tools for Fredholm Operators}
\label{apx:FredholmTools}

\begin{lemma}
\label{lem:generalGarding}
  Let $V$ be a real or complexified Hilbert space and $A \colon V \to V^*$ 
	a bounded linear operator that fulfills a generalized G\r{a}rding inequality
	with respect to an isomorphism $\mathsf{F} \colon V \to V$
	and a compact bounded linear operator $C \colon V \to V^*$, i.e.,
	there exists a constant $\gamma > 0$ such that
	\begin{align}
	\label{eq:generalizedGarding}
	  \big| \langle A v, \mathsf{F} \overline v \rangle + \langle C v, \overline v \rangle \big|
		\ge \gamma \| v \|_V^2
		\qquad \forall v \in V,
	\end{align}
  Then $A$ is Fredholm with index zero.
\end{lemma}
\begin{proof}
  We find that the (possibly complex-valued) operator $\mathsf{F}^\top A + C$ is positive bounded from below in the sense that
	\[
	  \big| \langle (\mathsf{F}^\top A + C) v, \overline v \rangle \big|
		\ge \gamma \| v \|_V^2 \qquad \forall v \in V.
	\]
	Since $\mathsf{F}^\top A + C$ is obviously bounded, a suitable version of the Lax-Milgram lemma
	(see e.g.\ \cite[Lemma~2.21]{Monk:2003a} implies that $\mathsf{F}^\top A + C$
	is an isomorphism. In particular, $\mathsf{F}^\top A + C$ is Fredholm with index zero.
	Since $C$ is compact, a standard result (see e.g.\ \cite[Thm~2.26]{McLean:Book})
	implies that $\mathsf{F}^\top A$ is Fredholm with index zero.
	Since $\mathsf{F}^\top$ is an isomorphism,
	another standard argument (see e.g.\ \cite[Thm.~2.21]{McLean:Book}) yields that
	$A$ itself is Fredholm with index zero.
\end{proof}

\begin{lemma}
\label{lem:abstractWaveFredholm}
  Let $V$ be a complexified Hilbert space
	and $A \colon V \to V^*$ a bounded linear operator of the form
	\[
	  A = A_0 + \complexi \sigma A_1 - A_2\,,
	\]
	with $\sigma \in \{ +1, -1 \}$
	and with linear, bounded, real-valued, symmetric, and non-negative operators $A_i \colon V \to V^*$.
	Moreover, assume real-valued projection operators $\mathsf{N}$ and $\mathsf{R} \colon V \to V$
	with $\mathsf{N} + \mathsf{R} = I$ such that
	\begin{enumerate}
	\item[(i)] $A_0 \mathsf{N} = 0$,
	\item[(ii)] $A_2 \mathsf{R}$ is compact,
	\item[(iii)] either (a) $A_1 \mathsf{N}$ is compact or (b) $A_1 \mathsf{R}$ is compact, and
	\item[(iv)] there exists a constant $c > 0$ such that
	  $\langle (A_0 + A_1 + A_2) v, \overline v \rangle \ge c\, \| v \|_V^2$ for all $v \in V$.
	\end{enumerate}
	Then $A$ is Fredholm with index zero.
	The same holds if (iii) and (iv) are replaced by the alternative conditions
	\begin{enumerate}
	\item[(iii')] $\mathsf{N}^\top A_1 \mathsf{R}$ is compact, and
	\item[(iv')] there exists a constant $c > 0$ such that
	  $\langle (A_0 + A_2) v, \overline v \rangle \ge c\, \| v \|_V^2$ for all $v \in V$.
	\end{enumerate}
\end{lemma}
\begin{remark}
  The case $\mathsf{N} = 0$ is admitted.
	In that case, only (ii) and (iv) are required.
\end{remark}
\begin{proof}[Proof of Lemma~\ref{lem:abstractWaveFredholm}]
	We define $\mathsf{F} := \mathsf{R} - \mathsf{N}$, which is an isomorphism:
	\[
	  \mathsf{F}^2 = (2\mathsf{R} - I)^2 = 4\mathsf{R}^2 - 4\mathsf{R} + I = I.
	\]
	Using property~(i)
	and the relations $\mathsf{F} = 2 \mathsf{R} - I = I - 2 \mathsf{N}$,
	we find that
	\begin{align*}
	  \langle A v, \mathsf{F} w \rangle
		& = \langle A_0 v, (\mathsf{R} - \mathsf{N}) w \rangle
		  - \langle A_2 v, (\mathsf{R} - \mathsf{N}) w \rangle
		  + \complexi \sigma \langle A_1 v, (\mathsf{R} - \mathsf{N}) w \rangle\\
		& = \langle A_0 v, (\mathsf{R} + \mathsf{N}) w \rangle
		  + \langle A_2 v, (I - 2 \mathsf{R}) w \rangle
		  + \complexi \sigma \langle A_1 v, (2 \mathsf{R} - I) w \rangle\\
		& = \langle (A_0 + A_2) v, w \rangle
		  - \complexi \sigma \langle A_1 v, w \rangle
		  - 2 \langle A_2 v, \mathsf{R} w \rangle
			+ 2 \complexi \sigma \langle A_1 v, \mathsf{R} w \rangle,
	\end{align*}
	which will be used for Case~(b).
	Alternatively, we have
	\begin{align*}
	  \langle A v, \mathsf{F} w \rangle
		& = \langle A_0 v, (\mathsf{R} + \mathsf{N}) w \rangle
		  + \langle A_2 v, (I - 2 \mathsf{R}) w \rangle
		  + \complexi \sigma \langle A_1 v, (I - 2 \mathsf{N}) w \rangle\\
		& = \langle (A_0 + A_2) v, w \rangle
		  + \complexi \sigma \langle A_1 v, w \rangle
		  - 2 \langle A_2 v, \mathsf{R} w \rangle
			- 2 \complexi \sigma \langle A_1 v, \mathsf{N} w \rangle,
	\end{align*}
	which will be used for Case~(a).
	We define $C \colon V \to V^*$ by
	\[
	  \langle C v, w \rangle := 
		  \begin{cases}
				2 \langle A_2 v, \mathsf{R} w \rangle + 2 \complexi \sigma \langle A_1 v, \mathsf{N} w \rangle & \text{in Case~(a),}\\
  		  2 \langle A_2 v, \mathsf{R} w \rangle - 2 \complexi \sigma \langle A_1 v, \mathsf{R} w \rangle & \text{in Case~(b),}
			\end{cases}
	\]
	which is a compact operator by property~(ii) and (iii).
	Then,
	\begin{align*}
	  \big| \langle (A + C) v, \mathsf{F} \overline v \rangle \big|
		= \big| \underbrace{\langle (A_0 + A_2) v, \overline v \rangle}_{\in \mathbb{R}_0^+}
		    {} + \complexi \underbrace{\delta \sigma}_{= \pm 1}
				  \underbrace{\langle A_1 v, \overline v \rangle}_{\in \mathbb{R}_0^+} \big|
		= \langle (A_0 + A_1 + A_2) v, \overline v \rangle
		    \ge c\, \| v \|_V^2\,,
	\end{align*}
	with $\delta = 1$ in Case~(a) and $\delta = -1$ is Case~(b).
	An application of Lemma~\ref{lem:generalGarding} shows that $\widetilde A$ is Fredholm.
	Under the conditions (iii') and (iv'), we can use
	\begin{align*}
	  & \langle A v, \mathsf{F} w \rangle
		= \langle A_0 v, (\mathsf{R} + \mathsf{N}) w \rangle + \langle A_2 v, (I - 2\mathsf{R}) w \rangle
		  + \complexi\sigma \langle A_1 (\mathsf{R} + \mathsf{N}) v, (\mathsf{R} - \mathsf{N}) w \rangle\\
		& = \langle (A_0 + A_2) v , w \rangle
		  - 2 \langle A_2 v, \mathsf{R} w \rangle
		  + \complexi\sigma \big( \langle A_1 \mathsf{R} v, \mathsf{R} w \rangle - \langle A_1 \mathsf{N} v, \mathsf{N} w \rangle
			  - \langle A_1 \mathsf{R} v, \mathsf{N} w \rangle + \langle A_1 \mathsf{N} v, \mathsf{R} w \rangle \big).
	\end{align*}
	We define $C \colon V \to V^*$ by
	$\langle C v, w \rangle := 2 \langle A_2 v, \mathsf{R} w \rangle
	  + \complexi\sigma \langle A_1 \mathsf{R} v, \mathsf{N} w \rangle
		- \complexi\sigma \langle A_1 \mathsf{N} v, \mathsf{R} w \rangle$,
	which is a compact operator by properties~(ii) and~(iii').
	Due to (iv')
	\begin{align*}
	  \big| \langle (A + C) v, \mathsf{F} \overline v \rangle \big|
		= \big| \underbrace{ \langle (A_0 + A_2) v, \overline v \rangle }_{\in \mathbb{R}_0^+} + \complexi \sigma \!
		  \underbrace{ \big( \langle A_1 \mathsf{R} v, \mathsf{R} \overline v \rangle \! - \! \langle A_1 \mathsf{N} v, \mathsf{N} \overline v \rangle \big)}_{ \in \mathbb{R} }
			\big|
		\ge \langle (A_0 + A_2) v, \overline v \rangle \ge c \| v \|_V^2\,,
	\end{align*}
	and so again Lemma~\ref{lem:generalGarding} implies that $\widetilde A$ is Fredholm.
\end{proof}

\subsection{Standard Robin problems}
\label{apx:invStdRobin}

For the following, let us assume that
$\Omega \subset \mathbb{R}^3$ is a bounded Lipschitz domain and $\Gamma_D$, $\Gamma_N$, $\Gamma_R \subset \partial\Omega$
disjoint surfaces such that $\partial\Omega = \overline{\Gamma_D \cup \Gamma_N \cup \Gamma_R}$
and such that $\Gamma_R$ has non-vanishing surface measure.
Note, however, that $\Gamma_D$ and/or $\Gamma_N$ are allowed to be empty.
Moreover, any of the sets $\partial\Gamma_D$, $\partial\Gamma_N$, $\partial\Gamma_R$ (unless empty)
should fulfill the requirements stated in \cite[Sect.~2]{HiptmairPechstein:2020a},
in particular being the union of closed curves that are piecewise $C^1$.
Later on, it is further assumed that $\Omega$ is a curvilinear Lipschitz polyhedron.

\subsubsection{The primal Helmholtz equation}

Let $\widehat U := H^1_D(\Omega) = \{ v \in H^1(\Omega) \colon v = 0 \text{ on } \Gamma_D \}$
and let $\widehat A \colon \widehat U \to \widehat U^*$ be given by
$\widehat A = \widehat A_0 + \complexi \widehat A_1 - \widehat A_2$ with
\begin{align*}
  \langle \widehat A_0 \widehat u, \widehat v \rangle = \int_\Omega \nabla \widehat u \cdot \nabla \widehat v \, dx, \qquad
	\langle \widehat A_1 \widehat u, \widehat v \rangle = \int_{\Gamma_R} \eta\, \widehat u\, \widehat v \, ds, \qquad
	\langle \widehat A_2 \widehat u, \widehat v \rangle = \int_\Omega \kappa^2 \widehat u\, \widehat v \, dx.
\end{align*}
We set $\widehat{\mathsf{N}} = 0$ and $\widehat{\mathsf{R}} = I$.
Since $\widehat A_2$ is compact and since
\[
  \langle (\widehat A_0 + \widehat A_1 + \widehat A_2) \widehat v, \overline{\widehat v} \rangle
	\ge \min(1, \kappa_{\min}^2) \| \widehat v \|_{H^1(\Omega)}^2\,,
\]
where $\kappa_{\min}$ is a positive lower bound for the coefficient $\kappa$ in $\Omega$,
Lemma~\ref{lem:abstractWaveFredholm} guarantees that $\widehat A$ is Fredholm.

\subsubsection{The dual Helmholtz equation}

Let $\widehat U := \{ \vv \in \Hv_{\Gamma_N}(\opdiv, \Omega) \colon \vv_{n|\Gamma_R} \in L^2(\Gamma_R) \}$,
where $\vv_n$ denotes the normal trace of $\vv$ in $H^{-1/2}(\partial\Omega)$.
The restriction $\vv_{n|\Gamma_R}$ is well-defined in $H^{-1/2}_{00}(\Gamma_R)$.
We use the norm
$\| \vv \|_{\widehat U}^2 = \| \vv \|_{\Lv^2(\Omega)}^2 + \| \opdiv \vv \|_{L^2(\Omega)}^2 + \| \vv_{n|\Gamma_R} \|_{L^2(\Gamma_R)}^2$.
The operator $\widehat A \colon \widehat U \to \widehat U^*$ is given by
$\widehat A = \widehat A_0 + \complexi \widehat A_1 - \widehat A_2$ with
\begin{align*}
  \langle \widehat A_0 \widehat \uv, \widehat \vv \rangle = \int_\Omega \kappa^{-2} \opdiv \widehat\uv\, \opdiv\widehat\vv \, dx, \qquad
	\langle \widehat A_1 \widehat \uv, \widehat \vv \rangle = \int_{\Gamma_R} \eta^{-1}\, \widehat \uv_n \widehat \vv_n \, ds, \qquad
	\langle \widehat A_2 \widehat \uv, \widehat \vv \rangle = \int_\Omega \widehat \uv \cdot \widehat \vv \, dx.
\end{align*}
Due to the \emph{regular decomposition} result in \cite{HiptmairPechstein:2020a},
there exist bounded, linear, and real-valued projections
$\underline{\mathsf{N}} \colon \Hv_{\Gamma_N}(\opdiv, \Omega) \to \Hv_{\Gamma_N}(\opdiv 0, \Omega)$
and
$\underline{\mathsf{R}} \colon \Hv_{\Gamma_N}(\opdiv, \Omega) \to \Hv^1_{\Gamma_N}(\Omega)$
with $\underline{\mathsf{N}} + \underline{\mathsf{R}} = I$ in $\Hv_{\Gamma_N}(\opdiv, \Omega)$.
For $\widehat\vv \in \widehat U$,
\[
  \widehat \vv_{n|\Gamma_R}
	  = (\underline{\mathsf{R}} \widehat\vv )_{n|\Gamma_R} + (\underline{\mathsf{N}} \widehat\vv)_{n|\Gamma_R} \in L^2(\Gamma_R).
\]
Assuming that $\Omega$ is a curvilinear polyhedron (cf.\ \cite{BuffaCiarlet:2001a}), the outer normal $\nv$ is piecewise smooth.
Therefore, since $\underline{\mathsf{R}} \widehat\vv \in \Hv^1_{\Gamma_N}(\Omega)$, we find that
$(\underline{\mathsf{R}} \widehat\vv)_{n|\Gamma_R} = (\underline{\mathsf{R}}\widehat\vv)_{|\Gamma_R} \cdot \nv
   \in H^{1/2}_{\text{pw}}(\Gamma_R) \subset L^2(\Gamma_R)$.
This shows that $(\underline{\mathsf{N}} \widehat\vv)_{n|\Gamma_R} \in L^2(\Gamma_R)$ as well.
Hence, we can restrict $\underline{\mathsf{R}}$, $\underline{\mathsf{N}}$
to operators $\widehat{\mathsf{R}}$, $\widehat{\mathsf{N}} \colon \widehat U \to \widehat U$,
and we meet the prerequisites of Lemma~\ref{lem:abstractWaveFredholm}:
\begin{itemize}
\item $\widehat{\mathsf{R}}$, $\widehat{\mathsf{N}}$ are projectors and
      $\widehat{\mathsf{R}} + \widehat{\mathsf{N}} = I$ in $\widehat U$,
\item $\widehat A_0 \widehat{\mathsf{N}} = 0$ since $\underline{\mathsf{N}}$ maps to $\Hv_{\Gamma_N}(\opdiv 0, \Omega)$,
\item $\widehat A_2 \widehat{\mathsf{R}}$ is compact since $\underline{\mathsf{R}}$ maps to $\Hv^1_{\Gamma_N}(\Omega)$
  which is compactly embedded in $\Lv^2(\Omega)$,
\item $\widehat A_1 \widehat{\mathsf{R}}$ is compact since
  $(\underline{\mathsf{R}} \widehat\vv)_{n|\Gamma_R} \in H^{1/2}_{\text{pw}}(\Gamma_R)$
	which is compactly embedded in $L^2(\Gamma_R)$,
\item
  $\langle (\widehat A_0 + \widehat A_1 + \widehat A_2) \widehat\vv, \overline{\widehat\vv} \rangle
    \ge \min(1, \kappa_{\max}^{-2}, \eta_{\max}^{-1})
	    \big( \| \opdiv \widehat\vv \|_{L^2(\Omega)}^2
			    + \| \widehat \vv \|_{\Lv^2(\Omega)}^2
					+ \| \widehat\vv_{n|\Gamma_R} \|_{L^2(\Gamma_R)}^2 \big)$,
\end{itemize}
where $\kappa_{\max}$, $\eta_{\max}$ are finite upper bounds for the coefficients $\kappa$, $\eta$ in $\Omega$.

\subsubsection{Maxwell's equations}

To avoid complications, it is again assumed that $\Omega$ is a curvilinear Lipschitz polyhedron (cf.\ \cite{BuffaCiarlet:2001a}).
Let $\widehat U := \{ \vv \in \Hv_{\Gamma_D}(\opcurl, \Omega) \colon \vv_{\tau|\Gamma_R} \in L^2(\Gamma_R) \}$,
where $\vv_\tau = \vv \times \nv$ denotes the tangential trace of $\vv$ in $\Hv^{-1/2}(\partial\Omega)$
and the restriction $\vv_{\tau|\Gamma_R}$ is well-defined,
for details see \cite{BuffaCiarlet:2001a,Monk:2003a}.
We use the norm
$\| \vv \|_{\widehat U}^2 = \| \vv \|_{\Lv^2(\Omega)}^2 + \| \opcurl \vv \|_{\Lv^2(\Omega)}^2
   + \| \vv_{\tau|\Gamma_R} \|_{\Lv^2(\Gamma_R)}^2$.
The operator $\widehat A \colon \widehat U \to \widehat U^*$ is given by
$\widehat A = \widehat A_0 + \complexi \widehat A_1 - \widehat A_2$ with
\begin{align*}
  \langle \widehat A_0 \widehat \uv, \widehat \vv \rangle = \! \int_\Omega  \!\! \mu^{-1} \opcurl\widehat\uv \cdot \opcurl\widehat\vv \, dx, \quad
	\langle \widehat A_1 \widehat \uv, \widehat \vv \rangle = \! \int_{\Gamma_R} \!\!\! \omega \eta^{-1}\, \widehat \uv_\tau \cdot \widehat \vv_\tau \, ds, \quad
	\langle \widehat A_2 \widehat \uv, \widehat \vv \rangle = \! \int_\Omega \!\! \omega^2 \varepsilon \widehat \uv \cdot \widehat \vv \, dx.
\end{align*}
There exist bounded, linear, and real-valued projectors
$\underline{\mathsf{N}} \colon \Hv_{\Gamma_D}(\opcurl, \Omega) \to \Hv_{\Gamma_D}(\opcurl 0, \Omega)$
and
$\underline{\mathsf{R}} \colon \Hv_{\Gamma_D}(\opcurl, \Omega) \to \Hv^1_{\Gamma_D}(\Omega)$
with $\underline{\mathsf{R}} + \underline{\mathsf{N}} = I$ in $\Hv_{\Gamma_D}(\opcurl, \Omega)$,
see \cite{HiptmairPechstein:2020a}.
For $\widehat\vv \in \widehat U$,
\[
  \widehat \vv_{\tau|\Gamma_R}
	  = (\underline{\mathsf{R}} \widehat\vv )_{\tau|\Gamma_R} + (\underline{\mathsf{N}} \widehat\vv)_{\tau|\Gamma_R} \in \Lv^2(\Gamma_R).
\]
Due to the assumptions on $\Omega$, the normal $\nv$ is piecewise smooth.
Therefore, since $\underline{\mathsf{R}} \widehat\vv \in \Hv^1_{\Gamma_D}(\Omega)$, we find that
$(\underline{\mathsf{R}} \widehat\vv)_{\tau|\Gamma_R} = (\underline{\mathsf{R}}\widehat\vv)_{|\Gamma_R} \times \nv
   \in \Hv^{1/2}_{\text{pw}}(\Gamma_R) \subset \Lv^2(\Gamma_R)$.
This shows that $(\underline{\mathsf{N}} \widehat\vv)_{\tau|\Gamma_R} \in \Lv^2(\Gamma_R)$ as well.
Hence, we can restrict $\underline{\mathsf{N}}$, $\underline{\mathsf{R}}$
to operators $\widehat{\mathsf{N}}$, $\widehat{\mathsf{R}} \colon \widehat U \to \widehat U$,
and we meet the prerequisites of Lemma~\ref{lem:abstractWaveFredholm}:
\begin{itemize}
\item $\widehat{\mathsf{N}}$, $\widehat{\mathsf{R}}$ are projectors and
      $\widehat{\mathsf{N}} + \widehat{\mathsf{R}} = I$ in $\widehat U$,
\item $\widehat A_0 \widehat{\mathsf{N}} = 0$ since $\underline{\mathsf{N}}$ maps to $\Hv_{\Gamma_D}(\opcurl 0, \Omega)$,
\item $\widehat A_2 \widehat{\mathsf{R}}$ is compact since $\underline{\mathsf{R}}$ maps to $\Hv^1_{\Gamma_D}(\Omega)$
  which is compactly embedded in $\Lv^2(\Omega)$,
\item $\widehat A_1 \widehat{\mathsf{R}}$ is compact since
  $(\underline{\mathsf{R}} \widehat\vv)_{\tau|\Gamma_R} \in \Hv^{1/2}_{\text{pw}}(\Gamma_R)$
	which is compactly embedded in $\Lv^2(\Gamma_R)$,
\item
  $\langle (\widehat A_0 + \widehat A_1 + \widehat A_2) \widehat\vv, \overline{\widehat\vv} \rangle
    \ge \min(\mu_{\max}^{-1}, \omega^2 \varepsilon_{\min}^2, \omega \eta_{\min})
	    \big( \| \opcurl \widehat\vv \|_{L^2(\Omega)}^2
			    + \| \widehat \vv \|_{\Lv^2(\Omega)}^2
					+ \| \widehat\vv_{\tau|\Gamma_R} \|_{\Lv^2(\Gamma_R)}^2 \big)$,
\end{itemize}
where $\mu_{\max}$, $\varepsilon_{\min}$, $\eta_{\min}$ are finite upper/positive lower bounds for the coefficients
$\mu$, $\varepsilon$, $\eta$ in $\Omega$.

\subsection{Generalized Robin problems}
\label{apx:invGenRobin}

In this section, we investigate whether the subdomain operator
\[
  \widetilde A_i := A_i + \alpha T_i^\top M_i T_i
\]
is Fredholm with index zero.

\subsubsection{Wave propagation case}

In accordance with Assumption~\ref{ass:A7} (with $\alpha = \complexi$),
we \emph{assume} that there exist projectors $\mathsf{N}_i$, $\mathsf{R}_i$ such that
\begin{enumerate}
\item[(i)] $A_{i,0} \mathsf{N}_i = 0$,
\item[(ii)] $A_{i,2} \mathsf{R}_i$ is compact,
\item[(iii)] either (a) $A_{i,1} \mathsf{N}_i$ is compact or (b) $A_{i,1} \mathsf{R}_i$ is compact, and
\item[(iv)] there exists a constant $c_i > 0$ such that
  $\langle (A_{i,0} + A_{i,1} + A_{i,2}) v, \overline v \rangle \ge c \| v \|_{U_i}^2$ for all $v \in U_i$.
\end{enumerate}
such that Lemma~\ref{lem:abstractWaveFredholm} implies that $A_i$ is Fredholm with index zero.
With the definitions
\[
  \widetilde A_i = A_{i,0} + \complexi \widetilde A_{i,1} - A_{i,2}\,, \qquad \text{where }
	\widetilde A_{i,1} = A_{i,1} + T_i^\top M_i T_i\,,
\]
Lemma~\ref{lem:abstractWaveFredholm} would imply that $\widetilde A_i$ is Fredholm with index zero as well
if, in addition to the above:
\begin{enumerate}
\item[(v)] either (a) $\widetilde A_{i,1} \mathsf{N}_i$ is compact or (b) $\widetilde A_{i,1} \mathsf{R}_i$ is compact, and
\item[(vi)] there exists a constant $c_i > 0$ such that
  $\langle (A_{i,0} + \widetilde A_{i,1} + A_{i,2}) v, \overline v \rangle \ge c \| v \|_{U_i}^2$ for all $v \in U_i$.
\end{enumerate}
The inequality (vi) follows from (iv) because $M_i$ is non-negative. We are left with the question
whether the operator $T_i^\top M_i T_i \mathsf{N}_i$ (in case~(a)) or $T_i^\top M_i T_i \mathsf{R}_i$ (in case~(b)) is compact:
\begin{itemize}
\item For the primal Helmholtz formulation, we can use $\mathsf{N}_i = 0$.
\item If the trace operator $T_i$ itself is compact (cf.\ Theorem~\ref{thm:convGeneral}), we are also done.
\item For the dual Helmholtz formulation, (iii) holds with case~(b), $\range(T_i \mathsf{R}_i) \subset H^{1/2}_{\text{pw}}(\Gamma_i)$,
  where $\Gamma_i$ is the chosen interface, possibly split into facets.
	The latter space is compactly embedded in $L^2(\Gamma_i)$,
	and thus also compactly embedded in any chosen trace space $\Lambda_i$ (which requires at most $H^{-1/2}$-regularity).
	Therefore, $T_i \mathsf{R}_i \colon U_i \to \Lambda_i$ is compact, and so (v) holds with case~(b).
\end{itemize}

The Maxwell case is to a large extent open, at least if $T_i$ is not compact.
However, one exceptional situation shall be mentioned:
If $A_{i,1} = 0$ and $M_i$ is \emph{orthogonal} with respect to the regular decomposition, i.e.,
$\mathsf{N}_i^\top T_i^\top M_i T_i \mathsf{R}_i = 0$,
then using the alternative conditions (iii'), (iv') in Lemma~\ref{lem:abstractWaveFredholm},
one can show that $\widetilde A_i$ is Fredholm with index zero.
For more results see \cite[Sect.~3.4.2]{Parolin:PhD}.

\subsubsection{Coercive case}

In accordance with Assuption~\ref{ass:A7} (with $\alpha = 1$),
we \emph{assume} that there exists a compact operator\footnote{In the typical applications, the existence of such an operator
  is shown using Poincar\'e-, Friedrichs-, or Korn-type inequalities.}
$C_i$ such that $A_i + C_i$ is bounded positively from below,
such that $A_i$ is Fredholm with index zero.
Since, by assumption $M_i$ is non-negative,
\[
  \langle (\widetilde A_i + C_i) v, \overline v \rangle
	= \langle (A_i + C_i) v, \overline v \rangle + \langle M_i T_i v, T_i \overline v \rangle
	\ge \gamma_i \| v \|_{U_i}^2,
\]
for some constant $\gamma_i > 0$ such that $\widetilde A_i$ is Fredholm.

\end{appendix}

\bibliographystyle{abbrv}
\bibliography{DDFormulations}

\end{document}

%% file: transmCrosspoint.pdf_tex
\begingroup%
  \makeatletter%
  \providecommand\color[2][]{%
    \errmessage{(Inkscape) Color is used for the text in Inkscape, but the package 'color.sty' is not loaded}%
    \renewcommand\color[2][]{}%
  }%
  \providecommand\transparent[1]{%
    \errmessage{(Inkscape) Transparency is used (non-zero) for the text in Inkscape, but the package 'transparent.sty' is not loaded}%
    \renewcommand\transparent[1]{}%
  }%
  \providecommand\rotatebox[2]{#2}%
  \ifx\svgwidth\undefined%
    \setlength{\unitlength}{939.45424805bp}%
    \ifx\svgscale\undefined%
      \relax%
    \else%
      \setlength{\unitlength}{\unitlength * \real{\svgscale}}%
    \fi%
  \else%
    \setlength{\unitlength}{\svgwidth}%
  \fi%
  \global\let\svgwidth\undefined%
  \global\let\svgscale\undefined%
  \makeatother%
  \begin{picture}(1,0.24347351)%
    \put(0,0){\includegraphics[width=\unitlength]{transmCrosspoint.pdf}}%
    \put(0.04195214,0.19468029){\color[rgb]{0,0,0}\makebox(0,0)[lb]{\smash{1}}}%
    \put(0.14408637,0.19530473){\color[rgb]{0,0,0}\makebox(0,0)[lb]{\smash{2}}}%
    \put(0.09012623,0.09412947){\color[rgb]{0,0,0}\makebox(0,0)[lb]{\smash{3}}}%
    \put(0.29509853,0.19563152){\color[rgb]{0,0,0}\makebox(0,0)[lb]{\smash{1}}}%
    \put(0.40229151,0.19709909){\color[rgb]{0,0,0}\makebox(0,0)[lb]{\smash{2}}}%
    \put(0.34664512,0.09676697){\color[rgb]{0,0,0}\makebox(0,0)[lb]{\smash{3}}}%
    \put(0.57248726,0.1981609){\color[rgb]{0,0,0}\makebox(0,0)[lb]{\smash{1}}}%
    \put(0.67883713,0.1979422){\color[rgb]{0,0,0}\makebox(0,0)[lb]{\smash{2}}}%
    \put(0.62656323,0.09339446){\color[rgb]{0,0,0}\makebox(0,0)[lb]{\smash{3}}}%
    \put(0.83807228,0.19900402){\color[rgb]{0,0,0}\makebox(0,0)[lb]{\smash{1}}}%
    \put(0.94948088,0.19962848){\color[rgb]{0,0,0}\makebox(0,0)[lb]{\smash{2}}}%
    \put(0.8407175,0.09423761){\color[rgb]{0,0,0}\makebox(0,0)[lb]{\smash{3}}}%
    \put(0.9486655,0.09845323){\color[rgb]{0,0,0}\makebox(0,0)[lb]{\smash{4}}}%
    \put(-0.00020419,0.03279991){\color[rgb]{0,0,0}\makebox(0,0)[lb]{\smash{$t_1 = 0$}}}%
    \put(0.25754499,0.03458026){\color[rgb]{0,0,0}\makebox(0,0)[lb]{\smash{$u_1 = u_2$}}}%
    \put(0.25749015,0.00928645){\color[rgb]{0,0,0}\makebox(0,0)[lb]{\smash{$t_1 + t_2 = 0$}}}%
    \put(0.53452831,0.03289401){\color[rgb]{0,0,0}\makebox(0,0)[lb]{\smash{$u_1 = u_2 = u_3$}}}%
    \put(0.53447348,0.0076002){\color[rgb]{0,0,0}\makebox(0,0)[lb]{\smash{$t_1 + t_2 + t_3 = 0$}}}%
    \put(0.80011334,0.0286784){\color[rgb]{0,0,0}\makebox(0,0)[lb]{\smash{$u_1 = u_2 = u_3 = u_4$}}}%
    \put(0.80005845,0.00338461){\color[rgb]{0,0,0}\makebox(0,0)[lb]{\smash{$t_1 + t_2 + t_3 + t_4 = 0$}}}%
  \end{picture}%
\endgroup%

%% file: discreteFacetVariants_a.pdf_tex
\begingroup%
  \makeatletter%
  \providecommand\color[2][]{%
    \errmessage{(Inkscape) Color is used for the text in Inkscape, but the package 'color.sty' is not loaded}%
    \renewcommand\color[2][]{}%
  }%
  \providecommand\transparent[1]{%
    \errmessage{(Inkscape) Transparency is used (non-zero) for the text in Inkscape, but the package 'transparent.sty' is not loaded}%
    \renewcommand\transparent[1]{}%
  }%
  \providecommand\rotatebox[2]{#2}%
  \ifx\svgwidth\undefined%
    \setlength{\unitlength}{889.98076024bp}%
    \ifx\svgscale\undefined%
      \relax%
    \else%
      \setlength{\unitlength}{\unitlength * \real{\svgscale}}%
    \fi%
  \else%
    \setlength{\unitlength}{\svgwidth}%
  \fi%
  \global\let\svgwidth\undefined%
  \global\let\svgscale\undefined%
  \makeatother%
  \begin{picture}(1,0.24239242)%
    \put(0,0){\includegraphics[width=\unitlength,page=1]{discreteFacetVariants_a.pdf}}%
    \put(0.0745677,0.16883158){\color[rgb]{0,0,0}\makebox(0,0)[lb]{\smash{1}}}%
    \put(0.1857339,0.16883158){\color[rgb]{0,0,0}\makebox(0,0)[lb]{\smash{2}}}%
    \put(0.07498004,0.06059298){\color[rgb]{0,0,0}\makebox(0,0)[lb]{\smash{3}}}%
    \put(0.18820792,0.06071669){\color[rgb]{0,0,0}\makebox(0,0)[lb]{\smash{4}}}%
    \put(0.00289518,0.11572251){\color[rgb]{0.50196078,0,0}\makebox(0,0)[lb]{\smash{$F_{13}$}}}%
    \put(0.1199908,0.22509916){\color[rgb]{0,0,0.50196078}\makebox(0,0)[lb]{\smash{$F_{12}$}}}%
    \put(0.12171774,0.00634425){\color[rgb]{0.50196078,0.50196078,0}\makebox(0,0)[lb]{\smash{$F_{34}$}}}%
    \put(0.240248,0.11638593){\color[rgb]{0.75294118,0.25098039,0}\makebox(0,0)[lb]{\smash{$F_{24}$}}}%
    \put(0.00121699,0.05328167){\color[rgb]{0,0.50196078,0}\makebox(0,0)[lb]{\smash{$F_{23}$}}}%
    \put(0.23995575,0.05142189){\color[rgb]{0.50196078,0,0.50196078}\makebox(0,0)[lb]{\smash{$F_{14}$}}}%
    \put(0,0){\includegraphics[width=\unitlength,page=2]{discreteFacetVariants_a.pdf}}%
    \put(0.00112835,0.22623201){\color[rgb]{0,0,0}\makebox(0,0)[lb]{\smash{max}}}%
    \put(0,0){\includegraphics[width=\unitlength,page=3]{discreteFacetVariants_a.pdf}}%
    \put(0.39135334,0.16678865){\color[rgb]{0,0,0}\makebox(0,0)[lb]{\smash{1}}}%
    \put(0.50251953,0.16678865){\color[rgb]{0,0,0}\makebox(0,0)[lb]{\smash{2}}}%
    \put(0.39176568,0.05855007){\color[rgb]{0,0,0}\makebox(0,0)[lb]{\smash{3}}}%
    \put(0.50499358,0.05867377){\color[rgb]{0,0,0}\makebox(0,0)[lb]{\smash{4}}}%
    \put(0.31968083,0.11367959){\color[rgb]{0.50196078,0,0}\makebox(0,0)[lb]{\smash{$F_{13}$}}}%
    \put(0.43677644,0.22305624){\color[rgb]{0,0,0.50196078}\makebox(0,0)[lb]{\smash{$F_{12}$}}}%
    \put(0.43850338,0.00430136){\color[rgb]{0.50196078,0.50196078,0}\makebox(0,0)[lb]{\smash{$F_{34}$}}}%
    \put(0.55703363,0.114343){\color[rgb]{0.75294118,0.25098039,0}\makebox(0,0)[lb]{\smash{$F_{24}$}}}%
    \put(0.31791398,0.22418909){\color[rgb]{0,0,0}\makebox(0,0)[lb]{\smash{pc}}}%
    \put(0,0){\includegraphics[width=\unitlength,page=4]{discreteFacetVariants_a.pdf}}%
    \put(0.72292331,0.16883162){\color[rgb]{0,0,0}\makebox(0,0)[lb]{\smash{1}}}%
    \put(0.83408953,0.16883162){\color[rgb]{0,0,0}\makebox(0,0)[lb]{\smash{2}}}%
    \put(0.72333567,0.06059304){\color[rgb]{0,0,0}\makebox(0,0)[lb]{\smash{3}}}%
    \put(0.83656358,0.06071674){\color[rgb]{0,0,0}\makebox(0,0)[lb]{\smash{4}}}%
    \put(0.6512508,0.11572255){\color[rgb]{0.50196078,0,0}\makebox(0,0)[lb]{\smash{$F_{13}$}}}%
    \put(0.76834641,0.22509921){\color[rgb]{0,0,0.50196078}\makebox(0,0)[lb]{\smash{$F_{12}$}}}%
    \put(0.77007337,0.00634431){\color[rgb]{0.50196078,0.50196078,0}\makebox(0,0)[lb]{\smash{$F_{34}$}}}%
    \put(0.88860363,0.11638597){\color[rgb]{0.75294118,0.25098039,0}\makebox(0,0)[lb]{\smash{$F_{24}$}}}%
    \put(0.64948395,0.22623206){\color[rgb]{0,0,0}\makebox(0,0)[lb]{\smash{nr}}}%
  \end{picture}%
\endgroup%

%% file: discreteFacetVariants_b.pdf_tex
\begingroup%
  \makeatletter%
  \providecommand\color[2][]{%
    \errmessage{(Inkscape) Color is used for the text in Inkscape, but the package 'color.sty' is not loaded}%
    \renewcommand\color[2][]{}%
  }%
  \providecommand\transparent[1]{%
    \errmessage{(Inkscape) Transparency is used (non-zero) for the text in Inkscape, but the package 'transparent.sty' is not loaded}%
    \renewcommand\transparent[1]{}%
  }%
  \providecommand\rotatebox[2]{#2}%
  \ifx\svgwidth\undefined%
    \setlength{\unitlength}{778.26672363bp}%
    \ifx\svgscale\undefined%
      \relax%
    \else%
      \setlength{\unitlength}{\unitlength * \real{\svgscale}}%
    \fi%
  \else%
    \setlength{\unitlength}{\svgwidth}%
  \fi%
  \global\let\svgwidth\undefined%
  \global\let\svgscale\undefined%
  \makeatother%
  \begin{picture}(1,0.24355587)%
    \put(0,0){\includegraphics[width=\unitlength]{discreteFacetVariants_b.pdf}}%
    \put(0.08228702,0.15912225){\color[rgb]{0,0,0}\makebox(0,0)[lb]{\smash{1}}}%
    \put(0.20941024,0.15912225){\color[rgb]{0,0,0}\makebox(0,0)[lb]{\smash{2}}}%
    \put(0.08275854,0.03534687){\color[rgb]{0,0,0}\makebox(0,0)[lb]{\smash{3}}}%
    \put(0.00032648,0.0983898){\color[rgb]{0.50196078,0,0}\makebox(0,0)[lb]{\smash{$F_{13}$}}}%
    \put(0.13423024,0.2234666){\color[rgb]{0,0,0.50196078}\makebox(0,0)[lb]{\smash{$F_{12}$}}}%
    \put(0.1582554,0.08522635){\color[rgb]{0,0.50196078,0}\makebox(0,0)[lb]{\smash{$F_{23}$}}}%
    \put(-0.00169397,0.22476207){\color[rgb]{0,0,0}\makebox(0,0)[lb]{\smash{max}}}%
    \put(0.44454473,0.15678608){\color[rgb]{0,0,0}\makebox(0,0)[lb]{\smash{1}}}%
    \put(0.57166794,0.15678608){\color[rgb]{0,0,0}\makebox(0,0)[lb]{\smash{2}}}%
    \put(0.44501625,0.03301071){\color[rgb]{0,0,0}\makebox(0,0)[lb]{\smash{3}}}%
    \put(0.36258419,0.09605363){\color[rgb]{0.50196078,0,0}\makebox(0,0)[lb]{\smash{$F_{13}$}}}%
    \put(0.49648795,0.22113044){\color[rgb]{0,0,0.50196078}\makebox(0,0)[lb]{\smash{$F_{12}$}}}%
    \put(0.36056373,0.2224259){\color[rgb]{0,0,0}\makebox(0,0)[lb]{\smash{pc}}}%
    \put(0.82370894,0.1591223){\color[rgb]{0,0,0}\makebox(0,0)[lb]{\smash{1}}}%
    \put(0.95083218,0.1591223){\color[rgb]{0,0,0}\makebox(0,0)[lb]{\smash{2}}}%
    \put(0.82418049,0.03534693){\color[rgb]{0,0,0}\makebox(0,0)[lb]{\smash{3}}}%
    \put(0.7417484,0.09838985){\color[rgb]{0.50196078,0,0}\makebox(0,0)[lb]{\smash{$F_{13}$}}}%
    \put(0.87565216,0.22346666){\color[rgb]{0,0,0.50196078}\makebox(0,0)[lb]{\smash{$F_{12}$}}}%
    \put(0.73972794,0.22476212){\color[rgb]{0,0,0}\makebox(0,0)[lb]{\smash{nr}}}%
  \end{picture}%
\endgroup%

%% file: discreteFacetVariants_c.pdf_tex
\begingroup%
  \makeatletter%
  \providecommand\color[2][]{%
    \errmessage{(Inkscape) Color is used for the text in Inkscape, but the package 'color.sty' is not loaded}%
    \renewcommand\color[2][]{}%
  }%
  \providecommand\transparent[1]{%
    \errmessage{(Inkscape) Transparency is used (non-zero) for the text in Inkscape, but the package 'transparent.sty' is not loaded}%
    \renewcommand\transparent[1]{}%
  }%
  \providecommand\rotatebox[2]{#2}%
  \ifx\svgwidth\undefined%
    \setlength{\unitlength}{781.33583747bp}%
    \ifx\svgscale\undefined%
      \relax%
    \else%
      \setlength{\unitlength}{\unitlength * \real{\svgscale}}%
    \fi%
  \else%
    \setlength{\unitlength}{\svgwidth}%
  \fi%
  \global\let\svgwidth\undefined%
  \global\let\svgscale\undefined%
  \makeatother%
  \begin{picture}(1,0.27689705)%
    \put(0,0){\includegraphics[width=\unitlength]{discreteFacetVariants_c.pdf}}%
    \put(0.08493256,0.19279509){\color[rgb]{0,0,0}\makebox(0,0)[lb]{\smash{1}}}%
    \put(0.19992136,0.21024772){\color[rgb]{0,0,0}\makebox(0,0)[lb]{\smash{2}}}%
    \put(0.10401837,0.05787081){\color[rgb]{0,0,0}\makebox(0,0)[lb]{\smash{3}}}%
    \put(0.21437448,0.0696468){\color[rgb]{0,0,0}\makebox(0,0)[lb]{\smash{4}}}%
    \put(0.06845044,0.1323012){\color[rgb]{0.50196078,0,0}\makebox(0,0)[lb]{\smash{$F_{13}$}}}%
    \put(0.13667175,0.25688669){\color[rgb]{0,0,0.50196078}\makebox(0,0)[lb]{\smash{$F_{12}$}}}%
    \put(0.13863882,0.00771387){\color[rgb]{0.50196078,0.50196078,0}\makebox(0,0)[lb]{\smash{$F_{34}$}}}%
    \put(0.19685918,0.13422036){\color[rgb]{0.75294118,0.25098039,0}\makebox(0,0)[lb]{\smash{$F_{24}$}}}%
    \put(0.00138242,0.06117795){\color[rgb]{0,0.50196078,0}\makebox(0,0)[lb]{\smash{$F_{23}$}}}%
    \put(0.27331786,0.05905956){\color[rgb]{0.50196078,0,0.50196078}\makebox(0,0)[lb]{\smash{$F_{14}$}}}%
    \put(0.00128146,0.25817706){\color[rgb]{0,0,0}\makebox(0,0)[lb]{\smash{max}}}%
    \put(0.44576731,0.19046809){\color[rgb]{0,0,0}\makebox(0,0)[lb]{\smash{1}}}%
    \put(0.56075608,0.20792072){\color[rgb]{0,0,0}\makebox(0,0)[lb]{\smash{2}}}%
    \put(0.4625261,0.05787084){\color[rgb]{0,0,0}\makebox(0,0)[lb]{\smash{3}}}%
    \put(0.57520924,0.06731982){\color[rgb]{0,0,0}\makebox(0,0)[lb]{\smash{4}}}%
    \put(0.4975065,0.2545597){\color[rgb]{0,0,0.50196078}\makebox(0,0)[lb]{\smash{$F_{12}$}}}%
    \put(0.49947357,0.00538691){\color[rgb]{0.50196078,0.50196078,0}\makebox(0,0)[lb]{\smash{$F_{34}$}}}%
    \put(0.36211619,0.25585008){\color[rgb]{0,0,0}\makebox(0,0)[lb]{\smash{pc}}}%
    \put(0.82344215,0.19279513){\color[rgb]{0,0,0}\makebox(0,0)[lb]{\smash{1}}}%
    \put(0.94075801,0.21141127){\color[rgb]{0,0,0}\makebox(0,0)[lb]{\smash{2}}}%
    \put(0.8413645,0.05321684){\color[rgb]{0,0,0}\makebox(0,0)[lb]{\smash{3}}}%
    \put(0.95288412,0.06964686){\color[rgb]{0,0,0}\makebox(0,0)[lb]{\smash{4}}}%
    \put(0.80928708,0.13230124){\color[rgb]{0.50196078,0,0}\makebox(0,0)[lb]{\smash{$F_{13}$}}}%
    \put(0.87518134,0.25688675){\color[rgb]{0,0,0.50196078}\makebox(0,0)[lb]{\smash{$F_{12}$}}}%
    \put(0.87714844,0.00771393){\color[rgb]{0.50196078,0.50196078,0}\makebox(0,0)[lb]{\smash{$F_{34}$}}}%
    \put(0.73979103,0.25817712){\color[rgb]{0,0,0}\makebox(0,0)[lb]{\smash{nr}}}%
  \end{picture}%
\endgroup%

%% file: Globs_a.pdf_tex
\begingroup%
  \makeatletter%
  \providecommand\color[2][]{%
    \errmessage{(Inkscape) Color is used for the text in Inkscape, but the package 'color.sty' is not loaded}%
    \renewcommand\color[2][]{}%
  }%
  \providecommand\transparent[1]{%
    \errmessage{(Inkscape) Transparency is used (non-zero) for the text in Inkscape, but the package 'transparent.sty' is not loaded}%
    \renewcommand\transparent[1]{}%
  }%
  \providecommand\rotatebox[2]{#2}%
  \ifx\svgwidth\undefined%
    \setlength{\unitlength}{322.56008606bp}%
    \ifx\svgscale\undefined%
      \relax%
    \else%
      \setlength{\unitlength}{\unitlength * \real{\svgscale}}%
    \fi%
  \else%
    \setlength{\unitlength}{\svgwidth}%
  \fi%
  \global\let\svgwidth\undefined%
  \global\let\svgscale\undefined%
  \makeatother%
  \begin{picture}(1,0.66315212)%
    \put(0,0){\includegraphics[width=\unitlength,page=1]{Globs_a.pdf}}%
    \put(0.23551326,0.46018934){\color[rgb]{0,0,0}\makebox(0,0)[lb]{\smash{1}}}%
    \put(0.54223373,0.46018934){\color[rgb]{0,0,0}\makebox(0,0)[lb]{\smash{2}}}%
    \put(0.23665095,0.16154648){\color[rgb]{0,0,0}\makebox(0,0)[lb]{\smash{3}}}%
    \put(0.54905984,0.1618878){\color[rgb]{0,0,0}\makebox(0,0)[lb]{\smash{4}}}%
    \put(-0.00469268,0.32035838){\color[rgb]{0.50196078,0,0}\makebox(0,0)[lb]{\smash{$F_{13}$}}}%
    \put(0.36084087,0.61543815){\color[rgb]{0,0,0.50196078}\makebox(0,0)[lb]{\smash{$F_{12}$}}}%
    \put(0.3656057,0.01186794){\color[rgb]{0.50196078,0.50196078,0}\makebox(0,0)[lb]{\smash{$F_{34}$}}}%
    \put(0.69264447,0.31548571){\color[rgb]{0.75294118,0.25098039,0}\makebox(0,0)[lb]{\smash{$F_{24}$}}}%
    \put(0.35711475,0.31342074){\color[rgb]{0,0.50196078,0}\makebox(0,0)[lb]{\smash{$V$}}}%
  \end{picture}%
\endgroup%

%% file: Globs_b.pdf_tex
\begingroup%
  \makeatletter%
  \providecommand\color[2][]{%
    \errmessage{(Inkscape) Color is used for the text in Inkscape, but the package 'color.sty' is not loaded}%
    \renewcommand\color[2][]{}%
  }%
  \providecommand\transparent[1]{%
    \errmessage{(Inkscape) Transparency is used (non-zero) for the text in Inkscape, but the package 'transparent.sty' is not loaded}%
    \renewcommand\transparent[1]{}%
  }%
  \providecommand\rotatebox[2]{#2}%
  \ifx\svgwidth\undefined%
    \setlength{\unitlength}{204.47985535bp}%
    \ifx\svgscale\undefined%
      \relax%
    \else%
      \setlength{\unitlength}{\unitlength * \real{\svgscale}}%
    \fi%
  \else%
    \setlength{\unitlength}{\svgwidth}%
  \fi%
  \global\let\svgwidth\undefined%
  \global\let\svgscale\undefined%
  \makeatother%
  \begin{picture}(1,1.04134538)%
    \put(0,0){\includegraphics[width=\unitlength,page=1]{Globs_b.pdf}}%
    \put(0.31745693,0.72117843){\color[rgb]{0,0,0}\makebox(0,0)[lb]{\smash{1}}}%
    \put(0.80129813,0.72117843){\color[rgb]{0,0,0}\makebox(0,0)[lb]{\smash{2}}}%
    \put(0.3192516,0.25007937){\color[rgb]{0,0,0}\makebox(0,0)[lb]{\smash{3}}}%
    \put(-0.00740255,0.49783793){\color[rgb]{0.50196078,0,0}\makebox(0,0)[lb]{\smash{$F_{13}$}}}%
    \put(0.51515701,0.96607819){\color[rgb]{0,0,0.50196078}\makebox(0,0)[lb]{\smash{$F_{12}$}}}%
    \put(0.55445063,0.48338141){\color[rgb]{0,0.50196078,0}\makebox(0,0)[lb]{\smash{$V$}}}%
    \put(0.48477504,0.01872128){\color[rgb]{1,1,1}\makebox(0,0)[lb]{\smash{$F_{12}$}}}%
  \end{picture}%
\endgroup%

%% file: Globs_c.pdf_tex
\begingroup%
  \makeatletter%
  \providecommand\color[2][]{%
    \errmessage{(Inkscape) Color is used for the text in Inkscape, but the package 'color.sty' is not loaded}%
    \renewcommand\color[2][]{}%
  }%
  \providecommand\transparent[1]{%
    \errmessage{(Inkscape) Transparency is used (non-zero) for the text in Inkscape, but the package 'transparent.sty' is not loaded}%
    \renewcommand\transparent[1]{}%
  }%
  \providecommand\rotatebox[2]{#2}%
  \ifx\svgwidth\undefined%
    \setlength{\unitlength}{166.0689386bp}%
    \ifx\svgscale\undefined%
      \relax%
    \else%
      \setlength{\unitlength}{\unitlength * \real{\svgscale}}%
    \fi%
  \else%
    \setlength{\unitlength}{\svgwidth}%
  \fi%
  \global\let\svgwidth\undefined%
  \global\let\svgscale\undefined%
  \makeatother%
  \begin{picture}(1,1.28805789)%
    \put(0,0){\includegraphics[width=\unitlength,page=1]{Globs_c.pdf}}%
    \put(0.15033326,0.89383792){\color[rgb]{0,0,0}\makebox(0,0)[lb]{\smash{1}}}%
    \put(0.69134276,0.97595059){\color[rgb]{0,0,0}\makebox(0,0)[lb]{\smash{2}}}%
    \put(0.24012989,0.25903424){\color[rgb]{0,0,0}\makebox(0,0)[lb]{\smash{3}}}%
    \put(0.75934306,0.31443895){\color[rgb]{0,0,0}\makebox(0,0)[lb]{\smash{4}}}%
    \put(0.39376038,1.19538177){\color[rgb]{0,0,0.50196078}\makebox(0,0)[lb]{\smash{$F_{12}$}}}%
    \put(0.40301524,0.02305142){\color[rgb]{0.50196078,0.50196078,0}\makebox(0,0)[lb]{\smash{$F_{34}$}}}%
    \put(0.37582122,0.61399297){\color[rgb]{0,0.50196078,0}\makebox(0,0)[lb]{\smash{$V$}}}%
  \end{picture}%
\endgroup%

%% file: 1DDecomp.pdf_tex
\begingroup%
  \makeatletter%
  \providecommand\color[2][]{%
    \errmessage{(Inkscape) Color is used for the text in Inkscape, but the package 'color.sty' is not loaded}%
    \renewcommand\color[2][]{}%
  }%
  \providecommand\transparent[1]{%
    \errmessage{(Inkscape) Transparency is used (non-zero) for the text in Inkscape, but the package 'transparent.sty' is not loaded}%
    \renewcommand\transparent[1]{}%
  }%
  \providecommand\rotatebox[2]{#2}%
  \ifx\svgwidth\undefined%
    \setlength{\unitlength}{488.8bp}%
    \ifx\svgscale\undefined%
      \relax%
    \else%
      \setlength{\unitlength}{\unitlength * \real{\svgscale}}%
    \fi%
  \else%
    \setlength{\unitlength}{\svgwidth}%
  \fi%
  \global\let\svgwidth\undefined%
  \global\let\svgscale\undefined%
  \makeatother%
  \begin{picture}(1,0.1497169)%
    \put(0,0){\includegraphics[width=\unitlength,page=1]{1DDecomp.pdf}}%
    \put(0.16448445,0.06667514){\color[rgb]{0,0,0}\makebox(0,0)[lb]{\smash{1}}}%
    \put(0.47545008,0.06667514){\color[rgb]{0,0,0}\makebox(0,0)[lb]{\smash{2}}}%
    \put(0.78641571,0.06667514){\color[rgb]{0,0,0}\makebox(0,0)[lb]{\smash{3}}}%
    \put(0.32395403,0.06597572){\color[rgb]{0,0,0}\makebox(0,0)[lb]{\smash{$F_{12}$}}}%
    \put(0.63673813,0.06625549){\color[rgb]{0,0,0}\makebox(0,0)[lb]{\smash{$F_{23}$}}}%
  \end{picture}%
\endgroup%

%% file: 1DDecompDirichlet.pdf_tex
\begingroup%
  \makeatletter%
  \providecommand\color[2][]{%
    \errmessage{(Inkscape) Color is used for the text in Inkscape, but the package 'color.sty' is not loaded}%
    \renewcommand\color[2][]{}%
  }%
  \providecommand\transparent[1]{%
    \errmessage{(Inkscape) Transparency is used (non-zero) for the text in Inkscape, but the package 'transparent.sty' is not loaded}%
    \renewcommand\transparent[1]{}%
  }%
  \providecommand\rotatebox[2]{#2}%
  \ifx\svgwidth\undefined%
    \setlength{\unitlength}{344.8bp}%
    \ifx\svgscale\undefined%
      \relax%
    \else%
      \setlength{\unitlength}{\unitlength * \real{\svgscale}}%
    \fi%
  \else%
    \setlength{\unitlength}{\svgwidth}%
  \fi%
  \global\let\svgwidth\undefined%
  \global\let\svgscale\undefined%
  \makeatother%
  \begin{picture}(1,0.27707366)%
    \put(0,0){\includegraphics[width=\unitlength,page=1]{1DDecompDirichlet.pdf}}%
    \put(0.53301804,0.01022445){\color[rgb]{0,0,0}\makebox(0,0)[lb]{\smash{$\Gamma_D$}}}%
    \put(0,0){\includegraphics[width=\unitlength,page=2]{1DDecompDirichlet.pdf}}%
    \put(0.23317865,0.15935088){\color[rgb]{0,0,0}\makebox(0,0)[lb]{\smash{1}}}%
    \put(0.67401392,0.15935088){\color[rgb]{0,0,0}\makebox(0,0)[lb]{\smash{2}}}%
    \put(0.44734964,0.17025775){\color[rgb]{0,0,0}\makebox(0,0)[lb]{\smash{$F_{12}$}}}%
  \end{picture}%
\endgroup%